\documentclass[oneside]{amsart}
\usepackage{amsaddr}
\usepackage[right=4cm, top=3.5cm, bottom=3.5cm]{geometry}

\usepackage{float}
\restylefloat{figure}
\usepackage{cite}
\usepackage{mathtools}

\usepackage{comment}
\usepackage[utf8]{inputenc}

\usepackage{enumitem}

\usepackage{color}
\usepackage{amsmath} 
\usepackage{amssymb} 
\usepackage{amsthm}
\usepackage{geometry}
\usepackage{graphicx}
\usepackage{esint}
\usepackage[colorlinks=true,linkcolor=blue]{hyperref}
\usepackage{tikz,tikz-cd}
\usetikzlibrary{patterns} 

\usepackage{mathabx}

\DeclareSymbolFont{symbolsC}{U}{pxsyc}{m}{n}
\SetSymbolFont{symbolsC}{bold}{U}{pxsyc}{bx}{n}
\DeclareFontSubstitution{U}{pxsyc}{m}{n}
\DeclareMathSymbol{\medcirc}{\mathbin}{symbolsC}{7}

\theoremstyle{plain}
\newtheorem{thm}{Theorem}[section]
\newtheorem{prop}[thm]{Proposition}
\newtheorem{lem}[thm]{Lemma}
\newtheorem{cor}[thm]{Corollary}

\theoremstyle{definition}
\newtheorem{remark}[thm]{Remark}
\newtheorem{definition}[thm]{Definition}

\newcommand{\Res}{\operatorname{Res}}
\newcommand{\Wr}{\operatorname{Wr}}

\newcommand{\N}{\mathbb{N}}
\newcommand{\Nz}{\N_0}

\newcommand{\MO}[1]{\mathcal{M}(#1)}

\newcommand{\cM}{\mathcal{M}}

\newcommand{\cR}{\mathcal{R}}

\newcommand{\DpQ}{\operatorname{Diff}_2(p,\cQ)}

\newcommand{\Lam}{\Lambda}
\newcommand{\lam}{\lambda}
\newcommand{\Tng}{{T_{\mathrm{ng}}}}
\newcommand{\Trg}{{T_{\mathrm{rg}}}}
\newcommand{\Tgg}{{T_{\mathrm{gg}}}}
\newcommand{\LamA}{\Lambda_{\mathrm{A}}}

\newcommand{\LamG}{\Lambda_{\mathrm{G}}}

\newcommand{\LamD}{\Lambda_{\mathrm{D}}}

\newcommand{\LamB}{\Lambda_{\mathrm{B}}}

\newcommand{\LamC}{\Lambda_{\mathrm{C}}}
\newcommand{\LamCB}{\Lambda_{\mathrm{CB}}}

\newcommand{\LC}{\operatorname{\mathbf{C}}}

\newcommand{\al}{\alpha}
\newcommand{\be}{\beta}

\newcommand{\hal}{{\hat{\alpha}}}
\newcommand{\hbe}{{\hat{\beta}}}
\newcommand{\tal}{{\tilde{\alpha}}}
\newcommand{\tbe}{{\tilde{\beta}}}
\newcommand{\trho}{{\tilde{\rho}}}
\newcommand{\hb}{{\hat{b}}}

\newcommand{\fu}{\mathfrak{u}}

\newcommand{\hkappa}{\hat{\kappa}}
\newcommand{\hnu}{\hat{\nu}}
\newcommand{\hrho}{\hat{\rho}}
\newcommand{\hphi}{\hat{\phi}}
\newcommand{\hpi}{\hat{\pi}}
\newcommand{\hPi}{\hat{\Pi}}
\newcommand{\Rho}{\mathrm{P}}
\newcommand{\hRho}{\hat{\Rho}}

\newcommand{\hvarphi}{\hat{\varphi}}
\newcommand{\hk}{\hat{k}}
\newcommand{\hp}{\hat{p}}
\newcommand{\hell}{\hat{\ell}}
\newcommand{\hi}{{\hat{\imath}}}

\newcommand{\hn}{\hat{n}}
\newcommand{\hu}{\hat{u}}
\newcommand{\hw}{\hat{w}}
\newcommand{\hq}{\hat{q}}
\newcommand{\hr}{\hat{r}}
\newcommand{\hW}{\hat{W}}

\newcommand{\hA}{\hat{A}}
\newcommand{\htau}{\hat{\tau}}

\newcommand{\R}{\mathbb{R}}

\newcommand{\imk}{{\imath k}}
\newcommand{\himath}{{\hat{\imath}}}
\newcommand{\himk}{{\himath \hk}}

\newcommand{\lp}{\left(}
\newcommand{\rp}{\right)}

\newcommand{\lb}{\left[}
\newcommand{\rb}{\right]}

\newcommand{\cA}{\mathcal{A}}

\newcommand{\cP}{\mathcal{P}}
\newcommand{\cQ}{\mathcal{Q}}

\newcommand{\Diff}{\operatorname{Diff}}

\newcommand{\Z}{\mathbb{Z}}
\newcommand{\Zm}{\Z_{-}}

\newcommand{\cQu}[1]{\cQ^{(#1)}}

\newcommand{\Ric}[1]{\mathrm{R}[#1]}
\newcommand{\Riclam}[1]{\operatorname{Ric}_\lam(#1)}

\newcommand{\bk}{\bar{k}}

\newcommand{\bK}{\bar{K}}
\newcommand{\bL}{\bar{L}}
\newcommand{\bI}{\bar{I}}

\newcommand{\rL}{\mathrm{L}}
\newcommand{\ord}{\operatorname{ord}}

\newcommand{\hI}{\hat{I}}
\newcommand{\hT}{\hat{T}}
\newcommand{\cPo}{\cP^{\pm}}
\newcommand{\cQo}{{\cQ_{\pm}}}

\newcommand{\qQ}{\mathrm{q}\cQ}

\newcommand{\sigq}{\sigma_{\tiny\qQ}}

\newcommand{\ttau}{\tilde{\tau}}

\newcommand{\tchi}{\tilde{\chi}}
\newcommand{\tphi}{\tilde{\phi}}

\newcommand{\tT}{\tilde{T}}
\newcommand{\tK}{\tilde{K}}

\newcommand{\tpi}{{\tilde{\pi}}}
\newcommand{\tnu}{{\tilde{\nu}}}
\newcommand{\tlam}{{\tilde{\lambda}}}

\newcommand{\hP}{\hat{P}}

\newcommand{\hchi}{\hat{\chi}}

\newcommand{\bt}{\boldsymbol{t}}

\newcommand{\bell}{{\boldsymbol{\ell}}}

\newcommand{\ckI}{{\check{I}}}
\newcommand{\cksig}{{\check{\sigma}}}
\newcommand{\ckpi}{{\check{\pi}}}
\newcommand{\ckT}{{\check{T}}}

\newcommand{\rdt}[1]{
  \mathrlap{\xleftharpoondown{\phantom{#1}}}%
  \xrightharpoonup{#1}
}

\newcommand{\I}{{\rm{I}}}

\newcommand{\mystar}{{\rlap+\times}}
\newcommand{\boxstar}{\boxasterisk}
\newcommand{\boxpm}{\boxdiv}
\newcommand{\boxotimes}{\mathrlap\boxcirc\boxtimes}

\newcommand{\bboxplus}{\boxed{\boxplus}}
\newcommand{\bboxminus}{\boxed{\boxminus}}
\newcommand{\bboxcirc}{\boxed{\boxcirc}}
\newcommand{\bboxtimes}{\boxed{\boxtimes}}

\newcommand{\boxtdown}{\mathbin{\ooalign{$\Box$\cr\hbox{$\mkern2.5mu\smalltriangledown$}\hidewidth}}}

\newcommand{\boxfsq}{\mathbin{\ooalign{$\Box$\cr\hbox{$\mkern2.5mu\sqbullet$}\hidewidth}}}

\newcommand{\tre}{\textcolor{red}}

\renewcommand{\tre}{}

\title{Classification of exceptional Jacobi polynomials}

\author{María Ángeles García-Ferrero}
\address{Instituto de Ciencias Matem\'aticas CSIC-UAM-UC3M-UCM, c/ Nicol\'as Cabrera 13--15,
28049 Madrid, Spain.}
\email{garciaferrero@icmat.es}
  
\author{David G\'omez-Ullate}
\address{School of Science and Technology, IE University, Paseo de la Castellana, 259. 28046 Madrid, Spain.}
\email{dgomezullate@faculty.ie.edu}

\author{Robert Milson}
\address{Department of Mathematics and Statistics, Dalhousie University, Halifax, NS, B3H 3J5, Canada.}
\email{rmilson@dal.ca}

\date{\today}

\begin{document}
\begin{abstract}
  We provide a full classification scheme for exceptional Jacobi
  operators and polynomials.  The classification contains six
  degeneracy classes according to whether $\al,\be$ or $\al\pm\be$
  assume integer values.  Exceptional Jacobi operators are in
  one-to-one correspondence with spectral diagrams, a combinatorial
  object that describes the quasi-rational eigenfunctions of the
  operator and their asymptotic behaviour at the endpoints of
  $(-1,1)$.  With a convenient indexing scheme for spectral diagrams,
  explicit Wronskian and integral construction formulas are given to
  build the exceptional operators and polynomials from the information
  encoded in the spectral diagram.  In the fully degenerate class
  $\al,\be\in\mathbb N_0$ there exist exceptional Jacobi operators
  with an arbitrary number of continuous parameters.  The
  classification result is achieved by a careful description of all
  possible rational Darboux transformations that can be performed on
  exceptional Jacobi operators.
  
\end{abstract}

\maketitle
\setcounter{tocdepth}{1}
\tableofcontents


\section{Introduction}\label{sec:intro}


Classical orthogonal polynomials are orthogonal polynomial bases of an $\rm L^2$ space that are also eigenfunctions of a Sturm-Liouville problem \cite{Bochner1929,Lesky1962}. Since the mid-nineteenth century until today, they appear in ubiquitous applications in mathematical  physics,   approximation theory, numerical analysis, or statistics, among other fields.
Relaxing the requirement that the polynomial eigenfunctions contain every possible degree has led to the discovery of \emph{exceptional orthogonal polynomials}, which form also complete bases of $\rm L^2$ spaces spanned by polynomial eigenfunctions of Sturm-Liouville problems, and thus constitute a very natural generalization of classical orthogonal polynomials. Although the terms might sound paradoxical, \emph{exceptional} orthogonal polynomial systems are thus the \textit{general} class of Sturm-Liouville problems with polynomial eigenfunctions, and they include classical polynomial systems as a particular case.

Some isolated examples existed earlier in the mathematical physics literature \cite{Dubov1994}, but the systematic study of exceptional orthogonal polynomials started in 2009 with the publication of two papers by G\'omez-Ullate, Kamran and Milson \cite{GKM09,Gomez-Ullate2010c}. Shortly after, Quesne realized that exceptional polynomials could be built via Darboux transformations, \cite{quesne09}, which are in fact a fundamental element of the theory. Sasaki and Odake iterated the Darboux transformation construction to derive more general families indexed by a set of integers,\cite{OS09,OS11}. Dur\'an showed an alternative construction based on perturbing measures for discrete classical orthogonal polynomials ensued by a limiting process, and he was able to extend the notion to exceptional discrete families of Charlier, Meixner and dual Hahn types, \cite{Du14a,Du14b,Du17}. In his doctoral thesis, Bonneux performed an exhaustive enumeration of all examples known by then, showing how exceptional Hermite polynomials were indexed by a single partition \cite{GGM14}, and exceptional Laguerre and Jacobi polynomials were indexed by two partitions, \cite{Du17,Bonneux2018,Bonneux20}. 

The mathematical properties of exceptional orthogonal polynomials are in certain ways very similar to their classical counterparts, and recent research has focused on the asymptotics and distribution of their zeros \cite{GMM13,KM15,simanek22}, equilibrium measures and electrostatic interpretation \cite{Ho15,horvath22}, polynomial symmetry algebras \cite{Mar25}, recurrence relations \cite{Odake13,Du15,Miki15,GKKM16}, and Wronskian determinants \cite{CASH17,GGM18, GGM18b, Bonneux20b,odake20}, among others. 
The notion of exceptional polynomials, originally derived for scalar orthogonal polynomials on the line, has also been extended to discrete polynomials \cite{Du14a,Du14b,Du17,miki23}, matrix-valued polynomials \cite{casper,koelink23}, Laurent polynomials \cite{Luo,Luo25}, Bannai-Ito polynomials \cite{Luo19} or Dunkl operators \cite{quesne23}.
Perhaps thanks to their closeness to classical orthogonal polynomials, exceptional polynomials have been applied in diverse fields, including quantum mechanics, quantum walks \cite{miki22}, bispectrality and the time-band limiting problem \cite{CG24,CGZ24}, rational solutions to Painlev\'e systems \cite{clarkson,GGM21}, superintegrable systems \cite{PTV12,YKM22,MPR20,HMPZ18}, coherent states in quantum optics \cite{contreras23}, financial mathematics \cite{yesiltas}, diffusion processes \cite{KY24} and statistics of DNA mutation \cite{HC21}.

The question of achieving a complete characterization and classification of exceptional orthogonal polynomials has been one of the main goals throughout the past 15 years. A key step in this task is the Bochner-type theorem that asserts that every exceptional polynomial family is necessarily connected to a classical family by a finite set of Darboux transformations, \cite{GFGM19}, proving a conjecture previously formulated in \cite{GKMconj}. The remaining logical step in the classification program becomes then to describe the complete catalogue of all Darboux transformations that preserve the polynomial character of the eigenfunctions (also called rational Darboux transformations), and to introduce a suitable combinatorial scheme to describe all such possible iterated transformations.
 When all elements seemed to be in place to achieve a complete classification, the list of rational Darboux transformations in the desired catalogue for the Jacobi operator had to be enlarged with new additions. 
 
 It was already known to Grandati and Quesne \cite{GQ15} that for some particular parameter values a degenerate situation could arise that allows new rational Darboux transformations. Indeed, consider the Jacobi eigenvalue equation
  \begin{equation}\label{eq:Jyab}
 (x^2-1)y''-[ \be-\al-(\al+\be+2)]y'=n(n+\al+\be+1)y.
  \end{equation}
When both $\al$ and $\be$  assume integer values, Calogero and Yi \cite{CY12} showed  that equation ~\eqref{eq:Jyab} has two linearly independent polynomial eigenfunctions at certain eigenvalues, which they called para-Jacobi polynomials. This fact  can be used to introduce real parameters in Darboux transformations by considering linear combinations of the two polynomial eigenfunctions as  factorizing function  \cite{bagchi}. Grandati and Quesne have recently extended their approach in \cite{GQ15} to include multi-step Darboux transformations with para-Jacobi polynomial seed functions, \cite{GQ24}. Another way of introducing continuous parameters in Darboux transformations is via \emph{confluent Darboux transformations}, which are two consecutive transformations at the same energy value. Confluent Darboux transformations were known in  the field of mathematical physics and integrable systems \cite{AM80,Su85,KSW89,GT96}, but they had not been systematically applied in the context of exceptional orthogonal polynomials. In this paper we show that these two transformations are strongly related. In \cite{GFGM21} it was shown how to build exceptional Legendre polynomials with an arbitrary number of real parameters, which are moreover interpreted as a continuous deformation from the classical Legendre class, since turning these real parameters to zero recovers the classical Legendre polynomials. In  \cite{gegen} the construction was extended to the Gegenbauer class, and new integral formulas were introduced that are closely related to the inverse scattering method of Gel'fand-Levitan-Marchenko in soliton theory and integrable systems (except in this polynomial setting the spectrum only contains bound states, no scattering data). Dur\'an has also studied this new class of exceptional polynomials with continuous parameters \cite{Du22a,Du22b,Du23}, by adapting his construction based on perturbing discrete measures to include these degenerate cases. The equivalence of Dur\'an's approach and the one described in this paper has not yet been established, but it will be communicated in a forthcoming paper. 

In connection with their group theoretic study of deformations of Calogero–Moser operators with rational potentials, in \cite{BC23}, Berest and Chalykh  have noted that when the non-degeneracy condition
 $\al,\be,\al\pm\be\notin\mathbb Z$ holds, rational Darboux transformations of Jacobi operators are completely described by  Wronskians indexed by 2 partitions,  i.e., the class of exceptional Jacobi operators described by Bonneux \cite{Bonneux2018}. They have also observed that in degenerate cases, when that condition does not hold, other Darboux transformations exist which produce new locus configurations. This is precisely the class that is fully described in this paper.

The main result of this paper is a complete classification of exceptional Jacobi operators and polynomials, which includes both the class described by Bonneux and the new degenerate families with continuous parameters. The keys to achieve a complete classification are
\begin{itemize}
\item[(i)] the characterization result \cite{GFGM19} that states that all exceptional Jacobi operators must be connected by rational Darboux transformations to a classical Jacobi operator,
\item[(ii)] the description of all possible rational Darboux transformations that can be successively applied at each step.
\end{itemize} 
The \textit{quasi-rational spectrum} is the set of eigenvalues for which an exceptional Jacobi operator (see Definition~~\ref{def:XT}) has a quasi-rational eigenfunction (an eigenfunction whose log-derivative is a rational function).
A rational Darboux transformation maps an exceptional Jacobi operator into another exceptional Jacobi operator, but the quasi-rational spectrum remains invariant under such transformation. The number and asymptotic behaviour of quasi-rational eigenfunctions at each eigenvalue can be labelled by a set of symbols, that we call \textit{asymptotic labels}.
The collection of all those symbols at every point in the quasi-rational spectrum defines a \textit{spectral diagram}.
Non-degenerate exceptional Jacobi operators are indexed by two Maya diagrams, but in order to describe the new degenerate exceptional families, Maya diagrams need to be generalized to spectral diagrams. To every exceptional Jacobi operator (whether it is generic or degenerate) we can associate a spectral diagram. A one-step Darboux transformation applied on it will render another exceptional Jacobi operator whose spectral diagram differs from the previous one on just a single asymptotic label. The set of all possible changes on an asymptotic label is called the \textit{flip alphabet}, and it effectively describes point (ii) above. We can now informally state the main results of this paper:
\vskip 0.3cm
\noindent\textbf{Main results:}
\begin{itemize}
    \item[(1)] Exceptional Jacobi operators and spectral diagrams are in one-to-one correspondence.
    \item[(2)] Depending on the values of $\al,\be$ six different degeneracy classes (G,A,B,C,CB,D) can be given, and each class has its own flip alphabet and rules to construct spectral diagrams.
    \item[(3)] In the fully-degenerate class (D) $\al,\be\in\mathbb N_0$, the spectral diagram does not fully determine a single exceptional Jacobi operator, but a $k$-continuous parameter family, where $k\in\mathbb N_0$ is an arbitrary positive integer. To specify a single operator, the norms of $k$ polynomial eigenfunctions need to be specified too.
    \item[(4)] For each of the six degeneracy classes, explicit Wronskian and integral formulas are given to construct both the operator and the eigenfunctions from the information encoded in the diagram. 
\end{itemize}
\vskip 0.3cm The paper is structured as follows. Quasi-rational
spectrum and eigenfunctions, asymptotic symbols, spectral diagrams and
the six degeneracy classes are introduced Section~~\ref{sec:pre}. They
are all the necessary concepts and definitions needed to state the
main theorems of the paper, which is done at the end of the
section. In Section~~\ref{sec:GABCD} we introduce a labelling system for
spectral diagrams in each of the six classes, and construction
formulas for the exceptional Jacobi operator and polynomials. Some
explicit examples are given in Section~~\ref{sec:examples}. No reference to Darboux
transformations is needed to state the main results, but of course
they are the key element that allows to understand and prove the
results of Sections~~\ref{sec:pre} and ~~\ref{sec:GABCD}. For this
reason, we revise the formal theory of Darboux transformations for
second order differential operators in Section~~\ref{sec:formal}, which
includes confluent Darboux transformations.  In
Section~~\ref{sec:classop} we study the spectral diagrams and spectral
functions of classical Jacobi operators. These simple spectral
diagrams of classical operators can be regarded as an empty canvas on which
to ``paint'' the richer spectral diagrams that describe exceptional
Jacobi operators. The allowed ``painting operations'' are the Darboux
transformations and corresponding flip alphabets, which are described
in detail for each degeneracy class in Section~~\ref{sec:JRDT}. The remaining
proofs to complete the main results stated in Section~~\ref{sec:pre}
are given in Section~~\ref{sec:proofs}. Since this paper is already
quite long and technical, two related questions are postponed to a
forthcoming publication.

The first one concerns equivalence between different representations
to describe the same exceptional Jacobi operator. \tre{For example, for the
subclass where $\al,\be\in \Nz$,} spectral diagrams encode the set of
Darboux transformations that need to be applied on the classical
Jacobi operator \tre{with $(\al,\be)=(0,0),(1,0)$ or $(0,1)$.} It
would also be possible to start from a different classical Jacobi
operator and thus the sequence of Darboux transformations would
differ. The reason for this freedom is that there exist
\textit{primitive} Darboux transformations that connect classical
Jacobi operators with different values of $(\al,\be)$ among
themselves. This set of alternative constructions, which show the
equivalence between Duran's approach and the one explained here, will
be fully explained.

The second question involves the \textit{regularity} of the
exceptional operator, or equivalently, in which cases can we assure
that the exceptional Jacobi operator defines an orthogonal polynomial
system with a well defined and positive measure supported on
$(-1,1)$. This problem has an elegant answer that is stated in Theorem
~\ref{thm:regpos} but whose proof is postponed to a forthcoming
publication.

\section{Preliminaries and main results}\label{sec:pre}

\subsection{Base notation and  definitions.}
We begin by fixing some notation.  Let $\cP=\R[x]$ denote the ring of
univariate, real polynomials, $\cP^\times$ the subset of non-zero
polynomials and $\cPo$ the subset of polynomials that do not vanish
at $\pm 1$. Let $\cQ=\R(x)$ denote the field of rational functions,
and $\cQo\subset \cQ$ the subspace of rational functions that are
regular at $x=\pm 1$.
We define a \emph{quasi-rational (qr) function} to
be a function $\phi(x)$ such that its log-derivative
$w(x) = \phi'(x)/\phi(x)$ is a rational function.  Let
\begin{equation}
  \label{eq:qQab}
  \qQ_{\al,\be} = (1-x)^\al (1+x)^\be\cQo ,\quad \al,\be \in \R,
\end{equation}
be the vector space consisting of qr-expressions of the form
$q(x) (1-x)^{\al}(1+x)^\be$, with $q\in \cQo$.  For
$q\in \cQ, \al,\be \notin \Z$, the notation
\[ \rho(x) = \int q(x) (1-x)^\al (1+x)^\be \] will signify that
$\rho(x)$ is the unique quasi-rational function whose derivative is
the indicated integrand\footnote{\tre{Since the primitive $\rho(x)$ is
  assumed to be quasi-rational with
  $\rho'(x) = q(x) (1-x)^\al (1+x)^b$, then necessarily
  $q(x)(1-x)^\al (1+x)^\be$ has vanishing residues away from
  $x=\pm1$. Thus, the integral in question is single-valued by
  assumption and does not produce logarithmic singularities.}}.  The
same notation is in place if one of $\al,\be$ is a non-negative
integer but the other isn't. If $\al,\be\in \Nz$, the notation
\[ \rho(x) = \int_{-1}^x q(x) (1-x)^\al (1+x)^\be \] will signify that
$\rho(x)$ is the unique rational function that vanishes at $x=-1$ and
whose derivative is the indicated rational integrand\footnote{\tre{In
    the subcase where $q(x)(1-x)^\al(1+x)^\be$ is integrable on
    $[-1,1]$ with vanishing residues there, we have
\[ \rho(x) = \int_{-1}^x q(u) (1-u)^\al (1+u)^\be du .\] This
observation motivates our particular choice of notation. However, in
general, $\int_{-1}^x$ should be regarded as a kind of definite
integral operator, albeit one that is constrained to produce
primitives that vanish at $x=-1$.}}.

The purpose of this paper is to provide a classification of exceptional
Jacobi operators and polynomials.  We define the former as follows.
\begin{definition}
  \label{def:keydefs}
  Let $\Diff(\cQ)$ denote the set of non-zero differential operator
  with rational coefficients and let $\Diff_n(\cQ)$ be the subset of
  such operators having order $n\in \Nz$.  For $p\in \cP^\times$, let
  $\DpQ\subset \Diff_2(\cQ)$ denote the set of second order operators
  of the form
  \begin{equation}
    \label{eq:Tpqr}
    T = p D^2 + qD + r,\quad q,r\in \cQ.
  \end{equation}
  We define an \emph{eigenpolynomial} of $T$ to be a non-zero
  polynomial solution $P\in \cP^\times$ of the eigenvalue equation
  $TP=\lam P$ where $\lam$ is a constant.  We call the corresponding
  $\lam$ a \emph{polynomial eigenvalue} of $T$ and refer to the set
  $\sigma_{\cP}(T)$ of all such $\lam$ as the \emph{polynomial spectrum} of
  $T$.  We define the \emph{degree set} $I=I(T)$ to be the set of the
  degrees of the eigenpolynomials of $T$ and index the latter as
  $\{ P_n\}_{n\in I}$, where $P_n(x)$ is an eigenpolynomial such that
  $\deg P=n$. 
    \end{definition}
    
\begin{definition}\label{def:XT}
  A second order differential operator
  $T\in \Diff_2(\cQ)$\footnote{One can show that if $T$ admits
    infinitely many eigenpolynomials, then necessarily $p(x)$ is a
    polynomial, and $q(x), r(x)$ are rational functions \cite{GKM12}.
    Therefore no generality is lost making the above assumptions
    regarding the coefficients $p,q,r$.} is an \emph{exceptional
    operator} if it has eigenpolynomials for all but finitely many
  degrees, i.e. if $\Nz\setminus I$ has finite cardinality.  $T$ is an
  \emph{exceptional Jacobi operator} if $p(x)=x^2-1$.
\end{definition}

The fundamental structure theorem of exceptional operators proved in
\cite[Theorem 5.3]{GFGM19} imposes a very specific structure for an
exceptional Jacobi operator.  The theorem below is a
  version of that result. The proof can be found in Section 7.
\begin{thm}
  \label{thm:tng}
  Every exceptional Jacobi operator is related by a gauge
  transformation and a spectral shift to an operator of the form
  \begin{equation}
    \label{eq:Tngdef}
    \Tng(\tau;\al,\be) = (x^2-1)\lp D_x^2 - 2 \frac{\tau'}{\tau} D_x +
    \frac{\tau''}{\tau} +\eta(x;\al,\be)\lp D_x- \frac{\tau'}{\tau} \rp  \rp 
    +2 x\frac{\tau'}{\tau}
  \end{equation}
  where
  \begin{equation}
    \label{eq:etadef}
    \eta(x;\al,\be) := \frac{\al+1}{x-1}+ \frac{\be+1}{x+1},\quad
    \al,\be \in \R,     
  \end{equation}
  and $\tau\in \cPo$, i.e. $\tau$  a polynomial such that $\tau(\pm 1) \ne 0$.
  Moreover, in that case, the $\Tng(\tau;\al,\be)$ introduced in ~\eqref{eq:Tngdef}
  is itself an exceptional Jacobi operator.
\end{thm}

\noindent
As per Definition 5.1 of \cite{GFGM19}, an operator
$\Tng(\tau;\al,\be),\; \tau\in \cPo$ will be said to be in the \emph{natural
gauge (ng)}.  In light of the above structure theorem, a classification
of exceptional Jacobi operators does not lose generality by focusing
on operators in the natural gauge.

In general, not every choice of $(\tau;\al,\be)$ in
~\eqref{eq:Tngdef} will constitute an exceptional Jacobi operator,
so the classification problem can be understood as characterizing all
of the $(\tau;\al,\be)$ tuples such that $\Tng(\tau;\al,\be)$ is an
exceptional operator.  
So far, exceptional Jacobi operators ~\eqref{eq:Tngdef} have been
defined as differential expressions. A subset of these expressions
will give rise to  exceptional Jacobi orthogonal polynomial systems; this
happens precisely when the weight function determined by
~\eqref{eq:Tngdef} defines a positive definite measure with finite
moments of all orders. This requires the notion of regularity of the
weight, which we define below.

\begin{definition}
  \label{def:regT}
An exceptional Jacobi operator $T=\Tng(\tau;\al,\be)$ is \emph{regular}
if $\tau(x)$ does not vanish for $x\in [-1,1]$ and if $\al,\be>-1$.
If regularity holds, then we will refer to the family
$\{ P_{n}\}_{n\in I}$ as an \emph{exceptional Jacobi Orthogonal Polynomial system}.
\end{definition}
\noindent
The motivation for the second definition is that, in the regular case,
\begin{equation}
  \label{eq:PiPjortho}
  \int_{-1}^1 \frac{P_m(x) P_n(x)}{\tau(x)^2} (1-x)^\al (1+x)^\be =
 \nu_n\,   \delta_{mn}
  ,\quad m, n\in I(T),
\end{equation}
with $\nu_n>0$ for every $n$.  As asserted in Theorem
~\ref{thm:regpos}, whose proof will be given elsewhere, the above
definition ensures that an exceptional Jacobi OP system is complete in
the corresponding weighted Hilbert space.

\subsection{Classical operators and polynomials.}
The classical Jacobi operator
\begin{equation}
  \label{eq:Tabdef}
  T(a,b):= (x^2-1)\lp D_x^2  +\eta(x;a,b)\, D_x\rp,\quad
  a,b\in \R,
\end{equation}
is a special case of definition ~\eqref{eq:Tngdef}, corresponding to
$\tau=1$. The classical Jacobi polynomials are
\begin{align}
  \label{eq:Pab2}
  P_n(x;a,b)
  &= \frac{2^{-n}}{n!} (x-1)^{-a}(x+1)^{-b} D_x^n \lp
  (x-1)^{n+a}(x+1)^{n+b}\rp\\
  \label{eq:Pabdef}
  &= 2^{-n}\sum_{k=0}^n \binom{n+a}{n-k}
    \binom{n+b}{k} (x-1)^k (x+1)^{n-k},\quad n\in \Nz.
\end{align}
The leading $n$th degree coefficient of $P_n(x;a,b)$ is given by
$(n+a+b+1)_n2^{-n}/n!$, where $(t)_n = t(t+1)\cdots (t+n-1)$ denotes
the Pochhammer symbol\footnote{Also known as the rising
  factorial.}. Thus, if $ (n+a+b+1)_n\ne 0$, then we can define the
monic form of the classical polynomials
\begin{equation}
  \label{eq:pindef}
  \pi_n(x;a,b) := \frac{2^n n!}{(n+a+b+1)_n} P_n(x;a,b).
\end{equation}
These may also be defined as the monic polynomial solutions of the
classical Jacobi eigenvalue equation
\begin{equation}
  \label{eq:TabPn}
  T(a,b)\pi_n = \lam_n\pi_n ,\quad\text{where } n = \deg \pi_n,
\end{equation}
Since
\begin{equation}
  \label{eq:Pn-1}
  P_n(-1;a,b) = \frac{(-1)^n (b+1)_n}{n!},
\end{equation}
an alternative formulation is \tre{to make the assumption that
$b\notin \Zm$ and to define} the $\pi_n(x;a,b)$ as the unique
polynomial solutions of ~\eqref{eq:TabPn} subject to the initial
condition
\[ \pi_n(-1;a,b) = \frac{(-2)^n (b+1)_n}{(1+a+b+n)_n}.\]
Necessarily, the allowed degrees are constrained by
$ (a+b+1+n)_n\ne 0$, while the eigenvalues have the form
$\lam_n = n(n+a+b+1)$.

If
$a,b>-1$ then we also have the classical orthogonality relation
\begin{equation}
  \label{eq:classortho}
  \int_{-1}^1 \pi_m(x;a,b)\pi_n(x;a,b) (1-x)^a (1+x)^b dx =
  \nu(n;a,b)   \delta_{mn}
  ,\quad m, n\in \Nz,
\end{equation}
where 
\begin{equation}
  \label{eq:nunab}
  \nu(z;a,b)
  := 2^{1+a+b+2z} 
  \frac{\Gamma(z+1)\Gamma(a+b+z+1)\Gamma(a+z+1)\Gamma(b+z+1)}{
    \Gamma(a+b+2z+1)\Gamma(a+b+2z+2)}.
\end{equation}
Using the Legendre duplication formula, the above expression may also
be given as
\begin{equation}
  \label{eq:nunab2}
  \nu(z;a,b)
  = \frac{\pi}{2^{2z+a+b}}
    \frac{\Gamma(z+1)\Gamma(z+1+a)\Gamma(z+1+b)\Gamma(z+1+a+b)}{
    \Gamma(z+(a+b+1)/2) \Gamma(z+(a+b+2)/2)^2 \Gamma(z+(a+b+3)/2)}.
\end{equation}
\noindent
The latter expression readily gives the asymptotic behaviour of the
classical norms.

By direct inspection of ~\eqref{eq:Pabdef}, the classical operator
$T(a,b)$ satisfies the definition of an exceptional operator for all
values of $a,b\in \R$.  Moreover, as proved in \cite{GFGM19},
\emph{all} exceptional operators can be obtained by a sequence of
rational Darboux transformations from a classical $T(a,b)$.

\subsection{The rational gauge}
Let $\Tng(\tau;\al,\be)$ be as defined in ~\eqref{eq:Tngdef}.  Let us
consider a gauge transformation of this operator with gauge factor
$\tau$ by setting \begin{equation}
    \label{eq:Trgdef}
     \Trg(\tau;\al,\be) := \MO{\tau^{-1}} \circ \Tng(\tau;\al,\be) \circ
   \MO{\tau},
 \end{equation}
 where $\MO{f}$ denotes the formal multiplication operator
 $y\mapsto f y$. 
\begin{prop}
  \label{prop:TRgauge}
  Let $\Trg(\tau;\al,\be)$ be as in ~\eqref{eq:Trgdef} and $T(\al,\be)$
  the classical Jacobi operator as per ~\eqref{eq:Tabdef}. Then,
  \begin{equation}
    \label{eq:Trgdef2}
    \Trg(\tau;\al,\be) = T(\al,\be) +2(x^2-1)u'+2xu,\quad
    u=\tau'/\tau.
  \end{equation}
\end{prop}
\begin{proof}
  The proof follows by a straight-forward calculation.  See also
  ~\eqref{eq:Tgauged} below.
\end{proof}

\begin{definition}\label{def:Trg}
  If $\Tng(\tau;\al,\be)$ is an exceptional operator in the natural gauge,
  as per Definition ~\ref{def:XT},   we will say that $\Trg(\tau;\al,\be)$ is
  an exceptional Jacobi operator in the \emph{rational gauge (rg)}.
  Moreover, we  will call $\Trg(\tau;\al,\be)$ \emph{regular} if the corresponding $\Tng(\tau;\al,\be)$ is regular.
\end{definition}
\noindent
The operator $\Trg$ has no longer polynomial but rational
eigenfunctions, which are polynomial up to a common factor. This
motivates to enlarge the eigenpolynomial notion introduced in
Definition ~\ref{def:keydefs} in the following manner.
\begin{definition}
  \label{def:quasipol}
  We define a \emph{quasi-polynomial eigenfunction} of an operator $T\in \DpQ$ to be a
  rational solution $\pi(x)$ of the eigenvalue equation
  $T\pi = \lambda \pi$ such that $\pi(x)$ has no zeros or poles at the
  zeros of $p(x)$.  We will refer to the corresponding $\lambda$ as a
  \emph{quasi-polynomial eigenvalue}.  Going forward, we define the degree of
  a rational function as the difference of degrees of the numerator
  and denominator and let $I_1(T)$ denote the set of the degrees of
  the quasi-polynomial eigenfunctions of $T$.  We will also use
  $\sigma_1(T)$ to denote the set of quasi-polynomial eigenvalues and
  refer to this set as the \emph{quasi-polynomial spectrum}\footnote{The
    reason for the subscript in $\sigma_1(T), I_1(T)$ will be
    clear later.} of $T$.
\end{definition}
The following result, which will be proved in Section
~\ref{subsec:proofs:qr}, characterizes the quasi-polynomial eigenfunctions of
an exceptional $\Trg$.
\begin{prop}
  \label{prop:piPtau}
  Let $\Tng(\tau;\al,\be) $ be an exceptional Jacobi operator in the
  natural gauge and
  $\Trg(\tau;\al,\be),\; \tau\in \cPo,\; \al,\be \in \R$ the
  corresponding exceptional Jacobi operator in the rational 
  gauge, as shown in ~\eqref{eq:Tngdef} and ~\eqref{eq:Trgdef2},
  respectively.  Every quasi-polynomial eigenfunction of $\Trg$ has
  the form
  \begin{equation}
  \label{eq:piidef}
  \pi_k(x) = \frac{P_{n}(x)}{\tau(x)},\quad k\in I_1(\Trg),\quad k=n-N,
\end{equation}
where $P_{n}\in \cPo$ is an eigenpolynomial of $\Tng$ of degree $n$
and $N=\deg\tau$.  Moreover, if $\al,\be \notin \Zm$, then
\begin{align}
  \nonumber
  I_1(\Trg) &= I(\Tng)-N\\
  \label{eq:sigTngTrg}
  \sigma_\cP(\Tng) &=\sigma_1(\Trg) =\{ k(k+\al+\be+1) \colon k \in
                 I_1(\Trg) \}.
\end{align}
\end{prop}

In light of the above result, we introduce the following definitions.
\begin{definition}
  We refer to the quasi-polynomial eigenfunctions of an exceptional Jacobi operator in the rational gauge, 
  $\pi_k(x),\; k\in I_1(\Trg)$,   as \emph{exceptional Jacobi
    quasi-polynomials} and to the set of its degrees,  $I_1(\Trg)$, as the \emph{quasi-polynomial index set}.
\end{definition}
\noindent
The classification problem for exceptional Jacobi operators evidently
reduces to the classification of exceptional operators in the rational
gauge.  It is in this form that we will pursue the classification
question.  The reason behind this choice is that the rational gauge
has some important technical advantages over the natural gauge such as
a simplified relation between the polynomial spectrum and the degree
set exhibited in ~\eqref{eq:sigTngTrg}.

  Another advantage of the rational
gauge is that, in the regular case, the orthogonality relation
~\eqref{eq:PiPjortho} takes the classical form:
\begin{equation}
  \label{eq:Rortho}
  \int_{-1}^1 \pi_i(x) \pi_j(x) (1-x)^\al (1+x)^\be dx = 
  \nu_i\, \delta_{ij},\quad i,j\in I_1(\Trg),
\end{equation}
with $\nu_i>0$ for every $i$.  


As specified in ~\eqref{eq:piidef},
$\tau(x), \{ P_n(x)\}_{n\in I(\Tng)},\,\{\pi_k(x)\}_{k\in I_1(\Trg)}$
and the corresponding norms $\{\nu_i\}_{i\in I_1(\Trg)}$ are only
defined up to scalar multiplication. In order to remedy this
ambiguity, one has to introduce some notion of normalization.
\begin{definition}
  Given a family of exceptional Jacobi quasi-polynomials $\pi_i(x),\;
  i\in I_1$ we will refer to the sequence of constants $\pi_i(-1),\;
  i\in I_1$ as a \emph{choice of normalization}.
\end{definition}
\noindent
Evidently, a choice of normalization is required in order to fix the
quasi-polynomial eigenfunctions and their norms.  One possibility
would be to require that $\pi_i(-1) = 1$ for all $i\in I_1$.  However,
for various technical reasons we will opt for a normalization that
produces monic quasi-polynomials; this means that is the leading
coefficient of the numerator $P_n(x)$ matches the leading coefficient
of the denominator $\tau(x)$.  Some families of exceptional
polynomials feature a finite number of continuous deformation
parameters.  In such cases, we will utilize a normalization that
results in asymptotically monic eigenfunctions; that is the
exceptional quasi-polynomials are monic in the limit where these
parameters go to infinity.  The concrete details and explicit
definitions will be given in Section ~\ref{sec:D}.

\subsection{Generalized orthogonality.}
If $T=\Trg(\tau;\al,\be)$ is not regular, then  the norm definition
\begin{equation}
  \label{eq:nuidef}
 \nu_i:= \int_{-1}^1 \pi_i(x)^2 (1-x)^\al (1+x)^\be dx,\quad i\in I_1(T) 
\end{equation}
is not valid because the right-hand side is a divergent integral.
There is a way to generalize the right side of ~\eqref{eq:nuidef} so as
to formally define $\nu_i$ for all exceptional operators for which
$\al,\be \notin \Zm$.  The construction
relies on the following result, which is proved\footnote{There is an
  alternative approach to generalized weights based on deformed
  contours \cite{KMFO05} \cite{HHHV16}.} in Section ~\ref{subsec:proofs:main}.

Going forward, set
\begin{equation}
  \label{eq:nudef}
  \nu(a,b) := \nu(0;a,b) =
  2^{1+a+b}\frac{\Gamma(a+1)\Gamma(b+1)}{\Gamma(a+b+2)}=
  2^{1+a+b}B(a+1,b+1),
\end{equation}
where the latter is the usual beta function.
For $a,b> -1$, we have
\[ \nu(a,b) = \int_{-1}^1 (1-x)^a (1+x)^b dx;\] so $\nu(a,b)$ may be
regarded as the analytic continuation of the basic weight integral.
\begin{prop}
  \label{prop:kappaT}
  Let $T=\Trg(\tau;\al,\be),\; \al,\be,\al+\be+1\notin \Zm$ be an
  exceptional operator and $\pi_i(x),\, i\in I_1$ corresponding Jacobi
  quasi-polynomials.  Then, there exist\footnote{The value of these
    constants depends on the choice of normalization.} constants
  $\nu_i$ such that
  \begin{equation}
    \label{eq:sfdef}
    \int \lp \pi_i(x)\pi_j(x)- \frac{\nu_i\,
      \delta_{ij}}{\nu(\al,\be)} \rp 
    (1-x)^{\al} (1+x)^{\be} \in \qQ_{\al+1,\be+1},\quad i,j\in I_1(T).
  \end{equation}
  If $\al,\be\notin \Zm$, but $m:=\al+\be+1\in \Z$ then there exist
  constants $\nu_i,\; i\in I_1$ such that
  \begin{equation}
    \label{eq:sfdef2}
    \int \lp \pi_i(x)\pi_j(x)- 
        (1+x)^{-m}\frac{\nu_i\, \delta_{ij}}{\nu(\al,-1-\al)}
    \rp(1-x)^{\al} (1+x)^{\be} 
    \in \qQ_{\al+1,\be+1},\quad i,j\in I_1(T).
  \end{equation}
\end{prop}
\noindent When $\Trg$ is regular, then ~\eqref{eq:sfdef} implies the
usual orthogonality condition ~\eqref{eq:Rortho}. Thus, the above
assertion extends the usual notion of a norm to the case when $\Trg$
is not regular.  For $\al,\be\notin \Zm$, set $m:=\al+\be+1$ and
observe that $\nu(\al,\be) = 0$ precisely when $m\in \Zm$. Thus,
relation ~\eqref{eq:sfdef} does not make sense in this case, \tre{and
  one has to use \eqref{eq:sfdef2} as the definition of $\nu_i$.  If
  $\al,\be \notin \Z$ but $m\in \Nz$, then the $\nu_i$ defined by
  \eqref{eq:sfdef} also satisfies relation \eqref{eq:sfdef2}.  It will
  therefore be convenient to take \eqref{eq:sfdef2} as the definition
  of the formal norm $\nu_i$ for all cases where $\al,\be \notin \Z$
  but $\al+\be+1\in \Z$.}

\begin{remark}
  Below, we establish explicit formulas for the norms of exceptional
  Jacobi polynomials relative to monic normalization.  From these
  formulas, it will be seen that the constants $\nu_i/\nu(\al,\be)$
  have a rational dependence on $\al,\be$.  Indeed, if
  $\al,\be\in \cQ$ are rational, then the norm $\nu_i$ is a rational
  number times $\nu(\al,\be)$. In the subcase where $\al+\be+1\in \Z$,
  we have
  \[ \nu(\al,-1-\al) = \Gamma(-\al) \Gamma(1+\al)= -\pi\csc(\pi
    \al).\] In this case also, for each $i\in I_1(T)$, the constants
  $\nu_i/\nu(\al,-1-\al)$ are also rational in $\al$.  Thus, if
  $\al\in \cQ$ is rational, then $\nu_i$ is a rational number times
  $\pi \csc(\pi \al)$.
\end{remark}

An exceptional Jacobi operator defines a well-behaved orthogonal
polynomial system (OPS) provided it is regular (in the sense of
Definition ~\ref{def:regT}). In order to complete the task of
classifying all exceptional Jacobi OPS we need to:
\begin{enumerate}
\item characterize the tuples
$(\tau,\al,\be)$ that define an exceptional Jacobi operator \break
{$T=\Trg(\tau;\al,\be)$},
\item select those tuples $(\tau,\al,\be)$ for which $T$ is regular.
\end{enumerate} 
The first step of this program is achieved in this paper. The
regularity question will be dealt with detail in a forthcoming
publication, but we state as an announcement the main criterion that
allows to determine when a given $(\tau,\al,\be)$ defines a regular
operator, which essentially amounts to the positivity of all the
norms.

\begin{thm}
  \label{thm:regpos}
  Let $T=\Trg(\tau;\al,\be),\; \al,\be,\al+\be+1\notin \Zm$ be an
  exceptional Jacobi operator, $\pi_i(x),\; i\in I_1$, the
  corresponding exceptional quasi-polynomials defined relative to some
  choice of normalization.  Then, each $\nu_i,i\in I_1(\Trg)$, as
  defined above, is a finite, non-zero real number.
  The operator $T$ is regular  if and only if $0<\nu_i<\infty$ for all
  $i\in I_1(T)$.  Moreover, if regularity holds then the
  quasi-polynomial eigenfunctions are also a complete $\rL^2$ basis in
  the corresponding weighted Hilbert space.
\end{thm}
\noindent
Note that even though the specific value of the $\nu_i,\; i\in I_1$
depends on a choice of normalization, the sign of the $\nu_i,\; i\in
I_1$ is independent of this choice.

\subsection{Quasi-rational spectrum.}

Most of the literature on the theory of exceptional polynomials and
operators on the past ten years has focused on the polynomial spectrum
and eigenfunctions. It is now clear that a complete, systematic
approach needs to consider the larger class of quasi-rational
eigenfunctions. Indeed, as we shall see, a proper characterization of
the quasi-rational eigenfunctions of a given exceptional operator and
their asymptotic behaviour essentially specifies the operator
uniquely. For this reason, we introduce in the following sections the
two key concepts of \emph{quasi-rational spectrum} and \emph{spectral
  diagram}.

\begin{definition}
  Let $T\in   \Diff_2(\cQ)$ be a second-order operator.   A quasi-rational
  solution $\phi(x)$ of the eigenvalue problem $T\phi = \lambda \phi$
  will be called a \emph{qr-eigenfunction} of $T$ and the corresponding
  $\lambda$ a \emph{qr-eigenvalue}. Let $\sigq(T)$ denote the set of all qr-eigenvalues of $T$. We will refer to $\sigq(T)$ as the
  \emph{qr-spectrum} of $T$.
\end{definition}

\begin{definition}
  For $T=pD^2+q D+r\in  \Diff_2(\cQ)$,  let $\Ric{T}:\cQ\to \cQ$ be the non-linear operator
  with action
  \begin{equation}
    \label{eq:Twlam}
    \Ric{T}w := p(w'+w^2) + q w + r,\quad w\in \cQ. 
  \end{equation}
   For $\lam\in \R$, define $\Riclam{T}:=\{ w\in \cQ: \Ric{T}w=\lam\}$
  to be the set of rational solutions of the indicated Ricatti-type
  equation.
\end{definition}
\noindent  
Observe that if $w= \phi'/\phi$, where $\phi(x)$ is quasi-rational, then
  \begin{equation}\label{eq:RicTw}
       \Ric{T}w = (T\phi)/\phi.
  \end{equation}
As an immediate consequence of this, we have the following identification.
\begin{prop}
  The qr-spectrum of $T\in \Diff_2(\cQ)$ can be characterized as
  \[ \sigq(T) = \{ \lam \in \R: \Riclam{T} \ne \emptyset \}.\]
\end{prop}


In the description of the qr-spectrum and eigenfunctions of an
exceptional Jacobi operator, it will be particularly important to
analyze the asymptotic behaviour of the qr-eigenfunctions at the
endpoints of the interval $[-1,1]$.  Accordingly, we divide the qr-eigenfunctions of $\Trg(\tau;\al,\be)$ into four types according to
the behaviour at $x=\pm 1$.
\begin{definition}
  \label{def:atype}
  We define the \emph{asymptotic type} of a rational function $w(x)$
  to be, respectively $1,2,3,4$ according to whether it is regular at
  $x=\pm 1$, has poles at both endpoints, has a pole at $x=+1$ but not
  at $x=-1$, or has a pole at $x=-1$ but not at $x=+1$.  We define the
  \emph{asymptotic type} of a qr-function $\phi(x)$ to be the
  asymptotic type of the corresponding log-derivative
  $w=\phi'/\phi$. We define the \emph{exponential order} of $\phi(x)$ to be
  the degree
     of $w(x)$ and  the \emph{degree} of $\phi(x)$, $\deg\phi$, to be the
  residue at  infinity of  $w(z)$.
\end{definition}

The four asymptotic types of qr-eigenfunctions of an exceptional
Jacobi operator are shown in Table~~\ref{tab:qreigen}
below\footnote{The symbol $\mu_\imath$ is defined in
  Proposition~~\ref{prop:qreigen}.}.  The quasi-polynomial
eigenfunctions fit the definition of a type 1 qr-function. Indeed, as
we now claim, the type 1 eigenfunctions are precisely the
quasi-polynomial eigenfunctions of an exceptional operator.  The
following result is proved in Section ~\ref{subsec:proofs:qr}.

\begin{prop} \label{prop:qreigen} Let $T=\Trg(\tau;\al,\be)$ be an
  exceptional Jacobi operator.  Let $\phi(x)$ be a type
  $\imath\in \{1,2,3,4\}$ qr-eigenfunction of $T$ and let $\lambda$ be
  the corresponding qr-eigenvalue.  Then, there exists a polynomial
  $\htau\in \cPo$ such that
  \begin{equation}
    \label{eq:phimutau}
     \phi(x) = \mu_\imath(x;\al,\be) \frac{\htau(x)}{\tau(x)},
  \end{equation}
  with the form of $\mu_\imath=\mu_\imath(x;\al,\be)$ as shown in
  Table ~\ref{tab:qreigen}.
  Moreover,
  \begin{equation}
    \label{eq:lamdegphi}
    \lambda =\deg\phi(\deg\phi+\al+\be+1),
  \end{equation}
  where
  \begin{align}
    \nonumber
    \deg\phi &=\deg\mu_\imath+\deg \htau - \deg \tau,\\
   \label{eq:degmui}
    \deg \mu_1 &= 0,\quad \deg \mu_2 = -\al-\be,\quad \deg\mu_3 =
    -\al,\quad \deg \mu_4 = -\be.
 \end{align}
\end{prop}

\begin{table}[h]
  \caption{Asymptotic types of qr-eigenfunctions of an exceptional
    Jacobi operator.  }
\begin{tabular}{|c|c|c|c|c|}
  \hline 
   $\imath$ & $x=-1$ & $x=1$ & $\mu_\imath(x;\al,\be)$\\
  \hline 
  1  & Regular & Regular &$1$\\
  \hline 
  2  & Singular & Singular& $(1-x)^{-\al}(1+x)^{-\be}$\\
  \hline 
  3  & Regular & Singular& $(1-x)^{-\al}$ \\
  \hline 
  4  & Singular & Regular & $(1+x)^{-\be}$\\
  \hline 
\end{tabular} 
  \label{tab:qreigen}
\end{table}

 The existence of four
families of qr-eigenfunctions is closely related to the following
parameter symmetries of the classical and exceptional Jacobi
operators.
\begin{prop}
  \label{prop:gaugesym}
  Let $\mu_\imath,\; \imath\in \{1,2,3,4\}$ be as in Table
  ~\ref{tab:qreigen}, and let $\cM(f):y\mapsto fy$ denote the formal
  multiplication operator. Then,
  \begin{equation}
    \label{eq:Tgaugesym}
    \begin{aligned}
      \cM(\mu_2)^{-1}\circ \Trg(\tau;\al,\be)\circ \cM(\mu_2)
      &=  \Trg(\tau;-\al,-\be) - \al-\be\\
      \cM(\mu_3)^{-1}\circ\Trg(\tau;\al,\be)\circ \cM(\mu_3)
      &= \Trg(\tau;-\al,\be)   - \al(\be+1)\\
      \cM(\mu_4)^{-1}\circ \Trg(\tau;\al,\be)\circ\cM(\mu_4)
      &=  \Trg(\tau;\al,-\be) - \be(\al+1)
    \end{aligned}
  \end{equation}
\end{prop}
\begin{proof}
  By ~\eqref{eq:Trgdef2} it suffices to verify the above identities for
 the subclass of classical Jacobi operators $T=T(\al,\be)$ in
  ~\eqref{eq:Tabdef}, which follows from a straightforward
  calculation.
\end{proof}

We now introduce an indexing scheme for the qr-eigenfunctions and
extend the definition of a quasi-polynomial spectrum to other quasi-rational types
as follows.
\begin{definition}
  Let $T=\Trg(\tau;\al,\be)$ be an exceptional Jacobi operator.  Let
  $\phi(x)$ be an eigenfunction of type $\imath\in \{1,2,3,4\}$ and
  let $\deg \phi$ be as above.
  Relying on ~\eqref{eq:phimutau}, we define the \emph{index} $k$ of $\phi$ to
  be the degree of the $\cQo$ factor of $\phi$.  To be more precise,
  \begin{equation}
    \label{eq:k1234}
    k = \deg\phi - \deg\mu_\imath = \deg\htau-\deg\tau=
    \begin{cases}
      \deg\phi & \text{if } \imath = 1\\
      \deg\phi+\al+\be & \text{if } \imath = 2\\
      \deg\phi+\al & \text{if } \imath = 3\\
      \deg\phi+\be & \text{if } \imath = 4
    \end{cases}.
  \end{equation}
  \end{definition}
  \begin{definition}
For every $\imath\in \{1,2,3,4\}$ we define the \emph{index sets}
  $I_\imath(T)$ as the set of indices of
   type $\imath$ eigenfunctions. We label the corresponding qr-eigenfunctions as
  $\phi_\imk,\; k\in I_\imath(T)$, and we extended this notation to the rational functions
  $\pi_\imk(x)$ and polynomials $\htau_\imk(x)$ as per Proposition
  ~\ref{prop:qreigen}, so that 
  \begin{align}
    \label{eq:phi1234}
    \phi_\imk
    &= \mu_\imath \pi_\imk = \mu_\imath
      \frac{\htau_\imk}{\tau},\quad     
      k\in  I_\imath(T) ,\quad \imath\in \{1,2,3,4\}.
        \end{align}
     In a similar fashion, we label the corresponding qr-eigenvalues and spectra according to these four types:   
        \begin{align}
    \label{eq:lamik}
    \lam_\imk
    &:= \deg\phi_\imk (\deg\phi_\imk+\al+\be+1),\; k\in I_\imath(T),\\
    \label{eq:sigiT}
    \sigma_\imath(T)
    &:= \{ \lam_\imk:          k\in  I_\imath(T) \},
  \end{align}
  and we refer to $\lam_\imk\in \sigma_\imath$ as  a type $\imath$
  eigenvalue and $\sigma_\imath(T)$ as the type $\imath$ spectrum. For
  the quasi-polynomial eigenfunctions we will also employ the
  simplified notation $\pi_{k} = \pi_{1k},\; k\in I_1(T)$.
    \end{definition}


\noindent
Since, by Proposition ~\ref{prop:qreigen}, every qr-eigenfunction is
one of the above asymptotic types, we have
  \begin{equation}\label{eq:union}
     \sigq(T)= \bigcup_{\imath\in \{1,2,3,4\}}\sigma_\imath(T).
  \end{equation}
\noindent
As a direct corollary of the definition ~\eqref{eq:k1234} and
Proposition ~\ref{prop:gaugesym} we see that
 \begin{equation}
  \label{eq:k1234pi}
  \begin{aligned}
    I_2(T) 
    &= I_1(\Trg(\tau;-\al,-\be)),\\
    I_3(T)     &= I_1(\Trg(\tau;-\al,\be)),\\
    I_4(T)    &=I_1(\Trg(\tau;\al,-\be)).
  \end{aligned}
\end{equation}
\noindent
\noindent
The typed spectra of an exceptional $T=\Trg(\tau;\al,\be)$ can be
determined directly in terms of the index sets as follows.  Set
\begin{equation}
  \label{eq:lam1234}
  \begin{aligned}
  \lam_{1}(k;\al,\be) &:= k(k+\al+\be+1),\\
  \lam_2(k;\al,\be)&=(k-\al-\be)(k+1)\\
  \lam_{3}(k;\al,\be) &:= (k-\al)(k+\be+1)\\
  \lam_{4}(k;\al,\be) &:= (k-\be)(k+\al+1) \\
\end{aligned}
\end{equation}
The following result is not surprising, but we postpone the proof to
Section ~\ref{subsec:proofs:qr}.
\begin{prop}
  \label{prop:sig1234}
  For $\imath\in \{1,2,3,4\},\; k\in I_\imath(T)$, we have
  $\lam_\imk = \lam_\imath(k;\al,\be)$.
\end{prop}

\noindent
Another virtue of the above indexing scheme is that, generically,
$I_1$ and $-I_2-1$ partition $\Z$ and can be conveniently visualized
as a Maya diagram.  The same remark holds for $I_3$ and $I_4$.  Thus,
for generic exceptional Jacobi operators the qr-spectrum and
eigenfunctions can be represented in terms of two Maya diagrams
\cite{Bonneux20}.  This way of representing the qr-eigenvalues and
eigenfunctions is described in detail in Section ~\ref{sec:sd} in terms
of a concept called the \emph{spectral diagram}.

Regardless of how the qr-eigenfunctions are distributed into the four
types above, their union always has the same simple form.  For
$\al,\be\in \R$, define\footnote{The notation $f(I)$ denotes a set
  obtained by applying a function $f$ to the elements of a set $I$.}
\begin{equation}
  \label{eq:sigqdef}
  \sigq(\al,\be)
  := \lam_1(\Z;\al,\be) \cup \lam_1(\Z-\al;\al,\be),
\end{equation}
The following result is proved in Section ~\ref{subsec:proofs:spec.diag}.
\begin{prop}
  \label{prop:sigq}
  Let $T=\Trg(\tau;\al,\be)$ be an exceptional Jacobi operator, then its quasi-rational spectrum is given by
  \begin{equation}
    \label{eq:sigqTalbe}
    \sigq(T) = \sigq(\al,\be).
  \end{equation}
\end{prop}

\noindent
By way of illustration, the quasi-rational eigenfunctions and
eigenvalues of a classical Jacobi operator are displayed in Table
~\ref{tab:stype}.
\begin{table}[h]
  \centering
  \caption{Quasi-rational eigenfunctions and eigenvalues of the classical Jacobi operator $T(a,b)$ with generic
  parameter values $a,b,a\pm b\notin\Z$. Here $P_n(x;a,b)$ is the
  classical Jacobi polynomial of degree $n$.}
  \label{tab:stype}
  \begin{tabular}{|c|l|l|}
  \hline
    Type & Eigenfunction & Eigenvalue   \\ \hline
    $1$
          & $\phi_{1n}:=P_n(x;a,b)$ & $\lam_{1}(n;a,b)$  \\ \hline 
    $2$
          & $\phi_{2n}:=(1-x)^{-a}(1+x)^{-b}P_{n}(x;-a,-b)$    
                          & $\lam_{2}(n;a,b)$  
    \\ \hline
    $3$ & $\phi_{3n}:=(1-x)^{-a} P_n(x;-a,b)$
                          & $\lambda_{3}(n;a,b)$\\ \hline 
    $4$ & $\phi_{4n}:=(1+x)^{-b}P_{n}(x;a,-b)$
                          & $\lam_{4}(n;a,b)$ \\\hline
\end{tabular}

\end{table}

\subsection{The generic, semi-degenerate and degenerate subclasses}
A key observation is to note that for some values of $(\al,\be)$, an
exceptional Jacobi operator $\Trg(\tau;\al,\be)$ might have linearly
independent qr-eigenfunctions of different types at the same
qr-eigenvalue.  Otherwise speaking, the union ~\eqref{eq:union} is not
always a disjoint union. The study of these degenerate cases motivates
to introduce the following definitions.

\begin{definition}
\label{def:degeneigen}
  If $\Riclam{T}$ contains exactly one element, we call $\lam$ a
 \emph{simple} qr-eigenvalue. If it contains multiple elements, then we
  speak of a \emph{degenerate} qr-eigenvalue.
\end{definition}
\noindent
Generically, each qr-eigenvalue of an exceptional $\Trg(\tau;\al,\be)$
is simple and can therefore be uniquely labelled by the asymptotic
type of the corresponding qr-eigenfunction.  However, if
$\al,\be$, $\al+\be$ and/or $\al-\be$ take integral values, then certain qr-eigenfunctions can coalesce to produce a degenerate eigenvalue. These
types of coalescence can occur due to various symmetries in the
qr-spectrum ~\eqref{eq:lam1234} and motivates the introduction of six
different classes.


\begin{definition}
  \label{def:degenclass}
  According to the values of $\al,\beta$, we define the following six
  distinct \emph{degeneracy classes} of exceptional Jacobi operators:
  \begin{itemize}
  \item Type G (Generic) if $\al,\be, \al\pm\be \notin\Z$,
  \item Type A (Semi-degenerate) if $\al\in\Nz$ and $\be\notin\Z$ or
    if $\al\notin \Z$ and $\be\in \Nz$,
  \item Type B (Semi-degenerate) if $\al,\be,\al+\be\notin\Z$ but
    $\al-\be\in\Z$,
  \item Type C (Semi-degenerate) if $\al,\be,\al-\be\notin\Z$ but
    $\al+\be\in\Z$,
  \item Type CB (Semi-degenerate, generalized Chebyshev) if
    $\al,\be\notin \Z$ but   $2\al,2\be\in \Z$,
  \item Type D (Degenerate) if  $\al,\be\in\Nz$.
  \end{itemize}
  We will refer to operators in class G as being \emph{generic}, to
  operators in classes A,B,C,CB as being \emph{semi-degenerate}, and
  to operators in class D as being \emph{degenerate}.
\end{definition}
\noindent
The two type A subclasses are connected by the transformation
$x\mapsto -x,\; \al\leftrightarrow \be$. Therefore, by Proposition
~\ref{prop:gaugesym}, there is no loss of generality in considering
only type A operators where $\al\in \Nz$ and $\be\notin \Z$.

The following result identifies the possible qr-eigenvalues, according to Definition ~\ref{def:degeneigen}, for each operator type.
\begin{prop}\label{prop:degen}
  Let $\lam\in\sigq(\al,\be)$ be a qr-eigenvalue of an exceptional
  Jacobi operator $T=\Trg(\tau;\al,\be)$. If $T$ is generic, then
  $\Riclam{T}$ contains exactly one element.  If $T$ is
  semi-degenerate, then $\Riclam{T}$ contains either one or two
  elements.  If $T$ is degenerate, then $\Riclam{T}$ contains either one
  or infinitely many elements.
\end{prop}
\noindent The result follows directly by
Proposition ~\ref{prop:lsfa}, which uses the notion of asymptotic label
introduced in Definition ~\ref{def:lamlabel}.
\begin{remark}\label{rem:degen}
  Since $\Trg(\tau;\al,\be)$ is a second order differential operator, its
  eigenfunctions at any given $\lambda$ span a 2-dimensional linear
  subspace. We are concerned only with qr-eigenfunctions, so the three
  possibilities in Proposition~~\ref{prop:degen} describe the cases
  where:
  \begin{enumerate}
    \item Only one eigenfunction is quasi-rational.
    \item There are two linearly independent qr-eigenfunctions but their linear combination is not quasi-rational.
    \item There is a 2-dimensional space of rational eigenfunctions.
\end{enumerate}
\end{remark}

\subsection{The spectral diagram.}
\label{sec:sd}
In this section we introduce the fundamental object in the
classification of exceptional Jacobi operators. We first introduce a
collection of labels to describe the number and asymptotic type of the
qr-eigenfunctions of an exceptional Jacobi operator $\Trg(\tau;\al,\be)$ at simple
and degenerate qr-eigenvalues.

\begin{definition}
  \label{def:lamlabel}
  Let $T=\Trg(\tau;\al,\be)$ be an exceptional Jacobi operator. For a
  qr-eigenvalue $\lam \in \sigq(T)$ we define the \emph{$\lam$-label} to be
  one of the following 8 symbols:
  $\boxcirc, \boxtimes,
  \boxplus,\boxminus,\boxstar,\boxdiv,\boxotimes,\boxfsq$, according
  to the asymptotic types of the elements of $\Riclam{T}$.  The
  meaning of these labels is indicated by Table~~\ref{tab:symbols},
  with an explicit definition below.
  \begin{itemize}
  \item The first four labels correspond to a singleton $\Riclam{T}$
    containing a rational function of type 1,2,3,4, respectively.
  \item The $\boxstar$ label, a confluence of the $\boxtimes$ and
    $\boxplus$ symbols, indicates that $\Riclam{T}$ is a two element set
    containing eigenfunctions of type 2 and 3.  
  \item The $\boxdiv$ label, a confluence of the $\boxplus$ and
    $\boxminus$ symbols, indicates that $\Riclam{T}$ is a two element
    set containing eigenfunctions of type 3 and 4.
  \item The $\boxotimes$ label, a confluence of the $\boxcirc$ and
    $\boxtimes$ symbols, indicates that $\Riclam{T}$ is a 2-element
    set containing eigenfunctions of type 1 and 2. 
  \item Finally, the $\boxfsq$ label, a confluence of the
    $\boxtimes,\boxplus,\boxminus$ symbols, indicates that
    $\Riclam{T}$ is an infinite set containing eigenfunctions of types
    2,3 and 4 (the third case in Remark~~\ref{rem:degen}).
  \end{itemize}
\end{definition}

\begin{center}
\begin{table}[h]
  \caption{Asymptotic types and labels}
  \begin{tabular}{|c|c|c|c|c|}
    \hline 
    $\lam$-label & $ |\Riclam{T}|$ & Class & Asymptotic types\\
    \hline 
    $\boxcirc$ & 1 & G,A,B,C,CB,D & 1\\
    \hline 
    $\boxtimes$ & 1 & G,B,C,CB & 2\\
    \hline 
    $\boxplus$ & 1 & G,A,B,C,CB,D & 3\\
    \hline 
    $\boxminus$ & 1 & G,A,B,C,CB, D & 4\\
    \hline 
    $\boxstar$ & $2$ & A & 2,3 \\
    \hline 
    $\boxpm$ & $2$ & B,CB & 3,4\\ 
    \hline 
    $\boxotimes$ & $2$ & C,CB & 1,2\\ 
    \hline 
    $\boxfsq$ & $\infty$ & D & 2,3,4 \\ 
    \hline 
  \end{tabular} 
  \label{tab:symbols}
\end{table}
\end{center}

We are now ready to introduce the fundamental object employed in the
classification of exceptional Jacobi operators. 
\begin{definition}
  The spectral diagram of an exceptional Jacobi operator
  $T=\Trg(\tau;\al,\be)$ is the assignment of the corresponding
  $\lam$-label to each $\lam\in\sigq(T)$.
\end{definition}
Thus, the spectral diagram of an exceptional Jacobi operator
$\Trg(\tau;\al,\be)$ includes the information on the asymptotic types
of all of its qr-eigenfunctions.

\subsection{The main results.}
In this section we  state the main results of this paper.
First of all, we see that the information included in the spectral diagram suffices
to completely determine the operator in all cases except in the fully
degenerate D class. This  result will be proved in Section ~\ref{subsec:proofs:main}.
\begin{thm}
  \label{thm:sd}
  Let $\Trg(\tau;\al,\be)$ be an exceptional Jacobi operator belonging
  to degeneracy classes G,A,B,C or CB.  Then, the spectral diagram of
  $\Trg(\tau;\al,\be)$ fully determines the operator.
\end{thm}

\noindent
In the D class, exceptional Jacobi operators can have an arbitrary
number of real parameters.  Thus, within the  D class, a complete
classification of class D operators requires an extension of the notion
of a spectral diagram.

\begin{definition}
  \label{def:isodef}
  We say that two exceptional Jacobi operators are \emph{isospectral}
  if they have the same quasi-polynomial spectrum.  We define an
  \emph{isospectral deformation} to be a family of isospectral
  exceptional operators in which the corresponding quasi-polynomial
  norms $\nu_i,\; i\in I_1$, relative to some choice of normalization,
  assume an arbitrary non-zero value for a finite number of $i\in I_1$
  and are constant for the other indices.  We will refer to the former
  subset of $I_1$ as the \emph{support} of the isospectral
  deformation.
\end{definition}

\begin{definition}
  \label{def:ESD}
  For a given exceptional class D operator
  $T=\Trg(\tau;\al,\be),\; \al,\be\in \Nz$ and the corresponding index
  set $I_1=I_1(T)$, let\footnote{\tre{See Remark \ref{rem:D} for the
    motivation behind these definitions.}}
  \begin{equation}
    \label{eq:I1pm}
    I_{1+} = \{ i\in I_1: i\ge (\al+\be+1)/2 \},\quad I_{1-}= \{ i
    \in I_1 : i< (\al+\be+1)/2 \}.    
  \end{equation}
  An \emph{extended spectral diagram} of $T$ is a modification of the
  spectral diagram where the eigenvalues of quasi-polynomials
  eigenfunctions of degree $i$ are labelled by $\boxcirc$ if
  $i\in I_{1+}$ and labelled by $\boxtdown$ if $i\in I_{1-}$.  We will
  refer to $\{ \boxcirc, \boxtdown,\boxfsq, \boxplus,\boxminus\}$ as
  the \emph{extended} type D $\lam$-label set.
\end{definition}
\begin{thm}
  \label{thm:esd}
  Isospectral deformations do not exist outside the degenerate D
  class. Within the D class, the support of an isospectral deformation
  is $I_{1-}$. \tre{The latter is finite, but can have an arbitrarily
    large cardinality.}  There is a one-to-one correspondence between
  isospectral deformations and extended spectral diagrams.  Within a
  particular isospectral deformation family, the value of the norms
  $\nu_i,\; i\in I_{1-}$ fully determines the corresponding
  exceptional operator.
\end{thm}
\noindent
The proofs of these results will be given in Section
~\ref{subsec:proofs:main}.

Thanks to Theorem ~\ref{thm:sd} and Theorem ~\ref{thm:esd}, the
classification methodology of exceptional Jacobi operators reduces to
two principles: (i) a combinatorial description of all possible
spectral diagrams; and (ii) a procedure for mapping a given spectral
diagram into the corresponding exceptional operator (or family of
operators in the case of an extended diagram).

The class of exceptional operators admits a natural operation called a
rational Darboux transformation (RDT) that relates pairs of
exceptional operators.  We postpone the discussion of RDTs until
Section ~\ref{sec:RDT} below.  At this point, it suffices to note the
following fundamental result.
\begin{thm}[Theorem 1.2 of \cite{GFGM19}]
  \label{thm:XOPFT2}
  Every exceptional
  operator is related to a classical operator by a finite chain of
  rational Darboux transformations.  
\end{thm}

\noindent
As a consequence of Theorem ~\ref{thm:XOPFT2}, we have the
following result, whose proof is given in Section ~\ref{subsec:proofs:spec.diag}.
\begin{thm}
  \label{thm:flip}
  The spectral diagram of every exceptional Jacobi operator is related
  to a spectral diagram of a classical Jacobi operator by a spectral
  shift and by finitely many label changes.
\end{thm}
\noindent
The combinatorial part of the classification then takes the following
simple shape: describe a certain canonical set of classical spectral
diagrams, and then devise a system of parameters for describing all
finite modifications of these canonical diagrams.

The sections that follow are devoted to such constructions on a
case-by-case basis for the GABCD degeneracy classes.  Each section
concludes with an explicit procedure for converting the combinatorial
data into the corresponding exceptional operator and corresponding
quasi-polynomial eigenfunctions.

\section{The GABCD classification and corresponding construction
  formulas}
\label{sec:GABCD}
This section is the heart of the classification results announced
above.  Each of the six degeneracy classes G,A,B,C,CB,D requires a
separate treatment, but the organizational scheme is the same for each
class.
\begin{itemize}
\item Each degeneracy class features a system of parameters to
  describe the spectral diagram of an exceptional operator as a finite
  modification of a classical operator belonging to the same class.
\item In each case, there is a constructive procedure for converting
  said parameters into the triple $(\tau;\al,\be)$.  This procedure
  always involves Wronskian differential operators, but for the (A)
  and (D) classes, there is a preliminary stage of the construction
  based on integral operators and determinantal representations.
\item For each degeneracy class, there is an assertion to the effect
  that this construction produces an exceptional operator in that class.
\item For each degeneracy class, there is a description of a canonical choice
  of parameters that gives a unique description of every exceptional
  operator in that class.
\end{itemize}

In essence, the parameters describe a chain of rational Darboux
transformations (RDTs) and confluent Darboux transformations (CDTs) that
connect a classical operator $T(a,b)$ with the exceptional operator
$\Trg(\tau;\al,\be)$; see Section ~\ref{sec:RDT} for details.  In each
case, there is an explicit description of how these RDTs transform the
index sets of the classical $T(a,b)$ into the index sets of the
exceptional $\Trg(\tau;\al,\be)$, and an enumeration of the label
flips that relate the classical and the exceptional spectral diagrams.

Each of the subsections begins with a technical description of the
relevant constructions and assertions.  This is done with a minimum of
explanation in order to gather the technical material in a compact and
accessible fashion.  For the same reason, the proofs are postponed
until Section ~\ref{sec:proofs}. 
Namely, the proofs of the results of Section~\ref{sec:genclass} can be found in Section~\ref{subsec:proofs:generic}, those of Sections~\ref{sec:semidegA},~\ref{sec:semidegB},~\ref{sec:semidegC} in Section~\ref{subsec:proofs:semideg}, and lastly, those of Section~\ref{sec:D} in Section ~\ref{subsec:proofs:deg}.

Each technical subsection is then accompanied by a more informal
discussion with illustrations and examples.  These remarks should be
viewed as a kind of ``user manual'' for spectral diagrams of that
class.  They explain the nature of the degeneracy that gives rise to
the class in question and motivate the relevant constructions.  A
useful approach for reading these sections is therefore to skim the
technical details and to pass quickly to the ``user manual''.  Having
perused these expositions, the reader can then return to the technical
details with an improved understanding and motivation.

Class G,B,C,CB operators do not require CDTs for their construction.
Essentially, the parameter sets for these classes consist of sets
$K_1,K_2,K_3,K_4$ that enumerate the type 1,2,3,4 RDTs that connect
the classical and the exceptional operators.    However, class G
operators can all be constructed using only type 1 and type 3
transformations.  Class B operators require the use of 1,3,4
transformations.  Class C operators require 1,2,3 type
transformations, while class CB operators require all four types of RDTs
in their construction.

The class G and class B formulas are sufficiently similar so that it
makes sense to combine them into a single statement.  The same remark
holds for class C and class CB.  This is the reason why the class G
section and formulas feature the $K_4$ parameter set.  The $K_4$
parameters are superfluous for the G operators, but including them
allows the class G formulas to be reused without repetition to
construct the class B operators.  The same remark applies to the C and
CB classes.

\subsection{The generic class G}
\label{sec:genclass}
Let $K_1,K_3,K_4\subset \Nz$ be sets of cardinalities $p_1,p_3,p_4$,
respectively\footnote{In the type G context, no generality is lost if
  we take $K_4=\emptyset$ and $p_4=0$.  The more general formulas will
  be needed for the description of the type B class, below.  We
  introduce them here in order to avoid repetition.}  , and
$a,b,a\pm b \notin \Z$.
Write $\bK_\imath = -K_\imath-1$ 
and define
\begin{equation}
  \label{eq:indexGB}
  \begin{aligned}
    &I_1=  (\Nz\setminus K_1)-p_1,& I_2 &= (\Nz\cup \bK_1)+p_1,\\
    &I_3=  ((\Nz\setminus K_3)\cup \bK_4)-p_3+p_4,& I_4 &=
    ((\Nz\setminus K_4) \cup
    \bK_3)+p_3-p_4  
  \end{aligned}
\end{equation}
Set
\begin{equation}
  \label{eq:albeG}
  \al = a+p_1-p_3+p_4,\quad  \be = b+p_1+p_3-p_4,
\end{equation}
and define
$\sigma_\imath = \lam_{\imath}(I_\imath;\al,\be),\; \imath =
1,2,3,4$, with the latter as per ~\eqref{eq:lam1234}.
\begin{prop}
  \label{prop:sigG}
  The type G assumptions on $a,b$ entail that $\al,\be,\al\pm \be \notin \Z$
  also.  Moreover,
  $\sigq(\al,\be) = \sigma_1 \sqcup \sigma_2\sqcup
  \sigma_3\sqcup\sigma_4$ (disjoint union).
\end{prop}
\begin{proof}
  By Proposition ~\ref{prop:classG}, the conclusions hold if
  $K_1=K_3=K_4=\emptyset$.  The general case follows by direct
  inspection of the above definitions.  It suffices to observe that
  $K_1$ moves eigenvalues from $\sigma_1$ to $\sigma_2$, while $K_3$
  moves eigenvalues from $\sigma_3$ to $\sigma_4$ and $K_4$ does just
  the opposite.
\end{proof}
\noindent
Let
$\LamG(K_1,K_3,K_4,a,b)\colon \sigq(\al,\be) \to \{
\boxcirc,\boxtimes,\boxplus,\boxminus\}$ be the mapping with action
  \begin{equation}
    \label{eq:sdG}
    (\sigma_1,\sigma_2,\sigma_3,\sigma_4)\mapsto
    (\boxcirc,\boxtimes,\boxplus,\boxminus ). 
  \end{equation}

\noindent
Let
$\phi_{\imath n}=\phi_{\imath n}(x;a,b),\; n\in
\Nz,\imath\in\{1,2,3,4\}$ denote the four types of qr-eigenfunctions
of the classical operator $T(a,b)$ as shown in Table ~\ref{tab:stype}.
Let $k_{\imath1}<\ldots < k_{\imath p}$ be the increasing
enumeration of $K_\imath$ and let
\begin{equation}
  \label{eq:philist}
  \phi_{\imath}(K_\imath) = \phi_{\iota k_1},\ldots, \phi_{\imath k_p},
\end{equation}
denote the indicated list.  Let $\phi(K_1,K_3,K_4)$ denote the
concatenated list
$\phi_1(K_1),\phi_3(K_3),\phi_4(K_4)$, and let
\begin{equation}
  \label{eq:tKB}
  \tK := K_1 \cup (K_3-a)\cup(K_4-b) 
\end{equation}
 be the set of degrees
of the qr-functions present in that list.  With the above definition
in place, set
\begin{align}
  \label{eq:tauGB}
  \tau(x;K_1,K_3,K_4,a,b)
  &:= (1-x)^{p_3(a+p_1+p_4)}(1+x)^{p_4(b+p_1+p_3)}
    \Wr[\phi(K_1,K_3,K_4)], \\ 
   \label{eq:pig}
  \pi_{i}(x;K_1,K_3,K_4,a,b)
  &=     \frac{(x-1)^{p_3}(1+x)^{p_4}}{\prod_{k\in \tK} (i+p_1-k)}
    \frac{\Wr[\phi(K_1,K_3,K_4),\ckpi_{i+p_1}]}{
    \Wr[\phi(K_1,K_3,K_4)]},\quad     i\in I_1,
\end{align}
where $\ckpi_n=\pi_n(x;a,b)$ is the classic monic Jacobi polynomial defined
in ~\eqref{eq:pindef}, and where $\Wr$ denotes the Wronskian
determinant with respect to $x$.

\begin{prop}
  \label{prop:TG}
  The $\tau=\tau(x;K_1,K_3,K_4,a,b),\; a,b,a\pm b\notin \Z$, as
  defined above, is a polynomial with
  \begin{equation}
    \label{eq:degG}
     \deg \tau =
     \!\!\sum_{k\in K_{134}}\!\!\!\! k
    - \binom{p_1}{2} -
    \binom{p_3}{2}- \binom{p_4}{2} + p_3 p_4,\quad
    K_{134} = K_1\cup K_3 \cup K_4.
  \end{equation}
  Moreover, $T=\Trg(\tau;\al,\be),$ with $\al,\be$ as defined in
  ~\eqref{eq:albeG}, is a type G exceptional operator with spectral
  diagram $\Lam_T=\LamG(K_1,K_3,K_4,a,b)$ and index sets
  $I_1,I_2,I_3,I_4$ as defined above. The corresponding
  quasi-polynomial eigenfunctions, with the monic choice of
  normalization, are given by the above
  $\pi_i(x;K_1,K_3,K_4,a,b),\; i\in I_1$.  The corresponding norms are
  \begin{align}
    \label{eq:nuiG}
    \nu_i &= \kappa(i+p_1) \nu(i+p_1;a,b),\; i\in I_1 \intertext{where
            $\nu(z;a,b)$ is defined in ~\eqref{eq:nunab}, and where}
    \label{eq:kappaG}
    \kappa(z) &= \prod_{k\in \tK} \frac{z-k^*}{z-k},\quad
    k^* := -k-a-b-1.
  \end{align}
\end{prop}

\begin{prop}
  \label{prop:GSD}
  Let $T=\Trg(\tau;\al,\be),\; \al,\be,\al\pm \be \notin Z$ be a type
  G exceptional operator and $\Lambda_T$ its spectral diagram. Then,
  $\Lambda_T=\LamG(K_1,K_3,\emptyset,a,b)$ for a unique choice of
  $K_1,K_3,a,b$ such that ~\eqref{eq:albeG} holds, and such that
  $0\notin K_1,\, 0\notin K_3$.  Moreover, up to a scale factor,
  $\tau=\tau(x;K_1,K_3,\emptyset,a,b)$ with the latter as per
  ~\eqref{eq:tauGB}.
\end{prop}

\begin{remark}
  \label{rem:G}
  It is convenient to visualize the class G spectral diagram as two
  Maya diagrams, one row for the type 1,2 eigenvalues, and another row
  for the type $3,4$ eigenvalues\footnote{This is very similar to
    Bonneux's system in \cite{Bonneux20}.}.  Each cell in these two rows
  represents a qr-eigenvalue $\lam\in\sigq(\al,\be)$ and is decorated
  by the asymptotic label associated to $\lam$.  The type 1,2
  eigenvalues are grouped together, because
  $\lam_2(i;\al,\be) = \lam_1(\bar{i};\al,\be)$, where
  $\bar{i} = -i-1$, and because $I_{12}=I_1\sqcup \bI_2 = \Z$, where
  $\bI_2=-I_2-1$ and where $\sqcup$ denotes disjoint union.  The
  $\boxcirc$ labels are placed at positions in $I_1$; the $\boxtimes$
  labels are placed at positions in $\bI_2$.  This works, because
  $\al+\be\notin \Z$ and hence the mapping
  \[ i\mapsto \lam_1(i;\al,\be)= i(i+\al+\be+1),\quad i\in \Z \] defines
  a bijection between the $I_1\sqcup \bI_2$ indices
  and the $\sigma_1\sqcup \sigma_2$ eigenvalues.

  Likewise, the type 3,4 eigenvalues are grouped together because
  $\lam_4(i;\al,\be) = \lam_3(\bar{i};\al,\be)$.  The $\boxplus$
  labels are placed at positions in $I_3$ and the $\boxminus$ labels
  at positions in $I_4$. Since $\al-\be\notin \Z$, the mapping
  \[ i \mapsto \lam_3(i;\al,\be) = (i-\al)(i+\be+1),\quad i\in \Z \]
  defines a bijection between the $I_3\sqcup \bI_4$ indices and the
  $\sigma_3\sqcup \sigma_4$ eigenvalues. 
  
  To recover the parameters $K_1,K_3,K_4,a,b$ from the spectral
  diagram, one places the origin $k=0$ of the first array to the
  leftmost $\boxcirc$ label, and places the origin of the second array
  to the leftmost $\boxplus$ label.  The elements of $K_1$ are the $k$
  positions of the $\boxtimes$ labels; the elements of $K_3$ are the
  $k$ positions of the $\boxminus$ labels; while $K_4=\emptyset$.
  \footnote{ Even though a type G diagram can be constructed without
    using $\boxminus \to \boxplus$ flips, such flips cannot be avoided
    in the case of type B diagrams.  We include $K_4$ in our
    definitions so that we can reuse the these formulas when we
    discuss class B operators below.}.  The values of $a,b$ can be
  recovered from the spectral diagram by considering the difference of
  the qr-eigenvalues at $k=0$ and $k=-1$ in the first diagram and the
  difference of the qr-eigenvalues at $k=0$ and $k=-1$.  This works
  because
  \begin{gather*}
    \lam_1(0;a,b) - \lam_1(-1;a,b) = a+b\\
    \lam_3(0;a,b) - \lam_3(-1;a,b) = a-b.
  \end{gather*}

  The diagram $\LamG(K,a,b)$ is related to the classical diagram
  \[\Lam(a,b) := \LamG(\emptyset,\emptyset,\emptyset,a,b)\] by
  flips $\boxcirc\mapsto \boxtimes$ at the indices in $K_1$ and by
  $\boxplus\leftrightarrow \boxminus$ flips at the indices at
  $K_3, K_4$, and by a spectral shift of $p_1(a+b+1+p_1)$.  As
  described in Section ~\ref{sec:prdt} below, classical Jacobi
  operators can be interrelated by RDTs.  Since exceptional operators
  are constructed by a chain of RDTs that begin with a classical
  operator, a given exceptional operator can be connected to a
  classical operator in multiple ways.  In other words, the spectral
  diagram $\LamG(K_1,K_3,K_4,a,b)$ does not uniquely determine the
  parameters $K_1,K_3,K_4,a,b$.  However, by placing the origin $k=0$
  at the leftmost $\boxcirc$ label and the leftmost $\boxplus$ label,
  is equivalent to the conditions that $0\notin K_1$ and that
  $0\notin K_3$ and $K_4=\emptyset$.  This ensures the uniqueness of
  the parameters.  An example of a class G spectral diagram is shown
  in Figure ~\ref{fig:gensd} below.  It is useful to compare it to the
  generic spectral diagram of a classical Jacobi operator  shown in
  Figure ~\ref{fig:trivspec}.
\end{remark}

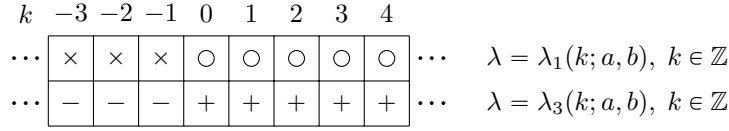
\begin{figure}[H]
  \centering
  \begin{tikzpicture}[scale=0.6]

    \draw (0.5,4.5) node {$k$};
    \foreach \x in {-3,...,4} \draw (\x+4.5,4.5) node {$\x$};
    \draw (1,2) grid +(8 ,2);

    \path [draw,color=black]
    (0.5,3.5) node {\huge ...}
    ++(1,0) node {$\times$}
    ++(1,0) node {$\times$}
    ++(1,0) node {$\times$}
    ++(1,0) circle (5pt)
    ++(1,0) circle (5pt)
    ++(1,0) circle (5pt)
    ++(1,0) circle (5pt)
    ++(1,0) circle (5pt)
    ++(1,0) node {\huge ...}
    +(1,0) node[anchor=west]
    {$\lam=\lam_{1}(k;a,b),\; k\in \Z$};

    \path [fill,color=black]
    (0.5,2.5) node {\huge ...}
    ++(1,0) node  {$-$}
    ++(1,0) node {$-$}
    ++(1,0) node {$-$}
    ++(1,0) node {$+$}
    ++(1,0) node {$+$}
    ++(1,0) node {$+$}
    ++(1,0) node {$+$}
    ++(1,0) node {$+$}
    ++(1,0) node {\huge ...}
    +(1,0) node[anchor=west]
    {$\lam=\lam_{3}(k;a,b),\; k\in \Z$};;
  \end{tikzpicture}
  \caption{Spectral diagram $\LamG(a,b)$ of the
    generic classical Jacobi operator
    $T(a,b),\; a,b,a\pm b\notin \Z$}
    \label{fig:trivspec}
\end{figure}

\begin{figure}[H]
  \centering
\begin{tikzpicture}[scale=0.6]

  \def\y{0}
  \draw (0,\y+2.5) node {$\bk$};
  \foreach \x in {-6,...,1}  \draw (-\x+3.5,\y+2.5) node {$\x$};

  \draw (0,\y+1.5) node {$k$};
  \foreach \x in {-2,...,5}  \draw (\x+4.5,\y+1.5) node {$\x$};

  \draw (2,\y-1) grid +(8 ,2);

  \path (11.5,\y+2.5) node[anchor=west]
  {$K_1 = \{ 2,4 \},\; K_3 = \{ 1,2,3,4\} $};

  \path (11.5,\y+1.5) node[anchor=west]
  {$I_1=\{0,1,3,5,\ldots\}-2$};
  \path (11.5,\y+0.5) node[anchor=west]
  {$I_2=\{-5,-3,0,1,\ldots\}+2$}   ;
  \path (11.5,\y-0.5) node[anchor=west]
  {$I_3=\{0,5,\ldots\}-4$};
  \path (11.5,\y-1.5) node[anchor=west]
  {$I_4=\{-5,-4,-3,-2,0,1,\ldots\}+4$}   ;

  \path [draw,color=black] (-0.5,\y+0.5)
  ++(2,0) node {\huge ...}
  ++(1,0) node {$\times$}
  ++(1,0) node {$\times$}
  ++(1,0) circle (5pt)
  ++(1,0) circle (5pt)
  ++(1,0) node {$\times$}
  ++(1,0) circle (5pt)
  ++(1,0) node {$\times$}
  ++(1,0) circle (5pt)
  ++(1,0) node {\huge ...};

  \path [fill,color=black]  (-0.5,\y-0.5) 
  ++(2,0) node {\huge ...}
  ++(1,0) node {$-$}
  ++(1,0) node {$-$}
  ++(1,0) node {$+$}
  ++(1,0) node {$-$}
  ++(1,0) node {$-$}
  ++(1,0) node {$-$}
  ++(1,0) node {$-$}
  ++(1,0) node {$+$}
  ++(1,0) node {\huge ...};

\end{tikzpicture}
  
\caption{Spectral diagram $\LamG(K_1,K_3,a,b)$ with
  $\al=a-2,\; \be=b+6$ and qr-eigenvalues
  $\lam_{1}(k-2;\al,\be) = \lam_{1}(k;a,b)+2(a+b+3),\; k\in \Z$, and
  $\lam_{3}(k-4;\al,\be) = \lam_{3}(k,a,b)+2(a+b+3),\; k\in \Z$}
  \label{fig:gensd}
\end{figure}


\subsection{Semi-degenerate  class A}
\label{sec:semidegA}
Suppose that $a\in \Nz, b\notin \Z$ and let $\ckI_1=\ckI_2=\Nz$ and
\begin{equation}
  \label{eq:ckIA}
  \ckI_{2} =  \ckI_3 = \{ 0,\ldots, a-1\}.
\end{equation}
Let $K,L\subset \Nz$ be finite sets of cardinalities $p,q$
respectively such that $K\cap L = \emptyset$.  Set
$\bK = -K-1, \bL=-L-1$ and
\[ \al = a+p,\quad \be = b+p+2q.\] Let
$\sigma_\imath = \lam_{\imath}(I_\imath;\al,\be),\; \imath = 1,2,3,4$,
where
\begin{align*}
  I_1 &= (\ckI_1\setminus (K\cup L))-p-q,
             &  I_4 &= (\ckI_4 \cup          \bL)+q,\\
  I_{2} &= (\ckI_2\cup \bK)+p+q &   I_3 &=(\ckI_3\cup (K+a))-q.
\end{align*}
\begin{prop}
  \label{prop:sdA}
  The assumptions that $a\in \Nz$ and $b\notin \Z$ entail that
  $\al\in \Nz$ and $\be\notin \Z$ also.  Moreover, $\sigma_2=\sigma_3$
  and $\sigq(\al,\be) = \sigma_1 \sqcup \sigma_2\sqcup\sigma_4$
  (disjoint union).
\end{prop}

\noindent
Making use of the above decomposition, we let $\LamA(K,L,a,b)$ be the
mapping 
\begin{equation}
  \label{eq:sdA}
  (\sigma_1,\sigma_2,\sigma_4)\to  (\boxcirc,\boxstar,\boxminus ).
\end{equation}

For type A operators, the generic Wronskian-based definitions
~\eqref{eq:tauGB} fail because the eigenvalues of the type 1 and the
type 3 classical eigenfunctions may not have distinct
eigenvalues\footnote{In order to apply Wronskian-type formulas in the
  confluent case, one has to make use of generalized eigenfunctions.
  This is fully explained in \cite{paper2}}. We have to introduce a
new mechanism to handle such spectral transformations and divide the
construction into two stages.  As the first stage of our construction,
consider the case where $K=\emptyset$. Let
\[ \rho_{ij}(x;a,b) =  \int \pi_i(x;a,b) \pi_j(x;a,b)
  (1-x)^a(1+x)^b ,\quad i,j\in \Nz,\; a\in \Nz,\,b\notin \Z.\] This
is uniquely defined, because the integrand is a polynomial times
$(1+x)^b$, and because
\begin{equation}
  \label{eq:intx+1}
  \int(1+x)^{b+k} = \frac{(1+x)^{b+k+1}}{b+k+1},\quad b\notin \Z,\; k\in \Z.
\end{equation}
is the unique quasi-rational anti-derivative.  Let
$\ell_1< \ell_2 < \cdots < \ell_q$ be the increasing enumeration of
$L$. Let $\cR=\cR(x;L,a,b)$ be the $q\times q$ matrix with entries
\[\cR_{ij}=\rho_{\ell_i\ell_j}(x;a,b),\; i,j=1,\ldots, q.\] Let
$\cA(n) = \cA(x;n,L,a,b),\; n\in \Nz$ denote the augmented
$(1+q)\times (1+q)$ matrix obtained from $\cR$ as follows:
\begin{align}
  \cA(n)_{00}
  &=\pi_n(x;a,b)
  & \cA(n)_{0i}
  &=\pi_{\ell_i}(x;a,b) \\
  \cA(n)_{i0}&=\rho_{\ell_i n}(x;a,b)
  & \cA(n)_{ij}
  &=
    \rho_{\ell_i\ell_j}(x;a,b)
\end{align} 
where $i,j$ have the range $1,\ldots, q$.  Next, set
$\hI_1 = (\Nz\setminus L)- q$, and define
\begin{align}
  \label{eq:htauA}
  \htau(x;L,a,b)
  &:=  (1+x)^{-q(q+b)} \det \cR(x;L,a,b),\\
  \label{eq:hpiA}
  \hpi_{i}(x;L,a,b)
  &:= \hchi(i+q;L,a,b) (1+x)^{-q}\frac{\det \cA(x;i+q,L,a,b)}{\det
    \cR(x;L,a,b)}, \quad i\in \hI_1,
    \intertext{where}
    \label{eq:hchiA}
    \hchi(z;L,a,b)
  &=    \prod_{k\in L}\frac{z-k^*}{z-k},    \quad    k^* = -k-a-b-1.
\end{align}
\noindent
Speaking informally, the effect of the first stage of the construction
is the label transformation  $\boxcirc \to \boxstar \to \boxminus$ at
positions indexed by $L$.  The idea of label transformations is fully
explained in Section ~\ref{sec:flips}.

The second stage of the construction corresponds to label
transformations $\boxcirc \to \boxstar$ at positions indexed by $ K$.
Let $k_1<\cdots <k_p$ be the the increasing enumeration of $K-q$, and
set
\begin{align}
  \label{eq:tauA}
  \tau(x;K,L,a,b)
  &:=  \htau\Wr[\hpi_{k_1},\ldots,\hpi_{k_p}]\\
  \label{eq:piA}
  \pi_{i}(x;K,L,a,b)
  &=   \chi(i+p+q;K) 
    \frac{\Wr[\hpi_{k_1},\ldots,\hpi_{k_p},\hpi_{i+p}]}{
    \Wr[\hpi_{k_1},\ldots,\hpi_{k_p}]}  
    ,\quad   
    i\in I_1
    \intertext{where}
    \label{eq:chizK}
    \chi(z;K)  &=\prod_{k\in K}(z-k)^{-1}.
\end{align}

\begin{prop}
  \label{prop:TA}
  The $\tau=\tau(x;K,L,a,b),\al,\be$, as defined above, is a
  polynomial with
  \begin{equation}
    \label{eq:degtauA}
     \deg \tau = 2\sum_{\ell\in L}\ell + \sum_{k\in K}k -
    \binom{p+q}{2} - \binom{q}{2} + qa.  
  \end{equation}
  Then, $T=\Trg(\tau;\al,\be)$ is a type A exceptional Jacobi operator
  with spectral diagram $\LamA(K,L,a,b)$ and index sets as defined
  above. The corresponding exceptional Jacobi quasi-polynomials, with
  the monic choice of normalization, are given by the above
  $\pi_i(x) = \pi_i(x;K,L,a,b),\; i\in I_1$.  The corresponding
  norms are given by
  \begin{equation}
    \label{eq:tnuA}
    \nu_i= \kappa(i+p+q;K,L,a,b)\, \nu(i+p+q,a,b),\quad i\in I_1.
  \end{equation}
  where
  \[ \kappa(z;K,L,a,b) =  \hchi(z;K,a,b) \hchi(z;L,a,b)^2.
    \]
\end{prop}

\begin{prop}
  \label{prop:ASD}
  Let $T=\Trg(\tau;\al,\be)$ be a type A exceptional Jacobi operator and
  $\Lam_T$ its spectral diagram. Then, $\Lam_T=\LamA(K,L,0,b)$ for a
  unique choice of $K,L,b$ such that $0\notin K$. Moreover, up to a
  scale factor, $\tau=\tau(x;K,L,0,b)$ as per ~\eqref{eq:hpiA} and
  ~\eqref{eq:tauA}.
\end{prop}

\begin{remark}
  \label{rem:A}
  Since the defining assumption of a type A operator is that
  $\al\in \Nz$, the type 1,2 and the type 3,4 eigenvalues are
  commensurate; c.f., ~\eqref{eq:lam1234}.  One of three things happens
  at every qr-eigenvalue as $\al$ tends to an integer: either a type 3
  eigenfunction degenerates to a type 1 eigenfunction, a type 2
  eigenfunction degenerates to a type 4 eigenfunction, or a degenerate
  eigenvalue with type 2 and 3 eigenfunctions appears.  Consequently,
  the spectral diagram of a class A operator is equivalent to a
  Maya-like diagram beginning with an infinite string of $\boxminus$
  labels followed by a finite string consisting of
  $\boxminus,\boxcirc,\boxstar$ labels followed by an infinite string
  of $\boxcirc$ labels. Proceeding more formally, we can assert that
  the mappings $i\mapsto \lam_{1}(i;\al,\be),\; i\in \Z$ and
  $j\mapsto\lam_3(j;\al,\be)= \lam_1(j-\al;al,\be),\; j\in \Z$ are
  both bijections of $\Z$ and $\sigq(\al,\be)$.  The placement of the
  asymptotic labels corresponds to the decomposition of the latter
  given by Proposition ~\ref{prop:sdA}.

  In order to understand class A diagrams, it is helpful to consider a
  classical A operator $T(a,b),\; a\in \Nz,\; b\notin \Z$ and its qr
  eigenfunctions as shown in Table ~\ref{tab:stype}.  The first and
  second kind of degeneration can be represented symbolically as
  \[ \phi_{3,k+a}(x;A,b)\dashrightarrow \phi_{1,k}(x;a,b),\; k\in
    \Nz,\qquad \phi_{2,k+a}(x;A,b) \dashrightarrow \phi_{4,k}(x;a,b),\; k\in
    \Nz\] as $A\to a\in \Nz$.  Here the we employ the notation
  $f\dashrightarrow g$ to signify that, modulo removable
  singularities, $f(x)$ tends to a constant multiple of $g(x)$.  At
  the level of classical Jacobi polynomials this is captured by the
  following classical identity \cite{Szego}:
\begin{equation}
  \label{eq:Pdegenid}
\pi_{n+a}(x;-a,b) =      (x-1)^{a}\pi_n(x;a,b),\quad
  a,n\in \Nz,
\end{equation}
Correspondingly,
\[ \lam_3(n+a;a,b)= \lam_1(n;a,b),\quad a,n\in \Nz\]
are simple eigenvalues.
Moreover,  since $a\in \Nz$, we have
\[ \lam_{2}(n;a,b)=\lam_{3}(a-n-1;a,b),\quad n\in \Nz.\] If
$n\in \ckI_2=\{ 0,1,\ldots, a-1\}$, then $\phi_{2n}$ and
$\phi_{3,a-n-1}$ are linearly independent and thus these eigenvalues
are degenerate.  For $n\ge a$, we have
\[ \pi_{n+a}(x;-a,-b) = (x-1)^a \pi_n(x;a,-b), \] and hence the
eigenvalues $\lam_2(n;a,b),\; n\ge a$ are simple.  Consequently, a
classical type A spectral diagram features $a$ degenerate eigenvalues
with a $\boxstar$ label, indicating the presence of eigenfunctions
with both $\boxplus$ and $\boxtimes$ behaviour.  All other qr
eigenvalues are simple and labelled by either $\boxminus$ or
$\boxcirc$.  By way of example, the spectral diagrams for the cases
$a=0$ and $a=2$ are shown in Figure ~\ref{fig:stypeA}

\begin{figure}[H]
  \centering
  \begin{tikzpicture}[scale=0.6]
    \def\y{2}
    \draw (0.5,\y+1.5) node[anchor=east] {$i=k$};
     \foreach \x in {-3,...,3} \draw (\x+4.5,\y+1.5) node {$\x$};
    \draw    (1,\y) grid +(7 ,1);
    \draw  (10,\y+1.5) node[anchor=west] {$a=0,\; b\notin
      \Z,\;\lam=i(i+b+1),\quad $};

    \path [draw,color=black] (0.5,2.5) node {\huge ...}
    ++(1,0) node {$-$}
    ++(1,0) node {$-$}
    ++(1,0) node {$-$}
    ++(1,0) circle (5pt)
    ++(1,0) circle (5pt)
    ++(1,0) circle (5pt)
    ++(1,0) circle (5pt)
    ++(1,0) node {\huge ...}
    ++(1.5,0) node[anchor=west]
    {$\ckI_1=\ckI_4=\Nz,\;\ckI_2=\ckI_3=\emptyset$};

    \def\y{0}

    \draw (0.5,\y+0.5) node[anchor=east] {$i=k+2$}; \foreach \x in
    {-5,...,1} \draw (\x+6.5,\y+0.5) node {$\x$}; \draw (1,\y-1) grid
    +(7 ,1); \draw (10,\y+0.5) node[anchor=west]
    {$a=2,\;b\notin \Z,\; \lam=i(i+b+3)$};

    \path [draw,color=black] (0.5,\y-.5)    node {\huge ...}
    ++(1,0)    node    {$-$}
    ++(1,0)    node {$-$}
    ++(1,0)    node {$-$}
    ++(1,0) node {$\mystar$}
    ++(1,0) node {$\mystar$}
    ++(1,0) circle (5pt)
    ++(1,0) circle (5pt)
    ++(1,0) node {\huge ...}
    ++(1.5,0) node[anchor=west]
    {$\ckI_1=\ckI_4=\Nz,\;\ckI_2=\ckI_3=\{0,1\}$}; 
  \end{tikzpicture}

  \caption{Classical type A spectral diagrams}
  \label{fig:stypeA}
\end{figure}

As asserted by Proposition ~\ref{prop:ASD}, the spectral diagram of a
type A exceptional operator is a finite modification of the classical
spectral diagram
\[ \LamA(a,b):= \LamA(\emptyset,\emptyset,a,b),\quad a\in
  \Nz,\,b\notin \Z \] by $q$ label flips
$\boxcirc\to\boxstar\to \boxminus$ indexed by $L$ and by $p$ flips
$\boxcirc \to \boxstar$ indexed by $K$, and by a spectral shift of
$(p+q)(p+q+a+b+1)$; c.f., ~\eqref{eq:lamshift}.  However, $\LamA(a,b)$
is itself a finite modification of the classical diagram
$\LamA(0,b-a)$ by $a$ consecutive $\boxcirc\mapsto \boxstar$ flips
(see Proposition ~\ref{prop:DCA}, below).  Thus, in order to assure
uniqueness of the parameters, we take $a=0$ in Proposition
~\ref{prop:ASD}.

The parameters $K,L,b$ and the index sets $I_1,I_2, I_3,I_4$ can be
recovered from the spectral diagram as follows.  One places the origin
$k=0$ at the leftmost $\boxcirc$ or $\boxstar$ label, and takes $K$ to
be the $k$ positions of the $\boxstar$ labels and $L$ the $k$
positions of the $\boxminus$ labels to the right of the origin.  Let
$p, q$ be the respective cardinalities of $K$ and $L$, and set
$i=k-p-q$. The elements of $I_1$ are the $i$ positions of the
$\boxcirc$ labels.  The elements of $I_2$ are the $\bar{i}=-i-1$
positions of the $\boxstar$ labels.  The elements of $I_3$ are $i+p$,
where $i$ is the position of the $\boxstar$ label.  The elements of
$I_4$ are $\bar{i}-p$, where $i$ is the positions of the $\boxminus$
labels.  Note that $I_3=\bI_2+\al$, because $\al=p$ and because the
type 2 and the type 3 eigenvalues coincide:
\[ \lam_2(k;\al,\be)  =
  \lam_3(-k-1+\al;\al,\be).\]
The value of $b$ is recovered by considering the difference of the qr
eigenvalues at positions $k=0$ and $k-1$.  Indeed,
$b= \lam_1(0;0,b) - \lam_1(-1;0,b).$
An example of a type A diagram is shown in the
figure ~\ref{fig:sda}.
\end{remark}

\begin{figure}[H]
  \centering
\begin{tikzpicture}[scale=0.6]

  \def\y{0}

    \path  (1.5,\y+3.5) node[anchor=east] {$k$};
  \foreach \x in {-2,...,6} \draw (\x+5.5,\y+3.5) node {$\x$};

  \path  (1.5,\y+2.5) node[anchor=east] {$i=k-4$};
  \foreach \x in {-6,...,2} \draw (\x+9.5,\y+2.5) node {$\x$};

  \path (13.5,\y+3.5) node[anchor=west] {$K=\{2,4\},\; L= \{ 1,3\}$};

  \path (13.5,\y+2.5)  node[anchor=west] 
  {$I_1= \{ -4,1,2,\ldots \},\; I_2=\{-1,1\}$};

  \path (13.5,\y+1.5)
  node[anchor=west] {$I_3=\{0,2\},\;I_4= \{-2,0,2,3,\ldots \}$};
  \path  (1.5,\y+0.5) node[anchor=east] {$\bar{i}=\bk+4$};
  \foreach \x in {-3,...,5} \draw (-\x+8.5,\y+0.5) node {$\x$};

   \draw  (3,\y+1) grid +(9 ,1);
   \path [draw,color=black] (0.5,\y+1.5)
   ++(2,0) node {\huge ...}
   ++(1,0) node {$-$}
   ++(1,0) node {$-$}
   ++(1,0) circle (5pt)
   ++(1,0) node {$-$}
   ++(1,0) node {$\mystar$}
   ++(1,0) node {$-$}
   ++(1,0) node   {$\mystar$}
   ++(1,0) circle (5pt)
   ++(1,0) circle (5pt)
   ++(1,0) node   {\huge ...};
   
\end{tikzpicture}
\caption{Type A spectral diagram with $\al=2, \; \be=b+6$ and qr
  eigenvalues
  $ \lam_{1}(i,\al,\be) = \lam_{1}(k,0,b)+4(b+5),\; i=k-4,\,k\in \Z.$}
  \label{fig:sda}
\end{figure}
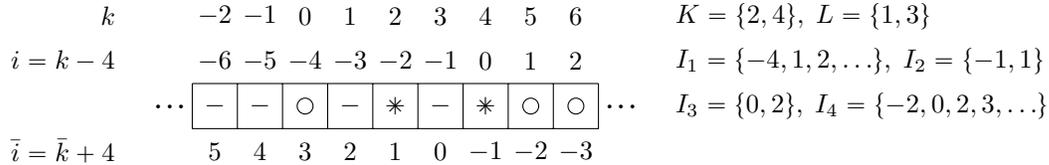



\subsection{Semi-degenerate class B}
\label{sec:semidegB}
Let $a,b\in \R$ such that $a,b,a+b \notin \Z,\;a-b\in \Z$. Set
$\ckI_1=\ckI_2=\Nz$,
\begin{equation}
\label{eq:ckI34}
\begin{aligned}
  &\ckI_3 = \{ n\in \Nz : (n-a+b)_n \ne 0\},&&
  \ckI_4  = \{ n\in \Nz : (n+a-b)_n \ne 0\},  \\
  & \ckI_{3-} = \{ n\in \Nz : 2n-a+b<0 \},&&
  \ckI_{3+}= \{ n\in \Nz: n-a+b\ge 0 \}\\
  &\ckI_{4-} = \{ n\in \Nz : 2n+a-b<0 \},&& \ckI_{4+}= \{ n\in \Nz:
  n+a-b\ge 0 \},
     \end{aligned}
\end{equation}
and observe that
$\ckI_3=\ckI_{3-}\sqcup\ckI_{3+},\; \ckI_4=\ckI_{4-}\sqcup \ckI_{4+}$,
and that $\ckI_{3+}-a = \ckI_{4+}-b$. Let
$K_1\subset \Nz,K_3\subset \ckI_{3+},K_4\subset\ckI_{4+}$ be finite
sets with respective cardinalities $p_1,p_3,p_4$, but with the
additional assumption that
\begin{equation}
  \label{eq:K34ab}
  \tK_{34} := (K_3-a)\sqcup (K_4-b)
\end{equation}
is a disjoint union.\footnote{The above defined $\tK_{34}$ is the set
  of degrees of the qr-functions present in the list
  $\phi_{3}(K_3;a,b),\phi_4(K_4;a,b)$.  The condition that
  ~\eqref{eq:K34ab} be a disjoint union ensures that the factorization
  eigenvalues are distinct.  See Section ~\ref{sec:RDT} for more
  details.}  Set $\bK_\imath = -K_\imath-1$ and 
\[ \al = a+p_1-p_3+p_4,\quad \be = b+p_1+p_3-p_4.\] Let
$\sigma_\imath=\lam_\imath(I_\imath;\al,\be)$,
$\sigma_{\imath\pm}=\lam_\imath(I_{\imath\pm};\al,\be)$,where
\begin{equation}
  \label{eq:indexB}
\begin{aligned}
  &I_1= (\ckI_1\setminus K_1)-p_1,
  && I_2 = (\ckI_2\cup \bK_1)+p_1,\\
  &I_{3-}= (\ckI_{3-}\cup \bK_4)-p_3+p_4
  && I_{3+} = (\ckI_{3+}\setminus (\tK_{34}+a))-p_3+p_4\\
  &I_{4-}= (\ckI_{4-}\cup \bK_3)+p_3-p_4
  && I_{4+} = (\ckI_{4+}\setminus (\tK_{34}+b))+p_3-p_4\\
  &I_3= I_{3-} \sqcup I_{3+} && I_4 = I_{4-}\sqcup I_{4+}
\end{aligned}
\end{equation}
\begin{prop}
  The assumptions $a,b,a+b\notin \Z, \; a-b\in \Z$ entail the same
  conditions on $\al,\be$.  The qr-spectrum $\sigq(\al,\be)$ has the
  decomposition
  \begin{equation}
    \sigq(\al,\be) = \sigma_1 \sqcup \sigma_2\sqcup
    \sigma_{3-} \sqcup \sigma_{4-}
    \sqcup \sigma_{34},
  \end{equation}
  where $\sigma_{34} := \sigma_{3+} = \sigma_{4+}$.
\end{prop}
\begin{proof}
  We prove the last assertion; the other assertions are straightforward.
  Recall that
  \[\lam_3(i;\al,\be) = \lam_1(i-\al;\al,\be),\quad
    \lam_4(i;\al,\be) = \lam_1(i-\be;\al,\be).\] By construction,
  \begin{align*}
    I_{3+}-\al &= ((\ckI_{3+}-a)\setminus \tK_{34})-p_1,\\
    I_{4+}-\be &= ((\ckI_{4+}-b)\setminus \tK_{34})-p_1,
  \end{align*}
  Hence, $I_{3+}-\al = I_{4+}-\be$ and hence
  $\sigma_{3+} = \sigma_{4+}$.
\end{proof}

Next, we to define $\LamB(K_1,K_3,K_4,a,b)$ to be the mapping
\begin{equation}
  \label{eq:sdB}
  (\sigma_1,\sigma_2,\sigma_{34}, \sigma_{3-},\sigma_{4-}) \to (
  \boxcirc,      \boxtimes,      \boxdiv,      \boxplus,     \boxminus)
\end{equation}

Let $\tau$ and
$\pi_i,\; i\in I_1$ be defined as in ~\eqref{eq:tauGB} and
~\eqref{eq:pig}. With these definitions and with the above assumptions,
Proposition ~\ref{prop:TG} holds verbatim, except that here
$\Trg(\tau;\al,\be)$ is a type B operator with spectral diagram
$\LamB(K_1,K_3,K_4,a,b)$.
\begin{prop}
  \label{prop:TB}
  The $\tau=\tau(x;K_1,K_3,K_4,a,b),\; a,b,a+b\notin \Z,\; a-b\in \Z$,
  as defined in ~\eqref{eq:tauGB}, is a polynomial whose degree is
  given in ~\eqref{eq:degG}. The operator $T=\Trg(\tau;\al,\be),$ with
  $\al,\be$ as defined in ~\eqref{eq:albeG}, is a type B exceptional
  operator with spectral diagram $\Lam_T=\LamG(K_1,K_3,K_4,a,b)$ and
  index sets $I_1,I_2,I_3,I_4$ as defined in ~\eqref{eq:indexB}. The
  corresponding exceptional Jacobi quasi-polynomials, with the monic
  choice of normalization, are given by the above defined
  $\pi_i(x;K_1,K_3,K_4,a,b),\; i\in I_1$.  The corresponding norms
  have the form shown in ~\eqref{eq:nuiG}.
\end{prop}
As before, the essential complication is that
different choices of $K,a,b$ can yield the same operator.  To make the
choice of these parameters unique, we require the following type B
analogue of Proposition ~\ref{prop:GSD}.
\begin{prop}
  \label{prop:BSD}
  Let $T=\Trg(\tau;\al,\be)$ be a type B exceptional operator and
  $\Lambda_T$ its spectral diagram. Then,
  $\Lambda_T=\LamB(K_1,K_3,K_4,a,b)$ for a unique choice of
  $K_1,K_3,K_4,a,b$ such that $a-b\in \{ -1,0,1\}$ and $0\notin K_1$.
  Moreover, up to a scale factor, $\tau=\tau(x;K_1,K_3,K_4,a,b)$ as
  defined in ~\eqref{eq:tauGB}.
\end{prop}

\begin{remark}
  \label{rem:B}
  The type B defining assumption is that
  $\al,\be,\al+\be\notin \Z,\al-\be\in \Z$.  From ~\eqref{eq:lam1234}, we
  see that
  \begin{equation}
    \label{eq:lam34mmu} \lam_{3}(i;\al,\be) =
    \lam_{3}(i^\ddag;\al,\be)= \lam_4(i-\al+\be;\al,\be), 
  \end{equation}
  where
  \begin{equation}
    \label{eq:idag}
    i^\ddag:=-i-1+\al-\be.
  \end{equation}
  This symmetry is at the heart of the class B degenerations.
  
  It is instructive to consider the consequences of the spectral
  symmetry ~\eqref{eq:lam34mmu} on the qr-eigenfunctions of a classical
  B operator $T(a,b),\; a-b\in \Z$.  We consider the subcase where
  $a\ge b$; the case where $b\ge a$ admits a similar discussion.
  Consider the limit of a generic classical operator $T(A,B)$ as
  $(A,B)\to (a,b)$.  In the limit, a finite number of type 3
  eigenfunctions degenerate to type 3 eigenfunctions with a lower
  index.  To explain this, observe that the leading coefficient of
  $P_n(x;-a,b)$ has a factor of $(n-a+b+1)_n$.  Thus,
  $\deg P_n(x;-a,b)<n$ if and only if
  \[ (n-a+b+1) \cdots (2n-a+b) = (0-n^\dag)\cdots (n-1-n^\dag)=0,\]
  where $n^\dag:=-n-1+a-b,\; n\in \Nz$.  This possibility occurs
  precisely for $n\in \Nz$ such that $n>n^\dag\ge 0$.  For such
  indices we have
  \begin{equation}
    \label{eq:phi34degen}
    \phi_{3,n}(x;A,B) \dashrightarrow
    \phi_{3,n^\dag}(x;a,b),\quad\text{as } (A,B)\to (a,b).
  \end{equation}
  This degeneration is captured by the following
  classical identity \cite{Szego}:
  \begin{equation}
    \label{eq:Pdegen2}
    (n^\dag+1)_{j}P_n(x;-a,b) = (n-a)_{j}
    P_{n^\dag}(x;-a,b),\quad   j:=n-n^\dag\ge 0.
  \end{equation}
  Consequently, the classical type 3 index set decomposes as
  $\ckI_3 = \ckI_{3-} \sqcup \ckI_{3+}$, as defined in
  ~\eqref{eq:ckI34}, and is missing $\lfloor (a-b)/2\rfloor$ indices.
  The missing indices are precisely the $n\in \Nz$ such that
  $n>n^\dag\ge 0$.  At the level of the spectral diagram, as
  $\LamG(A,B)\to \LamB(a,b)$, this symmetry manifests as a coalescence
  of a co-finite number of $\boxplus$ and $\boxminus$ eigenvalues into
  degenerate eigenvalues with the compound label $\boxpm$.  See Figure
  ~\ref{fig:stypeB} for some examples.
  \newpage

  \begin{figure}[H]
    \centering
    \begin{tikzpicture}[scale=0.6]

      \def\y{12}
      \draw (0.5,\y+2.5) node {$k$};
      \foreach \x in {-3,...,3} \draw (\x+4.5,\y+2.5) node {$\x$};
      \draw  (1,\y+1) grid +(7 ,1);

      \draw (10.5,\y+2.5) node[anchor=west] {$a-b=0$};
      
      \path [draw,color=black] (0.5,\y+1.5) node {\huge ...}
      ++(1,0) node {$\times$}
      ++(1,0) node {$\times$}
      ++(1,0) node {$\times$}
      ++(1,0) circle (5pt)
      ++(1,0) circle (5pt)
      ++(1,0) circle (5pt)
      ++(1,0) circle (5pt)
      ++(1,0) node {\huge ...};
      \path (10.5,\y+1.5) node[anchor=west] {$\ckI_1=\ckI_2=\Nz$};
      
      \draw  (1+3,\y) grid +(7-3 ,1);
      \path [draw,color=black] (0.5+3,\y+0.5) 
      ++(1,0) node {$\div$}
      ++(1,0) node {$\div$}
      ++(1,0) node {$\div$}
      ++(1,0) node {$\div$}
      ++(1,0) node {\huge ...};
      \path (10.5,\y+0.5) node[anchor=west]
      {$\ckI_{3+}=\ckI_{4+}=\Nz$};

      \def\y{8}
      \draw (0,\y+2.5) node {$k$};
      \foreach \x in {-3,...,3} \draw (\x+4.5,\y+2.5) node {$\x$};
      \draw  (1,\y+1) grid +(7 ,1);

      \path [draw,color=black] (0.5,\y+1.5)    node {\huge ...}
      ++(1,0)    node {$\times$}
      ++(1,0) node {$\times$}
      ++(1,0) node {$\times$}
      ++(1,0) circle (5pt)
      ++(1,0) circle (5pt)
      ++(1,0) circle (5pt)
      ++(1,0) circle (5pt)
      ++(1,0) node {\huge ...};
      \path (10.5,\y+1.5) node[anchor=west] {$\ckI_1=\ckI_2=\Nz$};

      \draw  (3,\y) grid +(5 ,1);

      \draw (10.5,\y+2.5) node[anchor=west] {$a-b=1$};
      \path [draw,color=black] (2.5,\y+0.5)
      ++(1,0) node[draw,rectangle]  {$+$}
      ++(1,0) node {$\div$}
      ++(1,0) node {$\div$}
      ++(1,0) node {$\div$}
      ++(1,0) node {$\div$}
      ++(1,0) node {\huge ...};
      \path (10.5,\y+0.5) node[anchor=west] {$\ckI_{3-}=\{0\},\;
        \ckI_{3+} = \{1,2,\ldots \},\; \ckI_{4+}=\Nz$};
      
      \def\y{4}
      \draw (0,\y+2.5) node {$k$};
      \foreach \x in {-3,...,3} \draw (\x+4.5,\y+2.5) node {$\x$};
      \draw  (1,\y+1) grid +(7 ,1);

      \path [draw,color=black] (0.5,\y+1.5) node {\huge ...}
      ++(1,0) node {$\times$}
      ++(1,0) node {$\times$}
      ++(1,0) node {$\times$}
      ++(1,0) circle (5pt)
      ++(1,0) circle (5pt)
      ++(1,0) circle (5pt)
      ++(1,0) circle (5pt)
      ++(1,0) node {\huge ...};
      \path (10.5,\y+1.5) node[anchor=west] {$\ckI_1=\ckI_2=\Nz$};

      \draw  (3,\y) grid +(5 ,1);

      \draw (10.5,\y+2.5) node[anchor=west] {$a-b=-1$};
      \path [draw,color=black] (2.5,\y+0.5)
      ++(1,0) node[draw,rectangle]  {$-$}
      ++(1,0) node {$\div$}
      ++(1,0) node {$\div$}
      ++(1,0) node {$\div$}
      ++(1,0) node {$\div$}
      ++(1,0) node {\huge ...};
      \path (10.5,\y+0.5) node[anchor=west] {$\ckI_{3+} =\Nz,\;
        \ckI_{4-}=\{0\},\; \ckI_{4+}= \{1,2,3,\ldots \}$};  
      
      \def\y{0}
      \draw (0,\y+2.5) node {$k$};
      \foreach \x in {-3,...,3} \draw (\x+4.5,\y+2.5) node {$\x$};
      \draw  (1,\y+1) grid +(7 ,1);

      \path [draw,color=black] (0.5,\y+1.5) node {\huge ...}
      ++(1,0) node {$\times$}
      ++(1,0) node {$\times$}
      ++(1,0) node {$\times$}
      ++(1,0) circle (5pt)
      ++(1,0) circle (5pt)
      ++(1,0) circle (5pt)
      ++(1,0) circle (5pt)
      ++(1,0) node {\huge ...}
      +(2,0) node[anchor=west] {$\ckI_1=\ckI_2=\Nz$};
      \draw  (3,\y) grid +(5 ,1);
      \draw (10.5,\y+2.5) node[anchor=west] {$a-b=3$}; 
      \path [draw,color=black] (2.5,\y+0.5) 
      ++(1,0) node[draw,rectangle] {$+$}
      ++(1,0) node {$+$}
      ++(1,0) node {$\div$}
      ++(1,0) node {$\div$}
      ++(1,0) node {$\div$}
      ++(1,0) node {\huge ...}
      +(2,0)   node[anchor=west]
      {$\ckI_{3-}=\{0,1\},\; \ckI_{3+} = \{3,4,\ldots \},\;
        \ckI_{4+}=\Nz$};

      \draw (10.5,\y-1.5) node[anchor=west] {$i^\dag=-i+2$};
      \draw (10.5,\y-0.5) node[anchor=west] {$j=-i^\ddag-1=i-3$};

      \draw (0,\y-0.5) node {$i$};
      \foreach \x in {1,...,5} \draw (\x+2.5,\y-0.5) node {$\x$};
      \draw (0,\y-1.5) node {$i^\dag$};
      \foreach \x in {-3,...,1} \draw (-\x+4.5,\y-1.5) node {$\x$};

      \def\y{-6}
      \draw (0,\y+2.5) node {$k$};
      \foreach \x in {-3,...,3} \draw (\x+4.5,\y+2.5) node {$\x$};
      \draw  (1,\y+1) grid +(7 ,1);

      \path [draw,color=black] (0.5,\y+1.5) node {\huge ...}
      ++(1,0) node {$\times$}
      ++(1,0) node {$\times$}
      ++(1,0) node {$\times$}
      ++(1,0) circle (5pt)
      ++(1,0) circle (5pt)
      ++(1,0) circle (5pt)
      ++(1,0) circle (5pt)
      ++(1,0) node {\huge ...}
      +(2,0) node[anchor=west] {$\ckI_1=\ckI_2=\Nz$};

      \draw  (4,\y) grid +(4 ,1);
      \draw (10.5,\y+2.5) node[anchor=west] {$a-b=4$}; 
      \path [draw,color=black] (3.5,\y+0.5) 
      ++(1,0) node {$+$}
      ++(1,0) node {$+$}
      ++(1,0) node {$\div$}
      ++(1,0) node {$\div$}
      ++(1,0) node {\huge ...}
      +(2,0)   node[anchor=west]
      {$\ckI_{3-}=\{0,1\},\; \ckI_{3+} = \{4,5,\ldots \},\;
        \ckI_{4+}=\Nz$}; 
      \draw (10.5,\y-1.5) node[anchor=west] {$i^\dag=-i+3$};
      \draw (10.5,\y-0.5) node[anchor=west] {$j=-i^\dag-1=i-4$};

      \draw (0,\y-0.5) node {$i$};
      \foreach \x in {2,...,5} \draw (\x+2.5,\y-0.5) node {$\x$};
      \draw (0,\y-1.5) node {$i^\dag$};
      \foreach \x in {-2,...,1} \draw (-\x+5.5,\y-1.5) node {$\x$};

    \end{tikzpicture}

    \caption{Classical type B spectral diagrams.  The type 1,2
      eigenvalues are given by $\lam=k(k+a+b+1)$.  The type 3,4
      eigenvalues are given by
      $\lam=(i-a)(i+b+1)=(i^\dag-a)(i^\dag+b+1)$, or equivalently by
      $\lam=(j-b)(j+a+1)$.}
    \label{fig:stypeB}
  \end{figure}


  Since $\al+\be\notin \Z$ the mapping $i\mapsto i(i+\al+\be+1)$
  defines a bijection of $\Z$ and $\sigma_1\sqcup \sigma_2$.  Since
  $\Z = I_1 \sqcup \bI_2$ we use this bijection to represent the type
  1,2 eigenvalues as a doubly-infinite row of $\boxtimes$ and
  $\boxcirc$ labels. Likewise, $i\mapsto (i-\al)(i+\be+1)$ defines a
  bijection of $\{i\in\Z: i\ge i^\ddag\}$ and
  $\sigma_{3-} \sqcup \sigma_{4-}\sqcup \sigma_{34}$\footnote{Recall
    that $\sigma_{3+}=\sigma_{4+} = \sigma_{34}$.}.  Since
  $\{ i\in \Z: i\ge i^\ddag \} =I_{3-}^\ddag \sqcup (I_{4-}+\al-\be)
  \sqcup I_{3+}$ and since $I_{3+} = I_{4+}+\al-\be$, we use the
  second bijection to represent the type 3,4 eigenvalues as a
  semi-infinite row of $\boxplus,\boxminus, \boxpm$ labels. 

  The B class is further divided into two subclasses, according to
  whether $\al-\be$ is even or odd.  The odd subclass is characterized
  by the fact that $i=(\al-\be-1)/2$ is the unique integer such that
  $i=i^\ddag$, with the latter defined as in ~\eqref{eq:idag}.  This
  index marks the vertex of the type 34 eigenvalue parabola
  $\{ (i,\lam_3(i;\al,\be)) : i\in \Z\}$.  We will refer to the
  corresponding $\lam_3(i;\al,\be)$ as \emph{the vertex eigenvalue}.
  In the visual representation described aboce, the vertex eigenvalue
  corresponds to the leftmost cell of the type 3,4 demi-diagram.  To
  make this distinction visually, we will mark the vertex eigenvalue
  by a doubly-lined $\bboxplus$ or $\bboxminus$ label, according to
  the asymptotic type of the eigenfunction at the vertex eigenvalue.
  In the even subclass, the vertex index is a half-integer, and so
  does not correspond to any qr-eigenfunctions.

  Every type B diagram may be canonically constructed using a finite
  number of $\boxcirc\to \boxtimes$ and
  $\boxpm \to \{ \boxplus, \boxminus \}$ flips applied to a classical
  diagram $\LamB(a,b)= \LamB(\emptyset,\emptyset,\emptyset,a,b)$ where
  $a,b\notin \Z,\; a-b\in \{ -1,0,1\}$.  The label changes
  $\boxcirc\mapsto \boxtimes$ are indexed by $K_1$, while the label
  changes $\boxdiv\mapsto \boxminus$ and $\boxdiv\mapsto \boxplus$ are
  indexed by $K_3$ and $K_4$, respectively.  The parameter set $K_1$
  and the index sets $I_1, I_2$ are recovered from the doubly-infinite
  diagram as follows. One takes the origin $k=0$ at the leftmost
  $\boxcirc$ symbol and lets $p_1$ be the number of $\boxtimes$ labels
  to the right of this origin. The parameter set $K_1$ consists of the
  $k$ positions of these labels. The elements of $I_1$ are the $k-p_1$
  positions of the $\boxcirc$ labels. The elements
  of $I_2$ are the $\bk:=-k-1$ positions of the $\boxtimes$ labels
  shifted by $+p_1$.

 The value of $\al-\be$, the parameter sets $K_3,K_4$, and the index
 sets $I_3,I_4$ are recovered from the demi-diagram as follows.  Let
 $\hp_3, \hp_4$ be the number of all $\boxminus$ and $\boxplus$
 labels, and let $p_3,p_4$ be the number of non-vertex
 $\boxminus,\boxplus$ labels.  Thus, $\hp_3=p_3+1$ if there is a
 vertex $\bboxminus$ label, and $\hp_4=p_4+1$ if there is a vertex
 $\bboxplus$ label.  In the even case, $\hp_3=p_3$ and $\hp_4=p_4$.
 The value of $a-b$ is given by $\hp_3-p_3-(\hp_4-p_4)$.  The value of
 $\al-\be$ is recovered as $p_4+\hp_4-(p_3+\hp_3)$; this is the
 difference between the number of $\boxplus$ and $\boxminus$ labels,
 with non-vertex labels counted twice, and the vertex label (if any)
 counted once.  Next, one places the origin $k=0$ of the demi-diagram
 at the leftmost non-vertex cell\footnote{The vertex cell, if present,
   is always placed at position $k=-1$.} and sets
 \[ i=k-p_3+\hp_4,\quad j=k+\hp_3-p_4=i-\al+\be. \] The elements of
 $K_3, K_4$ are, respectively, the $k+\hp_4-p_4$ and the $k+\hp_3-p_3$
 positions of the $\boxminus$ and the $\boxplus$ labels. The elements of
 $I_{3+},I_{4+}$ are, respectively, the $i$ and $j$ positions of the
 $\boxpm$ labels. The elements of $I_{3-}$ are the $i^\ddag$ positions
 of the $\boxplus$ labels, and the elements of $I_{4-}$ are the
 $\bar{i}=-i-1$ positions of the $\boxminus$ labels.  The elements
 of $K_3$ are the $k$ positions $K_4$ are the $k$ positions of the
 $\boxplus$ labels.

 Figure ~\ref{fig:sdb} illustrates examples of exceptional type B
 diagrams.
\end{remark}
\newpage
  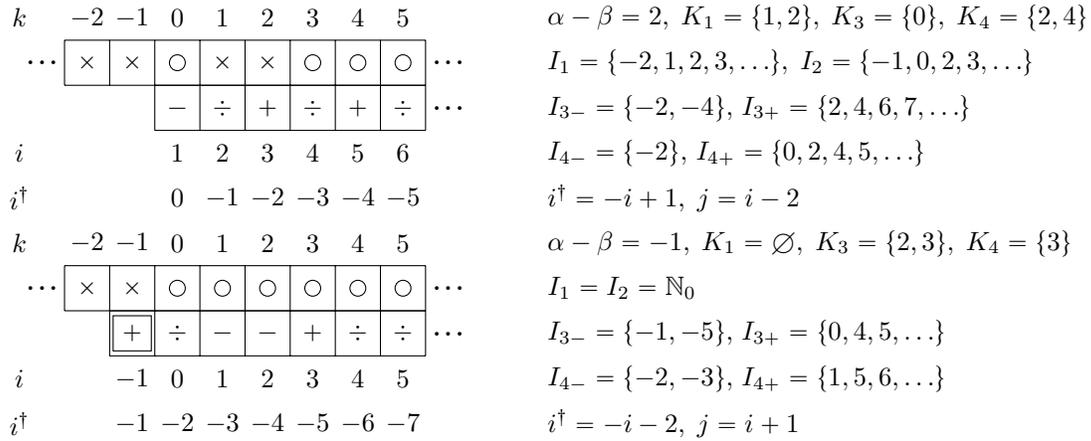
\begin{figure}[H]
    \centering
    
    \begin{tikzpicture}[scale=0.6]
      \def\y{4}

      \draw (0,\y+2.5) node[anchor=center] {$k$};
      \foreach \x in  {-2,...,5} \draw (\x+3.5,\y+2.5) node {$\x$};
      \draw (11.5,\y+3.5)
      node[anchor=west] {$$};  

      \draw  (1,\y+1) grid +(8 ,1);
      \path [draw,color=black]   (0.5,\y+1.5) node {\huge ...}
      ++(1,0) node {$\times$}
      ++(1,0) node {$\times$}
      ++(1,0) circle (5pt)
      ++(1,0) node {$\times$}
      ++(1,0) node {$\times$}
      ++(1,0) circle (5pt)
      ++(1,0) circle (5pt)
      ++(1,0) circle (5pt)
      ++(1,0) node {\huge ...};
      \draw (11.5,\y+2.5)    node[anchor=west]
      {$\al-\be=2,\;K_1=\{1,2\},\;   K_3=\{0\},\; K_4=\{2,4\}$};
      \draw (11.5,\y+1.5)    node[anchor=west]
      {$I_1=\{-2,1,2,3,\ldots\},\;
        I_2=\{-1,0,2,3,\ldots\}$};  

      \draw  (3,\y) grid +(6,1);
      \path [color=black]     (-0.5+3,\y+0.5) 
      ++(1,0) node {$-$}
      ++(1,0) node {$\div$}
      ++(1,0) node {$+$}
      ++(1,0) node {$\div$}
      ++(1,0) node {$+$}
      ++(1,0) node {$\div$}
      ++(1,0) node {\huge ...};
      \path (11.5,\y+0.5) node[anchor=west] {$I_{3-}=\{-2,-4\},\,
        I_{3+}=\{2,4,6,7,\ldots       \}$};
      \draw (0,\y-0.5) node[anchor=center] {$i$};
      \foreach \x in  {1,...,6} \draw (\x+2.5,\y-0.5) node {$\x$};
      \draw (0,\y-1.5) node[anchor=center] {$i^\ddag$};
      \foreach \x in  {-5,...,0} \draw (-\x+3.5,\y-1.5) node {$\x$};

      \draw (11.5,\y-0.5) node[anchor=west]
      {$I_{4-}=\{-2 \},\, I_{4+}=\{0,2,4,5,\ldots\}$}; 
      \draw (11.5,\y-1.5) node[anchor=west]
      {$i^\ddag=-i+1,\; j=i-2$};

      \def\y{-2}

      \draw (0,\y+2.5) node[anchor=center] {$k$};
      \foreach \x in  {-2,...,5} \draw (\x+3.5,\y+2.5) node {$\x$};
      \draw (11.5,\y+3.5)
      node[anchor=west] {$$};  
      \draw  (1,\y+1) grid +(8 ,1);
      \path [draw,color=black]   (0.5,\y+1.5) node {\huge ...}
      ++(1,0) node {$\times$}
      ++(1,0) node {$\times$}
      ++(1,0) circle (5pt)
      ++(1,0) circle (5pt)
      ++(1,0) circle (5pt)
      ++(1,0) circle (5pt)
      ++(1,0) circle (5pt)
      ++(1,0) circle (5pt)
      ++(1,0) node {\huge ...};
      \draw (11.5,\y+2.5)    node[anchor=west]
      {$\al-\be=-1,\;K_1=\emptyset,\;   K_3=\{2,3\},\; K_4=\{3\}$};
      \draw (11.5,\y+1.5)    node[anchor=west]
      {$I_1=I_2=\Nz$};  

      \draw  (2,\y) grid +(7,1);
      \path [color=black]     (-1.5+3,\y+0.5) 
      ++(1,0) node[draw,rectangle] {$+$}
      ++(1,0) node {$\div$}
      ++(1,0) node {$-$}
      ++(1,0) node {$-$}
      ++(1,0) node {$+$}
      ++(1,0) node {$\div$}
      ++(1,0) node {$\div$}
      ++(1,0) node {\huge ...};
      \path (11.5,\y+0.5) node[anchor=west] {$I_{3-}=\{-1,-5\},\,
        I_{3+}=\{0,4,5,\ldots       \}$};
      \draw (0,\y-0.5) node[anchor=center] {$i$};
      \foreach \x in  {-1,...,5} \draw (\x+3.5,\y-0.5) node {$\x$};
      \draw (0,\y-1.5) node[anchor=center] {$i^\ddag$};
      \foreach \x in  {-7,...,-1} \draw (-\x+1.5,\y-1.5) node {$\x$};

      \path (11.5,\y-0.5) node[anchor=west]
      {$I_{4-}=\{ -2,-3 \},\, I_{4+}=\{1,5,6,\ldots\}$};
      \path (11.5,\y-1.5) node[anchor=west]
      {$i^\ddag=-i-2,\; j=i+1$};


    \end{tikzpicture}
    \caption{Class B exceptional spectral diagrams.  }
    \label{fig:sdb}
  \end{figure}

\subsection{Semi-degenerate classes C and CB}
\label{sec:semidegC}
Let $a,b\in \R$ such that $a,b\notin \Z,\; a+b\in \Z$.
Set
\begin{equation}
  \label{eq:ckI12}
  \begin{aligned}
    &\ckI_1  = \{ n\in \Nz :
      (n+a+b+1)_n \ne 0\},&&
      \ckI_2  = \{ n\in \Nz : (n-a-b+1)_n \ne 0\},  \\ 
    &    \ckI_{1-}  = \{ n\in \Nz : 2n+a+b<0 \},&&
       \ckI_{1+}= \{ n\in \Nz: n+a+b+1\ge 1 \}\\     
    &
    \ckI_{2-}  = \{ n\in \Nz : 2n-a-b<0 \},&&
       \ckI_{2+}= \{ n\in \Nz: n-a-b+1\ge  1 \},
  \end{aligned}
\end{equation}
and let $\ckI_3,\ckI_4,\ckI_{3\pm}, \ckI_{4\pm}$ be as in
~\eqref{eq:ckI34}.
Let $K_\imath\subset \ckI_{\imath+},\; \imath = 1,2,3,4$ be finite
sets with respective cardinalities $p_\imath$ but with the additional
assumptions that
\[ \tK_{12} := K_1 \sqcup (K_2-a-b),\quad \tK_{34} = (K_3-a)\sqcup
  (K_4-b) \] be disjoint unions\footnote{As before, if the
  factorization degrees are distinct, then the factorization
  eigenvalues are also distinct.}.  Set $\bK_\imath = -K_\imath-1$ and
\begin{equation}
  \label{eq:albeCB}
  \al = a+p_1-p_2-p_3+p_4,\quad \be = b+p_1-p_2+p_3-p_4, 
\end{equation}
and let
$\sigma_\imath=\lam_\imath(I_\imath;\al,\be)$ and
$\sigma_{\imath\pm} = \lam_\imath(I_{\imath\pm};\al,\be),\; \imath =
1,2,3,4$, with
\begin{equation}
  \label{eq:indexCB}
  \begin{aligned}
    I_{1-} &= (\ckI_{1-}\cup \bK_2)-p_1+p_2
    & I_{1+} &= (\ckI_{1+}\setminus \tK_{12})-p_1+p_2\\
    I_{2-} &= (\ckI_{2-}\cup \bK_1)+p_1-p_2
    & I_{2+} &= (\ckI_{2+}\setminus \tK_{12})+p_1-p_2\\
    I_{3-} &= (\ckI_{3-}\cup \bK_4)-p_3+p_4
    & I_{3+} &= (\ckI_{3+}\setminus (\tK_{34}+a))-p_3+p_4\\
    I_{4-} &= (\ckI_{4-}\cup \bK_3)+p_3-p_4
    & I_{4+} &= (\ckI_{4+}\setminus (\tK_{34}+b))+p_3-p_4
  \end{aligned}
\end{equation}
\begin{prop}
  Let
  $\sigma_\imath = \lam_{\imath}(I_\imath;\al,\be),\; \imath =
  1,2,3,4$, with the latter as per ~\eqref{eq:lam1234}.  The type C and
  the type type CB conditions on $a,b$ entail, respectively, the same
  conditions on $\al,\be$.  Moreover, if the type C conditions hold, then
  \begin{align*}
    \sigq(\al,\be)
    &=   \sigma_{1-} \sqcup \sigma_{2-}
      \sqcup \sigma_{12} \sqcup \sigma_3 \sqcup \sigma_4
  \end{align*}
  where $\sigma_{12} := \sigma_{1+} = \sigma_{2+}$.  If the CB
  conditions hold, then
  \[\sigq(\al,\be)
    = \sigma_{1-} \sqcup \sigma_{2-}
    \sqcup \sigma_{12} \sqcup\sigma_{3-} \sqcup \sigma_{4-}\sqcup
    \sigma_{34},
  \]
  where $\sigma_{34} := \sigma_{3+} = \sigma_{4+}$.
\end{prop}

\noindent
For $a,b$ satisfying the type C conditions, let $\LamC(K,a,b)$ be the
mapping 
\[ (\sigma_{1-}, \sigma_{2-},\sigma_{12}, \sigma_3 , \sigma_4) \mapsto
  (\boxotimes, \boxcirc,\boxtimes,\boxplus,\boxminus);\] i.e., an
eigenvalue in $\sigma_{1-}$ is labelled by $\boxotimes$, etc.

For the CB class, let $\LamCB(K,a,b)$ be the mapping
\[ (\sigma_{1-}, \sigma_{2-},\sigma_{12}, \sigma_{3-}, \sigma_{4-},
  \sigma_{34})\mapsto (\boxotimes,
  \boxcirc,\boxtimes,\boxpm,\boxplus,\boxminus).\]

Let
$\phi_{1234}(K;a,b)$ denote the concatenation of lists
$\phi_1(K_1;a,b),\phi_2(K_2;a,b),\phi_3(K_3;a,b),\phi_4(K_4;a,b)$, and
let $\tK= \tK_{12} \sqcup \tK_{34}$ be the set of degrees of the
qr-functions present in that list.  With this definition in place,
define
\begin{align}
  \label{eq:tauC}
  \tau_\Lam(x)
  &=  (1-x)^{(p_2+p_3)(p_1+p_4+a)}(1+x)^{(p_2+p_4)(p_1+p_3+b)}\Wr[\phi_{1234}(K)]
  \\ 
  \label{eq:piC}
  \pi_{\Lam,i}(x)
  &=
    \frac{(x-1)^{p_2+p_3}(1+x)^{p_2+p_4}}{\prod_{k\in \tK}(i+p_1-p_2-k)} \frac{
    \Wr[\phi_{1234}(K),\phi_1(\fu(i+p_1-p_2))] }{
    \Wr[\phi_{1234}(K)]},\quad 
    i\in I_1, 
\end{align}
where $\fu(j):= \max \{ j, j^* \},\; j^*=-j-1-a-b$.
\begin{prop}
  \label{prop:TCB}
  The above defined $\tau=\tau(x;K,a,b),\; a,b\notin \Z$ is a
  polynomial with
  \[ \deg \tau = \sum_{\imath\in\{1,2,3,4\}}\sum_{k\in K_\imath}
    K_\imath+2p_1 p_2 + 2p_3p_4    -\binom{p_1+p_2}{2}    -\binom{p_3+p_4}{2}.
  \]
  If $a-b\notin \Z,\; a+b\in \Z$, then $T=\Trg(\tau;\al,\be)$, with
  $\al,\be$ as defined in ~\eqref{eq:albeCB}, is a type C exceptional
  Jacobi operator with spectral diagram $\Lambda_T=\LamC(K,a,b)$. If
  $2a,2b\in \Z$, then $T$ is a type CB exceptional operator with
  spectral diagram $\Lambda_T=\LamCB(K,a,b)$.  In both cases $T$ has
  the above-defined index sets.  The corresponding exceptional
  quasi-polynomials, with the monic choice of normalization, are given
  by ~\eqref{eq:piC}.  The
  corresponding norms are given by
  \begin{equation}
    \label{eq:Csf}
     \nu_i = \kappa(i+p_1-p_2) \nu(i+p_1-p_2;a,b),
  \end{equation}
  with $\kappa$ as defined in ~\eqref{eq:kappaG} but with $\tK$ as
  defined above.  Consequently, the norms are zero if $i\in I_{1-}$,
  and non-zero if $i\in I_{1+}$.
\end{prop}

\begin{prop}
  \label{prop:CSD}
  Let
  $T=\Trg(\tau;\al,\be),\; \al,\be,\al-\be\notin \Z,\, \al+\be\in \Z$
  be a type $C$ exceptional operator and $\Lambda$ its spectral
  diagram. Then $\Lambda_T=\LamC(K,a,b)$ for a unique choice of $K,a,b$
  such that $a+b\in \{-1, 0,1\}$ and ~\eqref{eq:albeCB} holds, and such
  that $K_4=\emptyset$.  Up to a scale factor, $\tau=\tau(x;K,a,b)$,
  as defined in ~\eqref{eq:tauC}.
\end{prop}

\begin{prop}
  \label{prop:CBSD}
  Let $T=\Trg(\tau;\al,\be),\; \al,\be\notin \Z,\; 2\al,2\be \in \Z$
  be a class CB exceptional operator and $\Lambda_T$ its spectral
  diagram. Then $\Lambda_T=\LamCB(K,a,b)$ for a unique choice of
  $K,a,b$ such that $a,b\in \{- 1/2,1/2\}$.  Up to a scale
  factor, $\tau=\tau(x;K,a,b)$, as defined in ~\eqref{eq:tauC}.
\end{prop}

\begin{remark}
  \label{rem:C}
  The class C spectral degeneracy entails a two-to-one correspondence
  between type 1,2 indices and type 1,2 eigenvalues.  Indeed, since
  $\al+\be\in \Z$ we have 
  \[ \lam_1(i;\al,\be) = \lam_1(i^\star;\al,\be)=\lam_2(\bar{i};\al,\be)=
    \lam_2(i-\al-\be;\al,\be)  \]
  where
  \[ i^\star = -i-1-\al-\be,\; \bar{i} = -i-1,\quad i\in \Z.\] As for class B,
  this symmetry entails two types of phenomena: the degeneration of a
  finite number of quasi-polynomial eigenfunctions to
  quasi-polynomials of a lower degree, and the presence of an infinite
  number of degenerate eigenvalues of type $\boxotimes$.  The visual
  representation of type C diagrams is similar to the B class; one
  merely exchanges the 12 eigenvalues and the 34 eigenvalues and
  replaces the $\boxplus, \boxminus, \boxpm$ labels with
  $\boxcirc,\boxtimes, \boxotimes$ labels.  

  The operators in the B and C classes are related by a gauge
  transformation that interchanges the asymptotic labels as follows:
  \[ \boxcirc \leftrightarrow \boxplus, \quad \boxtimes
    \leftrightarrow \boxminus,\quad \boxpm \leftrightarrow
    \boxotimes \] Thus, the visual representation of class C spectral
  diagram relies on a demi-diagram decorated with
  $\boxcirc, \boxtimes, \boxotimes$ labels for the type 1,2 eigenvalue
  and a doubly infinite diagram decorated with $\boxplus, \boxminus$
  labels for the type 3,4 eigenvalues.  As with the B class, the C
  class is divided into even and odd subclasses according to the
  parity of $\al+\be$.  In the odd case, one marks the vertex
  eigenvalue with either a $\bboxcirc$ or a $\bboxtimes$ label
  according to the asymptotics of the eigenfunction at that
  eigenvalue.  The considerations are exactly the same as the ones
  presented in Remark ~\ref{rem:B}.

  However, there is a peculiarity of the C class that deserves some
  further comment. One consequence of the class C norm formula
  ~\eqref{eq:Csf} is that the finite set of quasi-polynomial
  eigenfunctions with a non-degenerate $\boxcirc$ label have a zero
  norm.  Indeed, this can even be seen for classical type C
  polynomials.  An examination of the classical norm function
  $\nu(n;a,b)$, defined in ~\eqref{eq:nunab}, reveals that if
  $a+b+1\in \Zm$, then the expression
  \[ 
    \lim_{z\to n}
    \frac{\Gamma(z+1+a+b)}{\Gamma(2z+1+a+b)\Gamma(2z+2+a+b)} =0 \]
  precisely when $0\le n< n^*$, where $n^*:= -n-1-a-b$.  In other
  words, the condition $0\le n<n^*$ implies that $n\in \ckI_{1-}$,
  where the latter is defined in ~\eqref{eq:ckI12}.  This means that
  $\nu_n=0$ when $0\le n< n^*$.  More generally, it is possible to
  demonstrate that class C exceptional non-vertex eigenpolynomials
  indexed by $I_{1-}$ have zero norm. In consequence, the exceptional
  and classical type C operators with a $I_{1-}$ that contains
  non-vertex indices can \emph{never} be regular.

  The CB class (generalized Chebyshev) is the intersection of the B
  and C classes, and as such represents the case where $\al,\be$ are
  half-integers.  The CB spectral diagrams are thus represented by two
  demi-diagrams: the first consists of
  $\boxcirc,\boxtimes, \boxotimes$ labels, the second consists of
  $\boxplus, \boxminus, \boxpm$ labels.  The encoding of the index
  sets by the diagrams follows the same rules detailed in Remark
  ~\ref{rem:B}.  Since $\al,\be$ are half-integers, either $\al+\be$ is
  odd or $\al-\be$ is odd; the two possibilities are mutually
  exclusive.  That means that in the CB class, there is always exactly
  one vertex eigenvalue.  If $\al+\be$ is odd, then the vertex
  eigenvalue is marked by either $\bboxcirc$ or $\bboxtimes$,
  according to whether it belongs to $\sigma_{1-}$ or to
  $\sigma_{2-}$.  If $\al-\be$ is even, then the vertex eigenvalue is
  marked by either $\bboxplus$ or $\bboxminus$, according to whether
  it belongs to $\sigma_{3-}$ or to $\sigma_{4-}$.  Some examples of
  CB diagrams are shown in figure ~\ref{fig:stypeCB} below.

\newpage

  \begin{figure}[H]
  \centering
  \begin{tikzpicture}[scale=0.6]
    \def\y{8}
    \draw (2.5,\y+2.5) node {$k$};
    \foreach \x in {-1,...,3} \draw (\x+4.5,\y+2.5) node {$\x$};
    \draw  (3,\y+1) grid +(5 ,1);
    \draw (10.5,\y+2.5) node[anchor=west] {$a=b=-1/2$};
    \path [draw,color=black] (2.5,\y+1.5) 
    ++(1,0) circle (5pt)
    ++(0,0)  node[draw,rectangle] {\phantom{$+$}}
    ++(1,0) node {$\otimes$}
    ++(1,0) node {$\otimes$}
    ++(1,0) node {$\otimes$}
    ++(1,0) node {$\otimes$}
    ++(1,0) node {\huge ...};
    \path (10.5,\y+1.5) node[anchor=west] {$\ckI_{1-}=\{0\},\;
      \ckI_{1+} = \{1,2,\ldots \},\;\ckI_{2+}=\Nz$};
    
    \draw  (1+3,\y) grid +(7-3 ,1);
    \path [draw,color=black] (0.5+3,\y+0.5) 
    ++(1,0) node {$\div$}
    ++(1,0) node {$\div$}
    ++(1,0) node {$\div$}
    ++(1,0) node {$\div$}
    ++(1,0) node {\huge ...};
    \path (10.5,\y+0.5) node[anchor=west]
    {$\ckI_{3+}=\ckI_{4+}=\Nz$};

    \def\y{4}
    \draw (2.5,\y+2.5) node {$k$};
    \foreach \x in {-1,...,3} \draw (\x+4.5,\y+2.5) node {$\x$};
    \draw  (3,\y+1) grid +(5 ,1);
    \draw (10.5,\y+2.5) node[anchor=west] {$a=b=1/2$};
    \path [draw,color=black] (2.5,\y+1.5) 
    ++(1,0)  node[draw,rectangle] {$\times$}
    ++(1,0) node {$\otimes$}
    ++(1,0) node {$\otimes$}
    ++(1,0) node {$\otimes$}
    ++(1,0) node {$\otimes$}
    ++(1,0) node {\huge ...};
    \path (10.5,\y+1.5) node[anchor=west] {$\ckI_{1+}=\Nz,\;
      \ckI_{2-} = \{0 \},\;\ckI_{2+}=\{1,2,\ldots \}$};
    
    \draw  (1+3,\y) grid +(7-3 ,1);
    \path [draw,color=black] (0.5+3,\y+0.5) 
    ++(1,0) node {$\div$}
    ++(1,0) node {$\div$}
    ++(1,0) node {$\div$}
    ++(1,0) node {$\div$}
    ++(1,0) node {\huge ...};
    \path (10.5,\y+0.5) node[anchor=west]
    {$\ckI_{3+}=\ckI_{4+}=\Nz$}; 
    
    \def\y{0}
    \draw (2.5,\y+2.5) node {$k$};
    \foreach \x in {-1,...,3} \draw (\x+4.5,\y+2.5) node {$\x$};
    \draw  (4,\y+1) grid +(4 ,1);
    \draw (10.5,\y+2.5) node[anchor=west] {$a=1/2,\; b=-1/2$};
    \path [draw,color=black] (3.5,\y+1.5) 
    ++(1,0) node {$\otimes$}
    ++(1,0) node {$\otimes$}
    ++(1,0) node {$\otimes$}
    ++(1,0) node {$\otimes$}
    ++(1,0) node {\huge ...};
    \path (10.5,\y+1.5) node[anchor=west] {$\ckI_{1+} = \ckI_{2+}=\Nz$};
    
    \draw  (3,\y) grid +(5 ,1);
    \path [draw,color=black] (2.5,\y+0.5) 
    ++(1,0) node[draw,rectangle] {$+$}
    ++(1,0) node {$\div$}
    ++(1,0) node {$\div$}
    ++(1,0) node {$\div$}
    ++(1,0) node {$\div$}
    ++(1,0) node {\huge ...};
    \path (10.5,\y+0.5) node[anchor=west]
    {$\ckI_{3-} = \{ 0 \},\; \ckI_{3+}=\{1,2,\ldots, \},\; \ckI_{4+}=\Nz$}; 
    
    \def\y{-4}
    \draw (2.5,\y+2.5) node {$k$};
    \foreach \x in {-1,...,3} \draw (\x+4.5,\y+2.5) node {$\x$};
    \draw  (4,\y+1) grid +(4 ,1);
    \draw (10.5,\y+2.5) node[anchor=west] {$a=-1/2,\; b=1/2$};
    \path [draw,color=black] (3.5,\y+1.5) 
    ++(1,0) node {$\otimes$}
    ++(1,0) node {$\otimes$}
    ++(1,0) node {$\otimes$}
    ++(1,0) node {$\otimes$}
    ++(1,0) node {\huge ...};
    \path (10.5,\y+1.5) node[anchor=west] {$\ckI_{1+} = \ckI_{2+}=\Nz$};
    
    \draw  (3,\y) grid +(5 ,1);
    \path [draw,color=black] (2.5,\y+0.5) 
    ++(1,0) node[draw,rectangle] {$-$}
    ++(1,0) node {$\div$}
    ++(1,0) node {$\div$}
    ++(1,0) node {$\div$}
    ++(1,0) node {$\div$}
    ++(1,0) node {\huge ...};
    \path (10.5,\y+0.5) node[anchor=west]
    {$\ckI_{3+} = \Nz,\;\ckI_{4-} = \{0\},\; \ckI_{4+}=\{1,2,\ldots \}$}; 

    \def\y{-8}
    \draw (2.5,\y+2.5) node {$k$};
    \foreach \x in {-1,...,4} \draw (\x+4.5,\y+2.5) node {$\x$};
    \draw  (3,\y+1) grid +(6 ,1);
    \draw (10.5,\y+2.5) node[anchor=west] {$\al=-1/2,\be=7/2,\quad K_1=\{2\},\; K_2=\emptyset,\; K_3=\{1,2\},\; K_4=\emptyset$};
    \path [draw,color=black] (2.5,\y+1.5) 
    ++(1,0)  node[draw,rectangle] {$\times$}
    ++(1,0) node {$\otimes$}
    ++(1,0) node {$\times$}
    ++(1,0) node {$\otimes$}
    ++(1,0) node {$\otimes$}
    ++(1,0) node {$\otimes$}
    ++(1,0) node {\huge ...};
    \path (10.5,\y+1.5) node[anchor=west] {$\ckI_{1+}=\{-1,1,2,\ldots\},\;
      \ckI_{2-} = \{1 \},\;\ckI_{2+}=\{2,4,5,\ldots \}$};
    
    \draw  (4,\y) grid +(5 ,1);
    \path [draw,color=black] (0.5+3,\y+0.5) 
    ++(1,0) node {$\div$}
    ++(1,0) node {$-$}
    ++(1,0) node {$-$}
    ++(1,0) node {$\div$}
    ++(1,0) node {$\div$}
    ++(1,0) node {\huge ...};
    \path (10.5,\y+0.5) node[anchor=west]
    {$\ckI_{3+}= \{ -2,1,2,\ldots \},\; \ckI_{4-}= \{0,-1\},\;
      \ckI_{4+}=\{ 2,5,6,\ldots \}$};   
    
  \end{tikzpicture}

  \caption{Classical and exceptional  CB spectral diagrams}
  \label{fig:stypeCB}
\end{figure}
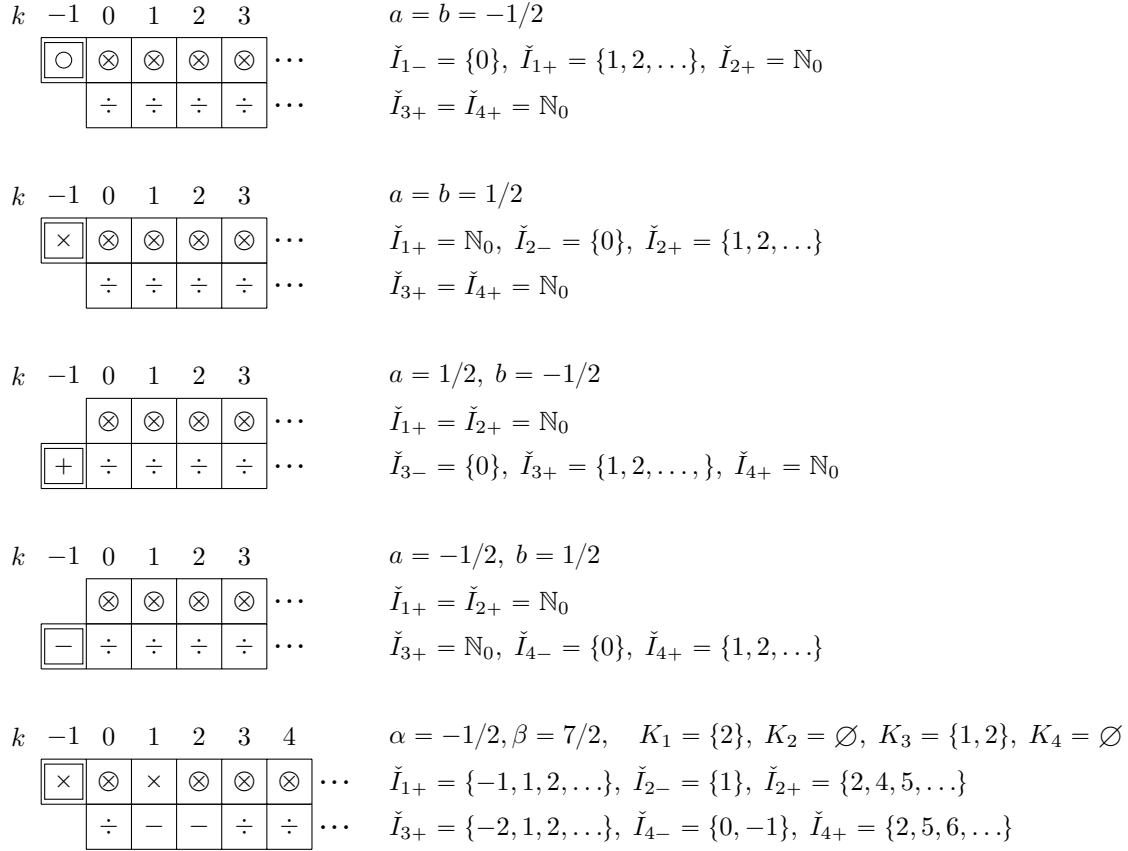
  

\end{remark}

\subsection{Degenerate class D}
\label{sec:D}
In order to describe and classify type D operators we introduce an
additional symbol $\boxtdown$, and define an \emph{extended
  spectral diagram} to be the spectral diagram of a type D operator
with a finite number of $\boxcirc$ labels replaced by $\boxtdown$
as per Definition ~\ref{def:ESD} below.  As asserted by Proposition
~\ref{prop:DSD}, below, the $\boxtdown$ labels index the support
of an isospectral deformation.

Let $a,b\in \Nz$, and let
\begin{equation}
  \label{eq:ckID1}
  \ckI_1 = \Nz,\quad
  \ckI_2 = \ckI_{2-} \sqcup \ckI_{2+},\quad 
  \ckI_3 = \ckI_{3-} \sqcup \ckI_{3+},\quad 
  \ckI_4= \ckI_{4-} \sqcup \ckI_{4+} \quad
\end{equation}
 where
 \begin{equation}
   \label{eq:ckID2}
  \begin{aligned}
    &\ckI_{2-}    = \{ 0,\ldots, \min\{a,b\}-1 \}, &&\ckI_{2+} = \{
    \max\{ a,b \},\ldots, a+b-1 \}\\
    &\ckI_{3-}
    =\{ n\in \Nz : 2n-a+b < 0 \}, &&\ckI_{3+}= \{ \max\{a-b,0\},\ldots,
    a-1 \} = \ckI_{2+}-b\\
    &\ckI_{4-} = \{ n\in \Nz : 2n+a-b < 0 \}, &&\ckI_{4+}=\{
    \max\{b-a,0\},\ldots, b-1 \} =\ckI_{2+}-a
  \end{aligned}
\end{equation}
Let $K,L_1,L_3,L_4\subset \Nz$ be finite and disjoint, with
$p, q_1,q_3, q_4$ the respective cardinalities.  Define
\[k^*=-k-a-b-1,\; K^* = -K-a-b-1,\; L^* = -L-a-b-1,\] and let
\begin{equation}
  \label{eq:I1234D}
  \begin{aligned}
    &I_{1+} = (\ckI_{1}\setminus (K\cup L_1\cup L_3\cup L_4))-p-q_3-q_4,&
    &I_{1-} = L^*_1-p-q_3-q_4,\\
    &I_{2+}=(\ckI_{2+}\cup K)+a+b+p+q_3+q_4,&
    &I_{2-}=(\ckI_{2-}\cup K^*)+a+b+p+q_3+q_4\\
    &I_{3+} = (\ckI_{3+}\cup K)-q_3+q_4+a,&
    &I_{3-} = (\ckI_{3-}\cup L_4)-q_3+q_4+a\\
    &I_{4+} =(\ckI_{4+} \cup K)+q_3-q_4+b , &
    &I_{4-} =(\ckI_{4-} \cup L_3)+q_3-q_4+b   
  \end{aligned}
\end{equation}
Set
\begin{equation}
  \label{eq:albeD}
  \al = a+p+2q_4, \quad \be = b+p+2q_3,\\
\end{equation}
and let $\sigma_\imath=\lam_\imath(I_\imath;\al,\be)$ and
$\sigma_{\imath\pm} = \lam_\imath(I_{\imath\pm};\al,\be),\; \imath =
1,2,3,4$.
\begin{prop}
  The assumptions that $a,b\in \Nz$ entail the same conditions on
  $\al,\be$.  Moreover,
  $\sigma_{234}:= \sigma_{2-} = \sigma_{2+} = \sigma_{3+} =
  \sigma_{4+}$ and
  \begin{equation}
    \label{eq:sigD}
    \sigq(\al,\be)
    = \sigma_{1+} \sqcup \sigma_{1-}\sqcup \sigma_{234}\sqcup
    \sigma_{3-} \sqcup \sigma_{4-}
  \end{equation}
\end{prop}

\noindent For $a,b\in \Nz$, let
$\LamD(K,L,a,b)$ be the mapping described by
\[ (\sigma_{1+},\sigma_{1-},\sigma_{234},\sigma_{3-},\sigma_{4-} )
  \mapsto (\boxcirc,\boxtdown,\boxfsq,\boxplus,\boxminus) \]
\noindent

As before, Wronskian-type definitions ~\eqref{eq:tauGB} will not
suffice to define the relevant $\tau$, because such Wronskians require
that the factorization eigenvalues be distinct.  Thus, just as for the
type A subclass, we have to divide our construction into two
stages.

The first stage corresponds to the case where $K=\emptyset$.  Let
$\ell_1,\ldots, \ell_{q_1+q_3+q_4}$ be an enumeration of
$L_{134}:=L_1\cup L_3\cup L_4$ ordered in such a way that
$\ell_1,\ldots, \ell_{q_1}$ is an enumeration of $L_1$, so that
$\ell_{q_1+1},\ldots, \ell_{q_1+q_3}$ is an enumeration of $L_3$, and
so that $\ell_{q_1+q_3+1},\ldots,\ell_{q_1+q_3+q_4}$ is an
enumeration of $L_4$.  Set
\begin{align}\label{eq:rhoij}
  \rho_{ij}(x;a,b) &= \int_{-1}^x \pi_i(x;a,b) \pi_j(x;a,b)\,
                   (1-x)^a(1+x)^b,\quad i,j\in \Nz.
\end{align}
Let $\bt_{\bell} := (t_{\ell_1},\ldots,
t_{\ell_{q_1}})$ be a finite list of indeterminates, and let
\[ t_\ell =
  \begin{cases}
    \rho_{\ell\ell}(-1;a,b)=0 & \text{if } \ell \in L_3\\
    \rho_{\ell\ell}(1;a,b)=\nu(\ell;a,b) & \text{if } \ell \in L_4.
  \end{cases} \]
Define
$\cR=\cR(x,\bt_\bell;L,a,b)$ to be the square matrix with entries 
\[ 
  \cR_{ij}  =
    t_{\ell_i} \delta_{ij}+ \rho_{\ell_i\ell_j}(x;a,b) ,\quad
    i,j=1,\ldots, q_1+q_3+q_4 
\]
Let $\cA(n)=\cA(x,\bt_\bell;n,L_1,L_3,L_4,a,b),\; n\in \Nz$ be the square matrix
obtained by augmenting 
$\cR$ with an extra row and column as follows:
\begin{align*}
  \cA(n)_{00}
  &= \pi_{n}(x;a,b),&
  \cA(n)_{0j}
  &= \pi_{\ell_j}(x;a,b),\\
  \cA(n)_{i0}  &=  \rho_{\ell_i n}(x;a,b), &
  \cA(n)_{ij}
  &=  \cR_{ij},
\end{align*}
with $i,j$ having the same range as above.  Set
\begin{equation}
  \label{eq:InchiD}
\begin{aligned} 
  \hI_1 &:=  ((\Nz\setminus L_{134})\cup
          L_1^*)-q_3-q_4,\\ 
  \fu(j) &:= \max\{ j, j^*  \},\\
  \hn_i&:=\fu(i+q_3+q_4),\\
  \hchi(z)&:=
  (-1)^{q_4} \prod_{\ell\in L_{34}}
  \frac{z-\ell^*}{z-\ell},\quad L_{34}:= L_3\cup L_4,
\end{aligned}
\end{equation}
and define
\begin{align}
  \label{eq:htauD}
  \htau(x,\bt_\bell;L_1,L_3,L_4,a,b)
  &:=  (1-x)^{-q_4(q_4+a)}(1+x)^{-q_3(q_3+b)}\det    \cR,\\
  \label{eq:hpiD}
      \hpi_{i}(x,\bt_\bell;L_1,L_3,L_4,a,b)
    &:=  (x-1)^{-q_4}(1+x)^{-q_3}   \hchi(n_i)
      \frac{\det \cA(\hn_i)}{\det\cR},\quad i\in \hI_1,\;
\end{align}

Next, let $k_1,\ldots,k_p$ be an enumeration of $K$, let
\[ \hpi(K-q_3-q_4):= \hpi_{k_1-q_3-q_4},\ldots, \hpi_{k_p-q_3-q_4}, \]
denote the indicated list.  Set
\begin{align*}
  n_i &:= \fu(i+p+q_3+q_4),\\
  \chi(z) &:=
              (-1)^{q_4} \prod_{k\in K} \frac{1}{z-k} \prod_{\ell\in L_{34}}
  \frac{z-\ell^*}{z-\ell},\quad L_{34}:= L_3\cup L_4,
\end{align*}
and define
\begin{align}
  \label{eq:tauD}
  \tau(x,\bt_\bell;K,L,a,b)
  &:=  \htau\Wr[\hpi(K-q_3-q_4)],\\
  \label{eq:piD}
  \pi_i(x,\bt_\bell;K,L,a,b)
  &:=
    \chi(n_i)\frac{\Wr[\hpi(K-q_3-q_4),\hpi_{i+p}]
    }{\Wr[\hpi(K-q_3-q_4)]}     ,\quad     i\in I_1.
\end{align}

\begin{prop}
  \label{prop:TD}
  The above defined
  $\{\tau\}_{\bt_\bell}=\tau(x,\bt_\bell;K,L,a,b),\al,\be$ is a
  $q_1$-parameter family of polynomials with
  \begin{equation}
    \label{eq:degtauD}
    \deg \tau =\sum_{k\in K} k +  2\!\!\!\!\sum_{\ell\in
      L_{134}}\!\!\!\!\ell -\binom{p}{2} -p(q_3+q_4)+ q_1 - 
     q_3(q_3-1)-q_4(q_4-1)+a(q_1+q_3)+b(q_1+q_4) 
  \end{equation}
  The corresponding operators
  $\{T\}_{\bt_\bell}=\Trg(\tau_{\bt_\bell},\al,\be)$ are a
  $q_1$-parameter isospectral deformation with extended spectral
  diagram $\LamD(K,L_1,L_3,L_4,a,b)$ and index sets as defined
  above. The corresponding families of exceptional Jacobi
  quasi-polynomials, with the asymptotically monic
  normalization\footnote{This means that (i) the value $\pi_i(-1)$ is
    independent of $\bt_\bell$ and that (ii) in limit as $\bt_\bell\to
    \infty$ the resulting polynomials are monic.}, are given by
  $\{\pi_i(x,\bt_\bell)\}_{\bt_\bell},\; i\in I_1$.  The corresponding
  norms have the form
  \begin{equation}
    \label{eq:nuD}
     \nu_i = \kappa(n_i;\bt_\bell;K,L,a,b)\,
    \nu(n_i;a,b),\quad i\in I_1     ,
  \end{equation}
where
\begin{equation}
  \label{eq:kappaD}
  \kappa(z)
  :=  \prod_{\ell\in L_1} \lp 1-
  \delta_{z,\ell^*}\frac{\tnu_\ell}{t_\ell+\tnu_\ell} \rp 
  \prod_{k\in K}
  \lp\frac {z-k^*}{z-k}\rp
  \prod_{\ell\in L_{34}}\!\!
  \lp\frac {z-\ell^*}{z-\ell}\rp^2   
  ,\quad \tnu_\ell = \nu(\ell;a,b).
\end{equation}
\end{prop}

\begin{prop}
  \label{prop:DSD}
  Let $T=\Trg(\tau;\al,\be)$ be a type D exceptional Jacobi operator
  and $\Lam$ its extended spectral diagram. Then,
  $\Lam=\LamD(K,L_1,L_3,L_4,a,b)$ for a unique choice of
  $K,L_1,L_3,L_4$ and $(a,b)\in \{ (0,0), (1,0), (0,1)\}$. Moreover,
  up to a scale factor, $\tau= \tau(x,\bt_\bell;K,L,a,b)$ for some
  non-zero value of $\bt_\bell$.
\end{prop}

\begin{remark}
  \label{rem:D}
  The type D class is a confluence of the type A and B classes.  As a
  consequence of the defining assumptions $\al,\be\in \Nz$, there
  exists degenerate eigenvalues where \emph{all} eigenfunctions are
  rational and all singular asymptotic types 2,3,4 eigenfunctions are
  present.  We use $\boxfsq$ to label such eigenvalues.  As well,
  class D operators admit spectral deformations supported at a finite
  subset of the quasi-rational spectrum.  Such eigenvalues are
  labelled by $\boxtdown$; the degrees of the corresponding
  quasi-polynomial eigenfunctions are elements of $I_{1-}$.


  This peculiarity of the D class is due to the fact that a
  quasi-polynomial eigenvalue does not uniquely determine the index of
  the corresponding quasi-polynomial eigenfunction.  The relationship
  between a type 1 index $i\in I_1$ and type 1 eigenvalue
  $\lam\in \sigma_1$ is
  \[ \lam = \lam_{1}(i;\al,\be) = -ii^\star,\]
  where
  \begin{equation}
    \label{eq:istar}
    i^\star:= -i-\al-\be-1,\quad i\in \Z
  \end{equation}
  Thus, if $\al+\be\in \Z$, then a given $\lam\in \sigma_1$ admits two
  possible values of $i$ according to the symmetry $i\mapsto i^\star$,
  which is the reflection across the vertex $i=(\al+\be+1)/2$ of the
  discrete parabola $\{ (i,i(i+\al+\be+1)): i\in \Z\}$.  This is the
  reason behind the decomposition of the type 1 quasi-polynomial index
  set as $I_1=I_{1+}\sqcup I_{1-}$, with the latter defined as in
  ~\eqref{eq:I1pm}\footnote{One can show that the vertex value
    $(\al+\be+1)/2$ can never be a quasi-polynomial eigenvalue.}.  The
  effect of an isospectral deformation is that it reflects the index
  of a quasi-polynomial eigenfunction from $I_{1+}$ to $I_{1-}$.  This
  fact allows us to mark the support of the isospectral deformation in
  terms of degrees.

  A class D exceptional operator depends on $q=|I_{1-}|$ continuous
  parameters $t_{\ell}$, where $ i=(\ell+\gamma)^\star,\; i\in I_{1-}$
  where $\gamma=p+q_3+q_4$ is the constant defined above.  These
  parameters determine the corresponding norms $\nu_i,\; i\in I_{1-}$
  in a one-to-one fashion.  Thus the norms together with the extended
  spectral diagram fully determine the corresponding exceptional
  operator.  The norms in question are subject to certain
  inequalities.  In the limit, as the parameters $t_{\ell}$ approach
  these boundary values, the isospectral deformation breaks down and
  the resulting transformation corresponds to a removal of an
  eigenvalue from the spectrum.

  It is instructive to consider the subclass of classical type D
  operators $T(a,b),\; a,b\in \Nz$ and their qr-eigenfunctions as
  shown in Table ~\ref{tab:stype}.  As a consequence of
  ~\eqref{eq:Pdegenid}, we have that
  \[ \phi_{3,k+a}(x;A,B),\phi_{4,k+b}(x;A,B),
    \phi_{2,k+a+b}(x;A,B)\dashrightarrow \phi_{1,k}(x;a,b),\quad k\in
    \Nz\] as $(A,B)\to (a,b)$.  Consequently, the qr-eigenvalues
  $\lam_{1}(k;a,b),\; k\in \Nz$ are simple with asymptotic label
  $\boxcirc$.  If $a\ge b$, then the qr-eigenfunctions
  $\phi_{4,k+b}(x;a,b), \phi_{3,k+a}(x;a,b),\; k\in \{-b,\ldots, -1\}$
  have a common eigenvalue
  \[
    \lam_1(k;a,b)=\lam_3(k+a;a,b)=\lam_4(k+b;a,b)=\lam_2(-k-1;a,b),\]
  but are linearly independent.  The eigenfunction
  $\phi_{2,-k-1}(x;a,b)$ also occurs at this eigenvalue as a linear
  combination of the other two.  We signify this kind of degenerate
  eigenvalue using the asymptotic label $\boxfsq$.  Finally, for
  $k\in \{ -a,\ldots, -b-1\}$, the qr-functions
  $\phi_{2,k+a+b}, \phi_{3,k+a}$ are linearly dependent. The
  corresponding eigenvalues $\lam_1(k;a,b)$ are simple with a
  $\boxplus$ label.

  We visualize the spectral diagram of a class D operator as a
  demi-diagram of $\boxcirc$ labels with finitely of these labels
  replaced by $\boxplus,\boxminus, \boxfsq$ and $\boxtdown$ labels.
  A demi-diagram suffices because the condition
  $\al,\be\in \Nz$ means that the degrees of the type 1,2,3,4
  eigenfunctions are all commensurate and all co-exist on the same
  semi-axis.

  The reflection $i\mapsto i^\star:= -i-1-\al-\be$ is
  a symmetry of the eigenvalue parabola
  $\{ (i,\lam_1(i;\al,\be)) : i \in \Z \}$ because
  \[ \lam_1(i;\al,\be) = \lam_1(i^\star;\al,\be) = i(i+\al+\be+1) =
    -i\, i^\star,\quad i\in \Z.\] Consequently, the mapping
  $i\mapsto i(i+\al+\be+1)$ is a bijection of the paired degrees
  $\{ (i,i^\star): i \in \Z \}$ and the eigenvalue set
  $\sigq(\al,\be)$.  We therefore number the cells of the demi-diagram
  by pairs of integers $(i,i^\star)$ with $i\ge i^\star$.  In effect,
  the visualization replaces the curved eigenvalue parabola
  $\{ (i,\lam_1(i;\al,\be)): i\in \Z \}$ by a linear demi-diagram
  indexed by $(i,i^\star)$.  We decorate the demi-diagram with the
  corresponding asymptotic labels, placing each label at the cell
  corresponding to the degree of that eigenfunction.

  If $\al+\be$ is odd, there is a unique integer index $i$ such that
  $i=i^\star$.  We then say that the corresponding eigenvalue
  $\lam_1(i;\al,\be)=\lam_3(i+\al;\al,\be)$ is the vertex eigenvalue.
  This eigenvalue, if present, is always located at the left-most
  position of the type D demi-diagram and corresponds to either a type
  3 or a type 4 eigenvalue.  Correspondingly, we mark the vertex
  eigenvalue by a doubly-lined $\bboxplus$ or $\bboxminus$ label,
  according to the asymptotic type of the eigenfunction at the vertex
  eigenvalue.  The rest of the $\boxplus, \boxminus$ labels are placed
  in cells at positions $I_{3-}-\al$ and $I_{4-}-\be$, respectively.
  The $\boxfsq$ labels are placed at positions $I_{2+}-\al-\be$.  The
  $\boxcirc$ labels are placed in positions $I_{1+}$, while the
  $\boxtdown$ labels are placed at positions $I_{1-}$.

  In particular, the spectral diagram of a classical operator
  $T(a,b), \; a\ge b$ consists of $a-b$ consecutive $\boxplus$ labels,
  followed by $b$ consecutive $\boxfsq$ labels, followed by an
  infinity of $\boxcirc$ labels.  The $\boxplus$ labels are attached
  to the $\sigma_{3-}$ eigenvalues and the corresponding type 3
  eigenfunctions are indexed by $I_{3-}$.  The $\boxfsq$ labels are
  attached to $\sigma_{234}$ eigenvalues.  The corresponding
  eigenspaces are 2-dimensional, containing type 2,3,4 quasi-rational
  eigenfunctions indexed by $I_{2-}, I_{2+}, I_{3+}, I_{4+}$.  Lemma
  ~\ref{lem:Drat} below gives an explicit description of the
  relationship between the degrees and eigenvalues of such
  eigenfunctions.  Figure ~\ref{fig:stypeD} shows some representative
  spectral diagrams of classical class D operators.

 \begin{figure}[H]
  \centering
  \begin{tikzpicture}[scale=0.6]
    \def\y{13+4}

    \foreach \x in {-1,...,4}    \draw  (\x-0.5,\y+1.5)    node {$\x$}; 
    \draw  (-3.5,\y+1.5)    node[anchor=west] {$k$};
    \draw  (5.5,\y+1.5)    node[anchor=west] {$a=b=0,\; \lam=k(k+1)$};
    \draw    (-1,\y) grid +(5 ,1);
    \path [draw,color=black]    (-1.5,\y+0.5)
    ++(1,0) circle (5pt)
    ++(1,0) circle (5pt)
    ++(1,0) circle (5pt)
    ++(1,0) circle (5pt)
    ++(1,0) circle (5pt)
    ++(1,0) node {\huge ...}
    ++(1,0) node[anchor=west] {$I_1=\Nz,\; I_2=I_3=I_4=\emptyset$};  
    \def\y{11+3}

    \foreach \x in {-1,...,4}    \draw  (\x-0.5,\y+1.5)    node {$\x$}; 
    \draw  (-3.5,\y+1.5)    node[anchor=west] {$k$};
    \draw  (5.5,\y+1.5)    node[anchor=west] {$a=1,b=0,\;\lam=k(k+2)$};
    \draw    (-2,\y) grid +(6,1);
    \path [draw,color=black]    (-2.5,\y+0.5)
    ++(1,0) node[draw,rectangle] {$+$}
    ++(1,0) circle (5pt)
    ++(1,0) circle (5pt)
    ++(1,0) circle (5pt)
    ++(1,0) circle (5pt)
    ++(1,0) circle (5pt)
    ++(1,0) node {\huge ...}
    ++(1,0) node[anchor=west] {$I_1=\Nz,\;
      I_{3-}=\{0\},\;I_2=I_{3+}=I_4=\emptyset$};    
    \def\y{9+2}

    \foreach \x in {-1,...,4}    \draw  (\x-0.5,\y+1.5)    node {$\x$}; 
    \draw  (-3.5,\y+1.5)    node[anchor=west] {$k$};
    \draw  (5.5,\y+1.5)    node[anchor=west] {$a=0,b=1,\; \lam=k(k+2)$};
    \draw    (-2,\y) grid +(6,1);
    \path [draw,color=black]    (-2.5,\y+0.5)
    ++(1,0) node[draw,rectangle] {$-$}
    ++(1,0) circle (5pt)
    ++(1,0) circle (5pt)
    ++(1,0) circle (5pt)
    ++(1,0) circle (5pt)
    ++(1,0) circle (5pt)
    ++(1,0) node {\huge ...}
    ++(1,0) node[anchor=west] {$I_1=\Nz,\; I_{4-}=\{0\},\;I_2=I_3=I_{4+}=\emptyset$};  

    \def\y{7+1}
    \draw  (-3.5,\y+1.5)    node[anchor=west] {$i$};
    \foreach \x in {-3,...,1}
    \draw  (\x+2.5,\y+1.5) node {$\x$};
    \draw  (-3.5,\y-0.5)    node[anchor=west] {$i^*$};
    \foreach \x in {-8,...,-4}
    \draw  (-\x-4.5,\y-0.5) node {$\x$};
    \draw  (5.5,\y+1.5)    node[anchor=west]
    {$a=5,b=1,\;\lam=i(i+7)=k(k+1)-12,\; i=k-3$};

    \draw    (-1,\y) grid +(5 ,1);
    \path [draw,color=black] (-1.5,\y+0.5) 
    ++(1,0) node {$+$}
    ++(1,0) node {$+$}
    ++(1,0) node {$\sqbullet$}
    ++(1,0) circle (5pt)
    ++(1,0) circle (5pt)
    ++(1,0) node {\huge ...}
    ++(1,0) node[anchor=west]
    {$I_1=\{0,1,\ldots\},I_{2+}=\{5\},\;I_{3+}=\{4\},\;I_{4+}=\{0\}$};
    \draw (5.5,\y-0.5) node[anchor=west]
    { $I_{2-}=\{0\},I_{3-}=\{0,1\},\; I_{4-}=\emptyset$};

    \def\y{4}
    \draw  (-3.5,\y+1.5)    node[anchor=west] {$i$};
    \foreach \x in {-3,...,2}
    \draw  (\x+1.5,\y+1.5) node {$\x$};
    \draw  (-3.5,\y-0.5)    node[anchor=west] {$i^*$};
    \foreach \x in {-8,...,-3}
    \draw  (-\x-4.5,\y-0.5) node {$\x$};
    \draw  (5.5,\y+1.5)    node[anchor=west]
    {$a=3,b=2,\;\lam=i(i+6)=k(k+2)-8,\; i=k-2$};

    \draw    (-6+4,\y) grid +(10-4 ,1);
    \path [draw,color=black] (-2.5,\y+0.5) 
    ++(1,0) node[draw,rectangle] {$+$}
    ++(1,0) node {$\sqbullet$}
    ++(1,0) node {$\sqbullet$}
    ++(1,0) circle (5pt)
    ++(1,0) circle (5pt)
    ++(1,0) circle (5pt)
    ++(1,0) node {\huge ...}
    ++(1,0) node[anchor=west]
    {$I_1=\Nz,\;I_{2-}=\{0,1\},I_{3-}=\{0\},\; I_{4-}=\emptyset$};
    \draw (5.5,\y-0.5) node[anchor=west]
    { $I_{2+}=\{3,4\},\;I_{3+}=\{1,2\},\;I_{4+}=\{-1,0\}$};
  \end{tikzpicture}

  \caption{Classical type D spectral diagrams}
  \label{fig:stypeD}
\end{figure}

  Conversely, one can recover the canonical parameters of Proposition
  ~\ref{prop:DSD} from the demi-diagram as follows.  Place the origin
  $k=0$ at the leftmost non-vertex label.  Thus, in the even case, the
  demi-diagram begins with cell $k=0$ and in the odd case it begins
  with cell $k=-1$. The elements of $K$ are the $k$ positions of the
  $\boxfsq$ labels, with $p$ be the number of such labels. The
  elements of $L_3, L_4$ are the $k$ positions of the $\boxminus$ and
  $\boxplus$ non-vertex labels, respectively, with $q_3, q_4$ the
  respective cardinalities. Let $\hq_3, \hq_4$ be the number of all
  $\boxminus$ and $\boxplus$ labels (including the vertex label, if
  present).  In accordance with ~\eqref{eq:albeD}, the value of $\al$
  is $p+q_4+\hq_4$, while the value of $\be$ is $p+q_3+\hq_3$.
  Finally, the elements of $L_1$ are the $k$ positions of the
  $\boxtdown$ labels.
  
  The index sets may now be recovered via ~\eqref{eq:I1234D}, or
  directly from the diagram as follows.  The degree index $i$ is
  recovered as $i=k-p-q_3-q_4$, a shift that counts the total number
  of $\boxplus, \boxminus, \boxfsq$ labels.  The contents of $I_{1+}$
  are simply the $i$ positions of the $\boxcirc$ labels.  The support
  of the isospectral deformation $I_{1-}$ consists of the $i^\star$
  positions of the $\boxtdown$ labels.  The elements of
  $I_{2+}, I_{3+}, I_{4+}$ are the $i$ positions of the $\boxfsq$
  labels shifted by $\al+\be, \al,\be$, respectively. The elements of
  $I_{2-}$ are the $i^\star$ positions of the $\boxfsq$ labels shifted
  by $\al+\be$.  The elements of $I_{3-},I_{4-}$ are the respective
  $i^\star$ positions of the $\boxminus,\boxplus$ labels shifted by
  $\al,\be$ respectively.  Figure ~\ref{fig:etypeD} shows some examples
  of such representations for spectral diagram of exceptional class D
  operators.

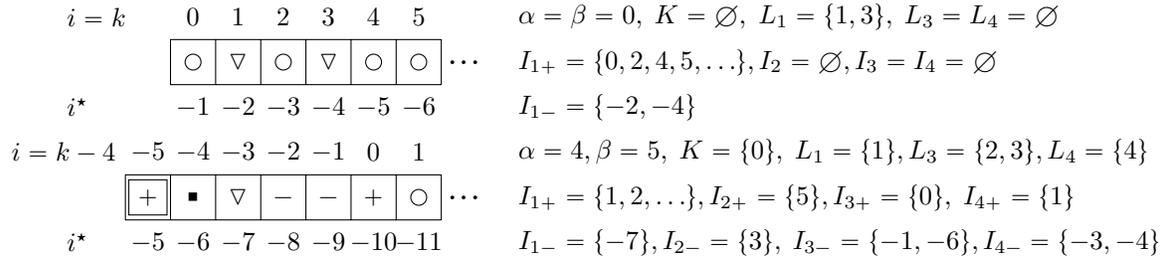
\begin{figure}[H]
  \begin{tikzpicture}[scale=0.6]
    \def\y{4+1}
    \foreach \x in {0,...,5}
    \draw  (\x-0.5,\y+2.5) node {$\x$};
    \draw  (-3.5,\y+2.5)    node[anchor=west] {$i=k$};
    \draw  (6.5,\y+2.5)    node[anchor=west]
    {$\al=\be=0,\; K=\emptyset,\; L_1=\{1,3\},\; L_3=L_4=\emptyset$};

    \draw    (-1,\y+1) grid +(6 ,1);
    \path [draw,color=black] (-1.5,\y+1.5) 
    ++(1,0) circle (5pt)
    ++(1,0) node {$\triangledown$}
    ++(1,0) circle (5pt)
    ++(1,0) node {$\triangledown$}
    ++(1,0) circle (5pt)
    ++(1,0) circle (5pt)
    ++(1,0) node {\huge ...}
    ++(1,0) node[anchor=west]
    {$I_{1+}=\{0,2,4,5,\ldots\},I_{2}=\emptyset,I_{3} = I_4=\emptyset$};


    \draw  (-3.5,\y+0.5)    node[anchor=west] {$i^\star$};
    \foreach \x in {-6,...,-1}
    \draw  (-\x-1.5,\y+0.5) node {$\x$};
    \draw  (6.5,\y+0.5)    node[anchor=west]
    {$I_{1-} = \{-2,-4\}$};

    \def\y{2}
    \draw  (-2,\y+1.5)    node[anchor=east] {$i=k-4$};
    \foreach \x in {-5,...,1}
    \draw  (\x+3.5,\y+1.5) node {$\x$};

    \draw  (-3.5,\y-0.5)    node[anchor=west] {$i^\star$};
    \foreach \x in {-11,...,-5}
    \draw  (-\x-6.5,\y-0.5) node {$\x$};

    \draw    (-2,\y) grid +(7 ,1);

    \draw (6.5,\y+1.5) node[anchor=west]
    {$\al = 4,\be=5,\; K=\{0\},\; L_1=\{ 1\},L_3= \{ 2,3 \},L_4=\{4\}$};

    \path [draw,color=black] (-2.5,\y+0.5) 
    ++(1,0) node[draw,rectangle] {$+$}
    ++(1,0) node {$\sqbullet$}
    ++(1,0) node {$\triangledown$}
    ++(1,0) node {$-$}
    ++(1,0) node {$-$}
    ++(1,0) node {$+$}
    ++(1,0) circle (5pt)
    ++(1,0) node {\huge ...}
    ++(1,0) node[anchor=west]
    {$I_{1+}=\{ 1,2,\ldots \}, I_{2+}=\{5\},I_{3+}=\{0\},\;I_{4+}=\{1\}$}; 
    \draw (6.5,\y-0.5) node[anchor=west]
    { $I_{1-} =\{-7\},I_{2-}=\{3\},\;I_{3-}=\{-1,-6\},I_{4-}=\{-3,-4\}$};
  \end{tikzpicture}
  \caption{Exceptional class D spectral diagrams}
  \label{fig:etypeD}
\end{figure}

The class D construction detailed above is, essentially, an
  isospectral deformation of a classical operator followed by some
  state-deleting transformations. However, since the factorization
  eigenvalues of the first stage are distinct from the factorization
  eigenvalues of the second stage, one could interchange the order of
  operations, and construct the same objects as isospectral
  deformations of exceptional type D operators.  This constitutes a
  novel class of exceptional polynomials and operators, although
  Duran's construction in \cite{Du23} should, in principle, yield the
  same objects.
  \end{remark}


\section{The formal theory of second-order operators.}
  \label{sec:formal}
In this section we gather some basic results in differential algebra
  related to Darboux transformations.   Most of the calculations here
  are formal, involving only rational or quasi-rational functions.
  Many of these results are well-known and have appeared in previous
  publications, but we reprise them here for the benefit of the reader.
\subsection{Rational Darboux Transformations.}
\label{sec:RDT}

\begin{prop}
  \label{prop:RDTeval}
  Let $T,\hT\in \DpQ$ and let $A=b(D-w),\; b,w\in \cQ$ be a first
  order operator with rational coefficients.  If
  \begin{equation}
    \label{eq:AThTA}
    A\circ T =  \hT\circ A,
  \end{equation}
  then the quantity $\lam=\Ric{T}w$ (see ~\eqref{eq:Twlam} for the latter
  notation) is a constant. 
\end{prop}
\begin{proof}
  Writing  $T = p D^2 + qD + r$, $\hT=p D^2 + \hq D + \hr$,
  where $p,q,\hq,r,\hr\in \cQ$, we have
  \[ A \circ T - \hT \circ A = b(q-\hq+p') D^2 + b(r-\hr+w(\hq-q)+
    q' +2 p w') D + b(w(\hr-r) + r' + \hq w' + p w''). \]
  It follows that
  \begin{align*}
    &q+p'-\hq=0,\\
    &r-\hr+ w p' + q' + 2pw'=0,\\
    &p(w''+2 w w') + p'(w' + w^2)+ w q' + qw'  + r'=0.
  \end{align*}
  The left-side of the last line is precisely the derivative of
  $p(w'+w^2)+qw + r$.  The latter, in turn is precisely $\Ric{T}w$.
\end{proof}

\begin{prop}
  Let $w(x)$ be a rational function and define the qr-function
  \begin{equation}
    \label{eq:phifromw}
    \phi(x) = \exp\lp\int w(x) dx\rp.
  \end{equation}
  Then, the Riccati-like relation $\lam = \Ric{T}w$ is equivalent to
  the eigenvalue relation  $T\phi = \lam \phi$.
\end{prop}
\begin{proof}
  By construction, $w=\phi'/\phi$, and hence $A\phi=0$.  Applying
  both sides of ~\eqref{eq:AThTA} to $\phi$ then gives $AT\phi = 0$.
  Since $\ker A$ is 1-dimensional, it follows that $T\phi = \lam \phi$
  for some constant $\lam$.  By ~\eqref{eq:Twlam}, that constant is
  precisely the expression $\lam=\Ric{T}w=(T\phi)/\phi$.
\end{proof}

\begin{prop}
  \label{prop:dualRDT}
  If ~\eqref{eq:AThTA} holds, then there exists a 1st order
  $\hA=\hb(D-\hw),\; \hb,\hw\in \cQ$ such that
  \begin{equation}
    \label{eq:TAAhT}
    \hA \circ \hT = T \circ \hA.
  \end{equation}
  Writing $A=b(D-w)$, $T = p D^2 + qD + r$, $\hT=p D^2 + \hq D + \hr$,
  where $b,w,q,\hq,r,\hq\in \cQ$, we have
  \begin{align}
    \label{eq:bhbqhq}
    &b\hb = p,\quad
    w+ \hw = -\frac{q}{p} + \frac{b'}{b} = -\frac{\hq}{p}  +
      \frac{\hb'}{\hb},\\
    \label{eq:rhr}
    &p(w'+w^2) + q w + r = p (\hw' + \hw^2)+ \hq \hw + \hr.
  \end{align}
\end{prop}
\begin{proof}
  Set $\hb=p/b$ and $\hw=-w-q/p-b'/b$ so that ~\eqref{eq:bhbqhq} holds
  by definition. By Proposition ~\ref{prop:RDTeval},
  \[ \lam =\Ric{T}w=p(w'+w^2) + qw + r\] is a constant. Hence, by direct
  calculation,
  \begin{align*}
    \hA \circ b
    &= p(D+b'/b) + p(w+q/p-b'/b) = p(D+w)+q,\\
    \hA\circ A
    &=  (p(D+w)+q)\circ (D-w)    \\
    &= p D^2 + q D- p(w'+w^2) -qw\\
    & = p D^2 + q D + r-\lam  = T-\lam.
  \end{align*}
  Consequently,
  \[ (\hA\circ \hT - T\circ \hA) \circ A = \hA\circ A \circ T - T\circ
    \hA \circ A = 0.\] The algebra of differential operators does not
  have zero divisors. Therefore, ~\eqref{eq:TAAhT} holds.  A similar
  calculation of $A\circ \hA$ yields 
  ~\eqref{eq:rhr}.
\end{proof}

\begin{definition}
  We say that $T,\hT\in \DpQ$ are related by an RDT (short for
  rational Darboux transformation) whenever there exists a first order
  operator $A=b(D-w),\; b,w\in \cQ$ such that
  ~\eqref{eq:AThTA} holds.  If this is the case, we will call $b(x)$
  the factorization gauge, and the quasi-rational function $\phi(x)$
  such that $w=\phi'/\phi$ the factorization eigenfunction.  The
  constant $\lam = (T\phi)/\phi = \Ric{T}w$ will be called the
  factorization eigenvalue, and we will write $T\rdt\lam\hT$.
\end{definition}
\noindent
As a direct Corollary of Propositions ~\ref{prop:RDTeval} and
~\ref{prop:dualRDT}, we have the following.
\begin{cor}
  \label{cor:RDTfac}
  Every $T\rdt{\lam}\hT$ entails the inverse relation
  $\hT\rdt{\lam} T$, and dual factorizations
  \begin{equation}
    \label{eq:ThTfac}
    T = \hA\circ A + \lambda,\quad \hT=A\circ\hA + \lambda.
  \end{equation}
\end{cor}

\begin{prop}
  \label{prop:Tgauged}
  Let $T=pD^2+qD+r$ be an operator in $\DpQ$, let $\mu(x)$ be a
  quasi-rational function.  Then,
  $\tT := \cM(\mu)\circ T\circ\cM(\mu^{-1})$ is also in $\DpQ$ with
  \begin{equation}
    \label{eq:Tgauged}
    \tT =  pD^2 +(q-2 p \sigma) D + (r-q \sigma-p \sigma'+p\sigma^2),\quad
    \sigma = \mu'/\mu.
  \end{equation}
  Moreover, if $\phi(x)$ is a qr-eigenfunction of $T$ with eigenvalue
  $\lam$, then $\tphi(x)= \mu(x)\phi(x)$ is a qr-eigenfunction of
  $\tT$ with the same eigenvalue.
\end{prop}
\begin{proof}
  The proof follows by a straightforward calculation.
\end{proof}

\noindent
We already saw that if ~\eqref{eq:AThTA} holds with
$A=b(D-w)$, then $w\in \Riclam{T}$ and $\lam = \Ric{T}w \in \sigq(T)$.
This condition is both necessary and sufficient.
\begin{prop}
  \label{prop:RDTqrefun}
  An RDT $T\rdt{\lam}\hT$ exists if and only if $\lam \in \sigq(T)$.
  For a given $\lam\in \sigq(T)$, the corresponding RDT is uniquely
  determined, up to a rational gauge transformation, by a choice of
  $w\in \Riclam{T}$.
\end{prop}
\begin{proof}
  If ~\eqref{eq:AThTA} holds, then the corresponding $\lam$ is a qr
  eigenfunction by Proposition ~\ref{prop:RDTeval}.  We now prove the
  converse.  Let $T=pD^2+qD+r$ be an operator in $\DpQ$. Let
  $\lam\in \sigq(T), \; w\in \Riclam{T}$, and let $b(x)$ be a non-zero
  rational function.  Let $\hb,\hw$ be as per ~\eqref{eq:bhbqhq}, and set
  \begin{align}
    \label{eq:AhAdef}
    &A=b(D-w),\quad \hA=\hb(D-\hw),\quad \hT:= A\circ\hA+\lam.
  \end{align}
  Relation ~\eqref{eq:AThTA} follows by construction.  Next, consider
  the case of $b(x)=1$ and set
  \[ \tilde{A} = D-w,\quad \check{A}= p(D+w+q/p),\quad \check{T}=
    \tilde{A}\circ\check{A}+\lam.\] Observe that $b\tilde{A} =
  A$. The second claim now follows because
  \[  \check{A}\circ \cM(b^{-1}) =  (p/b)(D+w+q/p-b'/b) = \hA\]
  \[ \cM(b)\circ \check{T}\circ \cM(b^{-1}) = b \tilde{A}
    \circ\check{A}\circ\cM(b^{-1})+\lam = A\circ\hA+\lam = \hT.\]
\end{proof}

\begin{definition}
  We say that two operators $T,\hT\in \DpQ$ are Darboux connected if
  there exists an $L\in \Diff(\cQ)$ such that $\hT L = L T$.
\end{definition}
\noindent
Observe that if, in the above relation, $L=\cM(\mu),\; \mu \in \cQ$ is
an operator of order zero, then $T$ and $\hT$ are related by rational
gauge transformation.  In particular, if this is the case, then 
quasi-rational eigenfunctions of $\hT$ are obtained by multiplying the
quasi-rational eigenfunctions of $T$ by $\mu$.
\begin{definition}
  We define an $n$-step RDT chain
  $T_0 \rdt{\lam_1} T_1 \rdt{\lam_2} T_2 \cdots T_{n-1} \rdt{\lam_{n}}
  T_n$ to be a sequence of operators $T_0,T_1,\ldots, T_n \in\DpQ$
  such that each arrow is an RDT; i.e., there exist
  $A_1,A_2,\ldots, A_{n}\in \Diff_1(\cQ)$ such that
  \[ A_{k} T_{k-1} = T_k A_{k},\quad k=1,\ldots, n.\]
  If this is the case, we will refer to
  $\lam_1,\ldots, \lam_n\in \sigq(T_0)$ as
  the \emph{eigenvalue sequence} of the chain.
\end{definition}
\begin{thm}[Theorem 3.10 of \cite{GFGM19}]
  Suppose that $T,\hT\in \DpQ$ are Darboux connected by an $n$th order
  intertwiner $L\in \Diff_n(\cQ)$ with $n\ge 1$. Then there
  exists an $n$-step RDT chain with $T_0 = T, \; T_n = \hT$ and
  $L = A_n \cdots A_1$.
\end{thm}

\noindent
The following construction characterizes
Darboux connectedness in the case where the eigenvalue sequence
consists of distinct eigenvalues.
\begin{prop}[Theorem  2.1 of \cite{gegen}]
  \label{prop:rdtchain}
  Let $T_0\in \DpQ$ be as per ~\eqref{eq:Tpqr}. Let
  $\phi_1, \ldots, \phi_n$ be quasi-rational eigenfunctions of $T_0$
  with \textbf{distinct} eigenvalues $\lambda_1,\lam_2 \ldots, \lambda_{n}$;
  and let $b_1, \ldots, b_n$ be non-zero rational functions. Define
  \begin{equation}
    \label{eq:Tkdef}
    T_k := p D^2 + q_k D + r_k,\qquad        A_k := b_k(D-w_k).
  \end{equation}
  where\footnote{Since we assume that the eigenvalues
    $\lambda_1, \ldots, \lambda_n$ are distinct,
    $\phi_1, \ldots, \phi_n$ are linearly independent. Hence, the
    Wronskians in the denominator of ~\eqref{eq:omk} are non-zero,
    and the functions $\upsilon_k$ are well-defined.
  }
  \begin{equation}
    \label{eq:omk}
    \sigma_k:=\sum_{j=1}^k (\log b_j)',\quad \upsilon_k := 
    \frac{\phi_{(1, 2, \ldots, k)}'}{\phi_{(1, 2, \ldots, k)}},\qquad
    k=1,\ldots, n; 
  \end{equation}
  and
  \begin{align}
    \begin{split}\label{eq:qkrk}
      &q_k := q_0+ k p' - 2 p \sigma_k, \\ 
      &r_k := r_0+k q_0'  + \frac12 k(k-1) p''+ \upsilon_k p'  -  \sigma_k(q_0+kp')  + (\sigma_k^2-\sigma_k' + 2\upsilon_{k}') p, 
    \end{split}\\
    \begin{split}\label{eq:wkgen}
      &w_1  :=\upsilon_1,\quad
      w_k :=
      \sigma_{k-1} +\upsilon_k - \upsilon_{k-1},\quad k=2,\dots,n.
    \end{split}
  \end{align}
  Then
  $T_0 \rdt{\lam_1} T_1 \rdt{\lam_2} T_2 \cdots T_{n-1} \rdt{\lam_{n}}
  T_n$ is an RDT chain.  Moreover, up to a rational gauge
  transformation, $T_n$ is invariant with respect to permutations of
  the list $\phi_1,\ldots, \phi_n$.
\end{prop}
\noindent
It is also worth noting that Proposition ~\ref{prop:rdtchain} is an
extension of the well-known Crum formula for Schr\"odinger operator.
The latter corresponds to the special case of $p=-1$ and $q_0=0$ and
gauges $b_1=\ldots = b_n = 1$.

As an immediate corollary, we obtain the following.
\begin{prop}[Permutability]
  \label{prop:permut}
  Let $\lam\ne \tlam$ be qr-eigenvalues of $T_0\in \DpQ$. Then, the
  corresponding RDTs $T_0\rdt{\lam} T_1$ and $T_0\rdt{\tlam} \tT_1$
  can be extended to a commutative square
  \[
    \begin{tikzcd}
      T_0 \arrow[d,"\tlam",rightharpoonup]  \arrow[d,leftharpoondown]
      \arrow[r, "\lambda", rightharpoonup]      \arrow[r,leftharpoondown]
      & T_1      \arrow[d,"\tlam",rightharpoonup] \arrow[d,leftharpoondown]\\ 
      \tT_1 \arrow[r,"\lam",rightharpoonup] \arrow[r,,leftharpoondown]
      & T_2           
    \end{tikzcd}
  \]
  If $\lam,\tlam$ are simple qr-eigenfunctions of $T_0$, then
  $T_2\in \DpQ$ is unique, up to a rational gauge transformation.
\end{prop}
\begin{prop}
  \label{prop:Awronsk}
  Let $T_0$ and $T_j,A_j,\phi_jb_j,\;j=1,\ldots,n$ be as in Proposition
  ~\ref{prop:rdtchain}, and $y(x)$ a sufficiently smooth function.
  Then,
  \[ (A_n \cdots A_1) y = (b_1\cdots b_n) \frac{\Wr[\phi_1,\ldots, \phi_n,
      y]}{\Wr[\phi_1,\ldots, \phi_n]} .\]
\end{prop}
\begin{proof}
  The proof rests on three remarks.  First, the kernel of an $n$th
  order differential operator is $n$-dimensional.  Thus, such an
  operator is characterized by its leading coefficient and by its
  kernel.  Second, by inspection, the leading order coefficient of
  $A_n \cdots A_1$ is $b_1\cdots b_n$.  Third, by well know properties
  of the Wronskian determinant, the differential operator
  \[ y \mapsto \frac{\Wr[\phi_1,\ldots, \phi_n, y]}{\Wr[\phi_1,\ldots,
      \phi_n]} \] is a monic $n$th order differential operator that
  annihilates $\phi_1,\ldots, \phi_n$. 
\end{proof}

\subsection{Formal orthogonality}
\begin{prop}
  The derivative of a quasi-rational $f(x)$ is itself quasi-rational.
\end{prop}
\begin{proof}
  Let $w=f'/f$ be the rational log-derivative of the given function.
  By direct calculation,
  \[ \frac{f''}{f'} = \frac{w'}{w}+w \]
  is also rational.
\end{proof}

\begin{lem}
  \label{lem:Tweight}
  Let $T=pD^2+qD+r$ be an operator in $\DpQ$.  The symmetric,
  Sturm-Liouville form of the eigenvalue equation $Ty= \lam y$ is
  $(Py)' + R y = \lam W y$,
  where
  \begin{equation}
    \label{eq:WTdef}
    W=W_T := p^{-1}\exp\lp \int \frac{q}{p}\rp,\quad P= p W,\quad R= r W
  \end{equation}
  are the indicated quasi-rational functions. Moreover, $T$ is
  formally self-adjoint with respect to $W$ in the sense that for
  every $f,g\in \cQ$, we have
  \begin{equation}
    \label{eq:intTfg}
    \int (fTg    -g Tf) W = \Wr[f,g] p W,
  \end{equation}
\end{lem}
\begin{proof}
  By definition, $(pW)' = q W$.  Hence, multiplying $Ty=\lambda y$ by
  W yields the symmetric form of the eigenvalue equation.  To prove
  the second claim, observe that
  \begin{equation}
    \label{eq:pWfg}
    \lp p W \Wr[f,g] \rp '
    = Wf ( p g'' + q g') - W g (pf''+q f') =  (f T g - g T f) W
  \end{equation}
\end{proof}

\begin{definition}
  Given rational $f(x),g(x)$ and a quasi-rational weight $W(x)$, we
  will refer to the formal anti-derivative $\int f g W$ as the
  \emph{incomplete inner product} of $f$ and $g$ with respect to $W$.
  Similarly, we will refer to $ \int f^2 W $ as the \emph{indefinite
    norm} of $f$ with respect to $W$.
\end{definition}

\begin{lem}
  \label{lem:forthog}
  Let $T\in \DpQ$ and let $W=W_T$ be as per ~\eqref{eq:WTdef}.  Let
 $\pi_1,\pi_2$ be rational eigenfunctions of $T$ with eigenvalues
 $\lam_1, \lam_2$.  If $\lam_1=\lam_2$, then
 $\Wr[\pi_1,\pi_2] p W$ is a constant.  If $\lam_1\ne \lam_2$, then
 the incomplete inner product of 
 $\pi_1,\pi_2$  with respect to $W_T$ is quasi-rational.  More formally:
  \begin{equation}
    \label{eq:phi12W}
    (\lam_2-\lam_1)\int \pi_1 \pi_2 W =      \Wr[\pi_1,\pi_2] p W .
  \end{equation}
\end{lem}
\begin{proof}
  Apply ~\eqref{eq:intTfg} to $f=\pi_1$ and $g=\pi_2$.
\end{proof}

\begin{lem}
  \label{lem:facform}
  Let $T\rdt{\lam} \hT$ be an RDT with $T=pD^2+qD+r$ and
  $\hT = p D^2 + \hq D + \hr$ the corresponding coefficients.  Let
  $\phi,\hphi$ be the dual factorization eigenfunctions and
  $A=b(D-w),\;\hA= \hb(D-\hw)$ the dual intertwiners as per
  ~\eqref{eq:bhbqhq}. Then, the following relations hold:
  \begin{align}
    \label{eq:phiW}
    &b \hb = p,\quad \phi \hphi = \frac{b}{p W} = \frac{\hb}{p \hW},    
  \end{align}
\end{lem}
\begin{proof}
  This follows from ~\eqref{eq:bhbqhq} of Proposition ~\ref{prop:dualRDT}.
\end{proof}
\begin{lem}
  \label{lem:RDTnorm}
  Let $T,\hT, A,\hA,W, \hW$ be as in Lemma ~\ref{lem:facform}.  Let
  $\pi$ be a rational eigenfunction of $T$ with eigenvalue
  $\tlam \ne \lam$.  Then, $\hpi = A\pi$ is a rational eigenfunction
  of $\hT$ with eigenvalue $\tlam$.
  Moreover, 
  the  indefinite norms of $\pi$ and $\hpi$ are
  related as follows: 
  \begin{equation}
    \label{eq:pihpinorms}
     \int \hpi^2 \hW - (\lam-\tlam) \int \pi^2 W =  \pi\hpi\hb W .
  \end{equation}
\end{lem}
\begin{proof}
  The first assertion follows by applying both sides of the
  intertwining relation ~\eqref{eq:AThTA} to $\pi$.  To prove the
  second assertion, observe that if $f,g$ are rational functions, then
  by ~\eqref{eq:WTdef}, ~\eqref{eq:phiW} and ~\eqref{eq:bhbqhq},
  \begin{align*}
    & (Af) g \hW + f (\hA g) W = b(f' g -w f g) \hW +\hb( g' f-\hw f
      g) W\\
    &\quad = (f'g+g'f- (w+\hw)fg)\hb W
      = \lp fg\hb W\rp'
  \end{align*}
  Applying the above relation to $f = \pi$ and $g=\hpi$ and using the
  fact that $\hA A = T-\lam$ yields the desired conclusion.
\end{proof}

\begin{lem}
  \label{lem:Wmult}
  Let $f_1(x),\ldots, f_j(x),g_1(x),\ldots, g_k(x)$ be quasi-rational
  functions  and $x_0,a\in \R$  constants.   Define $\ord_{x_0} f $ to be the
  residue of $f'(x)/f(x)$ at $x=x_0$, and assume that
  \[ \ord_{x_0} f_i  = a,\; i=1,\ldots, j,\quad \ord_{x_0} g_i = 0,\;
    i=1,\ldots, k.\]
  Then,
  \[ \ord_{x_0} \Wr[f_1,\ldots, f_j,g_1,\ldots, g_k] = j (a-k).\]
\end{lem}
\begin{proof}
  The statement is evident if either $j=0$ or $k=0$.  Consider the
  case of $k=1$. Observe that
  \[ h:=\Wr[f_1,\ldots, f_j, g] = (-1)^j g^{j+1} \Wr[ (f_1/g)',\ldots,
    (f_j/g)' ] .\] Since $\ord_{x_0} (f_i/g)' = a-1$ we must have
  $\ord_{x_0} h = j(a-1)$. To prove the general
  case, we use the identity
  \[ \Wr[f_1,\ldots, f_j,g_1,\ldots, g_k] = \lp \Wr[f_1,\ldots,
    f_j]\rp^{k-1} \Wr[h_1,\ldots, h_k ],\] where
  \[h_i = \Wr[f_1,\ldots, f_j,g_i],\; i=1,\ldots, k.\] Since
  $\ord_{x_0} h_i = j(a-1)$  and $\ord_{x_0} \Wr[f_1,\ldots, f_j] = a
  j$, we have
  \[ \ord_{x_0} \Wr[f_1,\ldots, f_j,g_1,\ldots, g_k] = jk(a-1) -
    aj(k-1) = aj-jk = j(a-k).\]
\end{proof}

\begin{definition}
  Let $f(x)$ be a quasi-rational function of exponential order $0$ and
  degree $d=\deg f$.  We define the leading coefficient of $f$ to be
    \begin{equation}
      \label{eq:LCdef}
      \LC(f) = \lim_{x\to \infty} x^{-d}f(x).
  \end{equation}
\end{definition}
\begin{lem}
  \label{lem:lc}
  Let $f_1(x),\ldots, f_p(x)$ be qr-functions of exponential order
  $0$. Then,
  \begin{align}
    \label{eq:degW}
    \deg \Wr[f_1,\ldots, f_p]
    &= \sum_{i} \deg f_i - \frac{p(p-1)}{2},\\
    \label{eq:LCWronsk}
    \LC \lp \Wr[f_1,\ldots, f_p]\rp
    &=
    \prod_{1\le i<j\le p} (\deg
    f_j -  \deg f_i) \prod_{i=1}^p \LC(f_i) .
  \end{align}
\end{lem}
\begin{proof}
  The definition ~\eqref{eq:LCdef} is equivalent to the condition that,
  asymptotically, $f(x) \sim \LC(f) x^{d}$ as $x\to \infty$ where
  $d=\deg f$.  Hence, it suffices to prove ~\eqref{eq:LCWronsk} for the
  case where $f_i(x) = x^{d_i},\;i=1,\ldots, p$ and $d_1,\ldots, d_p$
  constants.  We leave this as an exercise.
\end{proof}

\subsection{Confluent Darboux Transformations}
\label{sec:CDT}
\begin{definition}
  We say that 2-step $T_0\rdt{\lam_1} T_1 \rdt{\lam_2} T_2$ is a CDT
  (confluent Darboux transformation) if (i) $\lam_1=\lam_2$, and (ii)
  the RDT is non-trivial in the sense that $T_0$ and $T_2$ are not
  gauge-equivalent.
\end{definition}
\noindent
Recall that every RDT is invertible, with the inverse RDT having the
same factorization eigenvalue.  It therefore makes sense to regard a
2-step factorization chain consisting of an RDT and its inverse as
being trivial. This observation motivates the use of the term
``non-trivial'' in the above definition.
\begin{prop}
  \label{prop:CDTdegen}
  A given RDT $T_0\rdt{\lam} T_1$ can be extended to a CDT if and only
  if the factorization eigenvalue $\lambda$ is a degenerate qr-eigenvalue
  of $T_1$.
\end{prop}
\begin{proof}
  This follows directly by Proposition ~\ref{prop:RDTqrefun} and by the fact
  that $T_2$ is distinct from $T_0$.
\end{proof}

\begin{prop}
  \label{prop:CDTinorm}
  Let $T\rdt{\lam} \hT$ be an RDT. Let $\phi,\hphi$ be the dual
  factorization eigenfunctions and let $W=W_T$ be the quasi-rational weight as
  defined in Lemma ~\ref{lem:Tweight}.  Then $\lam$ is a degenerate
  eigenvalue of $\hT$ if and only if the indefinite norm
  $\rho := \int \phi^2 \, W$ defines a quasi-rational function.  If
  this holds, then $\rho \hphi$ is another, linearly independent qr
  eigenfunction of $\hT$ with eigenvalue $\lam$.
\end{prop}
\begin{proof}
  Suppose that $\rho$ is quasi-rational. Let $\hW=W_{\hT}$ be the dual
  weight function. Hence, by ~\eqref{eq:bhbqhq} and ~\eqref{eq:phiW},
  \[ p \hW \Wr[\hphi,\rho\hphi] = p\hW \hphi^2 \rho'
    = p  \phi^2 \hphi^2W\hW =1.
  \]
  Hence, $(p\hW\Wr[\hphi,\rho\hphi])'=0$.  Hence, by
  ~\eqref{eq:pWfg},
  \[ (\rho\hphi  \hT\hphi -  \hphi\hT(\rho\hphi))\hW   = 0.\]
  Since $\hT\hphi=\lam\hphi$,  we must have
  $\hT (\rho\hphi) = \lambda \rho\hphi$. The converse can be argued
  using the same manipulations, but only in reverse.
\end{proof}

\section{Some examples}\label{sec:examples}
\subsection{Isospectral deformation of exceptional class D operators}

It is instructive to include a simple example of such
  a construction to show the commutativity of the two types of Darboux transformations with explicit expressions for operators and eigenfunctions.

The starting point for this construction is the classical Legendre operator $\ckT$ given by 
\begin{equation}
 \check{\tau}=1,\quad \ckT = (x^2-1)D_{xx} + 2 x D_x,\quad \ckpi_n
      = \pi(x;0,0),
\end{equation}
  whose eigenfunctions are the classical Legendre polynomials $\ckpi_n
      = \pi(x;0,0)$ given by ~\eqref{eq:pindef}.
We would like to construct a 1-parameter family of class D exceptional operators $T$, indexed by the spectral diagram
      \[\Lam=\LamD(K,L_1,L_3,L_4,a,b)=\LamD(\{1 \},\{0 \},\emptyset,\emptyset,0,0),\]
depicted in Figure~~\ref{fig:etypeD2} below.

\begin{figure}[H]

  \begin{tikzpicture}[scale=0.6]
    \def\y{4}
    \foreach \x in {-1,...,4}
    \draw  (\x+0.5,\y+2.5) node {$\x$};
    \draw  (-4,\y+2.5)    node[anchor=west] {$i=k-1$};
    \draw  (6.5,\y+2.5)    node[anchor=west]
    {$\al=\be=1,\; K=\{ 1\},\; L_1=\{0\},\; L_3=L_4=\emptyset$};

    \draw    (-1,\y+1) grid +(6 ,1);
    \path [draw,color=black] (-1.5,\y+1.5) 
    ++(1,0) node {$\triangledown$}
    ++(1,0) node {$\sqbullet$}
    ++(1,0) circle (5pt)
    ++(1,0) circle (5pt)
    ++(1,0) circle (5pt)
    ++(1,0) circle (5pt)
    ++(1,0) node {\huge ...}
    ++(1,0) node[anchor=west]
    {$I_{1+}=\{1,2,3,4,\ldots\},I_{2+}=\{2\},I_{3} = I_4=\{ 1\}$};


    \draw  (-3.5,\y+0.5)    node[anchor=west] {$i^\star$};
    \foreach \x in {-7,...,-2}
    \draw  (-\x-2.5,\y+0.5) node {$\x$};
    \draw  (6.5,\y+0.5)    node[anchor=west]
    {$I_{1-} = \{-2\},\; I_{2-} = \{ -1 \}$};
    
  \end{tikzpicture}
 \caption{Spectral diagram for a type $D$ exceptional operator $T$}\label{fig:etypeD2}
\end{figure}
\noindent
In order to construct $T$ we have two possible routes which produce
the same operator in the end via different intermediate operators:
\begin{enumerate}
\item As explained above in this section, we can perform an
  isospectral CDT on the classical Legendre operator $\ckT$
  corresponding to $L_1=\{0\}$ to arrive at $\hT$, and then apply a
  state-deleting transformation corresponding to $K=\{1\}$.
\item Alternatively, we could perform first a state-deleting
  transformation corresponding to $K=\{1\}$ on the classical Legendre
  operator to arrive at $\tT$, and then an isospectral CDT introducing
  a real parameter corresponding to $L_1=\{0\}$.
\end{enumerate}

The two constructions can be visualized as the following commutative
diagram.
\begin{figure}[H]
    \begin{tikzcd}
      \ckT \arrow[d,"\lam_1",rightharpoonup]  \arrow[d,leftharpoondown]
      \arrow[r,"\lambda_0\lambda_0", rightharpoonup]   \arrow[r,leftharpoondown] 
      & \hT      \arrow[d,"\lam_1",rightharpoonup]
      \arrow[d,leftharpoondown]\\   
      \tT+2 \arrow[r,"\lambda_0\lambda_0",rightharpoonup] \arrow[r,,leftharpoondown]
      & T+2           
    \end{tikzcd}
\end{figure}
\noindent
The arrows with the $\lam_1$ label represent a rational Darboux
transformation (RDT), while the arrows with label $\lam_0\lam_0$
indicates a confluent Darboux transformation (CDT) performed at the
indicated eigenvalue.  RDTs are formally defined and discussed in
Section ~\ref{sec:RDT} below. A CDT is a non-trivial 2-step RDT
performed consecutively a the same eigenvalue.  This is discussed in
detail in Section ~\ref{sec:CDT} below.  We also visualize this
commutative diagram in terms of the corresponding spectral diagrams.
\begin{figure}[H]
  \begin{tikzpicture}[scale=0.4]
    \def\y{4}
    \draw    (-1,\y+1) grid +(6 ,1);
    \draw    (9,\y+1) grid +(6 ,1);
    \path [draw,color=black] (-1.5,\y+1.5) 
    ++(1,0)  circle (5pt)
    ++(1,0)  circle (5pt)
    ++(1,0) circle (5pt)
    ++(1,0) circle (5pt)
    ++(1,0) circle (5pt)
    ++(1,0) circle (5pt)
    ++(1,0) node {$\ldots$}
    ++(2,0) node {$\leftharpoondown\rightharpoonup$}
    ++(2,0) node {$\triangledown$}
    ++(1,0) circle (5pt)
    ++(1,0) circle (5pt)
    ++(1,0) circle (5pt)
    ++(1,0) circle (5pt)
    ++(1,0) circle (5pt)
    ++(1,0) node {$\ldots$};
    \def\y{1}
    \draw    (2,\y+3.2) node {$\upharpoonleft$};
    \draw    (2,\y+2.8) node {$\downharpoonright$};
    \draw    (12,\y+3.2) node {$\upharpoonleft$};
    \draw    (12,\y+2.8) node {$\downharpoonright$};
    \draw    (-1,\y+1) grid +(6 ,1);
    \draw    (9,\y+1) grid +(6 ,1);
    \path [draw,color=black] (-1.5,\y+1.5) 
    ++(1,0)  circle (5pt)
    ++(1,0) node {$\sqbullet$}
    ++(1,0) circle (5pt)
    ++(1,0) circle (5pt)
    ++(1,0) circle (5pt)
    ++(1,0) circle (5pt)
    ++(1,0) node {$\ldots$}
    ++(2,0) node {$\leftharpoondown\rightharpoonup$}
    ++(2,0) node {$\triangledown$}
    ++(1,0) node {$\sqbullet$}
    ++(1,0) circle (5pt)
    ++(1,0) circle (5pt)
    ++(1,0) circle (5pt)
    ++(1,0) circle (5pt)
    ++(1,0) node {$\ldots$};
 
  \end{tikzpicture}
\end{figure}
\noindent
In the first route $\ckT\to\hT\to T$, the intermediate operator $\hT$
and eigenfunctions are given by formulas
~\eqref{eq:rhoij}--~\eqref{eq:hpiD}. Since the matrix $\cR$ in this case
is just a scalar
\[ \cR(x,\bt_\bell;L,a,b))=\cR(x,(t_0);\{0\},0,0)= t_0+ \check{\rho}_{00}, \qquad \check{\rho}_{ij} = \int_{-1}^x \!\!\ckpi_i \ckpi_j, \]
we have the 1-parameter $\bt_\bell=(t_0)$ family of operators $\hT$ given by
\[ \hT = (x^2-1)D_{xx} + 2x D_x + \lp \frac{2x }{1+x+t_0} + \frac{
    2(1-x^2)}{(1+x+t_0)^2} \rp,\qquad t_0\in (-\infty,-2) \cup
   (0,\infty),\] which follows directly from ~\eqref{eq:Trgdef2}, since
 $\htau= t_0 + \check{\rho}_{00}= 1+x+t_0$.  The quasi-polynomial
 eigenfunctions $\hpi_i$ of $\hT$ are
 \begin{equation*}
   \hpi_i = \ckpi_{n_i} -
   \frac{\check{\rho}_{0n_i}}{\htau},\qquad i\in \hat I_1=\{
   -1,1,2,\ldots \},
 \end{equation*}
 as given by ~\eqref{eq:hpiD}, where
 \[ n_i= \max\{ i, -i-1\}. \] The first few quasi-polynomial
 eigenfunctions of the intermediate operator $\hT$ are:
\begin{align*}
  \hpi_{-1} &= 1-\frac{x+1}{x+t_0+1}=\frac{t_0}{1+t_0+x}\\
  \hpi_{1} &= x- \frac12 \frac{x^2-1}{x+t_0+1}\\
  \hpi_2 &= x^2-\frac13 +\frac13\frac{-x^3+x}{x+t_0+1}.
\end{align*}
Note that as $t_0\to \infty$, we recover the classical  Chebyshev
polynomials in their monic normalization.  Thus, these exceptional
polynomials can be justly regarded as deformations of classical
polynomials.  Their normalization isn't monic, but rather
asymptotically monic.

The final operator $T$ is obtained by performing a state-deleting transformation in $\hT$ with $K=\{1\}$, so expressions ~\eqref{eq:tauD} and ~\eqref{eq:Trgdef2} lead to
\begin{align}
  \tau(x;t_0) &= \hpi_1\htau= (x+1)^2+2xt_0,\qquad
                t_0\in(-2,0), \label{eq:tauex}\\ 
  T &=  (x^2-1)D_{xx} + 4x D_x - (x^2-1)\lp \frac{8
      t_0(2+t_0)}{\tau(x;t_0)^2} + \frac{2}{\tau(x;t_0)^2}\rp +2
\end{align}
The first few quasi-polynomial eigenfunctions of $T$, following ~\eqref{eq:piD} are:
\begin{equation}
  \label{eq:piex}
\begin{aligned}
 \pi_{-2} &=  \frac1x -\frac{(x+1)^2}{\tau(x;t_0)},\\
  \pi_{1} &= x+\frac{1}{3x} -\frac13 \frac{(x^2-1)^2}{\tau(x;t_0)},\\
  \pi_{2} &= x^2-\frac14 \frac{(x^2-1)^2}{\tau(x;t_0)},\\
  \pi_3 &= x^3-\frac17 x-\frac1{35x} -\frac1{35}
  \frac{(x^2-1)^2(7x^2-1)}{\tau(x;t_0)}. 
\end{aligned}
\end{equation}
The general formula is
\[ \pi_i = \frac{1}{n_{i+1}-1}
  \frac{\Wr[\hpi_1,\hpi_{i+1}]}{\hpi_1},\quad i\in
  \{-2,1,2,\ldots\},\] where $n_{-1} = 0$ and $n_i=i$ otherwise.

We observe that as $t_0\to \infty$, we recover monic polynomials,
suggesting that the above family may also be regarded as an
isospectral deformation, but this time a deformation of an exceptional
family.  This observation may be realized by following the second
route described above.  It involves first applying a standard
state-deleting transformation with $K=\{1\}$ on the classical Legendre
operator $\ckT\to \tT$, followed by a CDT $\tT\to T$.  The
intermediate operator $\tT$ can be simply written in terms of the
Wronskian formulas ~\eqref{eq:tauD}--~\eqref{eq:piD} where all ``hat''
variables $(\htau,\hpi_i)$ are substituted by ``check'' variables
$(\check \tau,\ckpi_i)$. More specifically, we have
\begin{align*}
\ttau &= \ckpi_1\check \tau= x\\
 \tT   &= (x^2-1)D_{xx} + 4 x D_x +      \frac{2}{x^2},\\
\tpi_{i} &=\tchi(i+1) \frac{\Wr[\ckpi_1,\ckpi_{i+1}]}{\ckpi_1},\qquad
      i\in \{ -1,1,2,\ldots \},\qquad \tchi(n) = \frac{1}{n-1}.
\end{align*}
The first few quasi-rational eigenfunctions of $\tT$ are:
\[ \tpi_{-1} = \frac1x,\qquad \tpi_1 = x +\frac1{3x},\qquad \tpi_2 =
  x^2,\qquad \tpi_3 = x^3-\frac27 x- \frac1{35x};\]
we recognize these as the $t_0$ independent terms in ~\eqref{eq:piex}.
We now need to apply a CDT $\tT\to T$, following the procedure described above albeit with ``check'' variables $(\check \tau,\ckpi_i)$ replaced by ``tilde'' variables  $(\ttau,\tpi_i)$.
Following ~\eqref{eq:rhoij}, define (note the shift $\al=a+1,\be=b+1$ after the first DT)
\[\trho_{ij} = \int_{-1}^x \tpi_i \tpi_j (1-x)(1+x). \]
The expressions for $\tau$ in ~\eqref{eq:tauex} can also be written as
\[ \tau = (2t_0- \trho_{-1,-1})\ttau= (x+1)^2 + 2 x t_0, \qquad t_0\in(-2,0),\]
and likewise the quasi-polynomial eigenfunctions $\pi_i$ of $T$ can be expressed as:
\[ \pi_i=\tpi_n - \frac{\trho_{-1,n}}{-2t_0+
            \trho_{-1,-1}},\qquad i=1,2,\ldots \]

\begin{remark}
It is not obvious at all that $\trho_{ij}$ are rational functions of $x$. This requires the vanishing of the residues at all poles of the integrand. While this can be seen to happen in the general case, a detailed proof will be given elsewhere. This is the reason why, despite the commutativity of both types of Darboux transformations, we chose to apply first the set of CDTs to the classical operator, where it is evident that $\rho_{ij}$ in ~\eqref{eq:rhoij} are rational (and therefore, also $\htau$ in ~\eqref{eq:htauD} is rational.
\end{remark}

\begin{remark}
Observe that in the second route the intermediate operator $\tT$ is not regular, but  the CDT $\tT\to T$ produces a regular operator $T$ for a certain range of values of the deformation parameter $t_0$. The whole study of the regularity of exceptional Jacobi operators is postponed to a forthcoming publication.
\end{remark}

\subsection{Norm calculations.}
In this section we collect some calculations to illustrate the ideas
behind generalized orthogonality and norms.  The basic idea is that if
$\al,\be,\al+\be+1\notin \Zm$, then for every quasi-polynomial eigenfunction
$\pi_i,\; i\in I_1$ of an exceptional operator $\Trg(\tau;\al,\be)$
there exists a unique constant $\kappa_i$ such that the quasi-rational
expression
\[ \lp \pi_i(x)^2 - \kappa_i\rp (1-x)^\al (1+x)^\be \] possesses a
quasi-rational anti-derivative.  If we accept this is true, then for
the regular cases where $\al,\be>-1$ and $\tau(x)$ does not vanish in
$[-1,1]$, we have
\[ \int_{-1}^1 \lp \pi_i(x)^2 - \kappa_i\rp (1-x)^\al (1+x)^\be dx =
  0.\]
Consequently, the corresponding $\rL^2$ norm is given by
\begin{equation}
  \label{eq:nuikappai}
  \begin{aligned}
  \nu_i = \kappa_i \nu(\al,\be) &=\int_{-1}^1 \pi_i(x)^2  (1-x)^\al
  (1+x)^\be  dx,\; i\in  
  I_1,\quad\text{ where }\\
  \nu(\al,\be) &= \int_{-1}^1 (1-x)^\al (1+x)^\be dx,
  \end{aligned}
\end{equation}
As asserted in Proposition ~\ref{prop:kappaT}, such constants exists
for every quasi-polynomial eigenfunction $\pi_i,\; i\in I_1$ of an
exceptional $T$.  The interesting and useful fact is that this
approach allows us to define norms even in the non-regular case.

By way of example, let's consider the norms of the classical Chebyshev
polynomials relative to a monic normalization. Let
$\pi_i(x) := \pi_i(x;1/2,1/2),\; i\in \Nz$ be the indicated classical
Jacobi polynomials.  According to the above prescription,
\[ \kappa_i = \frac{\nu(i;1/2,1/2)}{\nu(1/2,1/2)} =4^{-i}. \]
Consequently, as asserted by Proposition
~\ref{prop:sfclass} below, for every $i\in \Nz$, there exists
a polynomial $\Pi_i(x)$ such that
\[ \int \lp \pi_i(x)^2 - 4^{-i}\rp (1-x)^{1/2} (1+x)^{1/2}
  =\Pi_i(x)(1-x)^{3/2} (1+x)^{3/2}.\] 
Therefore,
\[ \int_{-1}^1 \pi_i(x)^2 (1-x)^{1/2} (1+x)^{-3/2} = 4^{-i}
  \nu(1/2,1/2) = 2^{-2i-1}\pi. \] 

Next, let us consider the exceptional operator obtained from the
classical $T(1/2,1/2)$ by a 1-step type 3 transformation with factorization eigenfunction $\pi_1(x;1/2,1/2)$.   Applying  ~\eqref{eq:piC} with $K_1=K_2=K_4=\emptyset$ and $K_3=\{ 1\}$ we obtain
\begin{align*}
  \tau(x) &= \pi_1(x;1/2,1/2) = x-1/2,\quad \al = -1/2,\; \be=3/2,\;
            I_1 = \{0,1,2,\ldots \}\\  
  \Trg(\tau;\al,\be)
          &= (x^2-1)D_x^2+(3x-2) D_x + \frac{ 2-x}{(x-1/2)^2}\\
  \pi_i
  &=
    \frac{x-1}{i-1/2}\frac{\Wr[\phi(x),\ckpi_i(x)]}{\phi(x)},\; 
    i=0,1,2,\ldots,\quad \text{where }\\
  \phi(x) &= (1-x)^{-1/2} (x-1/2),\quad \ckpi_i(x) = \pi_i(x;1/2,1/2)\\
  \pi_0(x)
          &= 1-\frac{1}{x-1/2}\\
  \pi_1(x)
          &= x-1 + \frac{1/2}{x-1/2}\\
  \pi_2(x)
          &= x^2-x + \frac14\\ \ldots
\end{align*}
Since $\tau(x)$ vanishes at $x=1/2$, the norm integrals
\[ \nu_i = \int_{-1}^1 \pi_i(x)^2 (1-x)^{-1/2}(1+x)^{3/2} dx \]
diverge and cannot be used to define the constants $\nu_i$.  The
$\nu_i$ can be meaningfully defined by deforming the contour of
integration to avoid the singularity at $x=1/2$.  This works because
the integrand has vanishing residue at the singularity. 
By ~\eqref{eq:Csf}, the value of these norms is given by
\begin{align*}
  \nu_i &= 4^{-i}\frac{i+5/2}{i-1/2} \nu(1/2,1/2)\\
  &= \kappa_i \nu(-1/2,3/2), \intertext{where}
    \kappa_i &= \frac{4^{-i}}{3} \frac{2i+5}{2i-1}.
\end{align*}
Thus, the assertion is that the following anti-derivatives
quasi-rational:
\[ \int \lp\pi_i(x)^2 - \kappa_i\rp (1-x)^{-1/2}(1+x)^{3/2},\;
  i=0,1,2,\ldots \] This implies that
\[ \int_I \pi_i(x)^2 (1-x)^{-1/2}(1+x)^{3/2} = \kappa_i \nu(-1/2,3/2)
  = \frac{2i+5}{2i-1} 2^{-2i-1} \pi,\; i=0,1,2,\ldots\] where $I$ is
any contour from $x=-1$ to $x=1$ that avoids the singularity at
$x=1/2$.  Moreover, observe that asserted by Theorem ~\ref{thm:regpos},
the lack of regularity in this example is signalled by the fact that
$\kappa_0$ is negative.

\section{Classical Jacobi operators}
\label{sec:classop}

In this section, we describe the index sets and the spectral diagrams
of the classical Jacobi operators.  The quasi-rational eigenfunctions
of a classical $T(a,b),\; a,b\in \R$ are shown in Table
~\ref{tab:stype}.  However, if $a,b\in \Zm$ or $a+b+1\in \Zm$ is a
negative integer, then the monic version $\pi_n(x;a,b)$ of the
classical polynomials $P_n(x;a,b)$ may no longer be well-defined.
This happens in one of two ways.  If $a+b+1\in \Zm$, then the leading
coefficient of $P_n(x;a,b)$ will vanish for a finite number of
$n\in \Nz$ such that $(n+a+b+1)_n= 0$ and consequently $\pi_n(x;a,b)$
will be undefined.  The other possibility occurs if $a,b\in \Zm$.
Then, for $n\ge a$ or $n\ge b$, respectively, the polynomial
$\pi_n(x;a,b)$ gains a zero of order $|a|, |b|$ respectively and can
no longer be classed as a type 1 eigenfunction, but rather must be
classed as a type 3 or type 4 eigenfunction with a prefactor of
$(1-x)^{-a}$ or $(1+x)^{-b}$, respectively.

\subsection{Classical index sets and spectral diagrams.}
\label{sec:classSD}
The following Propositions describe the index sets of classical
operators, and the corresponding spectral diagrams.  {The proofs
  are all elementary and are left to the reader.}  After this
case-by-case analysis, we confirm that the qr-spectrum of a classical
operator has the expected form ~\eqref{eq:sigqTalbe}.  This is done in
Proposition ~\ref{prop:sigqclass}, which is a restriction of
Proposition ~\ref{prop:sigq} to the case of classical operators.

Throughout, we let $\ckI_1,\ckI_2,\ckI_3,\ckI_4$ denote the index sets
of a classical operator $T(a,b)$,  let
\[\cksig_\imath := \lam_\imath(\ckI_\imath;a,b),\; \imath= 1,2,3,4\]
denote the corresponding spectral sets, and let $\Lam(a,b)$ denote the
spectral diagram of $T(a,b)$. In those cases where the index sets are
decomposed into lower and upper degrees as
$\ckI_\imath = \ckI_{\imath-} \sqcup \ckI_{\imath+}$, we write
\[ \cksig_{\imath-} = \lam_\imath(\ckI_{\imath-};a,b),\quad
   \cksig_{\imath+} = \lam_\imath(\ckI_{\imath+};a,b).\]

\begin{prop}
  \label{prop:classG}
  The index sets of a classical type G operator
  $\ckT=T(a,b),\;a,b\notin \Z$, are given by
  $\ckI_1=\ckI_2=\ckI_3=\ckI_4=\Nz$. The qr-eigenvalues of $\ckT$ are
  all simple.  Moreover,
  $\sigq(\ckT) = \cksig_1 \sqcup \cksig_2\sqcup
  \cksig_3\sqcup\cksig_4$, with
  $\boxcirc,\boxtimes, \boxplus,\boxminus$ the corresponding labels of
  $\Lam(a,b)$.
\end{prop}
\begin{proof}
  By inspection of ~\eqref{eq:lam1234}, since $a\pm b\notin \Z$, the
  mappings $i\mapsto \lam_1(i;a,b)$ and $i\mapsto \lam_3(i;a,b)$ are
  one-to-one when restricted to $i\in \Z$.  Since $a,b\notin \Z$, the
  sets $\lam_1(\Z;a,b)$ and $\lam_3(\Z;a,b)$ are disjoint.
\end{proof}
\begin{prop}
  \label{prop:classA}
  The index sets of a classical type A operator
  $\ckT=T(a,b),\;a\in \Nz,b\notin \Z$, are given by
  $\ckI_1=\ckI_4=\Nz$, and $\ckI_2,\ckI_3$ as in ~\eqref{eq:ckIA}.
  The type 1 and 4 qr-eigenvalues are simple. The type 2 and 3
  eigenvalues are degenerate with $\cksig_{23}:=\cksig_2=\cksig_3$.
  Moreover,
  $\sigq(\ckT) = \cksig_1 \sqcup \cksig_{23}\sqcup \cksig_4$, with
  $\Lam(a,b): (\cksig_1, \cksig_{23}, \cksig_4) \mapsto
  (\boxcirc, \boxstar,\boxminus)$ the corresponding spectral diagram.
\end{prop}

Let $\ckI_3,\ckI_4,\ckI_{3\pm}, \ckI_{4\pm}$ be as in
~\eqref{eq:ckI34}. Suppose that $a,b, a+b\notin \Z,\; a-b\in \Z$.  If
$a>b$, then a finite range of degrees will be missing from $\ckI_{3}$
and the resulting index set decomposes into a finite range of lower
degrees $\ckI_{3-}$ and an infinite range of higher degrees
$\ckI_{4+}$.  Similar remarks apply if $b>a$ for $\ckI_4,\ckI_{4\pm}$.
  If $a=b$, then $\ckI_3=\ckI_4=\ckI_{3+}=\ckI_{4+} = \Nz$.  For
  example, if $a-b = 4$, then $\ckI_{3-} = \{ 0,1\}$, and
  $\ckI_{3+} = \{ 4,5\ldots \}$; the degrees $\{2,3\}$ are missing
  from $\ckI_3$.

\begin{prop}
  \label{prop:classB}
  The index sets of a classical type B operator
  $\ckT=T(a,b),\; a,b,a+b\notin \Z,\; a-b\in \Z$ are given by
  $\ckI_1=\Nz, \ckI_2=\Nz$ and $\ckI_3, \ckI_4$ as in ~\eqref{eq:ckI34}.
  The type 1 and 2 eigenvalues are all simple.  The
  $\cksig_{3-}, \cksig_{4-}$ eigenvalues are simple.  The other type 3
  and 4 eigenvalues are degenerate with $\cksig_{34}:=\cksig_{3+} = \cksig_{4+}$.
  Moreover, if $a\ge b$, then
  \[ \sigq(\ckT) = \cksig_1 \sqcup \cksig_2\sqcup \cksig_{3-} \sqcup
    \cksig_{34}, \] with
  $\Lam(a,b): (\cksig_1, \cksig_2, \cksig_{3-} , \cksig_{34}) \mapsto
  (\boxcirc, \boxtimes, \boxplus,\boxpm)$ the corresponding spectral
  diagram.  If $b\ge a$, then
  \[ \sigq(\ckT) = \cksig_1 \sqcup \cksig_2 \sqcup \cksig_{4-} \sqcup
    \cksig_{34}, \] with
  $ (\cksig_1, \cksig_2, \cksig_{4-} , \cksig_{34}) \mapsto (\boxcirc,
  \boxtimes, \boxminus,\boxpm)$ the corresponding spectral diagram.
\end{prop}

\begin{prop}
  \label{prop:classC}
  The index sets of a classical type C operator
  $\ckT=T(a,b),\; a,b,a-b\notin \Z,\; a+b\in \Z$ are given by
  $\ckI_1,\ckI_2$ as given in ~\eqref{eq:ckI12} and by
  $\ckI_3=\Nz, \ckI_4=\Nz$.  The type 3 and 4 eigenvalues are all
  simple.  The $\cksig_{1-}, \cksig_{2-}$ eigenvalues are simple.  The
  other type 1 and 2 eigenvalues are degenerate with
  $\cksig_{12}:=\cksig_{1+} = \cksig_{2+}$.  If $a+b\ge 0$, then
  \[ \sigq(\ckT) 
    = \cksig_{2-} \sqcup \cksig_{12}  
    \sqcup \cksig_{3} \sqcup \cksig_{4}, \]
  with $\boxtimes, \boxotimes, \boxplus,\boxminus$ the corresponding
  labels.
  If $a+b+1\in \Zm$, then
  \[ \sigq(\ckT) = \cksig_{1-} \sqcup \cksig_{12} \sqcup \cksig_{3}
    \sqcup \cksig_{4}, \] with
  $\boxcirc, \boxotimes, \boxplus,\boxminus$ the corresponding labels.
\end{prop}

\begin{prop}
  \label{prop:classCB}
  The index sets of a classical  CB operator
  $\ckT=T(a,b),\; a,b\notin \Z,\; 2a,2b\in \Z$ are given by
  $\ckI_\imath = \ckI_{\imath-} \sqcup \ckI_{\imath+},\; \imath=
  1,2,3,4,$ with the latter as defined in ~\eqref{eq:ckI34} and
  ~\eqref{eq:ckI12}.  The
  $\cksig_{1-}, \cksig_{2-},\cksig_{3-}, \cksig_{4-}$ eigenvalues are
  simple with $\cksig_{1-} = \emptyset$ if $a+b\le 0$,
  $\cksig_{2-} = \emptyset$ if $a+b\ge 0$, $\cksig_{3-} = \emptyset$
  if $a\le b$, and $\cksig_{4-} = \emptyset$ if $b\le a$.  All the
  other eigenvalues are degenerate with
  $\cksig_{12}:= \cksig_{1+} = \cksig_{2+}$ and
  $\cksig_{34}:= \cksig_{3+} = \cksig_{4+}$.  The qr-eigenvalues
  decompose as
  \[ \sigq(\ckT) = \cksig_{1-} \sqcup \cksig_{2-} \sqcup
     \cksig_{12}\sqcup
    \cksig_{3-} \sqcup \cksig_{4-} \sqcup
    \cksig_{34},
  \]
  with
  $\boxcirc, \boxtimes, \boxotimes, \boxplus,\boxminus, \boxpm$
  as the corresponding labels of the spectral diagram.
\end{prop}

\begin{prop}
  \label{prop:classD}
  The index sets of a classical type D operator
  $\ckT=T(a,b),\; a, b\in \Nz$ are given by ~\eqref{eq:ckID1} ~\eqref{eq:ckID2}.
The $\cksig_1, \cksig_{3-}, \cksig_{4-}$ eigenvalues are simple with
$\cksig_{3-} = \emptyset$ if $b\ge a$ and $\cksig_{4-} = \emptyset$ if
$a\ge b$.  All the other eigenvalues are degenerate with
$\cksig_{234}:=\cksig_{2-} = \cksig_{2+} = \cksig_{3+} = \cksig_{4+}$.
The qr-spectrum decomposes as
  \[ \sigq(a,b) = \cksig_{1} \sqcup \cksig_{3-} \sqcup
    \cksig_{4-} \sqcup \cksig_{234} ,
  \]
  with $\boxcirc,\boxminus,\boxplus, \boxfsq$ as the corresponding
  labels of the spectral diagram.
\end{prop}

\noindent The proof is a consequence of the following Lemma.  The
upshot is that all of the rational eigenfunctions corresponding to a
$\boxfsq$ label have the same degree, but belong to 4 distinct
asymptotic classes.  The precise description is as follows.
\begin{lem}
  \label{lem:Drat}
  Suppose that $a\ge b\in
  \N$. As per Proposition ~\ref{prop:classD}, let
  \begin{align*}
    \ckI_{2-}
    & = \{0,\ldots, b-1 \}, 
    &\ckI_{2+} &= \{ a,\ldots, a+b-1\},\\
    \ckI_{3+}
    &= \{ a-b,\ldots,          a-1 \},
    &\ckI_{4+}&= \{ b-a,\ldots, b-1 \} 
  \end{align*}
  be the indicated index sets, let
  \[ \cksig_{2-} = \cksig_{2+} = \cksig_{3+} = \cksig_{4+} = \{
    (k+1)(k-a-b): k =0,1,\ldots, b-1\} \] the corresponding
  eigenvalues, and
  let $E_k,\; k=0,\ldots, b-1$ denote the
  corresponding eigenspace of $\ckT$.  Then, $E_k$ is 2-dimensional
  and consists of rational functions.  Aside from the zero element,
  the eigenfunctions in $E_k$ fall into 4 categories, the first
  generic, and the rest belonging to distinguished 1-dimensional
  subspaces:
  \begin{itemize}
  \item type 2 eigenfunctions having index $a+b-k-1 \in \ckI_{2+}$ and
    degree $-k-1$;
  \item type 2 eigenfunctions of index $k\in \ckI_{2-}$ and
    degree $k-a-b$;
  \item type 3 eigenfunction of index $a-k-1\in \ckI_{3+}$ and degree $-k-1$;
  \item type 4 eigenfunctions of index $b-k-1\in \ckI_{4+}$ and degree
    $-k-1$.
  \end{itemize}
\end{lem}
\begin{proof}
  Fix a $k\in \{0,1,\ldots, b-1\}$ and note that the eigenvalues in
  question, $ (k+1)(k-a-b)$, can be expressed in 4 different was as
  \[ \lam_2(k;a,b) = \lam_2(a+b-k-1;a,b) = \lam_3(a-k-1;a,b) =
    \lam_4(b-k-1;a,b).\] As a basis for $E_k$ we can take
  \begin{align*}
    \phi_{3,a-k-1}(x)
    &=(1-x)^{-a}P_{a-k-1}(x;-a,b),\quad \; k=0,\ldots, b-1,\\
    \phi_{4,b-k-1}(x)
    &= (1+x)^{-b} P_{b-k-1}(x,a,-b).
  \end{align*}
  These constitute the type 3 and type 4 subspaces of the above list.
  The linear combination
  \[ \phi_{2}(x,t;k):=
    2^b\phi_{4,b-k-1}(x)-(-2)^at\,\phi_{3,a-k-1}(x),\; t\ne 0\] is a
  type 2 eigenfunction whose index belongs to either $I_{2-}$ or to
  $I_{2+}$.  The classical identity
  \begin{equation}
    \label{eq:classid3}
     2^b (x-1)^{a} P_{b-k-1}(x;a,-b) - 2^a(x+1)^{b}  P_{a-k-1}(x;-a,b)  =
     2^{a+b}(-1)^a P_{k}(x;-a,-b),
   \end{equation}
   implies that
   \[ \phi_{2}(x,1;k) = 2^{a+b}(1-x)^{-a}(1+x)^{-b} P_k(x;-a,-b)\] has
   index $k$. By contrast, if $t\ne 1$, then $\deg \phi_2 = -k-1$
   and hence the index of $\phi_2$ is $l:=a+b-k-1$, which by
   inspection is an element of $I_{2+}$. 
\end{proof}

\noindent

\begin{prop}
  \label{prop:sigqclass}
  If $T=T(a,b),\; a,b\in \R$ is a classical operator, then its qr
  spectrum is given by $\sigq(T) = \sigq(a,b)$, with the latter as per
  ~\eqref{eq:sigqdef}.
\end{prop}
\begin{proof}
  As a special case of Lemma ~\ref{lem:Reigenfun} with $\tau=1$, every
  qr-eigenfunction of $T(a,b)$ has the form
  $\phi(x)=\mu_\imath(x;a,b)P(x)$ for some $\imath \in \{1,2,3,4\}$ and
  polynomial $P(x)$.  Hence, by Proposition ~\ref{prop:sig1234}, we
  have $\sigq(T)\subseteq \sigq(a,b)$.  The converse $\sigq(a,b)
  \subseteq \sigq(T)$ is a direct consequence of the above Propositions.
 
\end{proof}

\subsection{Formal norms of classical operators.}

Throughout this section, let $\nu(z;a,b)$ be as defined in ~\eqref{eq:nunab}, and let
\begin{equation}
  \label{eq:Wdef}
  W(x;a,b):= (1-x)^a(1+x)^b.
\end{equation}
\begin{lem}
  \label{lem:adqr}
  Let $q\in \cQo,\; a,b\in \R$ and suppose that $q(\pm 1) \ne 0$ and that
  \[ \rho(x) = \int q(x) W(x;a,b) \] is a quasi-rational
  function.  Then, one of the following mutually exclusive
  possibilities holds: (i) $\rho \in \qQ_{a+1,b+1}$; (ii)
  $\rho \in \qQ_{0,b+1},\;a\in \Nz,\; \rho(1)\ne 0$; (iii)
  $\rho\in \qQ_{a+1,0},\;b\in \Nz,\;\rho(-1) \ne 0$; (iv) $\rho\in
  \cQo,\; a,b\in \Nz,\; \rho(\pm 1)\ne 0$.
\end{lem}
\begin{proof}
  By assumption, $w(x)=\rho'(x)/\rho(x)$ is rational. Hence,
  \begin{align} 
    \label{eq:rhocond}
    &\rho(x) =\frac{q(x)}{w(x)} (1-x)^a(1+x)^b     \\
    \label{eq:wcond}
    &w(x)+ \frac{w'(x)}{w(x)}-\frac{q'(x)}{q(x)}= \frac{a}{x-1} +
      \frac{b}{x+1} . 
  \end{align}
  By inspection $w(x)$ has only simple poles.  Hence, either $x=1$ is
  a simple pole of $w(x)$; in which case, $\Res_{x=1} w(x) = a+1$ and
  $\ord_{x=1} \rho(x) = a+1$.  Otherwise, $x=1$ is a zero of $w(x)$;
  in this case necessarily $a\in \Nz$ is the multiplicity of that
  zero, and $\rho(1) \ne 0$.  Similarly, if $x=-1$ is a simple pole of
  $w(x)$, then $\Res_{x=-1} w(x) = b+1$ and
  $\ord_{x=-1} \rho(x) = b+1$.  Otherwise, $x=-1$ is a zero of $w(x)$
  with multiplicity $b\in \Nz$ and $\rho(-1)\ne 0$.
\end{proof}

\begin{prop}
  \label{prop:Rab}
  The first order operators $D_x$ and
  \begin{equation}
    \label{eq:Rabdef}
    R(a,b) := (x^2-1)D_x +a(x+1)+b(x-1) ,\quad a,b\in \R
  \end{equation}
  serve as a lowering and raising operators for the
  classical Jacobi polynomials in
  the sense that
  \begin{align}
    \label{eq:DPn}
    D_x P_n(x;a,b)
    &= \frac{1}{2}(n+a+b+1) P_{n-1}(x;a+1,b+1),\quad     n\in \N,\\
    \label{eq:Rab}
    R(a,b) P_n(x;a,b) &= 2(n+1)P_{n+1}(x;a-1,b-1),\quad    n\in \Nz
  \end{align}
\end{prop}
\begin{proof}
  By direct calculation,
  \begin{equation}
    \label{eq:Rab2}
    D_x \circ \cM( W(x;a,b)) = -\cM(W(x;a-1,b-1))\circ  R(a,b).
  \end{equation}
  By the Rodrigues formula ~\eqref{eq:Pab2},
  \begin{equation}
    \label{eq:Rab3}
    -D_x \left[W(x;a,b)  P_{n}(x;a,b) \right]
    = 2(n+1) W(x;a-1,b-1)  P_{n+1}(x;a-1,b-1). 
  \end{equation}
  Identity ~\eqref{eq:Rab}  follows.
\end{proof}

\begin{prop}
  \label{prop:classqr}
  For $a,b\in \R$, the formal anti-derivative $\rho(x):=\int W(x;a,b) $
  defines a quasi-rational function if and only if one of three
  mutually exclusive conditions hold  (A) either $a\in \Nz, b\notin \Zm$ or
  $b\in \Nz, a\notin \Zm$; (C) $a+b+1\in \Zm$ and
  $a,b\notin \Zm$; or (H)  $a+b+1\in \Zm$ and $a,b\in \Z$ with opposite signs.
\end{prop}
\begin{proof}
  If (A) holds, then $\rho(x)$ is quasi-rational by inspection.
  Observe that
  \begin{equation}
    \label{eq:sfid1}
    D_x W(x;a+1,b+1)=
    (2+a+b) W(x;a+1,b)-2(a+1)W(x;a,b)
  \end{equation}
  Hence, if 
  $a+b+1=-1,\; a\ne -1$, then $\rho(x)$ is quasi-rational by ~\eqref{eq:sfid1}. From
  there, the sufficiency of (C) and (H) can be shown by induction.

  Conversely, suppose that $\rho(x)$ is quasi-rational.  If
  $a,b\notin \Z$, then by Lemma ~\ref{lem:adqr}
  $\rho\in \qQ_{a+1,b+1}$.  Hence, by ~\eqref{eq:Rab2} and by
  ~\eqref{eq:Rab}, we must have $a+b+1\in \Zm$.  Suppose that
  $a\in \Zm$. Since $\rho(x)$ doesn't have a logarithmic singularity
  at $x=1$, we must have
  \begin{equation}
    \label{eq:ResBnu}
    \Res_{x=1} (1-x)^a(1+x)^b = (-1)^{a} 2^{1+a+b}
    \frac{\Gamma(b+1)}{\Gamma(-a)\Gamma(2+a+b)} = -\Res_{z=0}\nu(z;a,b)=0.
  \end{equation}
  This forces $b\in \Nz$ and $a+b+1\in \Zm$, which is condition (H).
  Analogous reasoning applies if $b\in \Zm$.
\end{proof}

\begin{lem}
  \label{lem:sfident}
  Let $n\in \N$. If $a,b\notin \Zm$ and if $n+a+b+1>0$, then
  \begin{align}
    \label{eq:sfident2}
    &\frac{\pi_n(x;a,b)^2}{\nu(n;a,b)}W(x;a,b)  -
    \frac{W(x;a+n,b+n)}{\nu(a+n,b+n)} 
    = -D_x\lp\sum_{j=0}^{n-1} \frac{
      \tpi_{j+1}(x)}{(j+1)} \frac{\tpi_{j}}{\tnu_j}W(x;a+n-j,b+n-j)\rp,
  \end{align}
  where
  \[
    \begin{aligned}
      \tpi_j(x)
      &:= \pi_j(x;a+n-j,b+n-j),\quad j=0,1,\ldots, n-1;  \\
      \tnu_j &:= \nu(j;a+n-j,b+n-j).
    \end{aligned}
    \]
\end{lem}
\begin{proof}
  Set $n^*= -n-1-a-b$, and observe that by
   ~\eqref{eq:DPn},
    ~\eqref{eq:Rab2} and ~\eqref{eq:Rab}, we have
    \begin{equation}
      \label{eq:sfid3}
      \begin{aligned}
        &D_x\left[ \pi_n(x;a,b) \pi_{n-1}(x;a+1,b+1)
          W(x;a+1,b+1) \right]\\
        &\quad =  n^*\,\pi_n(x;a,b)^2W(x;a,b)
        +n\,\pi_{n-1}(x;a+1,b+1)^2W(x;a+1,b+1) ,
      \end{aligned}
    \end{equation}
    Dividing through by $n\nu(n-1,a+1,b+1)$ and using the identity
    \[ (n+a+b+1) \nu(n,a,b) = n \nu(n-1,a+1,b+1),\]
    gives
  \begin{align*}
    &\frac{\pi_n(x;a,b)^2}{\nu(n;a,b)}W(x;a,b)
    -\frac{\pi_{n-1}(x;a+1,b+1)^2}{\nu(n-1;a+1,b+1)}
      W(x;a+1,b+1)  \\
    &\quad =D_x\left[ \frac{\pi_n(x;a,b)}{\nu(n;a,b)}
      \frac{\pi_{n-1}(x;a+1,b+1)}{n-1} 
      W(x;a+1,b+1) \right]
  \end{align*}
  Relation ~\eqref{eq:sfident2} now follows by induction.
\end{proof}

\begin{prop}
  \label{prop:sfclass}
  The formal norms corresponding to the classical Jacobi polynomials
  $\pi_n(x;a,b),\; a,b, a+b+1\notin \Zm$ are given by
  $\nu_n = \nu(n;a,b),\; n\in \Nz$ in the sense of ~\eqref{eq:sfdef}.
\end{prop}
\begin{proof}
  Let $W(x;a,b)$ be as in ~\eqref{eq:Wdef}. By Lemma  ~\ref{lem:sfident},
  \begin{equation}
    \label{eq:pinWabn}
    \int \frac{\pi_n(x;a,b)^2}{\nu(n;a,b)} W(x;a,b)
    -    \frac{W(x;a+n,b+n)}{\nu(a+n,b+n)}  \in \qQ_{a+1,b+1}.
  \end{equation}
  Identity ~\eqref{eq:sfid1} can be restated as
  \[ \frac{W(x;a+1,b)}{\nu(a+1,b)} - \frac{W(x;a,b)}{\nu(a,b)} = D_x
    \lb \frac{ W(x;a+1,b+1)}{(2+a+b)\nu(a+1,b)}\rb .\]
  By induction and by Lemma ~\ref{lem:adqr},
  \begin{equation}
    \label{eq:Wabn}
     \int \frac{W(x;a+n,b+n)}{\nu(a+n,b+n)} - \frac{W(x;a,b)}{\nu(a,b)}\in
    \qQ_{a+1,b+1} .
  \end{equation}
  It follows that
  \[ \int \lp \pi_n(x;a,b)^2-
    \frac{\nu(n;a,b)}{\nu(a,b)}\rp W(x;a,b) \in \qQ_{a+1,b+1}.\]
  This proves ~\eqref{eq:sfdef} for the case where $i=j=n$.
  For $i\ne j$, relation ~\eqref{eq:sfdef} follows by Lemma ~\ref{lem:forthog}.
\end{proof}

\begin{prop}
  \label{prop:sfclass2}
  Suppose that $a,b\notin \Z, a+b+1\in \Zm$ and let $\ckI_1$ be as
  defined in ~\eqref{eq:ckI12}.  The formal norms of classical
  polynomials $\pi_n(x;a,b),\;n\in \ckI_1$ is also given by
  $\nu_n = (-1)^n\nu(a+n,b+n)$, but in the sense of ~\eqref{eq:sfdef2}.
  Moreover,
  \begin{equation}
    \label{eq:nunC}
    \nu_n =
    \begin{cases}
      \nu(n;a,b), &\text{ if } n\in I_{1+}\\
      \frac12 \nu(n;a,b),& \text{ if }  2n+a+b+1 = 0,\\
       0& \text{ otherwise}.
    \end{cases}
  \end{equation}
\end{prop}
\begin{proof}
  If $n\in \ckI_{1+}$, that is if $n+a+b+1>0$, then ~\eqref{eq:pinWabn}
  holds by Lemma ~\ref{lem:sfident}. The argument used in the proof of
  Proposition ~\ref{prop:sfclass} applies to show
  \begin{equation}
    \label{eq:WabnWa}
    \int \frac{W(x;a+n,b+n)}{\nu(a+n,b+n)} -
    \frac{W(x;a,-a-1)}{\nu(a,-a-1)}\in \qQ_{a+1,-a} .
  \end{equation}
  Hence, writing $m= a+b+1$, we have
  \[ \int \lp \pi_n(x)^2 - (1+x)^{-m}\frac{\nu_n}{\nu(\al,-1-\al)} \rp
    (1-x)^a (1+x)^{b}\in
      \qQ_{a+1,b+1} \]
    if $n\in \ckI_{1+}$.  If
  $2n+a+b+1<0$, then repeated application of ~\eqref{eq:sfid3} and
  Proposition ~\ref{prop:classqr} show that
  $\int \pi_n(x;a,b)^2 W(x;a,b)$ defines a quasi-rational function.
  The desired conclusion now follows by Lemma ~\ref{lem:adqr}.
  Finally, consider the vertex case where $2n+a+b+1=0$.  Repeated
  application of ~\eqref{eq:sfid3} shows that
  \[ \int \pi_n(x;a,b)^2 W(x;a,b)- (-1)^n W(x;a+n,b+n)  \]
  is quasi-rational.  As above, ~\eqref{eq:WabnWa} also holds. Finally,
  we have
  \[ \nu(n;a,b) = 2 (-1)^n \nu(a+n,b+n),\] where the value of the
  left-side has to be understood in terms of removable singularities.
  Relation ~\eqref{eq:sfdef2} follows.
\end{proof}
\begin{prop}
  \label{prop:rhoA}
  Suppose that $a\in \Nz, b\notin \Z$, and let
  \begin{equation}
    \label{eq:rhoab}
     \rho(x;a,b) = \int W(x;a,b)\in \qQ_{0,b+1} 
  \end{equation}
  be the unique quasi-rational anti-derivative of $W(x;a,b)$. Then,
  \begin{equation}
    \label{eq:nurhoclass}
    \rho(1;a,b) = \nu(a,b)
  \end{equation}
\end{prop}
\begin{proof}
  If $a=0$, then by inspection,
  \[ \nu(0,b) =\frac{2^{b+1}\Gamma(b+1)}{\Gamma(b+2)} =
    \frac{2^{b+1}}{b+1},\quad \rho(x;0,b) = 
    \frac{(1+x)^{b+1}}{b+1},\quad \rho(1;0,b) = \nu(0;0,b).\]
  By ~\eqref{eq:sfid1},
  \[ \rho(1;a+1,b)  = \frac{2(a+1)}{2+a+b} \rho(1;a,b),\quad a\in
    \Nz,\; b\notin \Z.\]
  The desired conclusion now follows by induction on $a$.
\end{proof}

\noindent
\section{Rational Darboux transformations of exceptional Jacobi operators}
\label{sec:JRDT}

\subsection{The four types of Jacobi rational Darboux transformations.}
\label{sec:jacrdt}
The class of exceptional Jacobi operators enjoys four types of
  rational Darboux transformations, corresponding to the four asymptotic types of the Jacobi-class
  factorization eigenfunctions as detailed in Definition
  ~\ref{def:atype}. In each case, the choice of factorization gauge is
  fixed by the assumption that the operators in question are in the
  rational gauge.  In light of this fact, we will refer to these
transformations as Jacobi RDTs of type 1,2,3,4, respectively.  We also
show that all Jacobi RDTs involve a $\pm 1$ shift of the parameters
$(\al,\be)\mapsto (\hal,\hbe)$.  Type 1 and type 2 transformations
also involve a spectral shift $\epsilon$.  By contrast, type 3 and
type 4 transformations are isospectral.

Let $T = \Trg(\tau;\al,\be),\; \al,\be\notin \Zm$ be an exceptional
Jacobi operator.  Let $\imath\in \{1,2,3,4\}$, and set
\begin{equation}
  \label{eq:himath}
  (\himath,\hal_\imath,\hbe_\imath,\epsilon_\imath) =
  \begin{cases}
    (2,\al+1,\be+1,\al+\be+2) & \text{ if } \imath=1,\\
    (1,\al-1,\be-1,-\al-\be) & \text{ if } \imath=2,\\
    (4,\al-1,\be+1,0) & \text{ if } \imath=3,\\
    (3,\al+1,\be-1,0) & \text{ if } \imath=4.
  \end{cases}
\end{equation}
For a given $k\in I_\imath(T)$, let
$\phi_\imk(x)=\mu_\imath(x) \pi_\imk(x)$ be the corresponding
eigenfunction as per ~\eqref{eq:phi1234} with eigenvalue
$\lam_\imk = \lam_\imath(k;\al,\be)$.  Let $\htau_\imk$ be as defined
in ~\eqref{eq:phi1234} and set
\begin{equation}
  \label{eq:hTidef}
  \hT_\imk = \Trg(\htau_\imk;\hal_\imath,\hbe_\imath)
\end{equation}
Let $w_\imk := \phi_\imk'/\phi_\imk$ be the
corresponding log derivatives and set
\begin{equation}
  \label{eq:A1234}
  \begin{aligned}
    A_{1k}&:=D_x-w_{1k},\\
    A_{2k} &:= (x^2-1)(D_x-w_{2k})\\
    A_{3k} &:= (x-1)(D_x-w_{3k})\\
    A_{4k}&:=(x+1)(D_x-w_{4k})\\
  \end{aligned}
\end{equation}
Next, set
\begin{equation}
  \label{eq:hphidef}
  \hphi_\himk(x) =
  \frac{\mu_\himath(x;\hal_\imath,\hbe_\imath)}{\pi_\imk(x)},\quad
  \hk = -k
\end{equation}
and let $\hw_\himk=\hphi_\himk'/\hphi_\himk$ be the corresponding
rational log-derivatives.  Let $\hA_\himk$ be defined analogously to
~\eqref{eq:A1234}, but with $w_\imk$ replaced by $\hw_\himk$.

\begin{prop}
  \label{prop:jacrdt}
  With the above definitions, $T \rdt{\imk} \hT_\imk+\epsilon_\imath$
  is an RDT\footnote{Since the qr-eigenvalue $\lam_\imk$ is fully
    determined by the type $\imath$ and the index $k$, we sometimes
    denote the corresponding RDT by $\rdt{\imk}$ instead of the more
    cumbersome $\rdt{\lam_\imk}$.} with intertwining relations
  \begin{equation}
    \label{eq:intrel}
    A_\imk \circ T = (\hT_\imk+\epsilon_\imath) \circ A_\imk,\quad
    \hA_\himk \circ (\hT_\imk+\epsilon_\imath) = T
    \circ \hA_\himk,  
  \end{equation}
  and factorizations
  \begin{equation}
    \label{eq:ThTAhA}
           T =  \hA_\himk\circ A_\imk+\lam_\imk ,\quad
    \hT_\imk =A_\imk\circ \hA_\himk + \lam_\imk-\epsilon_\imath.
  \end{equation}
  Moreover, $\hk\in I_\himath(\hT_\imk)$ and
  $\hphi_\himk$, as defined in ~\eqref{eq:hphidef}, is the
  corresponding eigenfunction of $\hT_\imk$ with eigenvalue
  $\lam_\himk =\lam_\imk-\epsilon_\imath$.
\end{prop}

\begin{proof}
  We give the proof for the case $\imath=1$.  The other cases are
  argued similarly.  First, we establish that $\hphi_{2\hk}$ is an
  eigenfunction of $\hT_{1k}$ by showing that
  \begin{equation}
    \label{eq:T0w1T1w2}
    \Ric{T}w_1    =      \Ric{\hT_1}\hw_2 + \sigma_1
  \end{equation}
  By inspection,
  \[ w_1=\hu-u,\quad \hw_2 = -w_1-\eta,\] where
  $u=\tau'/\tau,\; \hu=\htau'/\htau$ and with $\eta=\eta(x;\al,\be)$
  as in ~\eqref{eq:etadef}. Also, observe that
  \begin{equation}
    \label{eq:Tabric}
    \Ric{T(\al,\be)}w_1 =
    (x^2-1)\lp w_1'+w_1^2 +\eta\,  w_1\rp
  \end{equation}
  Hence, by ~\eqref{eq:Trgdef2} ~\eqref{eq:hTidef}
  \begin{align*}
    \Ric{T}w_1- \Ric{\hT_1}\hw_2
    &=\Ric{T(\al,\be)}w_1-
      \Ric{T(\al+1,\be+1)}\hw_2\\
    &\quad
     + 2(x^2-1)(u-\hu)' + 2x(u-\hu)\\
    &= (x^2-1)\lp(w_1-\hw_2)' + w_1^2-\hw_2^2
      +\eta   (w_1-\hw_2)\rp\\
    &\quad -2x \hw_2
      -2(x^2-1)w_1'-2x w_1\\
    &= (x^2-1)( 2w'_1 +\eta') -2x \hw_2
      -2(x^2-1)w_1'-2x w_1\\
    &= (x^2-1)\eta' +2x \eta=\al+\be+2=\sigma_1
  \end{align*}
  By ~\eqref{eq:degmui}, we have
  \[ \deg\hphi_{2k} = -\hal_1-\hbe_1+\deg\tau-\deg\htau .\]
  Hence,
  \[ -\deg\hphi_{2k}-\hal_1-\hbe_1-1=\deg\htau+\deg\tau-1=k-1=\hk.\]
  Hence, $\hk\in I_2(\hT_{2k})$ by ~\eqref{eq:k1234}. Moreover,
  \[\lam_{2,\hk} = \lam_{12}(\hk;\hal_2,\hbe_2) = (k-1)(k-1+\al+\be+3)
    = k(k+\al+\be+1) -\al-\be-2 = \lam_{1,k}-\epsilon_1.\]
  
  Next, we prove ~\eqref{eq:ThTAhA}.  Applying ~\eqref{eq:AhAdef} with
  $b=1,\, p=x^2-1,\; w= w_{1,k}$ gives
  \[ \hb = x^2-1,\quad \hw=-w_{1}-\eta(x;\al,\be) = \hw_{2}. \] Hence,
  $A=b(D_x-w)$ and $\hA=\hb(D_x-\hw)$ match the definitions of
  $A_{1,k}$ and $\hA_{2,\hk}$ in ~\eqref{eq:A1234}. The intertwining
  relations ~\eqref{eq:intrel} now follows by Proposition
  ~\ref{prop:RDTqrefun}. The factorization relations follow by
  Corollary ~\ref{cor:RDTfac}. 
\end{proof}
\noindent


\begin{prop}
  \label{prop:covgaugexform}
  Jacobi-class RDTs are covariant with respect to the gauge
  transformations ~\eqref{eq:Tgaugesym}.  Explicitly, let
  $T=\Trg(\tau;\al,\be)$ be an exceptional Jacobi operator and let
  $\hT_\imk=\Trg(\htau_\imk;\hal_\imath,\hbe_\imath)$ a partner
  operator as defined in ~\eqref{eq:hTidef}.  Then,
  \[ \Trg(\tau;-\al,\be) \rdt{\tilde{\imath}k}
    \Trg(\htau_\imk;-\hal_\imath,\hbe_\imath ) \] is also a
  Jacobi-class RDT where $(\imath,\tilde\imath)$ is one of
  $(1,3), (3,1), (2,4), (4,2)$.  Similarly,
    \[\Trg(\tau;\al,-\be) \rdt{\tilde{\imath},-k-1}
      \Trg(\htau_\imk;\hal_\imath,-\hbe_\imath )\] is also a
    Jacobi-class RDT where $(\imath,\tilde\imath)$ is one of
    $(1,4), (4,1), (2,3), (3,2)$.  Finally,
    \[ \Trg(\tau;-\al,-\be) \rdt{\tilde{\imath},-k-1}
      \Trg(\htau_\imk;-\hal_\imath,-\hbe_\imath )\] is also a
    Jacobi-class RDT where $(\imath,\tilde\imath)$ is one of
    $(1,2), (2,1), (3,4), (4,3)$.
\end{prop}
\begin{proof}
  This follows from the form of $\mu_\imath$ show in Table
  ~\ref{tab:qreigen}, ~\eqref{eq:Tgaugesym}, Proposition
  ~\ref{prop:Tgauged}, from the definition of $\htau_\imk$ in
  ~\eqref{eq:phi1234}.
\end{proof}

\begin{prop}
  \label{prop:covreflect}
  Jacobi-class RDTs are covariant with respect to the transformation
  $x\mapsto -x$.  Explicitly, let
  $T=\Trg(\tau;\al,\be)$ be an exceptional Jacobi operator and let
  $\hT_\imk=\Trg(\htau_\imk;\hal_\imath,\hbe_\imath)$ a partner
  operator as defined in ~\eqref{eq:hTidef}.  Set
  \[ \ttau(x) = \tau(-x),\quad \ttau_\imk(x) = \htau_\imk(-x),\quad
    \tT = \Trg(\ttau;\be,\al),\quad
    \tT_\imk=\Trg(\ttau_\imk;\hbe_\imath,\hal_\imath).\] Then, $\tT$
  is also an exceptional operator and $\tT\rdt{\lam} \tT_\imk$ a type
  $\tilde{\imath}$ Jacobi RDT, where $(\imath,\tilde{\imath})$ is one
  of $(1,1), (2,2), (3,4), (4,3)$.
\end{prop}


\subsection{Primitive RDT}
\label{sec:prdt}
In this section, we take note  of the fact that classical Jacobi
operators can be Darboux-connected to other classical operators.  
\begin{definition}
  We say that $T_0\rdt{\lam} T_1$ is a \emph{primitive} RDT if both
  $T_0$ and $T_1$ are classical operators\footnote{In quantum
    mechanics, the presence of such primitive factorizations leads to
    $T_0$ being called ``shape invariant''.}.
\end{definition}
Observe that $D_x=A_{1,0}$ and $R(a,b)=A_{2,0}$, as defined in
~\eqref{eq:Rab} and ~\eqref{eq:A1234} with
\[ w_{1,0} = 0,\quad w_{2,0} = \frac{a}{x-1} + \frac{b}{x+1}.\] Thus,
$D_x$ and $R(a,b)$ should be justly regarded as the intertwiners of
type 1 primitive RDTs and are the factors of the following primitive
factorizations:
\[ T(a,b) = R(a,b) \circ D_x,\quad T(a+1,b+1)+2+a+b = D_x \circ R(a,b)
  .\] There is an analogous pair of shift operators related to the
type 3,4 primitive RDT.  We summarize the possibilities in the
following result. The proof is a special case of Proposition
~\ref{prop:jacrdt}.
\begin{prop}
  \label{prop:prdt}
  The primitive RDTs of a classical operator $T_0=T(a,b)$
  are $T_0\rdt{\lam_\imath} T_\imath,\; \imath\in \{1,2,3,4 \}$ where
  $\lam_\imath = \lam_\imath(0;a,b)$ and where
  \begin{align*}
    & T_1 = T(a+1,b+1) + 2+ a+ b,&& \lam_1 = 0,\\
    & T_2 = T(a-1,b-1) -a-b,&& \lam_2 = -a-b,\; a,b \ne 0 \\
    & T_3 = T(a-1,b+1) ,&& \lam_3 = -a(b+1),\; a\ne 0\\
    & T_4 = T(a+1,b-1),&& \lam_4 = -b(a+1),\; b\ne 0
  \end{align*}
\end{prop}

\noindent
Note that there are instances where the ``ground state'' index 0 may
be missing from one of the index sets.  In such cases, the classical
$T(a,b)$ will not have an RDT of that type.  For example if $a=0$, the
operator $T(a,b)$ will not have type 2 and type 3 qr-eigenfunctions
for the simple reason that $(1-x)^{-a}$ is just the constant function.
Such operators will have $\ckI_2=\ckI_3=\emptyset$ and will not have
type 2 and 3 primitive RDTS.

The following results (Propositions ~\ref{prop:DCA} -
  ~\ref{prop:DCD34-}) are all direct consequences of Proposition
  ~\ref{prop:prdt} and of the results in Section ~\ref{sec:classSD}.
  The details are left to the reader.
\begin{prop}
  \label{prop:DCG}
  Let $T_0=T(a,b),\; a+b\notin \Zm$ be a classical operator and let
  $k\in \N$. Then, $T_0$ is Darboux connected to
  $T_{k} = T(a+k,b+k)+ k(a+b+k+1)$ by a primitive type 1 chain whose
  eigenvalue sequence is $ \lam_j:=\lam_1(j;a,b), j=0,1,\ldots, k-1$.
  The spectral diagrams of $T_0$ and $T_k$ differ by
  $\boxcirc \to \boxtimes$ flips at the eigenvalues
  $\lam_j,\; j=0,\ldots, k-1$.
\end{prop}

\begin{prop}
  \label{prop:DCA}
  A classical type A operator $T(a,b),\; a\in \N,b\notin \Z$ is
  Darboux connected to the classical type A operator $T_a:=T(0,a+b)$
  by a primitive type 3 chain whose eigenvalue sequence is
  $\lam_j:= \lam_3(j;a,b), j=0,1,\ldots, a-1$.  The spectral
  diagrams of $T_0$ and $T_a$ differ by $\boxcirc\to \boxstar$ flips
  at these eigenvalues.
\end{prop}

\begin{prop}
  \label{prop:DCB+}
  Let $T_0=T(a,b),\; a,b\notin \Z,\; a-b\in \Nz$ be a classical type B
  operator and let $k\in \N$. Then, $T_0$ is Darboux connected to
  $T_{k}=T(a+k,b-k)$ by a primitive type 4 chain whose eigenvalue
  sequence is $\lam_j:=\lam_4(j;a,b), j=0,1,\ldots, k-1$. The spectral
  diagrams of $T_0$ and $T_k$ differ by $\boxpm\to \boxplus$ flips at
  these eigenvalues.
\end{prop}

\begin{prop}
  \label{prop:DCB-}
  Let $T_0=T(a,b),\; a,b\notin \Zm, a-b\in \Z$ be a classical type B
  operator.  If $a-b\ge 2$, then $T_0$ is Darboux-connected to the
  classical type B operator
  $T_k=T(a-k,b+k),\; k = \lfloor (a-b)/2\rfloor$ by a type 3 primitive
  chain whose eigenvalue set is
  $\lam_j:=\lam_3(j;a,b), j=0,\ldots, k-1$. The spectral diagrams of
  $T_0$ and $T_k$ differ by $\boxplus \to \boxpm$ flips at these
  eigenvalues.  If $b-a\ge 2$, then $T_0$ is Darboux-connected to the
  classical type B operator
  $T_k=T(a+k,b-k),\; k = \lfloor (b-a)/2\rfloor$ by a primitive type 4
  chain whose eigenvalue sequence is
  $\lam_j:=\lam_3(j;a,b), j=0,\ldots, k-1$. The spectral diagrams of
  $T_0$ and $T_k$ differ by $\boxminus \to \boxpm$ flips at these
  eigenvalues.
\end{prop}

\begin{prop}
  \label{prop:DCD2}
  Let $T_0=T(a,b),\; a,b\in \Nz$ be a type D classical operator.
  Then, $T_0$ is Darboux-connected to the classical operator
  $T_k=T(a-k,b-k)-k(a+b+k-1),\; k = \min\{a,b\}$ by a primitive type 2
  chain whose eigenvalue sequence is
  $\lam_j= (j+1)(j-a-b), j=0,1,\ldots, k-1$.  The spectral diagrams of
  $T_0$ and $T_k$ differ by $\boxfsq \to \boxcirc$ flips at these
  eigenvalues.
\end{prop}

\begin{prop}
  \label{prop:DCD34+}
  Let $T_0=T(a,b),\; a,b\in \Nz$ be a type D classical operator.  If
  $a\ge b$, then $T_0$ is Darboux-connected to the classical operator
  $T_k:=T(a+b,0),\; k=b$ by a primitive type 4 chain whose eigenvalue
  sequence is $\lam_j = (j+1)(j-a-b): j=0,1,\ldots, k-1$.  The
  spectral diagrams of $T_0$ and $T_k$ differ by
  $\boxfsq \to \boxplus$ flips at these eigenvalues.  If $b\ge a$,
  then $T_0$ is Darboux-connected to the classical
  $T_k=T(0,a+b),\; k=a$ by a primitive type 3 chain whose eigenvalue
  sequence has the form above.  In this case, the spectral diagrams of
  $T_0$ and $T_k$ differ by $\boxfsq \to \boxminus$ flips.
\end{prop}

\begin{prop}
  \label{prop:DCD34-}
  Let $T_0=T(a,b),\; a,b\in \Nz$ be a type D classical operator.  If
  $a-b\ge 2$, then $T_0$ is Darboux-connected to the classical
  operator $T_k:=T(a-k,b+k),\; k=\lfloor (a-b)/2\rfloor$ by a
  primitive type 3 chain whose eigenvalue sequence is
  $\lam_j:=\lam_3(j;a,b), j=0,\ldots, k-1$.  The spectral diagrams of
  $T_0$ and $T_k$ differ by $\boxplus \to \boxfsq$ flips at these
  eigenvalues.  If $b-a\ge 2$, then $T_0$ is Darboux-connected to the
  classical operator $T_k:=T(a+k,b-k),\; k=\lfloor (b-a)/2\rfloor$ by
  a primitive type 4 chain whose eigenvalue sequence is
  $\lam_j:=\lam_4(j;a,b), j=0,\ldots, k-1$.  The spectral diagrams of
  $T_0$ and $T_k$ differ by $\boxminus \to \boxfsq$ flips at these
  eigenvalues.
\end{prop}

\subsection{Transformations of the spectral diagram}
\label{sec:flips}

In the preceding section, we established that exceptional Jacobi
operators are interrelated by four types of RDTs according to the
asymptotic type of the corresponding factorization eigenfunction.  In
the present section, we describe the action of these 4 types of RDTs
on the corresponding spectral diagrams.  These results will allow us,
in the following section, to furnish proofs of Theorem ~\ref{thm:sd},
Theorem ~\ref{thm:esd}, and Theorem ~\ref{thm:flip}.

\begin{prop}
  \label{prop:1flip}
  Let $T=\Trg(\tau;\al,\be),\; \al,\be\notin \Zm$ be an exceptional
  operator and
  $T\rdt{\imk}\hT_\imk+\epsilon_\imath,\; \imath\in \{1,2,3,4\},\,k
  \in I_\imath(T) $ a Jacobi RDT as described by ~\eqref{eq:himath} and
  ~\eqref{eq:hTidef}. Then, $\hT_\imk$ is also an exceptional operator.
  The spectral diagrams of $T$ and $\hT_\imk+\epsilon_\imath$ differ
  only by a label change at $\lam_\imk=\lam_\imath(k;\al,\be)$.  
\end{prop}
\noindent
The proof requires the following Lemma.

\begin{lem}
  \label{lem:sdstable}
  Let $T, \hT_\imk$ be as above.  Let
  $\phi_{\jmath\ell}(x),\; \jmath\in \{1,2,3,4\},\, \ell\in
  I_\jmath(T)$ be another qr-eigenfunction of $T$ such that
  $\lam_\imk\ne \lam_{\jmath\ell}$. Then
  $\hvarphi_{\jmath\hell} := A_\imk \phi_{\jmath\ell}$ is a non-zero
  qr-eigenfunction of $\hT_\imk$ of asymptotic type $\jmath$ that
  satisfies
  \begin{equation}
    \label{eq:hTfhphif}
    \hT_\imk\hvarphi_{\jmath\hell} =
    (\lambda_{\jmath\ell}-\epsilon_\imath)\hvarphi_{\jmath\hell}. 
   \end{equation}
 \end{lem}

\begin{proof}
  By ~\eqref{eq:ThTAhA} and by assumption,
  \begin{equation}
    \label{eq:hAhvarphi}
     \hA_\himk \hvarphi_{\jmath\hell}= (T-\lam_\imk)\phi_{\jmath\ell} =
    (\lam_{\jmath\ell} - \lam_\imk) \phi_{\jmath\ell}\ne 0.
  \end{equation}
  Applying $A_\imk$ to both sides of the above equation and using the
  factorization ~\eqref{eq:ThTAhA} yields ~\eqref{eq:hTfhphif}.

  We now prove the second assertion.  By the assumption that
  $\al,\be\notin \Zm$, we see that $\phi_{\jmath\ell}(x)$ is singular
  at $x=\pm 1$ if and only if
  $w_{\jmath\ell}=\phi_{\jmath\ell}'/\varphi_{\jmath\ell}$ has a first
  order pole at that point.  By inspection of ~\eqref{eq:A1234}, the
  application of $A_\imk$ or $A_\himk$ cannot introduce a singularity
  at $x=\pm 1$.  Therefore, $\phi_{\jmath\ell}(x)$ is regular at
  $x=\pm 1$ if and only if $\hvarphi_{\jmath\hell}(x)$ is regular at
  that point.

\end{proof}

\begin{proof}[Proof of Proposition ~\ref{prop:1flip}]
  By Lemma ~\ref{lem:sdstable}, if $\lambda\ne \lam_\imk$ is a qr
  eigenvalue of $T$, then it is also a qr-eigenvalue of $\hT_\imk$
  with the same asymptotic label.  In particular, if
  $\lam\ne \lam_\imk$ is a quasi-polynomial eigenvalue of $T$, then it
  is also a quasi-polynomial eigenvalue of $\hT_\imk$.  Thus a 1-step
  RDT can add or remove at most 1 eigenvalue from the quasi-polynomial
  spectrum.  Therefore $\hT_\imk$ is also exceptional.

  Consider how the RDT $T\rdt{\imk} \hT_\imk+\epsilon_i$ modifies the
  spectral diagram at $\lam_\imk$.  By construction, we have
  $A_\imk\phi_\imk=0$.  Thus, $A_\imk\phi_\imk$ cannot serve as a qr
  eigenfunction of $\hT_\imk$ at $\lambda_\imk$.  Instead, by
  Proposition ~\ref{prop:jacrdt}, that role is taken up by
  $\hphi_{\himk}$, as defined in ~\eqref{eq:himath} and
  ~\eqref{eq:hphidef}.  Hence, the RDT changes the asymptotic label at
  $\lam_\imk$ and at no other qr-eigenvalue.
\end{proof}

In order to describe the transformation of the index sets induced by
an RDT, we set
\[ \bI_2 = -I_2-1,\quad \bI_4 = -I_4-1,\]
and define the combined index sets  as
\begin{equation}
  \label{eq:I12I34}
  I_{12} := I_1 \cup \bI_2,\quad I_{34}:= I_3 \cup
  \bI_4.
\end{equation}
\noindent
The combined index sets undergo the following transformations.
\begin{prop}
  \label{prop:indexflip}
  Let $T=\Trg(\tau;\al,\be),\; \al,\be\notin \Zm$ be an exceptional
  operator and
  $T\rdt{\imk}\hT_\imk+\epsilon_\imath,\; \imath\in \{1,2,3,4\},\,k
  \in I_\imath(T) $ a Jacobi RDT as described by ~\eqref{eq:himath} and
  ~\eqref{eq:hTidef}.  Further suppose that
  the factorization eigenvalue $\lam_\imk$ is a simple qr
  eigenvalue of $T$.
  Let $\hI_{12}, \hI_{34}$ be the combined index
  sets of $\hT_\imk$.  If $\imath=1$, we have
  $\hI_{12} \supseteq I_{12}-1,\; \hI_{34} \supseteq I_{34}$; for
  $\imath=2$, we have
  $\hI_{12} \supseteq I_{12}+1,\; \hI_{34} \supseteq I_{34}$; for
  $\imath=3$, we have
  $\hI_{12} \supseteq I_{12},\; \hI_{34} \supseteq I_{34}-1$; for
  $\imath=4$, we have
  $\hI_{12} \supseteq I_{12},\; \hI_{34} \supseteq I_{34}+1.$
\end{prop}

\begin{proof}
  Let
  $\phi_{\jmath\ell}(x),\; \jmath\in \{1,2,3,4\},\, \ell\in
  I_\jmath(T)$ be another qr-eigenfunction of $T$ such that
  $\lam_\imk\ne \lam_{\jmath\ell}$.  Let us consider the index of
  $\hvarphi_{\jmath\hell} := A_\imk \phi_{\jmath\ell}$.  We focus on
  the case of $\imath=1$; the other cases are argued similarly. By
  inspection of $A_{1k}$ defined in ~\eqref{eq:A1234},
  $\deg \hvarphi \le \deg \phi-1$.  Since $\himath=2$, by inspection
  of $\hA_{2k}$ defined in ~\eqref{eq:A1234} and by
  ~\eqref{eq:hAhvarphi}, $\deg\phi \le \deg\hvarphi + 1$.  Hence
  $\deg\hvarphi = \deg\phi-1$.  Suppose that $\jmath=1$.  By the
  definition of index in ~\eqref{eq:k1234} and ~\eqref{eq:phi1234}, we
  see that $\hvarphi$ has index $\hell=\ell-1$.  Next, suppose that
  $\jmath=2$. Since $\hal =\al+1,\hbe=\be+1$, it follows that
  \[ \hell=\deg\hvarphi +\hal+\hbe = \deg\phi-1+\al+\be+2 = \ell+1.\]
  Next, consider the case of $\jmath=3$.  Now
  \[ \hell=\deg\hvarphi+\hal =  \deg\varphi+\al = \ell.\]  Finally, consider
  the case of $\jmath=4$.  Now, also,
  \[ \hell=\deg\hvarphi+\hbe = \deg\varphi-1+\be+1 = \ell.\]

  Next, let us consider the transformation of the index sets
  $I_\imath\to \hI_\imath,\; \imath\in \{1,2,3,4\}$ by the RDT in
  question.  If $\jmath\ne \imath$, then the index of
  $A_\imk \phi_{j\ell},\; \ell \in I_\jmath$ undergoes the
  transformation described above.  If $\imath=1$, then the RDT removes
  $k$ from $I_1$ and adds $-k$ to $\hI_2$.  This conforms to a shift of
  $I_{12}$ by -1.  If $\imath=2$, then the RDT removes $k$ from $I_2$
  and adds $-k$ to $\hI_1$.  This conforms to a shift of $I_{12}$ by
  $+1$.  The case of type 3 and 4 RDTs is argued similarly.
\end{proof}

\begin{remark}
Remark: since the RDT relation is symmetric, the above result implies
that if the dual eigenfunction
$\lam_{\himath,\hk} = \lam_\imk - \epsilon_\imath$ is also a simple qr
eigenfunction of $\hT_\imk$ then the $\supseteq$ symbol above can be
replaced by equality.  For example, this is the case for operators
belonging to the generic G class. The reason behind this subtlety and
the reason behind the requirement that $\lam_\imk$ be simple is that
the qr-eigenfunctions associated with eigenvalues different from
$\lam_\imk$ retain their asymptotic type and merely undergo an index
shift by one of $-1,0,1$.  At the eigenvalue $\lam_\imk$ the effect of
the RDT is complicated.  If $\lam_\imk$ is simple, but $\lam_\himk$ is
degenerate, then new qr-eigenfunctions are created and new indices are
added to $\hI$.  Conversely, if $\lam_\imk$ is degenerate, but
$\lam_\himk$ is simple, then some qr-eigenfunctions are annihilated,
and some indices are removed from $\hI$.  The specifics of such
transformations are described by the various flip alphabets introduced
below.  The main utility of Proposition ~\ref{prop:indexflip} is to
describe the index shifts of the qr-eigenfunctions unrelated to the
factorization eigenvalue.
\end{remark}

\subsection{Labels and flips}
\begin{definition}
  \label{def:flipalph}
  Let $X$ be a collection of exceptional Jacobi operators.  We define
  the $X$ \emph{label set} to be the collection of the possible
  $\lam$-labels of all the qr-eigenvalues $\lam$ of all the operators
  in $X$.  If $X$ is closed under Jacobi RDTs, we further define the
  $X$ \emph{flip alphabet} to be the set of the possible
  transformations\footnote{As demonstrated in Proposition
    ~\ref{prop:1flip}, an RDT changes exactly one label of a spectral
    diagram.} of that label set engendered by the 4 types of Jacobi
  RDTs.
\end{definition}

\noindent

\begin{prop}
  \label{prop:lsfa}
  The GABCD degeneracy classes are all closed under Jacobi RDTs.  The
  G-class label set is $\boxcirc,\boxtimes, \boxplus,\boxminus$.  The
  G-class flip alphabet is
  \begin{equation}
    \label{eq:Gflips}
     1:\; \boxcirc\to \boxtimes;\quad 2:\; \boxtimes\to \boxcirc;\quad
    3:\; \boxplus\to \boxminus;\quad 4:\; \boxminus \to \boxplus.
  \end{equation}
  The A-class label set is $\boxcirc,\boxstar,\boxminus$. The A-class
  flip alphabet is:
  \begin{equation}
    \label{eq:Aflips}
     1:\; \boxcirc\to \boxstar;\quad 2:\; \boxstar\to \boxcirc;\quad
    3:\; \boxstar\to \boxminus;\quad 4:\; \boxminus \to \boxstar.
  \end{equation}
  The B-class label set is
  $\boxcirc,\boxtimes,\boxminus,\boxplus,\boxpm$.  The type B-class
  flip alphabet is:
  \begin{equation}
    \label{eq:Bflips}
    \begin{aligned}
      &1:\; \boxcirc\mapsto \boxtimes;
      && 2:\; \boxtimes\mapsto  \boxcirc,
      && 3:\boxplus\mapsto \boxdiv,\; \boxdiv\mapsto    \boxminus,\;
      && 4: \boxminus \mapsto    \boxdiv,\; \boxdiv \mapsto    \boxminus.\\
      &&&
      &&\phantom{3:\ }\bboxplus\mapsto \bboxminus
      &&\phantom{4:\ } \bboxminus\mapsto \bboxplus
    \end{aligned}
  \end{equation}
  The C-class label set is
  $\boxcirc,\boxtimes, \boxotimes,\boxminus,\boxplus$.  The  C-class
  flip alphabet is:
  \begin{equation}
    \label{eq:Cflips}
    \begin{aligned}
      &1: \boxcirc\mapsto \boxotimes,\; \boxotimes\mapsto \boxtimes,\;
      && 2: \boxtimes\mapsto \boxotimes,\; \boxotimes \mapsto   \boxcirc,
    && 3:\boxplus\mapsto    \boxminus,
    &&4: \boxminus     \mapsto \boxplus.    \\
    &\phantom{1:\ }     \bboxcirc\mapsto \bboxtimes
    &&\phantom{2:\ }    \bboxtimes\mapsto \bboxcirc
    \end{aligned}
  \end{equation}
  The CB-class  label set is 
  $\boxcirc,\boxtimes, \boxotimes,\boxminus,\boxplus,
  \boxdiv$.  The  CB-class flip alphabet is:
  \begin{equation}
    \label{eq:CBflips}
    \begin{aligned}
     &1: \boxcirc\mapsto \boxotimes,\; \boxotimes\mapsto \boxtimes,
     &&2: \boxtimes\mapsto \boxotimes,\; \boxotimes \mapsto
    \boxcirc,
    &&3:\boxplus\mapsto \boxdiv,\; \boxdiv\mapsto \boxminus,
    &&4: \boxminus \mapsto \boxdiv,\; \boxdiv \mapsto \boxminus,\\
    &\phantom{1:\ } \bboxcirc\mapsto \bboxtimes, 
    &&\phantom{2:\ } \bboxtimes\mapsto \bboxcirc,
    &&\phantom{3:\ }\bboxplus\mapsto \bboxminus,
    &&\phantom{4:\ } \bboxminus \mapsto \bboxplus.
    \end{aligned}
  \end{equation}
  The D-class label set is
  $\boxcirc,\boxtdown, \boxfsq, \boxplus, \boxminus$.  Type
  D-class flip alphabet is
  \begin{equation}
    \label{eq:Dflips}
    \begin{aligned}
      &1:\; \boxcirc,\boxtdown \mapsto \boxfsq,
      &&2:\;    \boxfsq \mapsto \boxcirc,\boxtdown,
      &&3: \boxfsq \mapsto \boxminus, \boxplus\mapsto \boxfsq,
    &&4:\boxfsq \mapsto \boxplus, \boxminus\mapsto \boxfsq,\;\\
    &&&
    &&\phantom{3:\ }    \bboxplus\mapsto \bboxminus,
    &&\phantom{4:\ }    \bboxminus\mapsto \bboxplus
    \end{aligned}
  \end{equation}
\end{prop}
\noindent
The proof is broken up into the following Lemmas ~\ref{lem:classlabels}
- ~\ref{lem:jrdtclosed}.

\begin{lem}
  \label{lem:classlabels}
  The label sets of a classical GABC Jacobi operators are as given in
  Proposition ~\ref{prop:lsfa}.  The label set of classical type D
  operators is $\boxcirc,\boxfsq,\boxplus,\boxminus$.
\end{lem}
\begin{proof}
  This is a direct consequence of Propositions ~\ref{prop:classG}
  -- ~\ref{prop:classD}.
\end{proof}

\begin{definition}
  Let $T_0=T(a,b),\; a,b\in \N$ be a type D operator with non-zero
  parameters and let $\ckI_2 = \ckI_{2-} \sqcup \ckI_{2+}$ be the
  decomposition of the type 2 index set defined in ~\eqref{eq:ckID2}.
  For $k\in \ckI_{2+}$, we call corresponding
  $T_0 \rdt{2,k} T_1,\;k\in \ckI_{2+}$ a \emph{para-Jacobi RDT}.
\end{definition}

As indicated in Table ~\ref{tab:symbols} and Propositions
~\ref{prop:classD} and ~\ref{lem:classlabels}, the $\boxfsq$ label
associated to a classical type D operator is characterized by the fact
that the corresponding 2-dimensional eigenspace consists entirely of
rational functions.  It is this degenerate nature of the $\boxfsq$
eigenvalues that lies at the core of the para-Jacobi phenomenon.  The
terminology comes from \cite{CY12}; this paper studied considered the
classical operators $T(-a,-b),\; a,b\in \Nz$. These operators are
gauge-equivalent to classical type D operators and the conjugation
~\eqref{eq:Tgaugesym} transforms the $\boxfsq$ eigenvalues of $T(a,b)$
into hyper-degenerate eigenvalues of $T(-a,-b)$.  The characteristic
property of these hyper-degenerate eigenvalues is that the
corresponding 2-dimensional eigenspace is entirely polynomial.

\begin{lem}
  \label{lem:classchain}
  Let $T_0$ be a classical operator GABCD operator, and
  $T_0\rdt{\lam_1} T_1$ a 1-step RDT that \textbf{is not} para-Jacobi.
  Then, $T_1$ can be Darboux-connected to a classical operator using a
  chain whose eigenvalue sequence excludes $\lam_1$ and that consists of
  operators belonging to the same degeneracy class as $T_0$.
\end{lem}
\begin{proof}
  The desired conclusion for type G operators follows by Proposition
  ~\ref{prop:DCG}, by gauge-covariance (Proposition
  ~\ref{prop:covgaugexform}), and by permutability (Proposition
  ~\ref{prop:permut}).

  For the case of type A operators, because of gauge-covariance, it
  suffices to consider the type 1 and type 3 RDTs. In the first
  instance, the conclusion follows by Proposition ~\ref{prop:DCG}.  In
  the second instance, the conclusion follows by Proposition
  ~\ref{prop:DCA}, by observing that $\cksig_3$ consists of $a$
  distinct eigenvalues, and by applying permutability.

  Consider the case of where $T_0$ is a type B classical operator.
  For type 1 and type 2 RDTs, the desired conclusion follows by
  Proposition ~\ref{prop:DCG}, by permutability and by
  gauge-covariance. For type 3 RDTs, the argument splits into two
  subcases.  If $\lam\in \cksig_{3-}$, the conclusion follows by
  Proposition ~\ref{prop:DCB-} and permutability.  If
  $\lam\in \cksig_{3+}$, we have to also invoke Proposition
  ~\ref{prop:DCB+}. The argument for type 4 RDTs is the same.  The
  proof for C and CB operators is analogous to the
  proof for the B class.

  Suppose that $T_0$ is a type D classical operator where WLOG
  $a\ge b$. For type 1 RDTs, the conclusion follows by Proposition
  ~\ref{prop:DCG}. For type 3 and type 4 RDTs the conclusion follows by
  Propositions ~\ref{prop:DCD34+} and ~\ref{prop:DCD34-}.  For a type 2
  RDT, that \emph{is not} para-Jacobi, the conclusion follows by
  Proposition ~\ref{prop:DCD2}.
\end{proof}

\begin{lem}
  \label{lem:1stepG}
  The flips admitted by classical G operators is given by
  ~\eqref{eq:Gflips}.
\end{lem}
\begin{proof}
  Suppose that $T_0=T(a,b),\; a,b,a\pm b\notin \Z$ is a class G
  operator and $T_0\rdt{1,k} T_1,\; k\in \Nz$ a type 1 RDT.  By
  Proposition ~\ref{prop:DCG} and by permutability, we can extend the
  RDT by a chain that connects $T_1$ to a classical $T_{k}$ such that
  the $\lam_{1,k}$-label of $T_k$ is $\boxtimes$ and such that the
  eigenvalue set of the chain that connects $T_1$ to $T_k$ excludes
  $\lam_{1,k}$. Hence, the $\lam_{1,k}$-labels of $T_1$ and $T_k$ are
  the same.  Hence the spectral diagrams of $T_0$ and $T_1$ differ by
  the flip $\boxcirc\to \boxtimes$ at $\lam_{1,k}$.  The case of type
  2,3,4 RDTs in the G class now follows by gauge-covariance
  Proposition ~\ref{prop:covgaugexform}.
\end{proof}

\begin{lem}
  \label{lem:1stepA}
  The flips admitted by classical A operators is given by ~\eqref{eq:Aflips}.
\end{lem}
\begin{proof}
  Suppose that $T_0=T(a,b),\;a\in \N,\; b\notin \Z$ is a class A
  operator.  The argument for type 1 and type 4 RDTs is the same as in
  the proof of Lemma ~\ref{lem:1stepG}.  By Proposition
  ~\ref{prop:classA}, the spectral diagram of $T_0$ has $\boxstar$
  labels at eigenvalues
  $\cksig_{23}=\{\lam_j:=\lam_{3}(j;a,b): j=0,\ldots, a-1\}$. Suppose
  then that $T_0 \rdt{3,j} T_1,\; j\in \{0,\ldots, a-1\}$ is a type 3
  RDT.  By Proposition ~\ref{prop:DCA} and by permutability, the RDT
  can be extended by a chain that connects $T_1$ to the classical
  $T_a=T(0,a+b)$ such that the $\lam_{j}$ label of $T_a$ is $\boxstar$
  and such that the eigenvalue set of the chain excludes
  $\lam_{j}$. Hence, the $\lam_{j}$ labels of $T_1$ and $T_a$ are the
  same. Hence, the spectral diagrams of $T_0$ and $T_1$ differ by a
  flip $\boxcirc \to \boxstar$. The subcase of a type 2 RDT follows by
  gauge-covariance.
\end{proof}

\begin{lem}
  \label{lem:1stepB}
  The flips admitted by classical B operators is given by ~\eqref{eq:Bflips}.
\end{lem}
\begin{proof}
  Suppose that $T_0=T(a,b),\;a,b\notin \Z, a-b\in \Z$ is a type B
  classical operator.  The argument for type 1 and type 2 RDTs is the
  same as in Lemma ~\ref{lem:1stepG}.  Without loss of generality,
  assume that $a-b\ge 2$. By Proposition ~\ref{prop:classB}, the
  spectral diagram of $T_0$ has $\boxplus$ labels at eigenvalues
  $\lam_{3,j},\; j\in \ckI_{3-}$ and $\boxpm$ labels at
  $\lam_{3,j},\; j\in \ckI_{3+}$ with
  $\ckI_{3-} = \{ j\in \Nz : 2j<a-b \}$ and
  $\ckI_{3+} = \{ j\in \Nz : j \ge a-b \}$.  Consider the type 3 RDT
  $T_0 \rdt{3,j} T_1,\; j\in \ckI_{3-}$.  By Proposition
  ~\ref{prop:DCB-} and by permutability, the RDT can be extended by a
  type 3 chain that connects $T_1$ to the classical
  $T_k = T(a-k,b+k),\; k=\lfloor (a-b)/2 \rfloor|$ such that the
  $\lam_{3,j}$ label of $T_k$ is $\boxpm$ and such that the eigenvalue
  set of the chain excludes $\lam_{3,j}$. Hence, the $\lam_{3,j}$
  labels of $T_1$ and $T_k$ are the same. Hence, the spectral diagrams
  of $T_0$ and $T_1$ differ by a flip $\boxplus \to \boxpm$.

  Next consider the type 3 RDT $T_0 \rdt{3,m} T_1,\; m\in \ckI_{3+}$.
  By Proposition ~\ref{prop:DCB+}, the above defined $T_k$ can be
  connected by a type 3 chain to the classical
  $T_{k+m}=T(a-k-m,b+k+m)$ whose spectral diagram has $\boxminus$
  labels at the eigenvalues $\lam_{3,j},\; j=0,\ldots, k+m$.  Using
  permutability, one can obtain a connection from $T_1$ to $T_{k+m}$
  using a chain whose eigenvalue set excludes $\lam_{3,m}$.  Hence,
  the $\lam_{3,m}$-label of $T_{1}$ matches the $\lam_{3,m}$ label of
  $T_{k+m}$.  Hence, the spectral diagrams of $T_0$ and $T_1$ differ
  by the flip $\boxpm \to \boxminus$.

  Finally, consider the type 4 RDT $T_0 \rdt{4,m} T_1,\; m\in
  \ckI_{4}$.  By assumption, $\ckI_{4-}=\emptyset$, so we can reuse
  the above argument to conclude that the spectral diagrams of $T_0$
  and $T_1$ differ by the flip $\boxpm \to \boxplus$.
\end{proof}

\begin{lem}
  \label{lem:1stepCB}
  The flips admitted by classical C and CB operators is given by
  ~\eqref{eq:Cflips} and ~\eqref{eq:CBflips}, respectively.
\end{lem}
\begin{proof}
  One applies gauge-covariance Proposition ~\ref{prop:covgaugexform}
  and re-uses the arguments in the proof of Lemma ~\ref{lem:1stepB}.
\end{proof}

\begin{lem}
  \label{lem:1stepD}
  The flips admitted by classical D operators is given by
  ~\eqref{eq:Dflips}.
\end{lem}
\begin{proof}
  Let $T_0=T(a,b),\; a,b\in \Nz$ be a classical type D operator.
  Proposition ~\ref{prop:classD} serves to prove that the spectral
  diagram of $T_0$ operator consists of the following labels:
  $\boxcirc,\boxfsq,\boxplus,\boxminus$.
  Without loss of generality, we assume that $a\ge b$, in which case,
  there are no $\boxminus$ labels.
  Thus, there are 6 possible RDTs to
  consider.  The claim is that the corresponding flips are as follows:
  \[
    1:\; \boxcirc \mapsto \boxfsq,\quad  2:\;
    \boxfsq \mapsto \boxcirc,\boxtdown,\quad
    3: \boxfsq \mapsto \boxminus, \boxplus\mapsto \boxfsq,\quad
    4: \boxfsq \mapsto \boxplus.
  \]
  The argument for the type 1,3,4 RDTs is more-or-less the same as in
  the proof to Lemma ~\ref{lem:1stepB}.  The argument for the type 2,
  non-para-Jacobi RDT utilizes similar reasoning, but is based on
  Proposition ~\ref{prop:DCD2}.  Finally, the para-Jacobi case is the
  subject of the following Lemma ~\ref{lem:pj1step}.
\end{proof}

\begin{lem}
  \label{lem:pj1step}
  A para-Jacobi RDT corresponds to the label transformation
  $\boxfsq\to \boxtdown$.
\end{lem}
\begin{proof}
  Without loss of generality, $a\ge b>0$. Fix a
  $k\in \{0,1,\ldots, b-1\}$ and let $E_k,\phi_2(x,t;k),\; $ be as
  defined to the proof of Lemma ~\ref{lem:Drat}.  The para-Jacobi case
  corresponds to using the generic
  $\phi_2(x,t;k)\in E_k,\; t\notin\{ 0,1\}$ as the factorization
  eigenfunction.  Set
  \[\tau(x,t;k) = (1-x)^a (1+x)^b \phi_2(x,t;k),\]
  fix a $t\neq 0,1$ and consider the para-Jacobi RDT
  $T_0 \rdt{2,l} T_1,\; l\in I_{2+}$ where $l=a+b-k-1$ and
  $T_1= \Trg(\tau;a-1,b-1)-a-b$. By Proposition ~\ref{prop:jacrdt},
  $\tau(x,t;k)^{-1}$ is a type 1 eigenfunction of $T_1$ with index
  $-k-1$ at the eigenvalue
   \[ \lam_{2}(l;a,b) = \lam_{1}(-k-1;a-1,b-1)-a-b= (k+1)(k-a-b).\] It
   remains to show that this eigenvalue is simple.

   This proof requires an argument based on indefinite norm
   anti-derivatives and Proposition ~\ref{prop:CDTinorm}.  Consider the
   RDT $T_0 \rdt{3,j} \tT_1$, where $j=a-k-1\in \ckI_{3+}$.  In the
   proof to Lemma ~\ref{lem:1stepD} we showed that the corresponding
   flip is $\boxfsq\to \boxminus$.  Hence, $\lam_{3,j}$ is a simple
   eigenvalue of $\tT_1$. It follows by Proposition
   ~\ref{prop:CDTinorm} that
   \[ \int \phi_{3,j}(x)^2 (1-x)^a (1+x)^b dx= \int
     P_{j}(x;-a,b)^2(1-x)^{-a} (1+x)^b dx,\quad j=a-b,\ldots, a-1\] is
   not quasi-rational.  Hence, the above integrand has a non-zero
   residue at $x=1$.  Since $\phi_{4,b-k-1}(x)$ is regular at $x=1$,
   the indefinite norm
   \[ \int \phi_2(x,t;k)^2(1-x)^a(1+x)^b dx \] is also not
   quasi-rational.  Therefore, $\lam_2(l;a,b)$ is a simple eigenvalue
   of $T_1$.
\end{proof}

\begin{definition}
  We say that an exceptional Jacobi operator is an $n$-step GABCD
  operator if it can be Darboux-connected to a classical operator of
  that class by an $n$-step RDT chain.
\end{definition}
\noindent
Note: at this point we are not yet asserting that an $n$-step GABCD
operator \emph{belongs} that class.  That will be demonstrated in
Lemma ~\ref{lem:jrdtclosed}.

\begin{lem}
  \label{lem:2step}
  The flip alphabets enjoyed by 1-step GABCD operators are as given in
  ~\eqref{eq:Gflips} - ~\eqref{eq:Dflips}.
\end{lem}
\begin{proof}
  Let $T_0$ be a classical GABCD operator and
  $T_0 \rdt{\lam_1} T_1 \rdt{\lam_2} T_2$ a 2-step Jacobi RDT. Suppose
  that $\lam_1\ne \lam_2$.  Then, the $\lam_2$-labels of $T_0$ and
  $T_1$ are the same.  Apply permutability and consider the chain
  $T_0\rdt{\lam_2} \tT_1 \rdt{\lam_1} T_2$.  The $\lam_2$-labels of
  $\tT_1$ and $T_2$ are the same.  Hence, the label transformation
  engendered by $T_1\rdt{\lam_2} T_2$ matches the label transformation
  engendered by $T_0\rdt{\lam_2} \tT_1$. By Lemmas ~\ref{lem:1stepG} -
  ~\ref{lem:1stepD}, the latter flip conforms with the possibilities in
  Proposition ~\ref{prop:lsfa}.

  Suppose then, that $\lam_1=\lam_2$.  If $\lam_1$ is a non-degenerate
  eigenvalue of $T_1$, then the second step is just the inverse of the
  first step, $T_2=T_0$ and the label transformation engendered by
  $T_1\rdt{\lam_2} T_2$ is the inverse of the label transformation
  engendered by $T_1\rdt{\lam_1} T_0$.  The flip alphabets in
  Proposition ~\ref{prop:lsfa} are closed under inverses, so the
  conclusion follows.

  It remains to consider the possibility that the 2-step chain is a
  CDT.  By Propositions ~\ref{prop:DCA}-~\ref{prop:DCD2}, $T_0$ is
  Darboux connected to a classical operator $\tT_n$ of the same
  degeneracy class by a chain
  \[ T_0 \rdt{\lam_1} T_1 \rdt{\lam_2} \tT_2 \cdots \tT_{n-1}
    \rdt{\lam_n} \tT_n \] where
  $\lam_1\notin \{ \lam_2,\ldots, \lam_n\}$.  By permutability, we can
  extend this chain by $\tT_n \rdt{\lam_1} \tT_{n+1}$ and use
  permutability again to obtain the chain
  \[ T_0 \rdt{\lam_1} T_1 \rdt{\lam_1} T_2 \rdt{\lam_2} T_3 \cdots
    T_{n-1} \rdt{\lam_{n-1}} T_n\rdt{\lam_n} T_{n+1} ,\] where
  $\tT_{n+1} = T_{n+1}$.  By Proposition ~\ref{prop:1flip}, the RDTs
  $\tT_{n}\rdt{\lam_1}\tT_{n+1}$ and $T_1\rdt{\lam_1} T_2$ engender
  the same label transformation.  Since $\tT_n$ is classical, that
  transformation belongs to one of the flip alphabets enumerated in
  Proposition ~\ref{prop:lsfa}.
\end{proof}

\begin{lem}
  \label{lem:classCDT} Let $T_0$ be a classical Jacobi operator and
  $T_0\rdt{\lam} T_1 \rdt{\lam} T_2$ a CDT.  The corresponding label
  transformations must be one of the following, or the inverse of one
  of the following:
  \begin{itemize}
  \item[(A)] $\boxcirc\mapsto \boxstar\mapsto \boxminus$
  \item[(B,CB)] $\boxplus\mapsto \boxpm\mapsto \boxminus$
  \item[(C,CB)] $\boxcirc\mapsto \boxotimes\mapsto \boxtimes$
  \item[(D)] $\boxcirc\mapsto \boxfsq \mapsto \boxtdown,\;
    \boxcirc\mapsto \boxfsq \mapsto \boxplus,\; \boxcirc\mapsto
    \boxfsq \mapsto \boxminus$.
  \end{itemize}
\end{lem}
\begin{proof}
  This follows by Lemma ~\ref{lem:2step} and by inspection of the label
  transformations displayed in Proposition ~\ref{prop:lsfa}.
\end{proof}

\begin{definition}
  \label{def:redchain}
  We say that a factorization chain
  $T_0 \rdt{\lam_1} T_1 \cdots T_{n-1} \rdt{\lam_n} T_n$ is
  \emph{reduced} if (i) $\lam_{k}\ne \lam_{k+j+1}$ for all $k,j\in \N$
  where defined, and (ii) in those cases where $\lam_k = \lam_{k+1}$,
  the subchain $T_{k-1}\rdt{\lam_k} T_k \rdt{\lam_{k+1}} T_{k+1}$ is a
  CDT.
\end{definition}

\begin{lem}
  \label{lem:redchain}
  If $T$ is Darboux-connected to a classical operator, it can be
  connected by a reduced factorization chain.
\end{lem}
\begin{proof}
  Consider a chain
  $T_0 \rdt{\lam_1} T_1 \cdots T_{n-1} \rdt{\lam_n} T_n$ of
  exceptional Jacobi operators with $T_0$ classical.  Suppose that one
  of the above eigenvalues is repeated. By permutability, WLOG assume
  that $\lam_1=\lam_2$.  If the initial 2-step subchain isn't a CDT
  then, $T_2=T_0$, and the chain can be shortened.  Suppose then that
  the initial 2-step subchain is a CDT, but that $\lam_1$ occurs at
  least 3 times in the eigenvalue set. Again by permutability, WLOG
  $\lam_1=\lam_2=\lam_3$.  By inspection of the list given in Lemma
  ~\ref{lem:classCDT}, $\lam_2$ is a simple eigenvalue of $T_2$, and
  hence $T_3=T_1$.  Again, the chain can be shortened.
\end{proof}


\begin{lem}
  \label{lem:nstep}
  Let $n\ge 2$. The $n$-step GABCD operators enjoy the flip alphabets
  displayed in ~\eqref{eq:Gflips} - ~\eqref{eq:Dflips}.
\end{lem}
\begin{proof}
  Consider a chain
  $T_0 \rdt{\lam_1} T_1 \rdt{\lam_2} T_2 \cdots T_{n-1} \rdt{\lam_{n}}
  T_n$ where $T_0$ is a classical Jacobi operator.  By Lemma
  ~\ref{lem:redchain}, we may assume that the chain is reduced. Hence,
  either $\lam_n\notin \{ \lam_1,\ldots, \lam_{n-1}\}$, or
  $\lam_n=\lam_j$ for a unique $j=1,\ldots, n-1$. In the first case,
  apply permutability to construct the chain
  \begin{equation}
    \label{eq:T0tT1}
    T_0 \rdt{\lam_n} \tT_1 \rdt{\lam_1} \tT_2 \cdots \tT_{n-1}
  \rdt{\lam_{n-1}} T_n.
  \end{equation}
  By Proposition ~\ref{prop:1flip}, the $\lam_n$-labels of $\tT_1$ and
  $T_n$ are the same.  For the same reason, the $\lam_n$-labels of
  $T_0$ and $T_{n-1}$ are the same.  Therefore $T_0\rdt{\lam_n} \tT_1$
  engenders the same label transformation as
  $T_{n-1} \rdt{\lam_n} T_n$.  The desired conclusion now follows by
  Lemmas ~\ref{lem:1stepG} - ~\ref{lem:pj1step}.

  Next, suppose that $\lam_n$ occurs twice in the eigenvalue set.
  Without loss of generality, assume that $\lam_1= \lam_n$. Again,
  apply permutability to construct the chain ~\eqref{eq:T0tT1}.  By the
  same argument as above, $\tT_1 \rdt{\lam_1} \tT_2$ engenders the
  same label transformation as $T_{n-1}\rdt{\lam_n} T_n$.  The desired
  conclusion now follows by Lemma ~\ref{lem:2step}.
\end{proof}

\begin{lem}
  \label{lem:jrdtclosed}
  The GABCD degeneracy classes are closed with respect to Jacobi RDTs.
\end{lem}
\begin{proof}
  Suppose that $T$ is an $n$-step GABCD operator for some $n\ge 1$.  We
  claim that, necessarily, $T$ belongs to the degeneracy class in
  question.  We also claim the converse, namely that every exceptional
  GABCD operator can be Darboux connected to a classical operator
  belonging to that class.  Both assertions are evident for type
  G,B,C,CB operators by inspection of ~\eqref{eq:himath} and by
  Definition ~\ref{def:degenclass}.  The conditions that define these
  degeneracy classes are invariant with respect to the $\pm 1$ shifts
  of the $\al,\be$ parameters engendered by Jacobi RDTs.

  Let $T_0=T(a,b),\; a\in \Nz,b\notin \Z$ be a classical type A
  operator, and let
  \begin{equation}
    \label{eq:T0Tnchain}
    T_0 \rdt{\lam_1} T_1 \rdt{\lam_2} T_2 \cdots T_{n-1} \rdt{\lam_n}
    T_n 
  \end{equation}
  be an $n$-step chain.  By Lemma ~\ref{lem:redchain}, without loss of
  generality that chain is reduced. By Proposition ~\ref{prop:classA},
  the spectral diagram of $T_0$ has exactly $a$ labels of type
  $\boxstar$.  Hence, by Lemma ~\ref{lem:nstep} the the total number of
  type 2 $\boxstar\mapsto \boxcirc$ and type 3
  $\boxstar \mapsto \boxminus$ flips engendered by this chain cannot
  exceed $a$.  Since these are the flips that shift $\al$ by $-1$,
  $T_n$ must have the form $T(\tau,\al,\be)+\epsilon$ where
  $0\le \al \le a$ is a non-negative integer. Therefore, $T_n$ is a
  class A exceptional operator.
  
  Conversely, consider a type A exceptional operator
  $\Trg(\tau;\al,\be),\; \al \in \Nz,\; \be \notin \Z$. By Theorem
  ~\ref{thm:XOPFT2}, there exists an $\epsilon\in \R$ such that
  $\hT:= T(\tau,\al,\be)+\epsilon$ is Darboux connected to a classical
  operator $T_0=T(a,b)$.  Since the $\al,\be$ shifts are $\pm 1$ for
  all RDTS, we must have $a\in \Z$ and $b\notin \Z$.  If $a\in \Nz$,
  we are done. Suppose then that $a\in \Zm$ is a negative integer.
  Hence, $T(-a,b)$ is a classical type A operator, which by
  gauge-covariance is Darboux connected to
  $\Trg(\tau;-\al,\be) + \epsilon$.  However, we just proved that the
  latter is a type A operator, with $0\le -\al \le -a$.  Hence
  $\al=0$, which means that $\Trg(\tau;\al,\be)+\epsilon$ can be Darboux
  connected to the type A classical operator $T(-a,b)$.

  Analogous reasoning serves to establish the desired conclusion for
  the D class.
\end{proof}

\section{Remaining Proofs}
\label{sec:proofs}

Having assembled the necessary definitions and key demonstrations we
now turn to the proofs of the propositions and theorems stated in the
introductory sections of this paper.  

\subsection{Proof of Theorem
~\ref{thm:tng}}
\label{subsec:proof:tng}

We begin by proving Theorem
~\ref{thm:tng}.  As was mentioned, a more limited version of this
result was proved in \cite{GFGM19}[Theorem 5.3].  Theorem
~\ref{thm:tng} is a stronger statement and our first task is to
show how this stronger statement follows from the results in \cite{GFGM19}

The demonstration is based on an operator form that generalizes both
the natural gauge form $\Tng$ ~\eqref{eq:Tngdef} and the rational gauge
form $\Trg$ ~\eqref{eq:Trgdef2}. For $\eta(x),\mu(x)\in\cP$  and
$\al,\be\in \R$, define
\begin{equation}
  \label{eq:Tgg}
  \Tgg(\eta,\mu;\al,\be)= \Tng(\eta;\al,\be)
  +  2(x^2-1)u'  + 2xu,\quad u= \mu'/\mu.
\end{equation}
\noindent
The following results were proved in \cite{GFGM19}[Lemmas 5.26 and
5.27].
\begin{lem}
  \label{lem:genform}
  Every exceptional Jacobi operator is gauge equivalent to an
  exceptional Jacobi operator (in the sense of Definition
  \ref{def:XT}) having the form $\Tgg(\eta,\mu;\al,\be)$, where
  $\eta(x),\mu(x)\in\cPo$.
\end{lem}

\begin{lem}
  Let $f(x)\in\cP$. Then,
  \begin{equation}
    \label{eq:ggxform}
     \MO{f} \circ \Tgg(\eta,\mu;\al,\be) \circ \MO{f^{-1}} = \Tgg(f
     \tau,f^{-1}\mu;a,b).
  \end{equation}
\end{lem}
\noindent
The natural gauge and rational gauge forms shown in ~\eqref{eq:Tngdef}
and ~\eqref{eq:Trgdef2} are extreme examples of the \emph{general gauge (gg)}
form ~\eqref{eq:Tgg}, with
\[ \Trg(\tau;\al,\be) = \Tgg(1,\tau^{-1};\al,\be),\quad
  \Tng(\tau;\al,\be) = \Tgg(\tau,1;\al,\be).\] However, as the following
identity demonstrates, natural gauge ~\eqref{eq:Tngdef} is not
uniquely determined when the parameters $\al,\be$ are integral.
\begin{prop}
  We have
  \begin{align}
    \Tng(\tau(1-x)^{\al+1};\al,\be) &= \Tng(\tau;-\al-2,\be)+ \be(\al+1),\\
    \Tng(\tau(1+x)^{\be+1};\al,\be) &= \Tng(\tau;\al,-\be-2)+ \al(\be+1)
  \end{align}
\end{prop}
\noindent
Therefore,  the natural form ~\eqref{eq:Tngdef} does not automatically imply
that $\tau(\pm 1) \ne 0$.  However, as we now show, the condition that
$\tau(\pm 1) \ne 0$ does not entail a loss of generality.
\begin{proof}[Proof of Theorem ~\ref{thm:tng}]
  By Lemma ~\ref{lem:genform} every exceptional Jacobi operator is
  gauge equivalent to some $\Tgg(\eta,\mu,\al,\be)$ where
  $\eta(\pm 1) \ne 0$ and $\mu(\pm 1 ) \ne 0$.  Hence, by
  ~\eqref{eq:ggxform},
  \begin{equation}
    \label{eq:ggxform2}
    \MO{\mu} \circ \Tgg(\eta,\mu;\al,\be) \circ \MO{\mu^{-1}} =
    \Tng(\mu \eta;\al,\be)  
  \end{equation}
  By inspection, $\tau(x) = \eta(x) \mu(x)$ does not vanish at
  $x=\pm 1$. Moreover, by ~\eqref{eq:ggxform2}, if $P(x)$ is an
  eigenpolynomial of $\Tgg(\eta,\mu,\al,\be)$, then $\mu(x) P(x)$ is
  an eigenpolynomial of $\Tng(\mu\eta;\al,\be)$.  Therefore, the
  latter is itself an exceptional operator.
\end{proof}

\subsection{Proofs of Propositions ~\ref{prop:piPtau}, ~\ref{prop:qreigen} and
~\ref{prop:sig1234}}
\label{subsec:proofs:qr}

Next, we prove Propositions ~\ref{prop:qreigen} and
~\ref{prop:sig1234}.  The former classifies the qr-eigenfunctions of an
exceptional Jacobi operator into four classes according to the asymptotic
type. The latter, classifies the possible qr-eigenvalues.  The proofs
are based on the following lemmas, which establish the fundamental
characterization of qr-eigenvalues and eigenfunctions.  Proposition
~\ref{prop:piPtau} then follows as a straightforward corollary.

\begin{lem}
  \label{lem:Reigenfun}
  Let $\Trg(\tau;\al,\be),\; \tau\in \cPo$ be an exceptional operator
  and $\phi(x)$ a qr-eigenfunction.  Let $\mu_1,\mu_2,\mu_3,\mu_4$ be
  as per Table ~\ref{tab:qreigen}. Then, necessarily
  $\phi(x)=\mu_\imath(x) \htau(x)/\tau(x)$ for some
  $\imath\in \{1,2,3,4\}$ and $\htau\in \cPo$.
\end{lem}
\begin{proof}
  Set $u= \tau'/\tau, w= \phi'/\phi$. By ~\eqref{eq:etadef},
  ~\eqref{eq:Tabdef}, ~\eqref{eq:Trgdef2}, ~\eqref{eq:Twlam} and
  ~\eqref{eq:Tabric} ,
  \begin{equation}
    \label{eq:ricR}
    (x^2-1)\lp w'+w^2\rp +(\al-\be+(2+\al+\be)x) w+ 2(x^2-1)u'+2 x u = \lambda. 
  \end{equation}
  By construction,
  \[ \deg  \big(2(x^2-1)u'+2 x u\big) \le 0.\]
  Hence, $\deg w<0$.  If not, then by inspection we would have
  \[ \deg \left( (x^2-1)\lp w'+w^2\rp +(\al-\be+(2+\al+\be)x) w
    \right)= 2+2\deg w >0,\] which would contradict ~\eqref{eq:ricR}.
  Also by construction, the poles of $2(x^2-1)u'+2xu$ have order 2 at
  most.  This implies that $w(x)$ has at most simple poles.  If $w(x)$
  were to have a non-simple pole, then by inspection,
  $(x^2-1)\lp w'+w^2\rp +(\al-\be+(2+\al+\be)x) w$ would have a pole
  of order 3 or higher.  This would, once again, contradict
  ~\eqref{eq:ricR}.
  
  Thus, $\phi(x) = (1-x)^{a}(1+x)^{b} \htau(x)/\tau(x)$, where $a,b$ are
  the residues of $w(x)$ at $x=+1,-1$, respectively, and where $\htau(x)$
  is a qr-function such that $\htau'(x)/\htau(x)$ is regular at $x=\pm 1$.
  Recall that, by assumption, $u(x)$ has no poles at $x=\pm 1$.
  Hence, by a straight-forward calculation,
  \[ \underset{x=1}{\Res}\;
    (x^2-1)\lp w'+w^2\rp +(\al-\be+(2+\al+\be)x) w
    =-2a(a+\al) = 0.\]
  Hence, either $a=0$ or
  $a=-\al$.  Similarly, either $b=0$ or $b=-\be$. Therefore $(1-x)^a(1+x)^b =
  \mu_\imath(x)$ for some $\imath\in \{1,2,3,4\}$.
  
  By ~\eqref{eq:Trgdef}, $\xi(x):=(1-x)^{a}(1+x)^{b}\htau(x)$ is a qr-
  eigenfunction of $\Tng(\tau;\al,\be)$.  In \cite{GFGM19} (see Remark
  5.22 and Lemma 5.16 of that reference) it was shown that such 
  $\xi(x)$ is regular at every $x\ne \pm 1$.  Hence, $\htau(x)$ is
  everywhere regular.  Since $\deg w<0$ we must have
  $\deg \htau'/\htau<0$ also.  Therefore, $\htau(x)$ is a polynomial
  that does not vanish at $x=\pm 1$.
\end{proof}

\begin{lem}
  \label{lem:rateval}
  Let $\pi_i(x),\; i\in I_1(\Trg)$ be a quasi-polynomial eigenfunction
  of degree $i$ of an exceptional $\Trg(\tau;\al,\be)$. Then,
  necessarily the corresponding eigenvalue has the form
  $\lambda=i(i+\al+\be+1)$.
\end{lem}
\begin{proof}
  Let $\cQu{k},\; k\in \Z$ be the subspace of rational functions of
  degree less or equal than $k$.  Set $w=\pi_i'/\pi_i,\; u= \tau'/\tau$ and
  $N=\deg\tau$.  Observe that
  \begin{align*}
    ((x^2-1)u'+ x u)
    &\equiv
      \frac{x^2\tau''}{\tau}-\lp \frac{x\tau'}{\tau}\rp^2+
      \frac{x\tau'}{\tau} \mod \cQu{-1}  \\
    &\equiv N(N-1)- N^2+N \equiv  0 \mod \cQu{-1}
  \end{align*}
  As above, ~\eqref{eq:ricR} holds by assumption.  Working mod
  $\cQu{-1}$, a direct calculation shows that
  \begin{align*}
    (x^2-1)\lp w'+w^2 +\eta(x;\al,\be)\,   w\rp
    &\equiv  \frac{x^2\pi_i''}{\pi_i}
      +\lp \al+\be+2\rp
      \frac{x\pi_i'}{\pi_i}\\
    &\equiv i(i-1)  + (\al+\be+2)i = \lambda
  \end{align*}
  The desired conclusion follows immediately.
\end{proof}

\begin{proof}[Proof of Proposition ~\ref{prop:qreigen}]
  The first assertion is a consequence of Lemma ~\ref{lem:Reigenfun}.
  The case $\imath=1$ of ~\eqref{eq:lamdegphi} is covered by Lemma
  ~\ref{lem:rateval}.  We turn to the case of $\imath=2$. Consider a
  type 2 eigenfunction
  $\phi_{2,k}(x) = \mu_2(x) \pi_{2,k}(x),\; k\in I_2$. By Proposition
  ~\ref{prop:gaugesym}, $\pi_{2,k}$ is an eigenfunction of
  $\Trg(\tau;-\al,-\be)$ with
  \[ \Trg(\tau;-\al,-\be)\pi_{2,k} - (\al+\be)\pi_{2,k} = \lambda
    \pi_{2,k}.\] Hence, by Lemma ~\ref{lem:rateval},
  \[ \lambda = (-k-1)(-k-1-\al-\be+1) - \al - \be =
    k(k+\al+\be+1)=d (d+1+\al+\be), \]
  where
  \[ d= \deg\phi=-k-1-\al-\be.\]
  The cases
  $\imath=3,4$ are argued similarly.
\end{proof}

\begin{proof}[Proof of Proposition ~\ref{prop:sig1234}]
  Consider a type 2 eigenfunction,
  $\phi_{2,k}= \mu_2 \pi_{2,k},\; k\in I_2$.  By definition
  ~\eqref{eq:k1234} and by ~\eqref{eq:Tgaugesym}, $\pi_{2,k}(x)$ is a
  rational eigenfunction of $\Trg(\tau;-\al,-\be)$ with
  $\deg\pi_{2,k} = k$.  Recall that
  $\deg\phi_{2,k} = k-\al-\be$. Hence, the corresponding eigenvalue
  is
  \[ \lam =(k-\al-\be)(k+1) = \lam_{2}(k;\al,\be).\] 

  Similarly, if $\phi_{3,j},\; j\in I_3$ is a type 3 eigenfunction,
  then by the convention of ~\eqref{eq:k1234} and ~\eqref{eq:phi1234},
  $\phi_{3,j} = \mu_3 \pi_{3,j}$ where the latter is a rational function
  of degree $j$.  Hence, $d=\deg\phi_{3,j} = j-\al$.  It follows that
  \[ d(d+\al+\be+1) = (j-\al)(j+\be+1) = \lam_{3}(j;\al,\be).\]
  Similarly, if  $\phi_{4,j} = \mu_4(x) \pi_{j},\; j\in I_4$ is a
  type 4 eigenfunction, then  $d=\deg\phi_{4,j}=j-\be$.  Hence, in
  this case also,
  \[ d(d+\al+\be+1) =(j-\be)(j+\al+1)= \lam_{4}(j;\al,\be).\]
\end{proof}

\begin{proof}[Proof of Proposition ~\ref{prop:piPtau}]
  The first assertion follows by the more general Proposition
  ~\ref{prop:qreigen}.  We turn to the second
  assertion.  Suppose that $\al,\be\notin \Zm$.  Let $P(x)$ be an
  eigenpolynomial of $\Tng(\tau;\al,\be)$.  We claim that $P(\pm 1) \ne 0$.  By
  ~\eqref{eq:Trgdef}, $\xi(x) = P(x)/\tau(x)$ is a rational
  eigenfunction of $\Trg(\tau;\al,\be)$.  By Proposition ~\ref{prop:qreigen},
  $\xi(x) = \mu_\imath(x) \tilde{P}(x)/\tau(x)$ for some
  $\imath\in \{1,2,3,4\}$ and  $\tilde{P}(x)\in\cPo$.  Since $\al,\be\notin \Zm$ we must have
  $\imath = 1$.  Since $\mu_1(x) =1$, we conclude that $P=\tilde{P}$.
  Hence, $I_1(\Trg) \supseteq I(\Tng) -N$.  Again, by Proposition
  ~\ref{prop:qreigen}, an quasi-polynomial eigenfunction of $\Trg$ must
  have the form $\xi(x) = P(x)/\tau(x)$ where $P(x)\in\cPo$.  Hence,
  $I_1(\Trg) \subseteq I(\Tng) -N$.
\end{proof}

\subsection{Proofs of  Proposition ~\ref{prop:sigq} and Theorem ~\ref{thm:flip}}
\label{subsec:proofs:spec.diag}
Finally, we prove two results concerning the qr-spectrum and   the spectral diagram of exceptional Jacobi operators. 
\begin{lem}
  Setting $s=(i+j)/2$, we have
  \begin{equation}
    \label{eq:lamshifta}
    \lam_1(z;\al,\be) - \lam_1(z-s;\al+i,\be+j) = \lam_1(s;\al,\be).
  \end{equation}
\end{lem}
\begin{proof}
  Applying the definition ~\eqref{eq:lam1234} and calculating, gives
  \begin{equation}
    \label{eq:lamshift}
    \begin{aligned}
    &\lam_1(z;\al,\be) - \lam_1(z-s;\al+i,\be+j) = 
    z(z+\al+\be+1) - (z-s)(z+s+\al+\be+1)\\
    &\quad = s(s+\al+\be+1) = \lam_1(s;\al,\be)
  \end{aligned}
  \end{equation}

\end{proof}

\begin{proof}[Proof of Proposition ~\ref{prop:sigq}]
  The claim is true for classical Jacobi operators by Proposition
  ~\ref{prop:sigqclass}.  By Theorem ~\ref{thm:XOPFT2}, it suffices to
  show that if ~\eqref{eq:sigqTalbe} is true for a given exceptional
  operator $T = \Trg(\tau;\al,\be)$, then it is true for the
  exceptional operators related to $T$ by a 1-step RDT.

  We prove this for the case of a type 1 RDT; the other cases are
  argued similarly.  By ~\eqref{eq:lamshift}, 
  the set $\sigq(\al,\be)$,
  as defined in ~\eqref{eq:sigqdef} enjoys the following transformation
  law:
  \begin{equation}
    \label{eq:sigqxform}
    \sigq(\al,\be)=\sigq(\al+i,\be+j)+\lam_1(s;\al,\be),\qquad i,j\in
    \Z,\; s=  \frac{i+j}{2} ,
  \end{equation}
  The spectral shifts defined in ~\eqref{eq:himath} and utilized in
  ~\eqref{eq:hTidef} match ~\eqref{eq:sigqxform}.  By ~\eqref{eq:himath}
  and by Proposition ~\ref{prop:1flip},
  \[ \sigq(\Trg(\tau;\al,\be)) = \sigq(\Trg(\htau_{1k};\al+1,\be+1))+
    \lam_1(1;\al,\be).\] 
  Therefore,
  $\sigq(\Trg(\htau_{1k};\al+1,\be+1)) = \sigq(\al+1,\be+1)$.
\end{proof}

\begin{proof}[Proof of Theorem ~\ref{thm:flip}]
  The argument is essentially the same as the induction proof given
  above, with Proposition ~\ref{prop:1flip} providing the inductive step.
\end{proof}

\begin{lem}
  \label{lem:JRDTnorm}
  Consider a Jacobi RDT
  $T_0\rdt{\imath,k} T_1,\; \imath\in \{1,2,3,4\}$ with factorization
  function $\phi_{\imath,k},\; k\in I_\imath(T_0)$ of degree
  $j=\deg\phi_\imk = \deg \mu_\imath + k$, intertwiner $A_\imath$ as
  per ~\eqref{eq:A1234}, and $\hal =\hal_\imath,\hbe=\hbe_\imath$ as
  per ~\eqref{eq:himath}.  For $i\in I_0(T_0)\setminus \{ j \}$, set
  \[ \hpi_{\hi} = \frac{A_\imath \pi_i}{i-j}, \] where
  $\hi= i-1,i+1,i,i$ respectively, according to whether
  $\imath=1,2,3,4$.  Then, $\hpi_\hi$ is a rational eigenfunction of
  $T_1$ satisfying
  \begin{gather}
    \deg \hpi_{\hi} = \hi,\quad
    \LC(\hpi_{\hi})=\LC(\pi_i),\\
    \label{eq:RDT1norm}
    (i-j)\hpi_{\hi}(x)^2
    (1-x)^{\hal}(1+x)^{\hbe} -
    (i-j^\star)\pi_i(x)^2(1-x)^{\al}(1+x)^{\be}
    =  \\ \nonumber
    \qquad D_x \lb
    \hpi_{\hi}(x)\pi_i(x)
    \,(1-x)^{\tal}(1+x)^{\tbe}  \rb,
  \end{gather}
  where
  $\tal = \max \{ \al,\hal\},\; \tbe = \max\{ \al,\hal \},\; j^\star =
  -j-1-\al-\be$.
\end{lem}
\begin{proof}
  We give the proof for the case $\imath=1$, where $j=k$. The other three
  cases are proved analogously.  By ~\eqref{eq:A1234}, 
  \[ A_\imath \pi_i = \pi_k^{-1}\Wr[\pi_{k},\pi_i ].\] The first
  assertions now follow 
  by ~\eqref{eq:degW} ~\eqref{eq:LCWronsk} of Lemma ~\ref{lem:lc}.  We
  now prove ~\eqref{eq:RDT1norm}.  Dropping some indices and
  simplifying the notation in ~\eqref{eq:A1234} ~\eqref{eq:hphidef},
  write
  \begin{align*}
    &\pi(x) = \pi_i(x),\; \hpi(x) = \hpi_{\hi}(x),\\
    &\phi(x)    =    \pi_{k}(x),\;   \hphi(x) =
      \pi_k(x)^{-1}(1-x)^{-\al-1}(1+x)^{-\be-1},\\
    &A = A_{1k},\; \hA= \hA_{2\hk},
  \end{align*}
  and recall that
  \begin{align*}
    T_0
    &= \hA A + k(k+1+\al+\be),\\
    A[y]
    &= \phi D_x\lb y \phi^{-1} \rb,\\
    \hA[y]
    &= (x^2-1)\hphi D_x\lb y\hphi^{-1}\rb,\\
    \hA[\hpi]
    &= (i(i+1+\al+\be)-k(k+1+\al+\be))/(i-k) \pi = (i-k^\star)\pi.
  \end{align*}
  Hence, by direct calculation,
  \begin{align*}
    &D_x  \lb \hpi(x) \pi(x) (1-x)^{\al+1} (1+x)^{\be+1} \rb\\
    &\quad = D_x  \lb \hpi(x) \pi(x) (\hphi(x)\phi(x))^{-1} \rb\\
    &\quad =\pi(x) \phi(x)^{-1} D_x  \lb \hpi(x) \hphi(x)^{-1} \rb 
      + \hpi(x) \hphi(x)^{-1}D_x \lb \pi(x)\phi(x)^{-1} \rb\\ 
    &\quad = \lp\pi(x)\hphi(x) D_x \lb \hpi(x)
      \hphi(x)^{-1} 
      \rb +   \hpi(x)\phi(x) D_x\lb \pi(x)\phi(x)^{-1} \rb  \rp
      (1-x)^{\al+1} (1+x)^{\be+1} \\
    &\quad = -\pi(x) \hA[\hpi(x)] (1-x)^\al (1+x)^\be + A[\pi(x)]
      \hpi(x) (1-x)^{\al+1}(1+x)^{\be+1}\\
    &\quad = 
      -(i-k^\star)\pi(x)^2 (1-x)^\al (1+x)^\be + (i-k) \hpi(x)^2
      (1-x)^{\al+1}(1+x)^{\be+1}.
  \end{align*}
\end{proof}


\subsection{Proofs of results on the generic class}
\label{subsec:proofs:generic}
  
\begin{proof}[Proof of Proposition ~\ref{prop:TG}]
  Since $a,b,a\pm b\notin \Z$, the eigenvalue set
  $\{ \lam_1(k;a,b) : k\in \tK \}$ consists of $p_1+p_3+p_4$ distinct
  eigenvalues.  Hence, we can apply Propositions ~\ref{prop:rdtchain}
  and ~\ref{prop:jacrdt} to construct an RDT chain connecting the
  classical $T(a,b)$ to the exceptional
  $\Trg(\tau;\al,\be)+p_1(1+a+b+p_1)$.

  The formula ~\eqref{eq:pig} for the corresponding Jacobi
  quasi-polynomials is justified by Proposition ~\ref{prop:Awronsk},
  Lemma ~\ref{lem:lc}, and Proposition ~\ref{prop:indexflip}.
  Proposition ~\ref{prop:Awronsk} accounts for the ratio of the
  Wronskians, while Proposition ~\ref{prop:indexflip} accounts for the
  degree shift between $\pi_i$ and $\ckpi_{i+p_1}$.  The factorization
  gauge at each step of the chain is determined by ~\eqref{eq:A1234}.
  The factor $(x-1)^{p_3} (x+1)^{p_4}$ in ~\eqref{eq:pig} is just a
  product of the factorization gauge functions in the chain: the type
  3 Jacobi RDTs utilize a gauge factor of $b(x)=x-1$, the type 4
  Jacobi RDTs utilize a gauge factor of $b(x) = x+1$, and the type 1
  Jacobi RDTs utilize a gauge factor of $b(x) =1$.  Lemma ~\ref{lem:lc}
  accounts for the form of the normalizing factor
  $\prod_{k\in \tK} (i+p_1-k)$.  The presence of this product in the
  denominator ensures that $\pi_i$ is monic.

  Next, we justify the form of $\tau(x)$ given in  ~\eqref{eq:tauGB}.
  Having established the validity of ~\eqref{eq:pig}, Proposition
  ~\ref{prop:qreigen} and, in particular, formula ~\eqref{eq:phimutau}
  indicate that the $\tau(x)$ in question is the polynomial part of
  the Wronskian $\Wr[\phi(K_1,K_3,K_4;a,b)]$.  Write
  \[ \Wr[\phi(K_1,K_3,K_4;a,b) ] = (1-x)^A (1+x)^B \tau(x) ,\] where
  $\tau(x)$ is a polynomial with no zeros at $x=\pm 1$.  The forms of
  $A$ and $B$ shown in ~\eqref{eq:tauGB} are a consequence of Lemma
  ~\ref{lem:Wmult}.  The degree formula ~\eqref{eq:degG} is a
  consequence of Lemma ~\ref{lem:lc} and of the definition
  ~\eqref{eq:tauGB}. The form of $\al,\be$ in ~\eqref{eq:albeG} is
  justified by repeated application of Proposition ~\ref{prop:jacrdt}
  and the table ~\eqref{eq:himath}.  Basically, the RDT chain that
  connects $T(a,b)$ and $\Trg(\tau;\al,\be)$ contains $p_1$ type 1
  RDTs, $p_3$ type 3 RDTS and $p_4$ type 4 RDTS.  Thus,
  \[ (\al,\be)-(a,b) = p_1(+1,+1) +p_3(-1,1) + p_4(1,-1).\] The form
  of the index sets ~\eqref{eq:indexGB} is a consequence of Proposition
  ~\ref{prop:indexflip} and of the type G flip alphabet
  ~\eqref{eq:Gflips} in Proposition ~\ref{prop:lsfa}.  Since the RDT
  chain contains $p_1$ type 1 steps and since all the other steps are
  type 3 and type 4, the relationship between $I_{12}$, the combined
  index set of $\Trg(\tau;\al,\be)$ and $\ckI_{12}$, the combined
  index set of $T(a,b)$ is simply $I_{12} = \ckI_{12} - p_1$.  We have
  equality because the type G qr-eigenvalues are all simple.  This is
  a consequence of Proposition ~\ref{prop:lsfa}.  A further consequence
  of Proposition ~\ref{prop:lsfa} is that the labels at the eigenvalues
  $\lam_1(k;a,b),\; k\in K_1$ undergo the flips
  $\boxcirc\mapsto \boxtimes$. Hence, when we break up $I_{12}$ into
  $I_1$ and $-I_2-1$ we have to (i) remove the indices $k,\; k\in K_1$
  from $\ckI_1$, (ii) add the indices $-k,\; k\in K_1$ to $\ckI_2$ and
  then shift the former by $-p_1$ and the latter by $+p_1$.  Similar
  reasoning serves to justify the forms of $I_3$ and $I_4$.

  By ~\eqref{eq:phi12W} of  Lemma ~\ref{lem:forthog}, we have
  \[ \int \pi_i(x) \pi_j(x) (1-x)^{\al} (1+x)^\be\in
    \qQ_{\al+1,\be+1} ,\quad i\ne j \in I_1(T_0).\] Let
  $\kappa(z;K,a,b)$ be as in ~\eqref{eq:kappaG}. Repeated application
  of Lemma ~\ref{lem:JRDTnorm} shows that the anti-derivative
  \[ \int\pi_i(x)^2(1-x)^\al (1+x)^\be - \kappa(i;K,a,b)
    \ckpi_{i+p_1}(x)^2 (1-x)^a (1+x)^b,\quad i\in I_1(T) \] is
  quasi-rational.  By Proposition ~\ref{prop:sfclass} and by
  ~\eqref{eq:sfid1}, the anti-derivatives
  \[ \int \lp\ckpi_{i+p_1}(x)^2  - \frac{\nu(i+p_1;a,b)}{\nu(0;a,b)}
    \rp (1-x)^a (1+x)^b,\quad \int
    \frac{(1-x)^\al(1+x)^\be}{\nu(0;\al,\be)} - 
    \frac{(1-x)^a (1+x)^b}{\nu(0;a,b)}  \]
  are also quasi-rational. Hence,
  \[ \int \lp\pi_{i}(x)^2 -
    \kappa(i;K,a,b)\frac{\nu(i;a,b)}{\nu(0;\al,\be)} \rp (1-x )^\al
    (1+x)^\be \] is quasi-rational. Since $\al,\be\notin \Z$, the
  above anti-derivative must be an element of $\qQ_{\al+1,\be+1}$ by
  Lemma ~\ref{lem:adqr}.  This justifies the form of the norm formula
  given in ~\eqref{eq:nuiG} and  ~\eqref{eq:kappaG}.
\end{proof}

\begin{proof}[Proof of Proposition ~\ref{prop:GSD}]
  By Theorem ~\ref{thm:flip}, Proposition ~\ref{prop:classG} and
  Proposition ~\ref{prop:lsfa}, the spectral diagram of a type G
  exceptional operator has the representation described in Remark
  ~\ref{rem:G}.  The same remark indicates how to uniquely recover the
  parameters $K_1,K_3,a,b$ from the spectral diagram.  The fact that
  $\tau(x)$ must have the form shown in ~\eqref{eq:tauGB} is justified
  above, in the proof of Proposition ~\ref{prop:TG}.
\end{proof}

\subsection{Proofs of results on the semi-degenerated classes}
\label{subsec:proofs:semideg}
\begin{lem}
  \label{lem:rhoA}
  Let $T=\Trg(\tau;\al,\be),\; \al\in\Nz,\be \notin \Zm$ be a class A
  exceptional operator with $\pi_i(x)=\pi_{1,i}(x),\; i\in I_1(T)$ the
  corresponding quasi-polynomial eigenfunctions, relative to some
  choice of normalization.  If $i\ne j$, then the incomplete
  inner-products
  \begin{equation}
    \label{eq:rhodefA}
    \rho_{ij}(x) := \int \pi_i(x)\pi_j(x) (1-x)^\al (1+x)^\be,\quad
    i,j\in I_1(T),
  \end{equation}
  define quasi-rational functions in $ \qQ_{\al+1,\be+1}$.
\end{lem}
\begin{proof}
  This is a direct consequence of ~\eqref{eq:phi12W} of Lemma
  ~\ref{lem:forthog}.
\end{proof}

\begin{lem}
  \label{lem:Anorm1}
  Let $T,\pi_i(x),\; i\in I_1(T)$, be as above and
  $\rho_{ij}(x),\; i,j\in I_1(T)$ be as defined in ~\eqref{eq:rhodefA}.
  Suppose that the indefinite norms $\rho_{ii}(x),\; i\in I_1(T)$ are
  quasi-rational functions.  Then, the corresponding norms are given
  by $\nu_i = \rho_{ii}(1),\; i\in I_1(T)$.
\end{lem}
\begin{proof}
  For each $i\in I_1(T)$,  define the quasi-rational anti-derivatives
  \[ \hrho_{ii}(x) := \int \lp\pi_i(x)^2 -
    \frac{\rho_{ii}(1)}{\nu(\al,\be)} \rp (1-x)^\al(1+x)^\be dx . \]
  Such anti-derivatives are uniquely defined because $\rho_{ii}(x)$ is
  assumed to be quasi-rational, and because $\al\in \Nz$.  By
  Proposition ~\ref{prop:rhoA} and by construction $\hrho_{ii}(1)=0$.
  Hence, by Lemma ~\ref{lem:adqr}, $\hrho_{ii}\in \qQ_{\al+1,\be+1}$.
  Therefore, following the criterion shown in ~\eqref{eq:nuidef},
  $\nu_i = \rho_{ii}(1)$, as was to be shown.
\end{proof}

\begin{lem}
  \label{lem:CDTA}
  Let $T,\pi_i(x),\; i\in I_1(T),\; \rho_{ij}(x),\; i,j\in I_1(T)$ be
  as in Lemma ~\ref{lem:rhoA}, and suppose that the indefinite norms
  $\rho_{ii}(x),\; i\in I_1(T)$ are quasi-rational functions such that
  $\rho_{ii}(1)\ne 0$.  Fix an $\ell\in I_1(T)$ and set
  \begin{align}
    \ttau(x)
    &= \tau(x) \pi_{\ell}(x),\\
    \label{eq:htaudefA}
    \htau(x)
    &= \tau(x)\rho_{\ell\ell}(x)(1+x)^{-\be-1}  ,\\
    \tT
    &=    \Trg(\ttau,\al+1,\be+1)+\al+\be+2,\\
    \hT
    &=    \Trg(\htau,\al,\be+2)+\al+\be+2.
  \end{align}
  Then, $\ttau(x), \htau(x)$ are polynomials and
  $T\rdt{\lam} \tT\rdt{\lam} \hT,\; \lam = \lam_1(\ell;\al,\be)$ is a
  type A CDT, as described in Lemma ~\ref{lem:classCDT}.  Moreover,
  \begin{equation}
    \label{eq:recdeghtauA}
    \deg\htau = \deg\tau + 2\ell +\al.
  \end{equation}
\end{lem}
\begin{proof}
  Let $\tT$ be defined as above.  By Proposition ~\ref{prop:jacrdt},
  $T\rdt{\lam}\tT,\; \lam= \lam_1(\ell;\al,\be)$ is a type 1 RDT.  By
  inspection, $\deg\ttau = \deg\tau + \ell$.  Hence by
  ~\eqref{eq:hphidef}, the seed eigenfunction for the inverse
  transformation $\tT\rdt{\lam} T$ is given by
  \begin{gather*}
    \tphi_{2}(x)=\frac{\tau(x)}{\ttau(x)}
    (1-x)^{-\al-1}(1+x)^{-\be-1},\\
    \lam = \lam_2(-\ell;\al+1,\be+1)+\al+\be+2    
  \end{gather*}
  By assumption, $\rho_{\ell\ell}(x)\in \qQ_{0,\be+1}$ is
  quasi-rational. Hence, By Proposition ~\ref{prop:CDTinorm} and since
  $\rho_{\ell\ell}(1)\ne 0$, the quasi-rational
  $\tphi_{3}:= \rho_{\ell\ell}(x) \tphi_{2}(x)$ is a type 3
  eigenfunction of $\tT$.  By ~\eqref{eq:rhodefA},
   \[  \deg \rho_{\ell\ell} = 2\ell+\al+\be+1.\]
   Hence, ~\eqref{eq:recdeghtauA} holds, and
   \[ \deg\phi_3= \ell-1,\quad \lam =
     \lam_3(\ell-1+\al;\al+1,\be+1)+\al+\be+2. \] It also follows that
   $\tau(x)/\ttau(x) \rho_{\ell\ell}(x) (1+x)^{-\be-1}$ is a
   quasi-polynomial.  Hence, by Proposition ~\ref{prop:qreigen},
   $\htau$ is the polynomial numerator of that rational function. Let
   $\hT$ be as defined above.  By Proposition ~\ref{prop:jacrdt},
   $\tT\rdt{\lam} \hT$ is a type 3 RDT, as was to be shown.
\end{proof}

\begin{lem}
  \label{lem:hpiA}
  Let $T,\pi_i(x),\; i\in I_1(T),\; \rho_{ij}(x),\; i,j\in I_1(T)$ be
  as in Lemma ~\ref{lem:rhoA}, and suppose that the indefinite norms
  $\rho_{ii}(x),\; i\in I_1(T)$ are quasi-rational functions such that
  $\rho_{ii}(1)\ne 0$. Fix an $\ell\in I_1(T)$ and let
  $\ttau(x),\htau(x),\hT$ be as in Lemma ~\ref{lem:CDTA}.  Set
  \begin{align}
    \label{eq:hpidefA}
    \hPi_{i}(x)
    &=
      (1+x)^{-1}\lp \pi_{i}(x) - \frac{\rho_{i\ell}(x)}{\rho_{\ell\ell}(x)}
      \pi_\ell(x)\rp ,\quad i\in I_1(T)\setminus\{\ell\},\\
    \label{eq:hrhodefA}
    \hRho_{ij}(x)
    &= \rho_{ij}-
      \frac{\rho_{i\ell}(x)\rho_{j\ell}(x)}{\rho_{\ell\ell}(x)},\quad i,j\in
      I_1(T)\setminus \{ \ell \}.
  \end{align}
 The degrees and leading coefficients of these
  eigenfunctions are given by
  \begin{equation}
    \label{eq:degLChpiA}
    \deg \hPi_{i}   =     i-1, \qquad
    \LC(\hPi_{i})
    =      \lp \frac{i-\ell}{i-\ell^\star}\rp \LC(\pi_{i}),\quad
    i\in I_1(T)\setminus \{ \ell \}, 
  \end{equation}
  where $\ell^\star = -\ell -1-\al-\be$.
  The $\hRho_{ij}(x)$ are the corresponding incomplete inner products;
  i.e.,
  \begin{equation}
    \label{eq:RhoPiIntA}
    \hRho_{ij}(x) = \int \hPi_{i}(x) \hPi_{j}(x)(1-x)^\al
    (1+x)^{\be+2},\quad i,j\in I_1(T)\setminus \{ \ell \}.
  \end{equation}
  These quasi-rational functions satisfy
  $ \hRho_{ij}(1)=\rho_{ij}(1)$.  Moreover,
  $I_1(\hT) = (I_1(T)\setminus \{ \ell\})-1$ and
  \begin{equation}
    \label{eq:hThpiA}
    \hT \hPi_{i+1} = \lam_1(i;\al,\be+2) \hPi_{i+1},\quad i\in
    I_1(\hT). 
  \end{equation}
\end{lem}
\begin{proof}
  Let $T,\tT,\hT,\tau(x),\ttau(x),\htau(x),\lam$ be as in the
  preceding lemma.  Set
  \begin{align*}
    w_1 &= \frac{\ttau'}{\ttau}-\frac{\tau'}{\tau},
    &\hw_1   &=-  \frac{\ttau'}{\ttau}+\frac{\tau'}{\tau}-\frac{\al+1}{x-1}
               - \frac{\be+1}{x+1} ,\\ 
    w_2 &= \frac{\htau'}{\htau}-\frac{\ttau'}{\ttau}- \frac{\al+1}{x-1}
    & \hw_2&=- \frac{\htau'}{\htau}+\frac{\ttau'}{\ttau}-\frac{\be+2}{x+1},\\
    A_1 &=  D_x - w_1,
    &B_1 &= (x^2-1)(D_x-\hw_1),\\
    A_2 &= (x-1)(D_x-w_2),
    & B_2 &= (x+1)(D_x-\hw_2).
  \end{align*}
  Then, by Proposition ~\ref{prop:jacrdt},
  \begin{align*}
    T&= B_1A_1+\lam,
    & \tT&=A_1B_1+\lam,\\
    \tT&= B_2 A_2 + \lam,
         & \hT&= A_2B_2+\lam. 
  \end{align*}
  By ~\eqref{eq:htaudefA},
  \[ \frac{\htau'}{\htau} = \frac{\tau'}{\tau}- \frac{\be+1}{x+1}+
    \frac{\rho_{\ell\ell}'}{\rho_{\ell\ell}} .\]
  Hence, 
  \begin{align*}
    A_2
    &= (x+1)^{-1}\lp  B_1 -
      \frac{\rho_{\ell\ell}'}{\rho_{\ell\ell}}\rp,\\
    A_2 A_1
    &= (x+1)^{-1}\lp B_1 A_1 - 
      \frac{\rho_{\ell\ell}'}{\rho_{\ell\ell}} A_1\rp 
    = (x+1)^{-1}\lp T - \lam- 
      \frac{\rho_{\ell\ell}'}{\rho_{\ell\ell}} A_1\rp\\
    A_2 A_1 \pi_i
    &= (x+1)^{-1} \lp (\lam_1(i;\al,\be)-\lam_1(\ell;\al,\be))\pi_i-
      \frac{\pi_\ell }{\rho_{\ell\ell}}
      \Wr[\pi_\ell,\pi_i](1-x)^\al (1+x)^\be \rp \quad \text{by
      ~\eqref{eq:rhodefA}},\\ 
    &= (x+1)^{-1} (i-\ell)(i+\ell+\al+\be+1)\lp \pi_i-
      \frac{\rho_{i\ell}}{\rho_{\ell\ell}} \pi_\ell
 \rp  \quad \text{by
      ~\eqref{eq:phi12W}}\\
    &=  (i-\ell)(i+\ell+\al+\be+1)\hPi_{i}.
  \end{align*}
  Eigenvalue relation ~\eqref{eq:hThpiA} follows by
  ~\eqref{eq:intrel}. Relations ~\eqref{eq:degLChpiA} follow by Lemma
  ~\ref{lem:JRDTnorm}.  Relation ~\eqref{eq:RhoPiIntA} follows by the
  following calculation:
    \begin{align*}
    (1-x)^{-\al}(1+x)^{-\be}\hRho_{ij}'
    &= (1-x)^{-\al}(1+x)^{-\be}\lp\rho_{ij}' - \frac{\rho_{i\ell}'
      \rho_{j\ell}}{\rho_{\ell\ell}}
      - \frac{\rho_{i\ell}'      \rho_{j\ell}}{\rho_{\ell\ell}}
      + \frac{\rho_{i\ell}
      \rho_{j\ell}\rho_{\ell\ell}'}{\rho_{\ell\ell}^2}  \rp   \\
    &= \pi_i \pi_j - \frac{\pi_i \pi_\ell
      \rho_{j\ell}}{\rho_{\ell\ell}}
      - \frac{\pi_j \pi_\ell      \rho_{i\ell}}{\rho_{\ell\ell}}
      +\frac{\pi_\ell\pi_\ell
      \rho_{i\ell}\rho_{j\ell}}{\rho_{\ell\ell}^2}\\
    &= \lp \pi_i - \frac{ \pi_\ell
      \rho_{i\ell}}{\rho_{\ell\ell}} \rp
      \lp \pi_j - \frac{ \pi_\ell
      \rho_{j\ell}}{\rho_{\ell\ell}} \rp\\
    &= (1+x)^{2}\hPi_i \hPi_j.
  \end{align*}
  The last assertion follows because, by Lemma ~\ref{lem:rhoA}, we have
  $\rho_{il}(1) = 0$ if $i\ne \ell$.
\end{proof}
\begin{lem}
  \label{lem:CDTAq}
  The $\htau(x)=\htau(x;L,a,b)$, as defined in ~\eqref{eq:htauA}, is a
  polynomial with
  \begin{equation}
    \label{eq:deghtauA}
    \deg \htau = 2\sum_{\ell\in L} \ell - q(q-1-a).
  \end{equation}
  The $\hpi_i(x)=\hpi_{i}(x;L,a,b),\; i\in \hI_1$ defined in
  ~\eqref{eq:hpiA} are monic rational functions of degree $i$ that satisfy
  \begin{equation}
    \label{eq:eigenvalA}
    \Trg(\htau;a,b+2q)\hpi_i = \lam_1(i;a,b+2q) \hpi_i,\quad i\in
    \hI_1.
  \end{equation}
  The incomplete inner-products
  \[ \hrho_{ij}(x) = \int \hpi_i(x) \hpi_j(x)(1-x)^a(1+x)^{b+2q},\quad
    i,j\in \hI_1, \] define quasi-rational functions, that satisfy
  \begin{equation}
    \label{eq:hrho1A}
     \hrho_{ij}(1) = \delta_{ij} \hchi(i+q;L,a,b)^2 \nu(i+q;a,b),\quad
    i,j\in \hI_1,
  \end{equation}
  with $\hchi$ defined as in ~\eqref{eq:hchiA}.
\end{lem}
\begin{proof}
  Enumerate $L$ as $\ell_1,\ldots, \ell_q$ and 
  inductively define
  \begin{align*}
    \tT_s
    &=    \Trg(\ttau_s,a+1,b+2s-1)+s(a+b+s+1),\quad s=1,\ldots, q,\\
    \hT_s
    &=    \Trg(\htau_s,a,b+2s)+s(a+b+s+1),
  \end{align*}
  with $\ttau_s(x), \htau_s(x)$ inductively defined using
  ~\eqref{eq:htaudefA} and ~\eqref{eq:hpidefA} ~\eqref{eq:hrhodefA}.  For
  the initial step, set $\hT_0 := T(a,b)$, let
  $\pi_{0,i}(x) = \pi_i(x;a,b)$ be the classical Jacobi polynomials
  and let
  \[ \hRho_{0;ij}(x) = \int \pi_{0,i}(x)\pi_{0,j}(x) (1-x)^a
    (1+x)^b,\quad i\in \Nz \] be the classical incomplete inner
  products.  We have $\hRho_{0;ij}(1)=0,\; i\ne j$ by Lemma
  ~\ref{lem:forthog}.  The corresponding indefinite norms
  $\hRho_{0;ii}(x)$ are quasi-rational by inspection.  We then use
  ~\eqref{eq:hpidefA} to inductively construct the quasi-polynomial
  eigenfunctions $\hPi_{s;i}(x),\; i \in \Nz \setminus \{
  \ell_1,\ldots, \ell_s\}$ and
  use ~\eqref{eq:hrhodefA}, to
  inductively construct the corresponding incomplete inner products
  $\hRho_{s;i,j}(x)$.  We then adjust the degrees, and normalize these
  quasi-polynomials to monic form by setting
  \begin{align*}
    \hpi_{s;i} &= \prod_{j=1}^s\lp \frac{ i-\ell_j^*}{i-\ell_j}  \rp 
    \hPi_{i+s},\\
    \hrho_{s;ij} &= \prod_{j=1}^s\lp \frac{ i-\ell_j^*}{i-\ell_j}
    \rp^2  \hRho_{s;i+s,j+s},\quad s=1,\ldots, q.
  \end{align*}
  Relation ~\eqref{eq:hrho1A} follows immediately.  Relation
  ~\eqref{eq:deghtauA} follows by ~\eqref{eq:recdeghtauA} and by
  induction.  Arrange these operators into the $2q$-step chain of type
  1,3 CDTs
  \[ \hT_{s-1} \rdt{\lam_s} \tT_s \rdt{\lam_s} \hT_s,\quad \lam_s =
    \lam_1(\ell_s;a,b),\; s=1,\ldots, q,\] that connects the classical
  $\hT_0$ with the exceptional $\hT_{q}$.  Relation
  ~\eqref{eq:eigenvalA} follows by the preceding Lemmas and by
  induction.

  Observe that the definition of $\htau(x)$ in ~\eqref{eq:htaudefA} and
  $\hpi_{i}(x)$ in ~\eqref{eq:hpidefA} is just the $q=1$ case of the
  definitions in ~\eqref{eq:htauA} and ~\eqref{eq:hpiA}.  The general
  determinantal formula ~\eqref{eq:htauA} for $\htau(x)$ and
  ~\eqref{eq:hpiA} for $\hpi_i(x)$ are justified by Sylvester's
  identity.  The full argument is related in Section 3 of
  \cite{GFGM21}.
\end{proof}

\begin{lem}
  \label{lem:Wronskpi}
  Let $\hT= \Trg(\htau,\hal,\hbe)$ be an exceptional operator with
  monic quasi-polynomial eigenfunctions $\hpi_i(x),\; i\in \hI_1$.
  Let $k_1,\ldots, k_p\in \hI_1$ be distinct.  Then,
  \[ \tau = \htau \Wr[\hpi_{k_1},\ldots, \hpi_{k_p}] \]
  is a polynomial such that
  \[ \deg \tau = \deg\htau +\sum_i k_i - \binom{p}{2} \]
  and such that the leading coefficient has the form
  \[\LC(\tau) = \prod_{1\le i<j<p} (k_j-k_i).\]
\end{lem}
\begin{proof}
  The proof of the first assertion is by induction.  The base case
  $p=1$ follows by Proposition ~\ref{prop:qreigen}.  Consider the type
  1 chain
  $T_0\rdt{\lam_1} T_1 \rdt{\lam_2} \cdots \rdt{\lam_{p-1}} T_{p-1}
  \rdt{\lam_p} T_p$ where $\lam_j = \lam_1(k_j;\al,\be)$, where
  $T_j=\Trg(\tau_j;\hal+j,\hbe+j) + \epsilon_j,\; j=1,\ldots, p$.  By
  Proposition ~\ref{prop:Awronsk} and by ~\eqref{eq:A1234}, the rational
  function
  \[ \tpi_{j}(x) := \frac{\Wr[\hpi_{k_1},\ldots, \hpi_{k_{j+1}}]}{
      \Wr[\hpi_{k_1},\ldots, \hpi_{k_{j}}]} =
    \frac{\tau_{j+1}}{\tau_j},\quad j=1,\ldots, p-1 \] 
  is an eigenfunction of $T_j$ whose numerator and denominator are the
  $\tau$-functions of $T_{j+1}$ and $T_j$, respectively.  Having
  established that $\tau_j=  \htau \Wr[\hpi_{k_1},\ldots,
  \hpi_{k_j}]$, it then follows that
  \[ \frac{\tau_{j+1}}{\tau_j} = \frac{\htau \Wr[\hpi_{k_1},\ldots,
      \hpi_{k_{j+1}}]}{\tau_j}.\] This is the inductive step of the
  argument.  The assertions regarding degree and the leading
  coefficient follow by Lemma ~\ref{lem:lc}.
\end{proof}
\begin{proof}[Proof of Proposition ~\ref{prop:TA}]
  The classical $T(0,b)$ is Darboux-connected to the exceptional
  $\hT= \Trg(\htau;0,b+2q)+q(q+b+1)$ by a $2q$-step chain consisting
  of type A CDTs as described in Lemma ~\ref{lem:CDTAq}.  The form of
  $\htau$ in ~\eqref{eq:htauA} and the quasi-polynomial eigenfunctions
  $\hpi_i(x),\; i\in \hI_1$ is justified by that Lemma. This
  intermediate $\hT$ is then Darboux connected to
  $T=\Trg(\tau;\al,\be) + (p+q)(p+q+b+1)$ by a $p$-step type 1 chain.
  The formula for the quasi-polynomial eigenfunctions
  $\pi_i(x),\; i\in I_1$ of $T$ given in ~\eqref{eq:piA} is justified
  by Proposition ~\ref{prop:Awronsk}.  The formula for $\tau$ in
  ~\eqref{eq:tauA} follows by Lemma ~\ref{lem:Wronskpi}.  The same Lemma
  justifies the degree formula ~\eqref{eq:degtauA} and the form of the
  normalization coefficient in ~\eqref{eq:chizK}.  Let
  $\hrho_{ij}(x), \rho_{ij}(x)$ be the incomplete inner products
  corresponding to the quasi-polynomial eigenfunctions of $\hT$ and
  $T$ respectively.  By ~\eqref{eq:hrho1A},
  \[ \hrho_{ii}(1) = \hchi(i+q;L,a,b)^2\nu(i+q;a,b),\quad i\in I_1(\hT).\]
  Hence, by ~\eqref{eq:RDT1norm} of Lemma ~\ref{lem:JRDTnorm},
  \[
    \begin{aligned}
      \rho_{ii}(1)
      &= \prod_{k\in K} \lp\frac{i+p+q-k^*}{i+p+q-k}\rp
      \hchi(i+p+q;L,a,b)^2 \nu(i+p+q;a,b),\quad i\in I_1(T),\\
      &= 
      \kappa(i+p+q;K,L,a,b) \nu(i+p+q;a,b).
    \end{aligned}
  \]
  Define $\nu_i,\; i\in I_1$ as per ~\eqref{eq:tnuA}.  Hence, by
  Proposition ~\ref{prop:classqr} and Proposition ~\ref{prop:sfclass},
  the anti-derivative
  \[ \int \lp \pi_i(x)^2 - \frac{\nu_i}{\nu(\al,\be)} \rp
    (1-x)^{\al}(1+x)^{\be},\quad \al = a+p,\; \be+b+p+2q, \] defines a
  quasi-rational function in $\qQ_{0,\be+1}$ that vanishes at
  $x=1$. In order to obtain the $(1+x)^{\al}$ factor the asymptotic
  behaviour of the antiderivative at $x=1$ must be
  $C+O\lp(1-x)^{\al+1}\rp,$ where $C$ is a constant.  Since this
  anti-derivative vanishes at $x=1$, we must have $C=0$.  Hence the
  above anti-derivative is an element of $\qQ_{\al+1,\be+1}$, and
  hence $\nu_i$, as defined above, are the norms of the $\pi_i$.
\end{proof}

\begin{proof}[Proof of Proposition ~\ref{prop:ASD}]
  By Theorem ~\ref{thm:XOPFT2}, by Proposition ~\ref{prop:DCA} and by
  Proposition ~\ref{prop:lsfa}, the given $T$ is Darboux-connected to a
  classical $T(0,b),\; b\notin \Z$.  Hence, by Proposition
  ~\ref{prop:lsfa}, $\Lam_T$ is a finite modification of $\Lam(0,b)$ of
  the type described in Remark ~\ref{rem:A}.  The construction and
  uniqueness of the parameters $K,L,b$ is spelled out in that same
  remark.  The form of $\tau$ then follows by the argument detailed
  about in the proof of Proposition ~\ref{prop:TA}.
\end{proof}

\begin{proof}[Proof of Proposition ~\ref{prop:TB}]
  See the proof of Proposition ~\ref{prop:TG}.
\end{proof}
\begin{proof}[Proof of Proposition ~\ref{prop:BSD}]
  By Theorem ~\ref{thm:XOPFT2}, by Propositions ~\ref{prop:DCB+}
  ~\ref{prop:DCB-} and by Proposition ~\ref{prop:lsfa}, the given $T$ is
  Darboux-connected to a classical
  $T(a,b),\; a,b\notin \Z,\; a-b\in \{-1,0,1\}$.  Hence, by
  Proposition ~\ref{prop:lsfa}, $\Lam_T$ is a finite modification of
  the corresponding $\Lam(a,b)$, as described in Remark ~\ref{rem:A}.
  The construction and uniqueness's of the parameters $K,a,b$ is
  spelled out in that same remark.  The form of $\tau$ then follows by
  the argument detailed about in the proof of Proposition
  ~\ref{prop:TG}.
\end{proof}

\begin{lem}
  \label{lem:JRDTnorm2}
  Let $T= \Trg(\tau;\al,\be),\; \al,\be\notin \Z,\al+\be+1 \in \Z$ be
  a class C or CB exceptional operator and consider the type 2 RDT
  $T\rdt{2,k} \hT$ where $k\in I_{2+}(T)$ and
  $\hT=\Trg(\pi_{2,k};\al-1,\be-1)-\al-\be$.  Let $A_{2k}$ be the
  corresponding intertwiner as defined in ~\eqref{eq:A1234}, and
  \[ \hpi_{\hi} = \frac{A_2 \pi_{i}}{i-j},\quad j=k-\al-\be,\; i=
    j^\star= -j-1-\al-\be=-k-1,\;    \] the indicated monic
  eigenpolynomial of $\hT$. Then, $\hpi_{\hi}$ has zero norm, in the
  sense that
  \[ \int \hpi_{\hi}(x)^2 (1-x)^{\al+1}(1+x)^{\be+1} \] defines a
  quasi-rational function.
\end{lem}
\begin{proof}
  This follows by ~\eqref{eq:RDT1norm} of Lemma ~\ref{lem:JRDTnorm},
  because in
  this case,  $i-j^\star=0$.
\end{proof}

\begin{proof}[Proof of Proposition ~\ref{prop:TCB}]
  Since $a,b\notin \Z$, we have $\tK_{12} \cap \tK_{34} = \emptyset$.
  Moreover, by the assumption that $\tK_{12}$ and $\tK_{34}$ are
  disjoint unions, the eigenvalue set $\{ \lam_1(k;a,b) : k\in \tK \}$
  consists of $p_1+p_2+p_3+p_4$ distinct eigenvalues.  Hence, we can
  apply Propositions ~\ref{prop:rdtchain} and ~\ref{prop:jacrdt} to
  construct an RDT chain connecting the classical $T(a,b)$ to the
  exceptional $\Trg(\tau;\al,\be)+(p_1-p_2)(1+a+b+p_1-p_2)$.  In the
  class C case, the chain consists of $p_1$ type 1 RDTs, $p_2$ type 2
  RDTs, and $p_3$ type 3 RDTS.  By Proposition ~\ref{prop:lsfa}, the
  corresponding flips are
  $\boxotimes \to \boxtimes, \boxotimes\to \boxcirc, \boxplus \to
  \boxminus$, respectively. In the most general CB case, the chain
  consists of $p_1$ type 1 RDTs, $p_2$ type 2 RDTs, $p_3$ type 3 RDTS,
  and $p_4$ type 4 RDTS. The corresponding flips are
  $\boxotimes \to \boxtimes, \boxotimes\to \boxcirc, \boxpm \to
  \boxminus, \boxpm \to \boxplus$, respectively.   The form of $\al,\be$ in
  ~\eqref{eq:albeG} is justified by repeated application of Proposition
  ~\ref{prop:jacrdt} and the table ~\eqref{eq:himath}.  In a nutshell:
  \[ (\al,\be)-(a,b) = p_1(+1,+1) +p_2(-1,-1)+p_3(-1,1) + p_4(1,-1).\]
  By Proposition ~\ref{prop:indexflip} and because the target labels of
  the above flips are all simple, the relationship between $I_{12}$,
  the combined index set of $\Trg(\tau;\al,\be)$ and $\ckI_{12}$, the
  combined index set of $T(a,b)$ is
  $I_{12} \supset \ckI_{12} - p_1+p_2$.  The relation isn't a simple
  shift because the flip $\boxotimes \to \boxtimes$ destroys a
  $\boxtimes$ label and moves a $\boxcirc$ label from $I_{1+}$ to
  $I_{1-}$.   This justifies the forms of $I_{1-}$ and $I_{1+}$ in
  ~\eqref{eq:indexCB}. The other index transformations detailed in
  ~\eqref{eq:indexCB} are justified in a similar manner.
  
  The formula ~\eqref{eq:piC} for the corresponding Jacobi
  quasi-polynomials is justified by Proposition ~\ref{prop:Awronsk},
  Lemma ~\ref{lem:lc}, and Proposition ~\ref{prop:indexflip}.
  Proposition ~\ref{prop:Awronsk} accounts for the ratio of the
  Wronskians, while Proposition ~\ref{prop:indexflip} accounts for the
  degree shift between $\pi_i$ and $\ckpi_{i+p_1-p_2}$.  The factorization
  gauge at each step of the chain is determined by ~\eqref{eq:A1234}.
  The factor $(x-1)^{p_2+p_3} (x+1)^{p_2+p_4}$ in ~\eqref{eq:pig} is just a
  product of the factorization gauge functions in the chain: the type
  3 Jacobi RDTs utilize a gauge factor of $b(x)=x-1$, the type 4
  Jacobi RDTs utilize a gauge factor of $b(x) = x+1$, the type 2 RDT
  has a gauge factor of $b(x)=(x-1)(x+1)$, and the type 1
  Jacobi RDTs utilize a gauge factor of $b(x) =1$.  Lemma ~\ref{lem:lc}
  accounts for the form of the normalizing factor
  $\prod_{k\in \tK} (i+p_1-p_2-k)$.  The presence of this product in the
  denominator ensures that $\pi_i$ is monic.  Observe that
  \[  (i+p_1-p_2)^* = -i-1-p_1+p_2-a-b = -i-1 - \al -\be +p_1-p_2
    = i^\star + p_1 + p_2.\]
  Hence,
 \[ \fu(i+p_1-p_2) =
   \begin{cases}
     i+p_1-p_2  & \text{ if }i\in I_{1+},\\
     i^\star+p_1-p_2 & \text{ if } i \in I_{1-}.
   \end{cases}\] This justifies the $\fu(i+p_1-p_2)$ index in the
 Wronskian formula ~\eqref{eq:piC}. The mapping
 $\ckpi_{(i+p_1-p_2)^*} \to \pi_i,\; i\in \I_{1-}$ describe the
 state-adding transformations $\boxotimes \to \boxcirc$ indexed by $K_2$.

  Next, we justify the form of $\tau(x)$ given in  ~\eqref{eq:tauC}.
  Having established the validity of ~\eqref{eq:piC}, Proposition
  ~\ref{prop:qreigen} and, in particular, formula ~\eqref{eq:phimutau}
  indicate that the $\tau(x)$ in question is the polynomial part of
  the Wronskian $\Wr[\phi(K_1,K_2,K_3,K_4;a,b)]$.  Write
  \[ \Wr[\phi(K_1,K_2,K_3,K_4;a,b) ] = (1-x)^A (1+x)^B \tau(x) ,\] where
  $\tau(x)$ is a polynomial with no zeros at $x=\pm 1$.  The forms of
  $A$ and $B$ shown in ~\eqref{eq:tauC} are a consequence of Lemma
  ~\ref{lem:Wmult}.  The degree formula ~\eqref{eq:degG} is a
  consequence of Lemma ~\ref{lem:lc}.   

  By ~\eqref{eq:phi12W} of  Lemma ~\ref{lem:forthog}, we have
  \[ \int \pi_i(x) \pi_j(x) (1-x)^{\al} (1+x)^\be\in \qQ_{\al+1,\be+1}
    ,\quad i\ne j \in I_1(T_0).\] Let $\kappa(z;K,a,b)$ be as in
  ~\eqref{eq:kappaG}. For $i\in I_{1+}$, repeated application of Lemma
  ~\ref{lem:JRDTnorm} shows that the anti-derivative
  \[ \int\pi_i(x)^2(1-x)^\al (1+x)^\be - \kappa(i+p_1-p_2;K,a,b)
    \ckpi_{i+p_1-p_2}(x)^2 (1-x)^a (1+x)^b \] is quasi-rational.  The
  desired conclusion now follows by Proposition ~\ref{prop:sfclass2}.
  If $i\in I_{1-}$, then the desired conclusion follows by Lemma
  ~\ref{lem:JRDTnorm2} and repeated application of Lemma
  ~\ref{lem:JRDTnorm}.
\end{proof}
\begin{proof}[Proof of Propositions ~\ref{prop:CSD} and ~\ref{prop:CBSD}]
  By Theorem ~\ref{thm:XOPFT2}, by Propositions ~\ref{prop:DCB+}
  ~\ref{prop:DCB-} and by Proposition ~\ref{prop:lsfa}, the given $T$ is
  Darboux-connected to a classical
  $T(a,b),\; a,b\notin \Z,\; a-b\in \{-1,0,1\}$.  Hence, by
  Proposition ~\ref{prop:lsfa}, $\Lam_T$ is a finite modification of
  the corresponding $\Lam(a,b)$, as described in Remark ~\ref{rem:A}.
  The construction and uniqueness of the parameters $K,a,b$ is
  spelled out in that same remark.  The form of $\tau$ then follows by
  the argument detailed about in the proof of Proposition
  ~\ref{prop:TG}.
\end{proof}

\subsection{Proofs of results on the degenerate class}
\label{subsec:proofs:deg}
\begin{lem}
  \label{lem:rhoD}
  Let $T=\Trg(\tau;\al,\be),\; \al,\be\in\Nz$ be a class D exceptional
  operator with $\pi_i(x),\; i\in I_1(T)$ the corresponding
  quasi-polynomial eigenfunctions, relative to some choice of
  normalization.  Then, for $i\ne j$, the incomplete inner products
  \begin{equation}
    \label{eq:rhointD}
    \rho_{ij}(x) := \int_{-1}^x \pi_i(x)\pi_j(x) (1-x)^\al (1+x)^\be,\quad
    i,j\in I_1(T),
  \end{equation}
  are quasi-rational functions that vanish at $x=-1$.  Indeed,
  \begin{equation}
    \label{eq:rhoijD}
    \rho_{ij}(x) = \frac{\Wr_x[\pi_i(x), \pi_j(x)] (1-x)^{\al+1}
      (1+x)^{\be+1}}{(j-i)(i+j+\al+\be+1)}\in \qQ_{\al+1,\be+1}.    
  \end{equation}
\end{lem}
\begin{proof}
  This follows by ~\eqref{eq:phi12W} of Lemma ~\ref{lem:forthog}.
\end{proof}

\begin{lem}
  \label{lem:Dnorm}
  Let $T,\pi_i(x),\; i\in I_1(T),\; \rho_{ij}(x),\; i,j\in I_1(T)$ be
  as in Lemma ~\ref{lem:rhoD}, and suppose that the indefinite norms
  $\rho_{ii}(x),\; i\in I_1(T)$ define quasi-rational functions that
  vanish at $x=-1$.  Then, the corresponding norms are given by
  $\nu_i = \rho_{ii}(1),\; i\in I_1(T)$.
\end{lem}
\begin{proof}
  See the proof of Lemma ~\ref{lem:Anorm1}.
\end{proof}

\begin{lem}
  \label{lem:CDTD}
  Let $T,\pi_i(x),\; i\in I_1(T)$ be as in Lemma ~\ref{lem:rhoD} and
  $\rho_{ij}(x),\; i,j\in I_1(T)$ as in ~\eqref{eq:rhointD}.  Suppose
  that the indefinite norms $\rho_{ii}(x),\; i\in I_1(T)$ define
  rational functions that vanish at $x=-1$. 
  Fix an $\ell\in I_1(T)$ and set
  \begin{equation}
    \label{eq:htaudefD}
    \begin{aligned}
    \ttau(x)
    &= \tau(x) \pi_{\ell}(x),\\
    \tT
    &=    \Trg(\ttau;\al+1,\be+1)+\al+\be+2,\\
    \htau(x)
    &= \tau(x)(t+\rho_{\ell\ell}(x)),\quad t\notin \{0,-\rho_{\ell\ell}(1)\}\\
    \hT
    &=    \Trg(\htau;\al,\be).
  \end{aligned}
  \end{equation}
  Then, $\ttau(x), \htau(x)$ are polynomials and
  $T\rdt{1,\ell} \tT\rdt{1,\ell} \hT$ is a type D CDT with label
  transformations $\boxcirc\to\boxfsq\to \boxtdown$, as described in
  Lemma ~\ref{lem:classCDT}.  Moreover,
  \begin{equation}
    \label{eq:recdeghtauD}
    \deg\htau = \deg\tau + 2\ell+\al+\be+1.
  \end{equation}
\end{lem}
\begin{proof}
  By Proposition ~\ref{prop:jacrdt} and by ~\eqref{eq:phi1234},
  $T\rdt{\lam}\tT,\; \lam= \lam_1(\ell;\al,\be)$ is a type 1 RDT.  By
  inspection, $\deg\ttau = \deg\tau + \ell$.   By
  ~\eqref{eq:hphidef}, the inverse seed eigenfunction is
  \begin{gather*}
    \tphi_{2-}(x)=\frac{\tau(x)}{\ttau(x)}
    (1-x)^{-\al-1}(1+x)^{-\be-1},
  \end{gather*}
  By assumption, by Lemma ~\ref{lem:adqr}, and by the restrictions
  imposed on $t$, we have $t+\rho_{\ell\ell}(\pm 1)\ne 0$. Hence, by
  Proposition ~\ref{prop:CDTinorm}, the rational
  $\tphi_{2+}:=(t+ \rho_{\ell\ell}(x)) \tphi_{2-}(x)$ is also a type 2
  eigenfunction of $\tT$.  By construction,
  $ \deg \rho_{\ell\ell} = 2\ell+\al+\be+1$. Hence,
  $\deg\tphi_{2+}=\ell-1$; this implies ~\eqref{eq:recdeghtauD}.  By
  Proposition ~\ref{prop:jacrdt}, $\ttau/\htau$ is a
  quasi-polynomial eigenfunction of $\hT$ with eigenvalue
  $\lam$. Since $\deg (\ttau/\htau) = -\ell-\al-\be-1 = \ell^\star$,
  the type 2 RDT $\tT\rdt{\lam} \hT$ corresponds to the label
  transformation $\boxfsq \to \boxtdown$, as was to be shown.
\end{proof}

\begin{lem}
  \label{lem:hpiD}
  Let
  $T,\pi_i,\rho_{ij},\ell,\hT$ be as in Lemma ~\ref{lem:CDTD}. Set
  \begin{align}
    \label{eq:hpidefD}
    \hPi_{i}(x;t)
    &=
      \pi_{i}(x) -
      \frac{\rho_{i\ell}(x)}{t+\rho_{\ell\ell}(x)}
      \pi_\ell(x),\quad i 
      \in I_1(T)\\
    \label{eq:hrhodefD}
    \hRho_{ij}(x;t)
      &= \rho_{ij}-
        \frac{\rho_{i\ell}(x)\rho_{j\ell}(x)}{t+\rho_{\ell\ell}(x)},\quad i,j\in I_1(T)
  \end{align}
  The degrees and leading coefficients of the former are given by
  \begin{align}
    \label{eq:hpiiD}
    \deg \hPi_i
    & =     i,\; i\ne \ell,&    \LC(\hPi_i)
    &=      \lp \frac{i-\ell}{i-\ell^\star} \rp \LC(\pi_i),\; i\ne \ell,\;\\
    \label{eq:hpiellD}
    \deg \hPi_\ell
    &= \ell^\star& \LC(\hPi_\ell)
    &=    (-1)^\al t (\ell-\ell^\star) \LC(\pi_\ell),
  \end{align}
  where $\ell^\star=-\ell-1-\al-\be$.
  The index set of $\hT$ is
  $I_1(\hT) = (I_1(T)\setminus \{ \ell \})\cup \{ \ell^\star \},$ with
  \begin{equation}
    \label{eq:hThpiD}
    \hT \hPi_{i} = \lam_1(i;\al,\be) \hPi_{i},\quad i\in
    I_1(\hT). 
  \end{equation}
  Moreover,
  \begin{align}
    \label{eq:hPipinorm}
    \hPi_i(-1;t) &= \pi_i(-1),\quad i\in I_1,\\
    \label{eq:hrhointD}
      \hRho_{ij}(x;t)
      &= \int_{-1}^x \hPi_{i}(x;t) \hPi_{j}(x;t)(1-x)^\al
      (1+x)^{\be},\quad i,j\in I_1(\hT),\\
    \label{eq:hrho1}
    \hRho_{ij}(1;t) 
      &=    \begin{cases}
      \frac{t\rho_{\ell\ell}(1)}{t+\rho_{\ell\ell}(1)} & \text{ if } i=j=\ell,\\
      \rho_{ij}(1) & \text{otherwise}.
    \end{cases}
  \end{align}
\end{lem}
\begin{proof}
  Let $T,\tT,\hT,\tau(x),\ttau(x),\htau(x),\lam$ be as in the preceding lemma.
  \begin{align*}
    w_1 &= \frac{\ttau'}{\ttau}-\frac{\tau'}{\tau},
    &\hw_1   &=-  \frac{\ttau'}{\ttau}+\frac{\tau'}{\tau}-\frac{\al+1}{x-1}
               - \frac{\be+1}{x+1}, \\ 
    w_2 &= \frac{\htau'}{\htau}-\frac{\ttau'}{\ttau}-
          \frac{\al+1}{x-1}- \frac{\be+1}{x+1} ,
    & \hw_2&=- \frac{\htau'}{\htau}+\frac{\ttau'}{\ttau},\\
    A_1 &=  D_x - w_1,
    &B_1 &= (x^2-1)(D_x-\hw_1),\\
    A_2 &= (x^2-1)(D_x-w_2),
    & B_2 &= D_x-\hw_2 .
  \end{align*}
  Then, by Proposition ~\ref{prop:jacrdt},
  \begin{align*}
    T&= B_1A_1+\lam,
    & \tT&=A_1B_1+\lam,\\
    \tT&= B_2 A_2 + \lam,
         & \hT&= A_2B_2+\lam .
  \end{align*}
  By ~\eqref{eq:htaudefD},
  $\displaystyle \frac{\htau'}{\htau} = \frac{\tau'}{\tau}+
  \frac{\rho_{\ell\ell}'}{t+\rho_{\ell\ell}}$.  Hence,
  \begin{align*}
    A_2
    &=   B_1 -
      \frac{\rho_{\ell\ell}'}{t+\rho_{\ell\ell}},\\
    A_2 A_1
    &=  B_1 A_1 - 
      \frac{\rho_{\ell\ell}'}{t+\rho_{\ell\ell}} A_1 
    =  T - \lam- 
      \frac{\rho_{\ell\ell}'}{t+\rho_{\ell\ell}} A_1\\
    \intertext{If $i\ne \ell$, then}
    A_2 A_1 \pi_i
    &=  (\lam_1(i;\al,\be)-\lam_1(\ell;\al,\be))\pi_i-
      \frac{\pi_\ell }{t+\rho_{\ell\ell}}
      \Wr[\pi_\ell,\pi_i](1-x)^\al (1+x)^\be  \quad \text{by
      ~\eqref{eq:rhointD}},\\ 
    &= (i-\ell)(i+\ell+\al+\be+1)\lp \pi_i-
      \frac{\rho_{i\ell}}{t+\rho_{\ell\ell}} \pi_\ell
      \rp  \quad \text{by
      ~\eqref{eq:phi12W}}\\
    &=  (i-\ell)(i+\ell+\al+\be+1)\hpi_{i}.
  \end{align*}
  If $i\ne \ell$, then eigenvalue relation ~\eqref{eq:hThpiD} follows by
  ~\eqref{eq:intrel}.  For $i = \ell$, we have
  \begin{equation}
    \label{eq:hpielldef}
    \hPi_\ell = \frac{t \pi_\ell}{t+\rho_{\ell\ell}}.
  \end{equation}
  In this case, ~\eqref{eq:hThpiD} holds by Proposition
  ~\ref{prop:jacrdt}, and because $\hPi_\ell'/\hPi_\ell = \hw_2$, as
  defined above.  Relations ~\eqref{eq:hpiiD} follow by ~\eqref{eq:degW}
  and ~\eqref{eq:LCWronsk} of Lemma ~\ref{lem:lc}.  Relations
  ~\eqref{eq:hpiellD} follow by inspection of ~\eqref{eq:hpielldef} and
  because $\deg \rho_{\ell\ell} = \ell-\ell^\star$ and because
  $\LC(\rho_{\ell\ell}) = (-1)^\al/(\ell-\ell^\star)$.
  Relation ~\eqref{eq:hrhointD} follows from the definition
  ~\eqref{eq:hpidefD} because, by assumption $\rho_{i\ell}(-1)=0$.  by
  the following calculation\footnote{The derivative is with respect to
    the $x$ variable.}
  \begin{align*}
    (1-x)^{-\al}(1+x)^{-\be}\hRho_{ij}'
    &= (1-x)^{-\al}(1+x)^{-\be}\lp\rho_{ij}' - \frac{\rho_{i\ell}'
      \rho_{j\ell}}{t+\rho_{\ell\ell}}
      - \frac{\rho_{i\ell}'      \rho_{j\ell}}{t+\rho_{\ell\ell}}
      + \frac{\rho_{i\ell}
      \rho_{j\ell}\rho_{\ell\ell}'}{(t+\rho_{\ell\ell})^2}  \rp   \\
    &= \pi_i \pi_j - \frac{\pi_i \pi_\ell
      \rho_{j\ell}}{t+\rho_{\ell\ell}}
      - \frac{\pi_j \pi_\ell      \rho_{i\ell}}{t+\rho_{\ell\ell}}
      +\frac{\pi_\ell\pi_\ell
      \rho_{i\ell}\rho_{j\ell}}{(t+\rho_{\ell\ell})^2}\\
    &= \lp \pi_i - \frac{ \pi_\ell
      \rho_{i\ell}}{t+\rho_{\ell\ell}} \rp
      \lp \pi_j - \frac{ \pi_\ell
      \rho_{j\ell}}{t+\rho_{\ell\ell}} \rp\\
    &= \hPi_i \hPi_j.
  \end{align*}
  Relation ~\eqref{eq:hrho1} follows directly from the definition
  ~\eqref{eq:hrhodefD}.
\end{proof}

\begin{lem}
  \label{lem:CDTD3}
  Let $T,\pi_i\rho_{ij},\ell,\ttau,\htau,\tT,\hT$ be as in
  ~\eqref{eq:htaudefA}.
  Then $\htau(x)$ is a polynomial  and 
  $T\rdt{1,\ell} \tT\rdt{1,\ell} \hT$ is a type D CDT with label
  transformations $\boxcirc\to\boxfsq\to \boxminus$, as described in
  Lemma ~\ref{lem:classCDT}.     Moreover, as in the class A case, the
  degree of $\htau$ is given by ~\eqref{eq:recdeghtauA}.
\end{lem}
\begin{proof}
  See the proof of Lemma ~\ref{lem:CDTA}.
\end{proof}
\begin{lem}
  \label{lem:CDTD4}
  Let $T,\pi_i\rho_{ij},\ell,\ttau,\tT$ be as in Lemma ~\ref{lem:CDTD}. Set
  \begin{align}
    \label{eq:htauD4}
    \htau(x) &= (1-x)^{-\al-1} \tau(x)(-\rho_{\ell\ell}(1)+ \rho_{\ell\ell}(x)),\\
    \hT &= \Trg(\htau;\al+2,\be).
  \end{align}
  Then $\htau(x)$ is a polynomial  and 
  $T\rdt{1,\ell} \tT\rdt{1,\ell} \hT$ is a type D CDT with label
  transformations $\boxcirc\to\boxfsq\to \boxplus$, as described in
  Lemma ~\ref{lem:classCDT}.      Moreover,
  \begin{equation}
    \label{eq:recdeghtauD4}
    \deg\htau = \deg\tau + 2\ell+\be.
  \end{equation}
\end{lem}
\begin{proof}
  Observe that $\rho_{\ell\ell}(x)- \rho_{\ell\ell}(1)$ is the unique
  anti-derivative $\int \pi_\ell(x)^2(1-x)^\al(1+x)^\be$ that vanishes
  at $x=1$. The rest of proof, mutatis mutandi, proceeds like the
  proof of Lemma ~\ref{lem:CDTA}.
\end{proof}
\begin{lem}
  \label{lem:hpiD3}
  Let $T,\pi_i,\rho_{ij},\ell,\hT$ be as in Lemma
  ~\ref{lem:CDTD}. Define $\hPi_i,\hRho_{ij}$ as in ~\eqref{eq:hpidefA}
  and ~\eqref{eq:hrhodefA}.  Then,
  $I_1(\hT) = (I_1(T)\setminus \{ \ell \})-1$, and ~\eqref{eq:hThpiA}
  holds.  The degrees and leading coefficients of these eigenfunctions
  are shown in ~\eqref{eq:degLChpiA}.  The $\hRho_{ij}(x)$ satisfy
  \begin{equation}
    \label{eq:hrhointD3}
    \begin{aligned}
      \hRho_{ij}(x)
      &= \int_{-1}^x \hPi_{i}(x) \hPi_{j}(x)(1-x)^\al
      (1+x)^{\be+2},\quad i,j\in I_1(T)\setminus \{ \ell \},\\
      \hRho_{ij}(1)
      &= \rho_{ij}(1).
    \end{aligned}
  \end{equation}
\end{lem}
\begin{proof}
  For the first assertion, see the proof of Lemma ~\ref{lem:hpiA}.  The
  second assertion follows from the definition ~\eqref{eq:hrhodefA} and
  because if $i\ne \ell$, then $\rho_{i\ell}(1) = 0$.
\end{proof}
\begin{lem}
  \label{lem:hpiD4}
  Let
  $T,\pi_i,\rho_{ij},\ell,\hT$ be as in Lemma ~\ref{lem:CDTD}. Set
  \begin{align}
    \label{eq:hpidefD4}
    \hPi_{i}(x)
    &=
      (x-1)^{-1}\lp\pi_{i}(x) -
      \frac{\pi_\ell(x)\rho_{i\ell}(x)}{-\rho_{\ell\ell}(1)+\rho_{\ell\ell}(x)}
      \pi_\ell(x)\rp,\quad i 
      \in I_1(T)\setminus \{ \ell \}\\
    \label{eq:hrhodefD4}
    \hRho_{ij}(x)
    &= \rho_{ij}-
      \frac{\rho_{i\ell}(x)\rho_{j\ell}(x)
      }{-\rho_{\ell\ell}(1)+\rho_{\ell\ell}(x)},\quad i,j\in
      I_1(T)\setminus\{\ell\} .
  \end{align}
  Then, $I_1(\hT) = (I_1(T)\setminus \{ \ell \})-1$, and
  \begin{equation}
    \label{eq:hThpiD4}
    \hT \hPi_{i+1} = \lam_1(i;\al+2,\be) \hPi_{i+1},\quad i\in    I_1(\hT). 
  \end{equation}
  The degrees and leading coefficients of these are  shown in
  ~\eqref{eq:degLChpiA}.  The $\hRho_{ij}(x)$ satisfy
  \begin{equation}
    \label{eq:hrhointD4}
    \begin{aligned}
      \hRho_{ij}(x)
      &= \int_{-1}^x \hPi_{i}(x) \hPi_{j}(x)(1-x)^{\al+2}
      (1+x)^{\be},\quad i,j\in I_1(\hT),\\
      \hRho_{ij}(1)
      &= \rho_{ij}(1).
    \end{aligned}
  \end{equation}
\end{lem}
\begin{proof}
  See the proof of Lemma ~\ref{lem:hpiA}.
\end{proof}

\begin{lem}
  \label{lem:CDTDq}
  The $\htau(x)=\htau(x;L,a,b)$, as defined in ~\eqref{eq:htauD}, is a
  polynomial with
  \begin{equation}
    \label{eq:deghtauD}
     \deg \htau = 2\!\!\sum_{\ell\in L_{134}}\!\!\ell  + q_1 -
     q_3(q_3-1)-q_4(q_4-1)+a(q_1+q_3)+b(q_1+q_4) 
  \end{equation}
  The $\hpi_i(x)=\hpi_{i}(x;L,a,b),\; i\in \hI_1$ defined in
  ~\eqref{eq:hpiD} are asymptotically monic, rational functions of
  degree $i$ that satisfy
  \[
    \Trg(\htau;\hal,\hbe)\hpi_i = \lam_1(i;\hal,\hbe) \hpi_i,\quad i\in
    \hI_1.
  \]
  where $\hal=a+2q_4,\hbe=b+2q_3$.
  The indefinite norms
  \[ \hrho_{ii}(x) = \int_{-1}^x \hpi_i(x)^2
    (1-x)^{\hal}(1+x)^{\hbe},\quad i\in \hI_1, \]
  define rational functions, that satisfy $\hrho_{ii}(-1)=0$ and
  \begin{equation}
    \label{eq:hrho1D}
    \hrho_{ii}(1) =
    \begin{cases}
          \hchi(i+q_3+q_4;L_1,L_3,L_4,a,b)^2 \nu(i+q_3+q_4;a,b) &
          \text{ if } i\in \hI_{1+}\\
          \hchi(i+q_3+q_4;L_1,L_3,L_4,a,b)^2 
          \frac{t_{\hn_i} \nu(\hn_i;a,b)}{t_{\hn_i}+\nu(\hn_i;a,b)}  &
          \text{ if } i\in \hI_{1-}\\
    \end{cases}
  \end{equation}
  with $\hn_i,\hchi$ as defined in ~\eqref{eq:InchiD}, and with
  $\hI_1 = \hI_{1-}\sqcup \hI_{1+}$ where
  \[ \hI_{1-} = L_1^*-q_3-q_4,\quad \hI_{1+} = (\Nz\setminus L_{134})-q_3-q_4.\]
\end{lem}
\begin{proof}
  let $T_0 := T(a,b)$, let $\hPi_{0,n}(x) = \pi_n(x;a,b),\; n\in \Nz$
  be the classical Jacobi polynomials and let
  \[ \hRho_{0;ij}(x) = \int_{-1}^x \pi_{0,i}(x)\pi_{0,j}(x) (1-x)^a
    (1+x)^b,\quad i,j\in \Nz \] be the classical incomplete inner
  products.  The corresponding indefinite norms $\hRho_{0;ii}(x)$ are
  polynomial by inspection.  Enumerate $L_{134}$ as described in
  Section ~\ref{sec:D}.  For $s=1,\ldots, q_1-1$, inductively define
  $\hPi_{s,i}(x; t_{\ell_1},\ldots, t_{\ell_s}),\; i\in \Nz$ using
  ~\eqref{eq:hpidefD} and
  $\hRho_{s,ij}(x;t_{\ell_1},\ldots, t_{\ell_s}),\; i,j\in \Nz$ using
  ~\eqref{eq:hrhodefD} by taking
  $\pi_i\to \hPi_{s-1,i}, \rho_{ij} \to \hRho_{s-1,ij}, t \to
  t_{\ell_s}$.  Define
  $\ttau_s(x), \htau_s(x), \tT_s, \hT_s\; s=1,\ldots q_1-1 $ using
  ~\eqref{eq:htaudefD} with $\al\to a, \be \to b$.  By Lemmas
  ~\ref{lem:CDTD} and ~\ref{lem:hpiD}, and by induction, the $\htau_s$
  are polynomials, the $\hP_{s,i}$ are rational eigenfunctions of
  $\hT_s$, and the $\hRho_{s,ij}$ the corresponding incomplete inner
  products.  The operators $\tT_s, \hT_s$ form a $2q_1$-step chain of
  type 1,2 CDTs
  \[ \hT_{s-1} \rdt{\lam_s} \tT_s \rdt{\lam_s} \hT_s,\quad \lam_s =
    \lam_1(\ell_s;a,b),\; s=1,\ldots, q_1.\]

  Next, let $\al_s= a,\; \be_s = b+2(s-q_1),\; s=q_1,\ldots q_1+q_3-1$
  and inductively define
  $\hPi_{s;i}(x;\bt_\bell), \hRho_{s;ij}(x;\bt_\bell),\; s=q_1,\ldots,
  q_1+q_3-1, \; i,j\in \Nz$ using ~\eqref{eq:hpidefA}
  ~\eqref{eq:hrhodefA} by taking
  $\pi_i \to \hPi_{s-1,i},\; \rho_{ij} \to \hRho_{s-1,ij}, \ell \to
  \ell_s$.  Then, use ~\eqref{eq:htaudefA} to define
  $\htau_s(x),\; s=q_1,\ldots, q_1+q_3-1$ with
  $\tau\to \htau_{s-1},\; \rho_{\ell\ell} \to
  \hRho_{s-1,\ell_s\ell_s},\; \beta\to \be_s$, and define the
  operators
  \begin{align*}
    \tT_s
    &=    \Trg(\ttau_s,\al_{s-1}+1,\be_{s-1}+1)+(s-q_1)(a+b+s-q_1+1),\\
    \hT_s
    &=    \Trg(\htau_s,\al_s,\be_s)+(s-q_1)(a+b+s-q_1+1)
  \end{align*}
  By Lemmas ~\ref{lem:CDTD3} and ~\ref{lem:hpiD3}, and by induction, the
  $\htau_s$ are polynomials, the $\hP_{s,i}$ are rational
  eigenfunctions of $\hT_s$, and the $\hRho_{s,ij}$ the corresponding
  incomplete inner products.  Moreover, the above CDT chain is
  extended by a $2q_3$ step chain of type 1,3 CDTs
  \[ \hT_{s-1} \rdt{\lam_s} \tT_s \rdt{\lam_s} \hT_s,\quad \lam_s =
    \lam_1(\ell_s;a,b),\; s=q_1+1,\ldots, q_1+q_3.\]
  Next, let $\al_s=a+2(s-q_1-q_3),\; \be=b+2q_3,\; s=q_1+q_3+1,\ldots,
  q_1+q_3+q_4$
  and inductively define
  $\hPi_{s;i}(x;\bt_\bell), \hRho_{s;ij}(x;\bt_\bell),\; s=q_1+q_3,\ldots,
  q_1+q_3+q_4-1, \; i,j\in \Nz$ using ~\eqref{eq:hpidefD4}
  ~\eqref{eq:hrhodefD4} by taking
  $\pi_i \to \hPi_{s-1,i},\; \rho_{ij} \to \hRho_{s-1,ij}, \ell \to
  \ell_s$.  Then, use ~\eqref{eq:htauD4} to define
  $\htau_s(x),\; s=q_1+q_3,\ldots, q_1+q_3+q_4-1$ with
  $\tau\to \htau_{s-1},\; \rho_{\ell\ell} \to
  \hRho_{s-1,\ell_s\ell_s},\; \beta\to \be_s$, and define the
  operators
  \begin{align*}
    \tT_s
    &=
      \Trg(\ttau_s,\al_{s-1}+1,\be_{s-1}+1)+(s-q_1-q_3)(a+b+s-q_1-q_3+1),\\
    T_s
    &=    \Trg(\tau_s,\al_s,\be_s)+(s-q_1-q_3)(a+b+s-q_1-q_3+1).
  \end{align*}
  By Lemmas ~\ref{lem:CDTD4} and ~\ref{lem:hpiD4}, and by induction, the
  $\htau_s$ are polynomials, the $\hP_{s,i}$ are rational
  eigenfunctions of $\hT_s$, and the $\hRho_{s,ij}$ the corresponding
  incomplete inner products.  Moreover, the  CDT chain is further
  extended by a $2q_4$ step chain of type 1,4 CDTs
  \[ \hT_{s-1} \rdt{\lam_s} \tT_s \rdt{\lam_s} \hT_s,\quad \lam_s =
    \lam_1(\ell_s;a,b),\; s=q_1+q_3+1,\ldots, q_1+q_3+q_4.\]
  We now adjust the indices and the normalizations to match the
  prescriptions of ~\eqref{eq:hpiD} and of Proposition ~\ref{prop:TD}.
  To do so, we set
  \begin{equation}
    \label{eq:hpiiD2}
    \hpi_i = \prod_{\ell\in L_{34}}
    \lp\frac{\hn_i-\ell^*}{\hn_i-\ell}\rp \, \hPi_{q_1+q_3+q_4,
      \hn_i},\quad i\in \hI_1, 
  \end{equation}
 where $\hI_1,n_i$ are as defined in
  ~\eqref{eq:InchiD}.  By ~\eqref{eq:hpiiD} ~\eqref{eq:hpiellD}, we also
  have $\deg \hPi_{s,i} = i$ for
  $i\notin \{ \ell_1,\ldots, \ell_{s}\}$ and $\deg \hPi_{s,i} = i^*$
  for $i \in \{ \ell_1,\ldots, \ell_s\}$ for this range of $s$.  By
  Lemmas ~\ref{lem:hpiD3} and ~\ref{lem:hpiD4} and by induction, it then
  follows that $\deg \hpi_i = i$ for all $i\in \hI_1$.

  By inspection of ~\eqref{eq:hpidefD}, and by induction, the
  quasi-polynomials $\hPi_{s,i},\; s=1,\ldots q_1,\; i\in \Nz$ are
  asymptotically monic.   Hence by the assertions regarding the
  leading coefficients found in these same two Lemmas and by the use
  of the normalization in ~\eqref{eq:hpiiD2} the same is true for the
  $\hpi_i$ defined in ~\eqref{eq:hpiiD2}.

  The definition of $\htau_s(x)$ in ~\eqref{eq:htaudefD}
  ~\eqref{eq:htauA} ~\eqref{eq:htauD4} is the 1-step case of the
  definitions in ~\eqref{eq:htauD} with $t$, respectively, the
  indeterminate $t_{\ell_s}$, $t=0$ and $t=-\rho_{\ell\ell}(x)$.  Likewise, the
  definition of $\hPi_{s,i}$ in ~\eqref{eq:hpidefD} ~\eqref{eq:hpidefA}
  ~\eqref{eq:hpidefD4} are $2\times 2 $ cases of the determinantal
  formula ~\eqref{eq:hpiD} with $t$ as above.  The general
  determinantal formulas follow by Sylvester's identity; see Section 3
  of \cite{GFGM21} for details of the argument.

  Finally, relation ~\eqref{eq:hrho1D} follows by induction from
  ~\eqref{eq:hrhointD} ~\eqref{eq:hrhointD3} ~\eqref{eq:hrhointD4} and
  from the application of the normalization shown in
  ~\eqref{eq:hpiiD2}.
\end{proof}

\begin{lem}
  \label{lem:hnuD}
  Let $\htau,\, \hal,\hbe,\,\hpi_i,\; i\in \hI_1$ be as in
  Lemma ~\ref{lem:CDTDq}. The corresponding norms are given by
  \begin{equation}
    \label{eq:hnuD}
    \begin{aligned}
      \hnu_i
      &= \hkappa(\hn_i) \nu(\hn_i;a,b)\quad \text{where }\\
      \hkappa(z)
      &= \prod_{\ell\in L_1} \lp 1-
    \delta_{z,\ell^*}\frac{\tnu_\ell}{t_\ell+\tnu_\ell} \rp 
      \prod_{\ell\in L_{34}}\!\!
      \lp\frac {z-\ell^*}{z-\ell}\rp^2    ,\quad \tnu_\ell = \nu(\ell;a,b).
    \end{aligned}
  \end{equation}
\end{lem}
\begin{proof}
  This follows by Lemma ~\ref{lem:Dnorm} and by
   ~\eqref{eq:hrho1D}.
\end{proof}
\begin{proof}[Proof of Proposition ~\ref{prop:TD}]
  The subcase where $p=0$ is covered by Lemma ~\ref{lem:CDTDq}.  The
  formula for $\tau$ in ~\eqref{eq:tauD} follows by Lemma
  ~\ref{lem:Wronskpi}.  The same Lemma justifies the degree formula
  ~\eqref{eq:degtauD}.  The formula for the quasi-polynomial
  eigenfunctions given in ~\eqref{eq:piD} is justified by Proposition
  ~\ref{prop:Awronsk}.  The form of the normalization constant $\chi$
  follows by Lemma ~\ref{lem:lc}.  The norm formula ~\eqref{eq:nuD}
  ~\eqref{eq:kappaD}  follows from
  ~\eqref{eq:hnuD} of Lemma ~\ref{lem:hnuD} and by ~\eqref{eq:RDT1norm}
  of Lemma ~\ref{lem:JRDTnorm}.
\end{proof}

\begin{proof}[Proof of Proposition ~\ref{prop:DSD}]
  By Theorem ~\ref{thm:XOPFT2}, by Propositions ~\ref{prop:DCD2}
  ~\ref{prop:DCD34-} and by Proposition ~\ref{prop:lsfa}, the given $T$
  is Darboux-connected to a classical $T(a,b)$, where
  $(a,b) \in \{ (0,0),(1,0),(0,1)\}$.  The first possibility holds if
  $\al=\be$; the second possibility holds if $\al >\be$; and the last
  possibility holds if $\al< \be$.  Hence, by Proposition
  ~\ref{prop:lsfa}, $\Lam_T$ is a finite modification of such a
  $\Lam(a,b)$. The construction and uniqueness of the parameters
  $K,L_1,L_3,L_4,a,b$ is spelled out in Remark ~\ref{rem:D}.  The form
  of $\tau$ then follows by the argument detailed about in the proof
  of Proposition ~\ref{prop:TD}.
\end{proof}

\subsection{Proof of Proposition ~\ref{prop:kappaT} and Theorems ~\ref{thm:sd} and  ~\ref{thm:esd}}
\label{subsec:proofs:main}

\begin{proof}[Proof of Proposition ~\ref{prop:kappaT}]
  By assumption, the exceptional operator in question belongs to one
  of the ABCD classes.   The desired conclusion now follows by the
  Propositions stated in section ~\ref{sec:GABCD}.  For example, for a
  class G operator, Proposition ~\ref{prop:GSD} asserts that the $T$ in
  question has the form described in Proposition ~\ref{prop:TG}.  That
  same Proposition asserts that ~\eqref{eq:sfdef} holds for such
  operators.  The proof is given earlier in this section.  The same
  reasoning and justification applies to the operators in the other
  degeneracy classes.
\end{proof}

\begin{proof}[Proof of Theorem ~\ref{thm:sd}]
  The theorem is a direct consequence of the Propositions in Section
  ~\ref{sec:GABCD}.  For the G class, the theorem follows from
  Propositions ~\ref{prop:TG} and ~\ref{prop:GSD}.  For the A class, the
  theorem follows by Propositions ~\ref{prop:TA} and ~\ref{prop:ASD}.
  For the B class, the theorem follows by Propositions ~\ref{prop:TB}
  and ~\ref{prop:BSD}.  For the C and CB classes, the theorem follows
  by Propositions ~\ref{prop:TCB} ~\ref{prop:CSD} and ~\ref{prop:CBSD}.
\end{proof}

\begin{proof}[Proof of Theorem ~\ref{thm:esd}]
  By the Propositions detailed in the preceding proof, the norms of
  the quasi-polynomials of GABC operators are fully determined in
  terms of their degree.  Therefore, an isospectral deformation cannot
  exists for such operators.  For the D class, the desired conclusion
  is a consequence of Propositions ~\ref{prop:TD} and ~\ref{prop:DSD}.
  The one-to-one correspondence between extended spectral diagrams and
  isospectral deformations is also a consequence of Propositions
  ~\ref{prop:TD} and ~\ref{prop:DSD}.  By ~\eqref{eq:hnuD} of Lemma
  ~\ref{lem:hnuD}, for $i\in I_{1-}$ we have
  \[ \nu_i = \frac{t_\ell}{t_\ell+\tnu_\ell}\times \text{const.}
    ,\quad\text{where } \ell = i^\star+p+q_3+q_4 .\] Thus, the mapping
  $t_\ell \mapsto \nu_i$ is one-to-one.  The conclusion now follows
  because the support of the isospectral deformation is a finite set.
\end{proof}
\appendix
\subsection*{Acknowledgements}
MAGF has been partially supported by the grants CEX2023-001347-S, PID2021-125021NAI00, PID2021-124195NB-C32, PID2021-122154NB-I00 and PID2021-122156NB-I00, funded by MCIN/AEI/10.13039/501100011033.
DGU has been partially supported by the grants PID2021-122154NB-I00 and TED2021-129455B-I00 funded by ``ERDF A Way of making Europe'' via MICIU/AEI/10.13039/501100011033. He also acknoledges support from the BBVA Foundation under its program ``Ayudas a Proyectos de Investigación Científica'' in the area of Mathematics.
RM would like to acknowledge support from the grant  PID2021-122154NB-I00 and MITACS Accelerate grant IT26380.



\end{document}